\documentclass{article}
\usepackage[utf8]{inputenc}
\usepackage{authblk}
\usepackage[margin=1in]{geometry}
\usepackage[titletoc,title]{appendix}
\usepackage[dvipsnames,table,xcdraw]{xcolor}
\usepackage{color}
\usepackage{setspace}

\usepackage{amsmath,amsfonts,amssymb,amsthm,mathtools}
\usepackage{algorithm2e}
\RestyleAlgo{ruled}

\usepackage{graphicx,float}
\usepackage{subcaption}  
\usepackage{booktabs}
\usepackage{booktabs,multirow}

\usepackage[colorlinks=true, pdfborder={0 0 0}, linkcolor=red]{hyperref}

\usepackage[style=authoryear,date=year,maxcitenames=2,maxbibnames=20,giveninits,url=false,backend=biber,uniquename=false,uniquelist=false]{biblatex}
\bibliography{revised_main}
\usepackage{wasysym}
\usepackage{lipsum}

\allowdisplaybreaks
\title{\textbf{$\ell^2$ Inference for Change Points in High-Dimensional Time Series via a Two-Way MOSUM}}
\author{Jiaqi Li\thanks{Department of Mathematics and Statistics, Washington University in St. Louis; Email: lijiaqi@wustl.edu}\qquad Likai Chen\thanks{Department of Mathematics and Statistics, Washington University in St. Louis; Email: likai.chen@wustl.edu}\qquad Weining Wang\thanks{Department of Economics and Related Studies, University of York; Email: weining.wang@york.ac.uk}\qquad Wei Biao Wu\thanks{Department of Statistics, University of Chicago; Email: wbwu@galton.uchicago.edu}}
\date{}

\begin{document}

\maketitle

\begin{abstract}
We propose an inference method for detecting multiple change points in high-dimensional time series, targeting dense or spatially clustered signals. Our method aggregates moving sum (MOSUM) statistics cross-sectionally by an $\ell^2$-norm and maximizes them over time. We further introduce a novel Two-Way MOSUM, which utilizes spatial-temporal moving regions to search for breaks, with the added advantage of enhancing testing power when breaks occur in only a few groups. The limiting distribution of an $\ell^2$-aggregated statistic is established for testing break existence by extending a high-dimensional Gaussian approximation theorem to spatial-temporal non-stationary processes. Simulation studies exhibit promising performance of our test in detecting non-sparse weak signals. Two applications, analyzing equity returns and COVID-19 cases in the United States, showcase the real-world relevance of our proposed algorithms.

\vspace{0.9cm}

\end{abstract}

\setlength{\parindent}{15pt}

\let\altH\H

\def\One{\mathbf{1}}

\def\III{\text I}

\def\A{\mathcal A}
\def\B{\mathcal B}
\def\C{\mathcal C}
\def\D{\mathcal D}
\def\E{\mathcal E}
\def\F{\mathcal F}
\def\G{\mathcal G}
\def\H{\mathcal H}
\def\I{\mathcal I}
\def\K{\mathcal K}
\def\L{\mathcal L}
\def\M{\mathcal M}
\def\N{\mathcal N}
\def\P{\mathcal P}
\def\Q{\mathcal Q}
\def\S{\mathcal S}
\def\T{\mathcal T}
\def\U{\mathcal U}
\def\V{\mathcal V}
\def\W{\mathcal W}
\def\Z{\mathcal Z}

\def\BB{\mathbb B}
\def\EE{\mathbb E}
\def\II{\mathbb I}
\def\NN{\mathbb N}
\def\PP{\mathbb P}
\def\RR{\mathbb R}
\def\ZZ{\mathbb Z}

\def\ss{\mathbf s}

\def\bbell{\boldsymbol \ell}

\newtheorem{definition}{Definition}
\newtheorem{assumption}{Assumption}
\newtheorem{theorem}{Theorem}
\newtheorem{proposition}{Proposition}
\newtheorem{corollary}{Corollary}
\newtheorem{lemma}{Lemma}
\newtheorem{remark}{Remark}
\newtheorem{example}{Example}

\newpage

\section{Introduction}\label{sec_intro}
Change-point analysis is a fundamental problem in various fields of applications: in economics, the break effects of policy are of particular interest (\cite{chen_dynamic_2021}); in biology, high-amplitude co-fluctuations are utilized in cortical activity to represent dynamics of brain functional connectivity (\cite{faskowitz_edge-centric_2020,zamani_esfahlani_high-amplitude_nodate}); in network analysis, change-point detection can be employed for the anomaly of network traffic data caused by attacks (\cite{levy-leduc_detection_2009}), etc. The above list of scenarios spans a wide range of data structures, including high-dimensional data with temporal and cross-sectional dependence, which pose substantial challenges to change-point analysis. The paper aims to address this issue by providing theory on multiple break inference for high-dimensional time series allowing both temporal and spatial dependence.

There is a sizable literature on high-dimensional change-point detection. Various studies consider data aggregation, and many of them consider $\ell^{\infty}$-based {methods}. See, for example, \textcite{shao_self-normalized_2010,jirak2015uniform,chen_inference_2022,yu_finite_2021}. Most aforementioned studies focus on sparse signals, while an $\ell^2$-based approach favors non-sparse weak signals, and this is also the focus of this study. In the meanwhile, the $\ell^2$-type aggregation is quite common in the literature. \textcite{bai_common_2010} evaluates the performance of a least square estimation and establishes a distribution theory for single change-point estimator in panel data without cross-sectional dependence; \textcite{zhang_detecting_2010} develops a recursive algorithm based on {sums} of chi-squared statistics across samples with independent and identically distributed (i.i.d.) Gaussian noises, which could be viewed as an extension of the circular binary segmentation algorithm by \textcite{olshen_circular_2004}. In addition, \textcite{horvath2012change,chan2013darling,horvath2017asymptotic,bai2020estimation,liu2022change} study $\ell^2$-based cumulative sum (CUSUM) statistics to estimate and make inference for change points in linear regression or panel models. Although their methods primarily concentrate on scenarios with a single break per time series, their approaches bear similarities to ours, with the key difference being our utilization of a moving sum (MOSUM) variant. More recently, \textcite{enikeeva_high-dimensional_2019} proposes a linear and a scan CUSUM statistic with the minimax bound established for the change-point estimator of i.i.d. Gaussian data; \textcite{chen_high-dimensional_2020} introduces a coordinate-wise likelihood ratio test for online change-point detection for independent Gaussian data and present the response delay rate. In addition to aggregation, there are other well-known techniques for high-dimensional change-point analysis, including the U-statistics as demonstrated by \textcite{wang2020hypothesis, wang2020dating,yu_finite_2021}, threshold-based approaches proposed by \textcite{cho2015multiple,cho_change-point_2016}, and a projection-based method developed by \textcite{wang_samworth_2018}. In this paper, we consider a maximized $\ell^{2}$-type test statistic to adapt to different datasets containing signals of distinct temporal-spatial properties and errors with complex dependency structures.

Besides the challenge of change-point test brought by high-dimensionality, studies on multiple change-point detection have a long-standing tradition. In general, two broad classes of methods have been developed: model selection and hypothesis testing. Model selection approaches aim to treat change-point signals as parameters and derive estimators for them, such as the PELT algorithm (\cite{killick_optimal_2012}) and the fused LASSO penalty (\cite{tibshirani_spatial_2008,li_panel_2016,lee_lasso_2016}). \textcite{cho2022two} proposes a localised application of the Schwarz criterion for multiscale change points. \textcite{kuchibhotla2021uniform,xu2022change,wang2022optimal} consider change-point analysis for linear regression models featuring varying parameters, encompassing a broad range of nonlinear time series. As for testing, a traditional approach is binary segmentation developed by \textcite{scott_cluster_1974}. Its variants are considered in \textcite{bai_estimating_1998,olshen_circular_2004}. Moreover, \textcite{fryzlewicz2014wild} introduces a wild binary segmentation and \textcite{cho2015multiple} proposes a sparsified binary segmentation algorithm.
\textcite{yu2020review} reviews diverse minimax rates in change-point analysis literature.

In the context of testing, MOSUM is a notably popular technique for both univariate and multivariate time series, such as \textcite{huskova_permutation_2001} on i.i.d. data, \textcite{wu_inference_2007,eichinger_mosum_2018} on temporal dependent data, and \textcite{kirch_moving_2021} on multivariate time-continuous stochastic processes. MOSUM is attractive due to the simplicity of implementation and an overall control of significance level which avoids issues in multiple testing. However, for high-dimensional time series, when a MOSUM statistic aggregates all the series by an $\ell^2$-norm, the testing power would suffer if breaks only occur in a {portion} of them. Hence, in this paper, we propose a novel spatial-temporal moving sum algorithm called Two-Way MOSUM. This method utilizes moving spatial-temporal  regions to search for temporal breaks and locate spatial neighborhoods where temporal breaks occur. Such moving regions can be viewed as a generalized concept of the moving windows in previous MOSUM, which can aggregate signals adaptive to cross-sectional group structures to enhance testing power. We emphasize that overlapping groups are allowed in our method, and therefore the prior knowledge of groups is not a requirement for effective detection of breaks, since one can always search for all possible grouping scenarios. Nevertheless, prior grouping information is available in numerous data applications, which can boost the testing power by decreasing the number of searching windows. See, for example, in neuroscience, regions of interest (ROI) in human brains can be assigned to networks by different functions and the ROIs from the same network will undergo simultaneous functional change points (\cite{barnett_change_2016}); in finance, stock prices of industries are often grouped by market capitalization and a few number of sectors may experience market shocks at the same time (\cite{onnela_dynamics_2003}). Note that all the theoretical comparison of testing power in this paper is only among MOSUM-based statistics.

Although both $\ell^2$ aggregation and MOSUM statistics have been well investigated respectively, it is quite challenging to rigorously develop an inference theory for $\ell^2$-based MOSUM statistics to detect the existence of breaks for the high-dimensional data. To be more specific, when we take the maximum of $\ell^2$ statistics obtained from all the rolling windows over time, these aggregated statistics are temporally dependent even though the underlying errors may be independent. Most of the previous works concerning $\ell^2$-based statistics only provide inference for the break estimators by assuming the existence of a break, such as an $\ell^2$-type break location estimator and its inference introduced by \textcite{bai_common_2010} for single change-point estimation with cross-sectionally independent errors.
To the best of our knowledge, this study is the first to establish the limiting distribution of an $\ell^2$-type test statistic to facilitate the inference for change-point detection, which allows both spatial and temporal dependence.

To summarize, we contribute to the literature in both theory and algorithms. On the theory front, we propose an $\ell^2$-type MOSUM test statistic for multiple break detection in high-dimensional time series, allowing both temporal and spatial dependence. The Gaussian approximation (GA) result under the null is provided as a theoretical foundation to backup our detection of breaks (cf. Theorems \ref{thm1_constanttrend} \& \ref{thm3_nonli}). Correspondingly, we introduce an innovative Two-Way MOSUM statistic to account for spatially-clustered signals (cf. Theorem \ref{thm2_constanttrend}). Consistency results of estimators for number of breaks, temporal and spatial break locations, as well as break sizes are all established (cf. Theorem \ref{thm_consistency} \& Proposition \ref{prop_consistency}).

\textit{Roadmap.}
The rest of this article is organized as follows. Section \ref{sec_test} is devoted to the test specification and asymptotic properties with cross-sectional independence assumed. 
Section \ref{sec_test_spatial} serves as an extension to the cases where breaks might exist only in a subset of component series (clustered signals).
We follow with Section \ref{sec_test_nonlinear} as a generalization to nonlinear time series with spatial space in $\ZZ^v$, allowing both temporal and cross-sectional dependence.
In Section \ref{sec_data}, we deliver two empirical applications on testing structural breaks for the stock return and COVID-19 data.
The simulation studies and proofs are deferred to Appendix.

\textit{Notation.}
For a vector $v=(v_1,\ldots,v_d)\in\RR^d$ and $q>0$, we denote $|v|_q=(\sum_{i=1}^d|v_i|^q)^{1/q}$ and $|v|_{\infty}=\max_{1\le i\le d}|v_i|$. 
% For a matrix $A=(a_{i,j})_{1\le i\le m, 1\le j\le n}$, we define the max norm $|A|_{\text{max}}=\max_{i,j}|a_{i,j}|$. 
For $s>0$ and a random vector $X$, we say $X\in\L^s$ if $\lVert X\rVert_s=[\EE(|X|_2^s)]^{1/s}<\infty$, and denote $\E_0(X)=X-\EE(X)$. For two positive number sequences $(a_n)$ and $(b_n)$, we say $a_n=O(b_n)$ or $a_n\lesssim b_n$ (resp. $a_n\asymp b_n$) if there exists $C>0$ such that $a_n/b_n\le C$ (resp. $1/C\le a_n/b_n\le C$) for all large $n$, and 
say $a_n=o(b_n)$ or $a_n\gg b_n$ if $a_n/b_n\rightarrow0$ as $n\rightarrow\infty$. Let $(X_n)$ and $(Y_n)$ to be two sequences of random variables. Write $X_n=O_{\PP}(Y_n)$ if for $\forall \epsilon>0$, there exists $C>0$ such that $\PP(|X_n/Y_n|\le C)>1-\epsilon$ for all large $n$, and say $X_n=o_{\PP}(Y_n)$ if $X_n/Y_n\rightarrow 0$ in probability as $n\rightarrow\infty$.

\section{Testing and Estimating High-Dimensional Change Points}\label{sec_test}

In this section, we propose a test statistic based on an $\ell^2$ aggregated MOSUM and investigate its theoretical properties to test the presence of structural breaks. To formulate our model, let $Y_1,\ldots,Y_n$ be observed $p$-dimensional random vectors satisfying 
\begin{equation}
    \label{eq_model}
    Y_t=\mu(t/n)+\epsilon_t, \qquad t=1,\ldots,n,
\end{equation}
where $(\epsilon_t)_t$ is a sequence of $p$-dimensional stationary errors with zero-mean and $\mu(\cdot)$ is a $p$-dimensional vector of unknown trend functions. Our main interest is to detect the potential change points occurring on the trend function
\begin{equation}
    \label{eq_trend}
    \mu(u)=\mu_0+\sum^{K }_{k=1}\gamma_k\One_{u\ge u_k},
\end{equation}
where $K \in\NN$ is the number of structural breaks which is unknown and could go to infinity as $n$ increases; $u_1,\ldots,u_{K }$ are the time stamps of the breaks with $0=u_0<u_1<\ldots<u_{K }<u_{K +1}=1$ and the minimum gap $\kappa_n=\min_{0\le k\le K }(u_{k+1}-u_k)$, where $\kappa_n>0$ is allowed to tend to 0 as $n\rightarrow\infty$; $\mu_0\in\RR^p$ represents the benchmark level when no break occurs and $\gamma_k\in\RR^p$ is the jump vector at the time stamp $u_k$ with size $|\gamma_k|_2$.

It should be noted that not all entries of $\gamma_k$ need to be nonzero, which allows for cases where only a subset of time series experience a jump at the time stamp $u_k$. In such situations, it is preferable to aggregate only the series with breaks rather than all of them, as it can improve the testing power. A more detailed discussion of this scenario is given in Section \ref{sec_test_spatial} where the Two-Way MOSUM method is introduced. In this section, for readability, we focus solely on the improved MOSUM that aggregates all time series. For brevity, we assume the time series to be linear and cross-sectionally independent throughout Sections \ref{sec_test} and \ref{sec_test_spatial}, which will be relaxed to nonlinear and cross-sectionally dependent cases in Section \ref{sec_test_nonlinear}.

\subsection{$\ell^2$-Based Test Statistics}
This subsection is devoted to test the null hypothesis:
$$\H_0: \quad \gamma_1=\gamma_2=\ldots=\gamma_{K }=0,$$
which denotes the case with no breaks, against the alternative $\H_A$: there exists $k\in\{1,\ldots,K \}$, such that $\gamma_k\neq0$. Note that the number of breaks $K $ is allowed to go to infinity under some condition on the separation of breaks. We refer to a detailed discussion below Definition \ref{def_separation} in Section \ref{sec_GA}.

The primary reason for testing the presence of structural breaks is to prevent model misspecification. If we apply a change-point algorithm to a data generating process without any actual breaks, we may obtain false break estimates, leading to erroneous conclusions. Therefore, it is {necessary} to test for the existence of breaks before conducting further analysis. However, previous studies on change points in high-dimensional time series mostly focus on inference for break location estimators, such as Theorem 2.2 in \textcite{horvath2017asymptotic}, which assumes the existence of breaks. Although there are some available literature on change-point testing for high-dimensional time series (see \textcite{jirak2015uniform,chen_inference_2022,wang2022inference}), there is no existing theory of $\ell^2$-based statistics for testing break existence, which can be {necessary and} highly beneficial in identifying dense signals.

We define a jump vector $J(u)$ at the time point $u$ as $J(u)=0$ when no break occurs at the time stamp $u$, and 
$J(u)=\gamma_i$ when $u=u_i$ for some $1\le i\le K $. Intuitively, we can test the existence of breaks by evaluating the jump estimate
$\hat J(i/n)=\hat{\mu}^{(l)}_i-\hat{\mu}^{(r)}_i$, where 
\begin{equation}
    \label{eq_localest}
    \hat{\mu}^{(l)}_i=\hat{\mu}^{(l)}(i/n)=\frac{1}{bn}\sum_{t=i-bn}^{i-1}Y_t,\quad \hat{\mu}^{(r)}_i=\hat{\mu}^{(r)}(i/n)=\frac{1}{bn}\sum_{t=i}^{i+bn-1}Y_t,
\end{equation}
are the local averages on the left and {the} right hand sides of the time point $i/n$, respectively, and $b$ is a bandwidth parameter satisfying $b\rightarrow 0$ and $bn\rightarrow\infty$. {Without loss of generality, we assume that $bn$ is an integer.} We shall reject the null if $|\hat J(i/n)|_2$ is too large. To this end, we shall develop an asymptotic distributional theory which appears to be highly nontrivial.

Throughout Sections \ref{sec_test} and \ref{sec_test_spatial}, we assume that there is no dependence between component processes $(\epsilon_{t,j})_{t\in\ZZ}$, $1\le j\le p$, and we later relax this restriction to allow for weak cross-sectional dependence in Section \ref{sec_test_nonlinear}. To make all $p$ components in $\hat J(i/n)$ comparable, we need to specify the error process $(\epsilon_t)_{t\in\ZZ}$ and obtain the long-run variances for standardization. In particular, we model $\epsilon_t$ as a {vector moving average (VMA) process} in Sections \ref{sec_test} and \ref{sec_test_spatial}, which embraces many important
time series models such as a vector autoregressive moving averages (VARMA) model, {and we discuss a more general $\epsilon_t$ in Section \ref{sec_test_nonlinear} using functional dependence measure, which allows nonlinear forms.} Let 
\begin{equation}
    \label{eq_epsilon_linear}
    \epsilon_t=\sum_{k\ge 0}A_k\eta_{t-k},
\end{equation}
where $\eta_t\in\RR^{\tilde p}$ are i.i.d. random vectors with zero mean and an identity covariance matrix with $p\le \tilde p\le c p$, for some constant $c>1$. The coefficient matrices $A_k$, $k\ge 0$, take values in $\RR^{p\times \tilde p}$ such that $\epsilon_t$ is a proper random vector. {Define $\tilde A_0=\sum_{k\ge0}A_k$.} Then the long-run covariance matrix of $\epsilon_t$ and the diagonal matrix of the long-run standard deviations are
\begin{equation}
    \label{eq_longrunCov}
    \Sigma = \tilde A_0\tilde A_0^{\top}=(\sigma_{i,j})_{i,j=1}^p, \quad \Lambda = \text{diag}(\sigma_{1},\sigma_{2},\ldots,\sigma_{p}),
\end{equation}
% We further define the diagonal matrix $\Lambda$ as
% \begin{equation}
%     \label{eq_longrunCov_diag}
%     \Lambda = \text{diag}(\sigma_{1},\sigma_{2},\ldots,\sigma_{p}),
% \end{equation}
respectively, where $\sigma_j^2=\sigma_{j,j}\ge c_{\sigma}$, for some constant $c_{\sigma}>0$, is the long-run variance of the $j$-th component series. Note that in Sections \ref{sec_test} and \ref{sec_test_spatial}, we assume $p=\tilde p$, and all $A_k$, $k\ge0$, are diagonal matrices, which indicates the independence between the component processes $(\epsilon_{t,j})_{t\in\ZZ}$.
We relax this assumption to weak cross-sectional dependence for a general $\epsilon_t$ in Section \ref{sec_test_nonlinear}, and also provide the results for the same linear $\epsilon_t$ in (\ref{eq_epsilon_linear}) as a special case in Appendix \ref{subsec_linear_sec_dep}.

Following the previous intuition, we test the existence of breaks by evaluating the gap vectors $\hat J(\cdot)$. Namely, we standardize $\hat J(\cdot)$ by the long-run standard deviations of each time series, that is, for $bn+1\le i\le n-bn$,
\begin{align}
    \label{eq:defvi}
    V_i=\Lambda^{-1}\hat J(i/n)=\Lambda^{-1}(\hat\mu^{(l)}_i-\hat\mu^{(r)}_i).    
\end{align}
To conduct the change-point detection with $p\rightarrow\infty$, we take the $\ell^2$ aggregation of each $V_i$ in the cross-sectional dimension, i.e. $|V_i|_2^2$, to capture dense signals. Note that by model (\ref{eq_model}), the random vector $V_i$ involves both the signal part $\EE V_i$ and the error part $V_i-\EE V_i$. We define
\begin{equation}
    \label{eq_thm1_c_def}
    \bar{c}=\sum_{j=1}^pc_j, \quad \text{where }c_j= \text{Var}(V_{i,j}),
\end{equation}
and $V_{i,j}\in\RR$ is the $j$-th coordinate of $V_i$. Since no break exists under the null hypothesis, i.e. $\EE V_i=0$, it follows that $|V_i|_2^2-\bar{c}$ is a centered statistic under the null. The detailed evaluation of $\bar{c}$ is deferred to Remark \ref{rmk_centering}. Finally, we move the windows in the temporal direction to find the maximum and formulate our $\ell^2$-based test statistic as follows:
\begin{equation}
    \label{eq_teststats}
    \Q_n=\max_{bn+1\le i\le n-bn}\big(|V_i|_2^2-\bar{c}\big).
\end{equation}
We consider $\Q_n$ as a feasible test statistic by assuming that the long-run standard deviation $\Lambda$ is known. The estimated long-run variances via a robust M-estimation method proposed by \textcite{chen_inference_2022} 
% (Section 4, Pages 10-12) 
are utilized in Section \ref{sec_data} for applications, and the details are deferred to Appendix \ref{subsec_longrun}.
% and estimating the long-run variance covariance matrix is not the focus of this study.

It is worth noticing that when breaks are sparse in the cross-sectional components, an $\ell^{\infty}$-type statistic, i.e.,
$\Q_{n,\infty}=\max_{bn+1\le i\le n-bn}|V_i|_{\infty}$, could be more powerful (\cite{chen_inference_2022}) than an $\ell^2$-based one. However, in the presence of weak dense signals, the $\ell^{\infty}$ {test} would have power loss (\cite{chen_monitoring_2022}) while the $\ell^2$-type statistic can boost the power {due to the aggregation of weak signals}; see Remark \ref{11} for a simple comparison. The current study targets change-point detection with non-sparse or spatially clustered signals. An $\ell^2$-based test statistic $\Q_n$ is therefore being proposed.

\begin{remark}[Comparison of $\ell^2$ and $\ell^{\infty}$ statistics]\label{11}
Here we present a simulation study to intuitively show that an $\ell^2$-based test statistic is generally more powerful in detecting weak dense signals compared to an $\ell^{\infty}$ one, while in the case of sparse signals, an $\ell^{\infty}$ type statistic appears better. Specifically, we perform a single change-point detection using test statistics $\Q_n$ and $\Q_{n,\infty}$ respectively, and compare their testing powers with varied proportions of cross-sectional dimensions containing jumps. The errors are generated from MA($\infty$) models defined in (\ref{eq_epsilon_linear}) with $\eta_t\sim t_9$ and the sample size $n=100$. We consider $p=50, 100$ and the window size $bn=20$. We set the coefficient matrix $A_k$ in (\ref{eq_epsilon_linear}) to be $A_k=\text{diag}(a_1k^{-3/2},a_2k^{-3/2},\ldots,a_pk^{-3/2})$, where $a_1,\ldots,a_p$ are uniformly ranging from $0.5$ to $0.9$. For each time series with a break, the jump size is $1$. All the reported powers in Figure \ref{fig_sim_compare} are averaged over $1000$ replicates. 
Intuitively, our test statistic $\Q_n$ incorporates the term $|V_i|_2$, which aggregates dense signals in a linear fashion with respect to the number of components $p$, while the standard deviation of weakly dependent random noises is aggregated on the order of $\sqrt{p}$. As a result, an $\ell^2$-type statistic is better suited for identifying dense signals, while its performance for sparse signals may be inferior since aggregating sparse signals will only add noise without any significant signal. After introducing Theorem \ref{thm1_constanttrend}, we shall provide a theoretical power comparison (cf. Remark \ref{remark_linf}). 
\end{remark}

\begin{figure}[!htbp]
\centering
\includegraphics[width=1\linewidth]{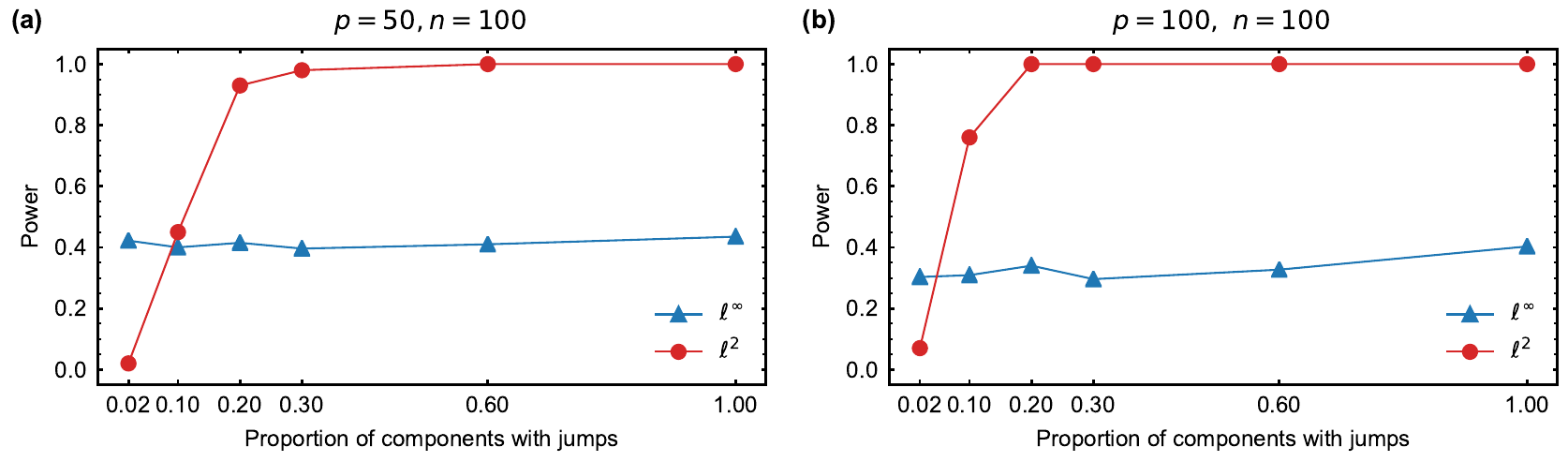}  
\caption{Power comparison of $\ell^{\infty}$ MOSUM and $\ell^2$ MOSUM.}
\label{fig_sim_compare}
\end{figure}

\subsection{Asymptotic Properties of Test Statistics}

To conduct the test, it is essential to understand the asymptotic behavior of the test statistic $\Q_n$. However, deriving the limiting distribution of $\Q_n$ under the null is highly challenging. This is because, even if the underlying errors are i.i.d., the standardized jump estimator $V_i$ defined in \eqref{eq:defvi} is still dependent over $i$ due to the overlapped observations among different moving windows. In this section, we provide an intuition for the theoretical proofs of our first main theorem, which extends the high-dimensional GA for dependent data. 

First, recall that the random vector $V_i$ can be decomposed into the expectation $\EE V_i$ and the deviation part $V_i-\EE V_i$. We have $\EE V_i=0$ for any $i$ under the null hypothesis. Let
\begin{equation}
    \label{eq_thm1_x_def}
    x_{i,j}=(V_{i,j}-\EE V_{i,j})^2-c_j \quad {\text{and}\quad X_j=(x_{bn+1,j},\ldots,x_{n-bn,j})^{\top},}
\end{equation}
where $c_j$ is defined in (\ref{eq_thm1_c_def}). By (\ref{eq:defvi}), we can write the test statistic $\Q_n$ under the null into
\begin{equation}
    \label{eq_thm1_GAform}
    \Q_n =\max_{bn+1\le i\le n-bn} \sum_{j=1}^p x_{i,j}.
\end{equation}
When the errors are cross-sectionally independent, $X_1, \ldots,X_p$ are also independent. Therefore, as $p$ goes to infinity, we can apply the high-dimensional GA theorem to (\ref{eq_thm1_GAform}) to derive the asymptotic distribution of $\Q_n$. In this way, the temporal dependence caused by the overlapped moving windows can be properly dealt with. 
We generalize this result in Section \ref{sec_test_nonlinear} with cross-sectionally dependence allowed between $X_1,\ldots,X_p$.

We introduce the centered Gaussian random vector $\Z=(\Z_{bn+1},\ldots,\Z_{n-bn})^{\top}\in\RR^{n-2bn}$ with covariance matrix $\Xi=\EE(\Z\Z^{\top})\in\RR^{(n-2bn)\times (n-2bn)}$, and denote the $i$-th element in $\Z$ by $\Z_i$. Here $\Xi=(\Xi_{i,i'})_{1\le i,i'\le n-2bn}$ with expression
\begin{equation}
    \label{eq_cov}
    \Xi_{i,i'}= p(bn)^{-2}g\big(|i-i'|/(bn)\big),
\end{equation}
where the function {$g(\cdot):[0,\infty)\mapsto \RR$} is defined as
\begin{equation}
    \label{eq_def_g}
    g(\zeta)=\begin{cases}
    18\zeta^2-24\zeta +8,\quad &0\le \zeta <1, \\
    2\zeta^2-8\zeta +8, \quad &1\le \zeta <2, \\
    0, \quad &\zeta \ge 2.
    \end{cases}
\end{equation}
Note that $g(\zeta)$ has bounded second derivatives except for the point $\zeta=1$.
The matrix $\Xi$ is asymptotically equal to the covariance matrix of $\sum_{j=1}^pX_j\in\RR^{n-2bn}$. The detailed evaluation of $g(\zeta)$ in (\ref{eq_def_g}) is deferred to Lemma \ref{lemma_cov_thm1} in Appendix \ref{sec_proofs}. The regime with $0\leq \zeta<1$ corresponds to correlations of statistics within length of $bn$, while $1\le \zeta <2$ is concerning statistics which are further apart ($bn<|i-i'|<2bn$) and still have weaker correlations. Finally, $\zeta \ge 2$ suggests that statistics beyond $2bn$ are uncorrelated.

By the GA theorem, we shall expect the asymptotic distribution of $\Q_n$ under the null to be approximated by the maximum coordinate of a centered Gaussian vector $\Z$, that is,
\begin{equation}
    \label{eq_GA1}
    \PP(\Q_n\le u)\approx \PP\big(\max_{bn+1\le i\le n-bn}\Z_i\le u\big).
\end{equation}
This result allows us to find the critical value of our proposed test statistic $\Q_n$ by Monte-Carlo replicates. We refer to a simulation study in Appendix \ref{subsec_sim_null}.

\subsection{Gaussian Approximation}\label{sec_GA}

In this section, we provide a theory which implies the critical value of the proposed change-point test. We first consider the simplest setting where the errors are assumed to be cross-sectionally independent, and a MOSUM statistic aggregating all the $p$ series is adopted. The novel Two-Way MOSUM follows in Section \ref{sec_test_spatial}. A generalized case with cross-sectionally dependent errors is investigated in Section \ref{sec_test_nonlinear}.

We begin with two necessary assumptions as follows.
\begin{assumption}[Finite moment]
    \label{asm_finitemoment}
    Assume that the innovations $\eta_{i,j}$ defined in (\ref{eq_epsilon_linear}) are i.i.d. with $\mu_q:=\lVert\eta_{1,1}\rVert_q<\infty$ for some $q\ge8$.
\end{assumption}
\begin{assumption}[Temporal dependence]
    \label{asm_temp_dep} There exist constants $C>0$, $\beta>0$ such that 
    $$\max_{1\le j\le p}\sum_{k\ge h}|A_{k,j,\cdot}|_2/\sigma_j\le C(1\vee h)^{-\beta},$$
    for all $h\ge0$, where $A_{k,j,\cdot}$ is the $j$-th row of $A_k$.
\end{assumption}
\begin{assumption}[Cross-sectional independence]
    \label{asm_sec_indep}
    Assume that for all $k\ge0$, the coefficient matrices $A_k$ defined in (\ref{eq_epsilon_linear}) are diagonal matrices. 
    % component processes $\{\epsilon_{t,j}\}_{t\in\ZZ}$, $1\le j\le p$, are independent. Accordingly, $A_kA_k^{\top}$, $k\ge0$, are all diagonal matrices.
\end{assumption}
Assumption \ref{asm_finitemoment} puts tail assumptions on the moment of the noise sequences in our moving average model (\ref{eq_epsilon_linear}). Assumption \ref{asm_temp_dep} is a very general and mild temporal dependence condition, which requires the polynomial decay rate of the temporal dependence. It also ensures that the long-run variance is finite. Many interesting processes fulfill this assumption. We refer to a specific example in Appendix \ref{subsec_ex_temp_Dep}. Note that we introduce Assumption \ref{asm_sec_indep} for the simplicity to show the GA and it shall be relaxed to allow weak cross-sectional dependence in Section \ref{sec_test_nonlinear}.

Provided all the aforementioned conditions, we state our first main theorem as follows.
\begin{theorem}[GA with cross-sectional independence]
    \label{thm1_constanttrend}
    Suppose that Assumptions \ref{asm_finitemoment}--\ref{asm_sec_indep} are satisfied. Then, under the null hypothesis, for $\Delta_0=(bn)^{-1/3}\log^{2/3}(n)$,
    $$\Delta_1 = \Bigg(\frac{n^{4/q}\log^7(pn)}{p} \Bigg)^{1/6}, \quad \Delta_2 = \Bigg(\frac{n^{4/q}\log^3(pn)}{p^{1-2/q}}\Bigg)^{1/3},$$
    we have
    \begin{equation}
        \sup_{u\in\RR}\Big|\PP(\Q_n\le u)-\PP\big(\max_{bn+1\le i\le n-bn}\Z_i\le u\big)\Big|\lesssim \Delta_0+\Delta_1+\Delta_2,
    \end{equation}
    where the constant in $\lesssim$ is independent of $n,p,b$. If in addition, $\log(n)=o\{(bn)^{1/2}\}$ and
    \begin{equation}
        \label{eq_thm11_o1}
        n^4p^{2-q}\log^{3q}(pn)   {\rightarrow0, }
    \end{equation}
    then we have
    \begin{equation}
        \label{eq_thm11_result}
        \sup_{u\in\RR}\Big|\PP(\Q_n\le u)-\PP\big(\max_{bn+1\le i\le n-bn}\Z_i\le u\big)\Big|\rightarrow 0.
    \end{equation}
\end{theorem}
The symbol $o(1)$ and $\rightarrow0$ in this theorem and the rest of the paper is in the asymptotic regime $n\wedge p\rightarrow\infty$. It is worth noting that the convergence rate of the GA in Theorem \ref{thm1_constanttrend} depends on the number of cross-sectional components $p$, and a larger $p$ is no longer a curse when utilizing an $\ell^2$-type test statistic. The intuition behind this is that when applying the GA, we effectively treat our $p$ cross-sectional components as equivalent to the "$n$ observations" in \textcite{chernozhukov_central_2017}. As such, a larger $p$ provides more information that can be used to detect change points, which can in turn reduce the approximation error.

Compared to the finite moment condition (E.2) in \textcite{chernozhukov_central_2017} which assumes that $q\ge4$, we require $q\ge8$ since $X_j$ in our test statistic $\Q_n$ is quadratic with regard to the random noise $\epsilon_{t,j}$. As for the GA rate in Proposition 2.1 in \textcite{chernozhukov_central_2017}, our $\Delta_1$ and $\Delta_2$ correspond to their rate with $p$ replaced by $n$ and $B_n = n^{2/q}$. It shall be noted that our additional term $\Delta_0$ is due to an additional step to compare non-centered Gaussian variables (cf. Lemma \ref{lemma_chen2019} in Appendix \ref{sec_proofs}).
% It shall be noted that our additional term $\Delta_0$ is due to the aggregation over the dimension $p$, since we need additional steps to compare non-centered Gaussian variables (cf. Lemma \ref{lemma_chen2019} in Appendix \ref{sec_proofs}).

\begin{remark}[Allowed dimension $p$ relative to $n$] 
In Theorem \ref{thm1_constanttrend}, we can allow $p$ to be of some polynomial order of $n$, and its order depends on the moment parameter $q$ defined in Assumption \ref{asm_finitemoment}. In particular, let $p\asymp n^{\nu_1}$, for some $\nu_1>0$. If $\nu_1>4/(q-2)$, then expression (\ref{eq_thm11_result}) holds. The larger moment parameter $q$ is, the weaker condition on $p$ we can allow.
\end{remark}

\begin{remark}[Comment on the convergence rate in Theorem \ref{thm1_constanttrend}]  We standardize $X_j$ defined in (\ref{eq_thm1_x_def}) and denote it by $X_j^*$, i.e. $X_j^*=bnX_j$. \textcite{chernozhukov_nearly_2021} derives a nearly optimal bound in the case when the smallest eigenvalue $\sigma_*^2$ of the covariance matrix of $X_j^*$ is bounded below from zero. However, this sharp approximation rate cannot be achieved in our case. To see this, we consider the simple case where the errors are i.i.d.. Then, for any integer $h$, the $(i,i+h)$-th element of the covariance matrix $\EE(X_j^*X_j^{*\top})$ takes the form of $g(|h|/(bn))$ in (\ref{eq_cov}). It shall be noted that $\EE(X_j^*X_j^{*\top})$ is a $2bn$-banded matrix and it is symmetric. Since $g(|h|/(bn))$ has bounded second derivative except for one point, for any four coordinates at the positions $(s,s+2i+1)$, $(s,s+2i+2)$, $(s+2i,s)$ and $(s+2i+1,s)$ in $\EE(X_j^*X_j^{*\top})$, $1\le s\le n-2bn$, $0\le i\le \lfloor(n-2bn-s-2)/2\rfloor$, we can bound them in the following way,
\begin{equation}
    \label{eq_thm1_covshape}
    -g\Big(\frac{2i+1}{bn}\Big)+g\Big(\frac{2i+2}{bn}\Big)-g\Big(\frac{2i}{bn}\Big)+g\Big(\frac{2i+1}{bn}\Big)=O\Big\{\frac{1}{(bn)^2}\Big\}.
\end{equation}
Inspired by (\ref{eq_thm1_covshape}), we define a vector $y=(-1,1,-1,1,\ldots)^{\top}\in\RR^{n-2bn}$. When $b^2n\rightarrow\infty$, the upper bound of the smallest eigenvalue $\sigma_*^2$ of $\EE(X_j^*X_j^{*\top})$ tends to $0$, that is 
\begin{equation}
    \label{eq_thm1_eigen}
    \sigma_*^2\le \frac{y^{\top}\EE(X_j^*X_j^{*\top})y}{y^{\top}y}=O\Big\{\frac{1}{b^2n}\Big\}\rightarrow0.
\end{equation}
Therefore, \textcite{chernozhukov_nearly_2021} is not applicable to our case and we instead extend the GA in \textcite{chernozhukov_central_2017} to achieve the rate in Theorem \ref{thm1_constanttrend}, which does not require the covariance matrix of $X_j^*$ to be non-degenerate.
\end{remark}

Theorem \ref{thm1_constanttrend} implies that, for any level $\alpha\in(0,1)$, we can choose the threshold value $\omega$ to be the quantile of the Gaussian limiting distribution indicated by Theorem \ref{thm1_constanttrend}:
\begin{equation}
    \label{eq_thm1_critical}
    \omega=\inf_{r\ge0}\Big\{r:\PP\big(\max_{bn+1\le i\le n-bn}\Z_i>r\big)\le \alpha\Big\}.
\end{equation}
We shall reject the null hypothesis if the test statistic $\Q_n$ exceeds the threshold value $\omega$, i.e. $\Q_n>\omega$. Under the alternative hypothesis, we show that when the break size is sufficiently large, we can achieve the testing power asymptotically tending to $1$ (cf. Corollary \ref{cor_power}).
Recall the trend function defined in (\ref{eq_trend}). We denote the break location by $\tau_k=nu_k$, $1\le k\le K $, and introduce the following assumption for the identification of breaks.
\begin{definition}[Temporal separation]
    \label{def_separation} Define the minimum gap between breaks as $\kappa_n=\min_{0\le k\le K }(u_{k+1}-u_k)$, for some $\kappa_n>0$, and we allow $\kappa_n\rightarrow0$ as $n$ diverges.     
\end{definition}
Definition \ref{def_separation}  concerns the separation of temporal break locations to ensure the consistency of break estimation, and it implies that $K $ cannot be larger than $1/\kappa_n$. When $\kappa_n\rightarrow0$, $K $ can grow to infinity,
% Comment on the assumption is similar to Yi Yu. 
which is in line with Assumption (1b) in \textcite{xu2022change}, where they allow the minimum spacing to be a function of $n$ and to vanish as $n$ diverges. In addition, we require the bandwidth parameter $b$ in (\ref{eq_localest}) to satisfy $b\ll \kappa_n$ as $n\rightarrow\infty$. Otherwise, if more than one break exists within a window of temporal width $bn$, the adopted MOSUM statistics might fail to distinguish the break time points in the same window. 
For any time point $i$ satisfying $|i-\tau_k|\leq bn$, we define the weighted break vectors as
\begin{equation}
    \label{eq_EV_appr}
    d_i=\EE V_i=\sum_{k=1}^K\Big(1-\frac{|i-\tau_{k}|}{bn}\Big)\Lambda^{-1}\gamma_{k}\One_{|i-\tau_k|\le bn},
\end{equation}
and let $\underline{d}=(d_{bn+1}^{\top},d_{bn+2}^{\top},\ldots,d_{n-bn}^{\top})^{\top}$. Under the alternative hypothesis, there exists at least one break, that is $\underline{d}\neq 0$. We evaluate our testing power in the following corollary.
\begin{corollary}[Power]
    \label{cor_power}
    Under Assumptions \ref{asm_finitemoment}--\ref{asm_sec_indep}, if (\ref{eq_thm11_o1}) holds, $b\ll \kappa_n$, and
    $$\max_{1\le k\le K}n(u_{k+1}-u_k)|\Lambda^{-1}\gamma_k|_2^2\gg \sqrt{p\log (n)},$$
    then the testing power $\PP(\Q_n>\omega)\rightarrow 1$, as $n\wedge p\rightarrow\infty$.
\end{corollary}
The symbols $\ll$ and $\gg$ hold here and throughout the rest of the paper in the asymptotic regime that $n\wedge p\rightarrow\infty$. Corollary \ref{cor_power} provides a condition for the test power tending to $1$. It allows for cases with nontrivial alternatives. Namely, in some component series, the break sizes could be small and tend to $0$, as long as the aggregated size of the break is sufficiently large. 
It is generally challenging to make an exact comparison between different break statistics with different assumptions for complicated panel data. Nevertheless, we can observe that our detection lower bound is quite sharp in the situation where many weak signals with similar magnitude are present. For example, suppose the jump size of each component series is $\vartheta\in\RR$ and then we only require $n\kappa_n\vartheta^2\gg \sqrt{\log (n)/p}$. Compared with Table 1 in \textcite{cho2022two} which summarizes the procedures under different settings, $\sqrt{\log (n)/p}$ is smaller than all entries in the table since we have $\sqrt{p}$ in the denominator resulting from the $\ell^2$-aggregation. 

\begin{remark}[Detailed power comparison of $\ell^2$ and $\ell^{\infty}$ statistics]
\label{remark_linf} 
This remark is a complement to Remark \ref{11}. 
Now we clarify these differences by a detailed power comparison in those two cases. First, we consider sparse signals. Suppose that there is only one time series with a single break and the break size is $\vartheta^*$, and we use the MOSUM with bandwidth $b$ for detection. Then to ensure the power tending to $1$, $\Q_{n,\infty}$ only requires $\vartheta^*\gg \log^{1/2}(pn)(bn)^{-1/2}$ (see, for example,  \textcite{chen_inference_2022}), while $\Q_n$ needs a stronger condition by Corollary \ref{cor_power} that $ \vartheta^*\gg (p\log (n))^{1/4}(bn)^{-1/2}$. Secondly, we check dense signals. Suppose all series jump with the same size $\vartheta'$. Then, for $\Q_{n,\infty}$, we still need $\vartheta'\gg \log^{1/2}(pn)(bn)^{-1/2}$, while $\Q_n$ only requires $\vartheta' \gg \log^{1/4} (n)p^{-1/4}(bn)^{-1/2}$. 
\end{remark}

\subsection{Estimation of Change Points}\label{sec_estimation}
Based on the theoretical underpinnings of the test in prior subsections, we {can now present} our detection algorithm (cf. Algorithm \ref{algo1}). To begin with, we explicate our strategy for detecting and estimating change points. Furthermore, we showcase the consistency for the estimators pertaining to the number, time stamps, and sizes of breaks. A simulation study can be found in Appendix \ref{subsec_sim_alter}.

We consider the size of the break at time point $\tau_k$ normalized by the long-run standard deviation $\Lambda$, that is $|\Lambda^{-1}\gamma_k|_2$. Define the minimum of normalized break sizes over time as
\begin{equation}
    \label{eq_min_breaksize}
    \delta_p =\min_{1\le k\le K }|\Lambda^{-1}\gamma_k|_2,
\end{equation}
which can be viewed as the minimum strength of signals in the setting of Corollary \ref{thm1_constanttrend}.

\begin{algorithm}[hbt!] 
   \caption{$\ell^2$ Multiple Change-Point Detection via a MOSUM}
   \label{algo1}
   \KwData{Observations $Y_1,Y_2,\ldots,Y_n$; bandwidth parameter $b$; threshold value $\omega$}
   \KwResult{Estimated number of breaks $\hat K $; estimated break time stamps $\hat\tau_k$, $k=1,\ldots,\hat K $; estimated jump vectors $\hat \gamma_k$; estimated minimum break size over time $\hat\delta_p$}
   $\Q_n\gets \max_{bn+1\le i\le n-bn}(|V_i|_2^2-\bar{c})$\;
   \eIf{$\Q_n<\omega$}{
    $\hat K =0$; STOP\;
   }{
   $k \gets 1$; $\A_1\gets \{bn+1\le \tau\le n-bn: (|V_{\tau}|_2^2-\bar{c})>\omega\}$\;
   \While{$\A_k\neq\emptyset$}{
   $\hat\tau_k \gets \arg\max_{\tau\in\A_k}(|V_{\tau}|_2^2-\bar{c})$; $\hat\gamma_k \gets \hat\mu_{\hat\tau_k-bn}^{(l)} - \hat\mu_{\hat\tau_k+bn-1}^{(r)}$\;
   $\A_{k+1} \gets \A_k\setminus \{t:|t-\hat\tau_k|\le 2bn\}$; $k \gets k+1 $\;
   }
   $\hat K =\max_{k\ge1}\{k:\A_k\neq\emptyset\}$; $\hat\delta_p \gets\min_{1\le k\le \hat K }\big||\Lambda^{-1}\hat\gamma_k|_2^2-\bar{c}\big|^{1/2}$\;
   }
\end{algorithm}

\begin{remark}[Comments on the centering term $\bar{c}$]
\label{rmk_centering}
To implement Algorithm \ref{algo1}, we need to {provide} the centering term $\bar{c}$ defined in (\ref{eq_thm1_c_def}). Due to the weak temporal dependence of $\{\epsilon_t\}$, one can show that 
\begin{equation}
    \label{eq_thm1_center_order}
    \bar{c}=2p/(bn)\big(1+O\{1/(bn)\}\big).
\end{equation}
In practice, we can simply take the centering term to be $\overline{c}\approx 2p/(bn)$. This still ensures the consistency results in Theorem \ref{thm_consistency} when the window size $bn$ is slightly larger than the dimension $p$, since the approximation error of using $2p/(bn)$ is of $O\{p/(bn)^2\}$, which is smaller than the order of $(p\log{(n)})^{1/2}(bn)^{-1}$. 
\end{remark}

\begin{remark}[Selection of bandwidth $b$]
    \label{remark_b}
    Technically, the larger the bandwidth $b$ is, the more data points could be used. Thus, a larger $b$ would reduce the magnitude of noise, and then the signals shall be easier to detect. However, we also need to restrict $b$, since the identification condition $b\ll \kappa_n$ below Definition \ref{def_separation} requires us to have fewer than one break within each window. Therefore, we suggest starting with a small bandwidth, for example $b=1/\sqrt{n}$ to ensure that $bn\rightarrow\infty$. One could continue to increase the bandwidth until the estimated number of breaks decreases. We refer to Appendix \ref{sec_sim} for a simulation study including the influences of different bandwidth parameters, as well as a discussion in Remark \ref{rmk_multiscale} for the potential extension of our method to a multi-scale MOSUM.
\end{remark}

Now we outline the consistency results of estimators for break numbers, break time stamps and break sizes. 
We shall impose a minimum requirement of break sizes to guarantee that all of the breaks can be precisely captured and estimated.
\begin{assumption}[Signal]
    \label{asm_min_breaksize}
    Assume $\min_{0\le k\le K}n(u_{k+1}-u_k)|\Lambda^{-1}\gamma_k|_2^2 \gg \sqrt{p\log(n)}$.
\end{assumption}
We highlight that Assumption \ref{asm_min_breaksize} imposes a moderate condition on the signals, which relates the break separation $(u_{k+1}-u_k)$ to the break signal strength $|\gamma_k|_2^2$, and does not require $p$ series to jump simultaneously. 
% {\color{red} distinguishing our scenario from the $p$ one-dimensional problems??}. 
A comprehensive discussion on the advantage of our setting can be found following the consistency results presented below.

\begin{theorem}[Temporal consistency]
    \label{thm_consistency}
    Let $q\ge8$. Under Assumptions \ref{asm_finitemoment}--\ref{asm_min_breaksize}, if (\ref{eq_thm11_o1}) hold, $b\ll \kappa_n$, $\delta_p^2 \ge 3\omega$, $\omega\gg (p\log{(n)})^{1/2}(bn)^{-1}$ and $\max_{1\le k\le K}|\Lambda^{-1}\gamma_k|_q/|\Lambda^{-1}\gamma_k|_2=O(1/K^{1/q})$,
    then we have the following results:
    \begin{itemize}
        \item[(i)] (Number of breaks). $\PP(\hat{K }=K )\rightarrow1$.
        \item[(ii)] (Time stamps of breaks). Let $\Gamma_k=p/(bn|\Lambda^{-1}\gamma_k|_2^2)$. Then, we have $\max_{1\le k\le K }|\hat\tau_k-\tau_{k^*}|\cdot|\Lambda^{-1}\gamma_k|_2^2/(1+\Gamma_k)= O_{\PP}\{\log^2(n)\}$, where $k^*=\arg\min_i|\hat\tau_k-\tau_i|$. If in addition, $\delta_p^2/p\gtrsim1/(bn)$, we have
        $$\max_{1\le k\le K }|\hat\tau_k-\tau_{k^*}|\cdot|\Lambda^{-1}\gamma_k|_2^2=O_{\PP}\big\{\log^2(n)\big\}.$$
        \item[(iii)] (Break sizes). $\max_{1\le k\le K }\big||\Lambda^{-1}(\hat\gamma_k-\gamma_{k^*})|_2^2-\bar{c}\big|  = O_{\PP}\big\{(p\log{(n)})^{1/2}(bn)^{-1}\big\}$,  which also implies that $|\hat\delta_p-\delta_p| = O_{\PP}\big\{(p\log{(n)})^{1/4}(bn)^{-1/2}\big\}.$
    \end{itemize}
\end{theorem}
The results in Theorem \ref{thm_consistency} are in an asymptotic context as $n\wedge p\rightarrow \infty$. To see the advantage of our method, let us consider the simple case with a jump size $\vartheta\in\RR$ for each component series. To achieve the consistency results in (ii), we only require $n\kappa_n\vartheta^2\gg \sqrt{\log(n)/p}$ by Assumption \ref{asm_min_breaksize} and $K=O(p^{q/2-1})$ by the condition $\max_{1\le k\le K}|\Lambda^{-1}\gamma_k|_q/|\Lambda^{-1}\gamma_k|_2=O(1/K^{1/q})$. 
% In contrast, if we were to identify the breaks in the $p$ component series individually without the $\ell^2$-aggregation, a much stronger condition that $\kappa_n\vartheta^2\gg \sqrt{\log(n)}$ would be necessary. 
% {\color{red}Hence, when the jump size $\vartheta$ is small, one-dimensional methods would fail in both detection and estimation tasks, whereas the $\ell^2$-type method can effectively address this challenge.}
Notably, as discussed following Corollary \ref{cor_power}, our detection lower bound is weaker compared to those reported in Table 1 by \textcite{cho2022two}. 
% such as $\log(n)$ as presented in \cite{wang_samworth_2018}. Also, a larger $p$ makes the detection and estimation easier by our proposed method.
% We also emphasize that unlike the one-dimensional methods, our approach does not require simultaneous jumps in all $p$ series as indicated by Assumption \ref{asm_min_breaksize}. 
In (iii), since $\Lambda^{-1}\hat\gamma_k$ contains both the signal part $\EE V_{\hat\tau_k}$ and the error part $V_{\hat\tau_k}-\EE V_{\hat\tau_k}$, we center $|\Lambda^{-1}\hat\gamma_k|_2^2$ under the null by subtracting $\bar{c}=\EE|V_{\hat\tau_k}-\EE V_{\hat\tau_k}|_2^2$. This guarantees that the break sizes can be estimated consistently.
We have $\log(n)$ in the consistency rates in (ii) and (iii) since we take the maximum overall $n-2bn$ moving windows. It is also worth noticing that, for single change-point detection, \textcite{bai_common_2010} achieves $|\hat\tau_k-\tau_k|=O_{\PP}(1)$ by assuming that $\delta_p$ is a constant. Under this condition, by our Theorem \ref{thm_consistency} (ii), we can achieve the same consistency rate up to a logarithm factor. Indeed, $\delta_{p}$ does not need to be a constant for consistent estimators (of a single or finitely many breaks) with an order of $O_{\PP}(1)$; see for instance \textcite{horvath2017asymptotic}.
Intuitively, a larger $\delta_p$ makes the estimation problem easier, and in the same example mentioned earlier, where the break size for each series is denoted as $\vartheta\in\RR$, one can see that our $\delta_p$ can diverge as $p$ grows, leading to improved consistency rates for $\hat\tau_k$.

\section{Testing and Estimation via a Two-Way MOSUM}\label{sec_test_spatial}

The previous section focuses on change-point statistics for data generating processes with dense breaks across all component series. In this section, we propose a novel Two-Way MOSUM to address cases where change points may exist in only a {subset} of time series. In such case, aggregating all component series using the $\ell^2$-norm dilutes the testing power due to the overwhelming aggregated noises compared to signals. To handle this issue, we construct \textit{temporal-spatial windows} (cf. Definition \ref{def_temporal_spatial}) to aggregate series within spatial neighborhoods and maximize over 
neighborhoods and time. Our method allows for the estimation of both the temporal break stamp and the spatial neighborhood that contains the breaks. We extend the Two-Way MOSUM to nonlinear processes and a general spatial space in Section \ref{sec_test_nonlinear}.

\subsection{Two-Way MOSUM}
As \textcite{xie_sequential_2013} suggested, in certain applications, data streams could represent observations from multiple sensors, where breaks only occur in some but not all of them. However, testing procedures aggregating all sensors include noise from unaffected sensors in detection statistics, leading to poor performance. To address this problem, we introduce cross-sectional neighborhoods comprised of spatially adjacent series, and propose the Two-Way MOSUM to account for spatial group structure and detect existence of breaks. This statistic improves test performance for signals dense within clusters but sparse among them. We emphasize that Two-Way MOSUM can work even in the absence of a-prior clustering information. See the discussion below Definition \ref{def_temporal_spatial} for detailed reasoning after introducing the temporal-spatial moving window.

To define the relative spatial locations of component series, we consider a spatial space $\L_0=\L_{0,n}$ for all series. In particular, we start with a simple case in this section by assuming a linear ordering in the coordinates, i.e., $\L_0\subset\{1,\ldots,p\}$. Then in Section \ref{sec_test_nonlinear}, we extend it to a more general space with $\L_0\subset\ZZ^v$ for a general $v$, where the spatial location of each series shall be determined by a $v$-dimensional vector. We define the linear spatial neighborhood and the corresponding neighborhood-norm as follows. 

\begin{definition}[Linear spatial neighborhood]
    \label{def_linear_nbd}
    Let $\L_s\subset\{1,\ldots,p\}$ be the set of coordinates in a spatial neighborhood, $1\le s\le S$, where $S$ is the total number of spatial neighborhoods. We denote the size of each $\L_s$ by $|\L_s|$. In particular, we define 
    \begin{equation}
        \label{eq_min_nbdsize}
        |\L_{\text{max}}|=\max_{1\le s\le S}|\L_s| \quad \text{and} \quad |\L_{\text{min}}|=\min_{1\le s\le S}|\L_s|.
    \end{equation}
\end{definition}
\begin{definition}[Linear nbd-norm]
    \label{def_linear_nbdnorm}
    For a $p$-dimensional vector $v_i=(v_{i1},\ldots,v_{ip})^{\top}$ with a linear ordering in coordinates, we define the linear neighborhood-norm (nbd-norm) as $\vert v_i\vert_{2,s}=\big(\sum_{j=1}^pv_{i,j}^2\One_{j \in \L_s}\big)^{1/2}$, $1\le s\le S$.
\end{definition}
It shall be noted that the spatial neighborhoods defined in Definition \ref{def_linear_nbd} can be overlapped, which means that each component series can belong to multiple different spatial neighborhoods. Therefore, these spatial neighborhoods actually can be viewed as an analogue of the temporal moving windows, but moving in the spatial direction and allowing different window sizes (i.e. neighborhood sizes). All $s=1,2,\ldots,S$ are only the indices of different spatial groups and do not necessarily reflect the spatial order of these groups. The size of each neighborhood could tend to infinity as $p\rightarrow\infty$. 

We highlight that the spatial moving windows can be adapted to different data scenarios. Apart from an identification condition, we do not need more knowledge (e.g. clustered breaks) on the spatial structure of the signals.
It is worth noting that similar definitions of neighborhoods are considered in the literature, see for example, \textcite{arias2008searching,addario2010combinatorial,arias2011detection}. These setups exclusively focus on the topological structure of neighborhood with simple Gaussian or i.i.d. assumptions. Comparably, we are more flexible in modeling spatial temporal dependency of the data. In fact, our spatial windows can be extended to more complicated shapes depending on the demand of applications as long as Assumption \ref{asm_nbd_size_ratio} is satisfied.
Many real-life data streams, such as those in geographical or economic contexts, provide prior knowledge about which spatial groups are likely to include breaks, as demonstrated in the geostatistics data examples in Chapter 4 of \textcite{cressie2015statistics}. Our definition of the temporal-spatial window (cf. Definition \ref{def_temporal_spatial}) does not require this prior knowledge, but if it is available, it can be utilized for the more relevant detection and estimation procedures.

Our goal is to model the breaks occurring on the vector of unknown trend functions. When there potentially exists a group structure, we formulate the trend function in (\ref{eq_model}) as
\begin{equation}
    \label{eq_model_part2}
    \mu(u)=\mu_0+\sum^{R }_{r=1}\gamma_r\One_{u\ge u_r},
\end{equation}
where $R \in\NN$ is an unknown integer represents the number of localized breaks which could go to infinity as $n$ or $S$ increases; $u_1,\ldots,u_{R }$ are the time stamps of the breaks with $0=u_0<u_1<\ldots < u_{R }<u_{R +1}=1$ and $\kappa_n=\min_{0\le r\le R }(u_{r+1}-u_r)$, for some constant $\kappa_n>0$, where $\kappa_n$ is allowed to tend to zero as $n\rightarrow\infty$; $\gamma_r=(\gamma_{r,1},\ldots,\gamma_{r,p})^{\top}\in\RR^p$ is the jump vector at the time stamp $u_r$ with $\gamma_{r,j}=0$ if $j\notin\L_{s_r}$, where $s_r$ is the index of the spatial location of the $r$-th break. We define the break size as $|\gamma_r|_2$. In the rest of this paper, we use $(\tau_r,s_r)$ to denote the temporal-spatial location of the $r$-th break.

To test the existence of spatially localized breaks, it suffices to test the null hypothesis
$$\H_0^{\diamond}: \quad \gamma_r=0, \quad 1\le r\le R , $$
which denotes the case with no breaks, against the alternative that at least one break exists, that is, $\H_A^{\diamond}$: there exists $r\in\{1,\ldots,R \}$, such that $\gamma_r\neq0$.
This enables us to identify both the time stamps and the spatial neighborhoods with significant breaks. Our Two-Way MOSUM statistic aims to adopt more flexible moving windows to efficiently capture both temporal and spatial information of breaks. To achieve this goal, we shall derive a localized test statistic which first aggregates the time series within each spatial neighborhood by an $\ell^2$-norm and then take the maximum over all the neighborhoods and time points. Accordingly, we define \textit{temporal-spatial windows} as follows:
\begin{definition}[Temporal-spatial window]
    \label{def_temporal_spatial}
    For $bn+1\le i\le n-bn$, $1\le s\le S$, define an index set $\V_{i,s}=\big\{(t,l): t=i, \, l\in\L_s\big\}$. Then, define the temporal-spatial moving window as $\S_{i,s}=\{\V_{t,l}:\, i-bn\le t\le i+bn-1,l\in \L_s\}\subset\ZZ^2$.
\end{definition}
Note that the index set $\V_{i,s}$ can be regarded as a vertical line, which is the center of the temporal-spatial moving window $\S_{i,s}$. Specifically, $\S_{i,s}$ spans the neighborhood $\L_s$ in the spatial direction and centered at the time point $i$ with radius $bn$ in the temporal direction. In the rest of this paper, we shall depict the index set $\V_{i,s}$ as a vertical line at time $i$ and neighborhood $\L_s$. We refer to Figure \ref{fig_twoway} in Appendix \ref{subsec_twoway_example} for a more straightforward illustration of this Two-Way moving window. It is worth emphasizing that even without prior knowledge of clusters, the number of potential windows is relatively small, because only the adjacent series can be assigned to the same group, which leads to the number of potential windows at most $O(p^2)$. See a more detailed explanation in Appendix \ref{subsec_twoway_example} and consider Figure \ref{fig_ordering} as an example. In general, for any $p$ time series, there are only at most $O(p^{2v})$ possible windows in a $\mathbb{Z}^v$ space, $v\ge1$. This fact does not diminish the validity of our results obtained through GA (cf. Theorem \ref{thm2_constanttrend}). Consequently, while prior knowledge about the clusters is beneficial for boosting power, it is not a prerequisite.

\begin{definition}[Influenced set]
    \label{def_set}
    We define the set of vertical lines influenced by the break located at $(\tau,s)$ as
\begin{equation}
    \label{eq_influence}
    \W_{\tau,s}=\{\V_{t,l}: 1\le t\le n,\,1\le l\le S,\, \S_{t,l}\cap\V_{\tau,s}\neq \emptyset\}.
\end{equation}
\end{definition}
\begin{assumption}[Neighborhood size]
    \label{asm_nbd_size_ratio}
    Assume that $|\L_{\text{max}}|/|\L_{\text{min}}|\le c$ holds for some constant $c\geq 1$, where $|\L_{\text{max}}|$ and $|\L_{\text{min}}|$ are defined in Definition \ref{def_linear_nbd}.
\end{assumption}
Assumption \ref{asm_nbd_size_ratio} requires that the sizes of all spatial neighborhoods do not differ too much, which still allows the flexibility of different neighborhood sizes but in a reasonable range. This assumption embraces many interesting cases in practice. For example, according to the geographical locations, the Centers for Disease Control and Prevention (CDC) divides the states in the United States (US) into four regions with similar spatial sizes, which is also taken into consideration in our application (cf. Section \ref{sec_data}); based on the patterns of synchronous activity and communication between different brain regions, \textcite{thomas_yeo_organization_2011} identifies and segregates the brain into seven distinct functional networks with similar scales.

Similar to (\ref{eq_thm1_c_def}), we define the centering term of the statistics as 
\begin{equation}
    \label{eq_thm2_c_def}
    \bar{c}_s^{\diamond}=\sum_{j=1}^p c_{s,j}^{\diamond}, \quad \text{where } c_{s,j}^{\diamond}=c_j\One_{j\in\L_s}.
\end{equation}
Following the intuitions that we could adopt temporal-spatial moving windows to account for spatially clustered jumps, we formulate our Two-Way MOSUM test statistic as follows:
\begin{equation}
    \label{teststats2}
    \Q_n^{\diamond} = \max_{1\le s\le S}\max_{bn+1\le i\le n-bn}\frac{1}{\sqrt{|\L_s|}}\big(|V_i|_{2,s}^2 - \bar{c}_s^{\diamond}\big)
\end{equation}
where the nbd-norm $|\cdot|_{2,s}$ is introduced in Definition \ref{def_linear_nbdnorm}. Recall (\ref{eq_thm1_x_def}) for $x_{i,j}$. Let 
\begin{equation}
    \label{eq_thm2_x_def}
   x_{i,s,j}^{\diamond}=x_{i,j}\One_{j\in\L_s} \,\text{and}\,\, X_j^{\diamond}=(x_{bn+1,1,j}^{\diamond},\ldots,x_{n-bn,1,j}^{\diamond},\ldots,x_{bn+1,S,j}^{\diamond},\ldots,x_{n-bn,S,j}^{\diamond})^{\top}.
\end{equation}
Then, under the null hypothesis, we can rewrite $\Q_n^{\diamond}$ into
\begin{equation}
    \label{eq_thm2_GAform}
    \Q_n^{\diamond}=\max_{1\le s\le S}\max_{bn+1\le i\le n-bn} \frac{1}{\sqrt{|\L_s|}}\sum_{j=1}^p x_{i,s,j}^{\diamond}.
\end{equation}
When the time series are cross-sectionally independent, $X_j^{\diamond}$ are independent, $1\leq j\leq p$. Hence, by applying the GA to (\ref{eq_thm2_GAform}), we can properly address the dependence resulting from the overlapped temporal-spatial windows. 
We also note that cluster-based statistics can be seen as a special case of Two-Way MOSUM as in (\ref{eq_thm2_GAform}) when the prior knowledge of clusters is given, which can also be seen as an extension of the statistic introduced by \textcite{jirak2015uniform} where one essentially gets back to his idea by taking each component as its own cluster.

We introduce the centered Gaussian vector
\begin{equation}
    \label{eq_Z_local}
    \Z^{\diamond}=(\Z_{bn+1,1}^{\diamond},\ldots,\Z_{n-bn,1}^{\diamond},\ldots,\Z_{bn+1,S}^{\diamond},\ldots,\Z_{n-bn,S}^{\diamond})^{\top},
\end{equation}
with the covariance matrix 
\begin{equation}
    \label{eq_cov_Z_local}
    \Xi^{\diamond}=(\Xi^{\diamond}_{i,s,i',s'})_{1\le i,i'\le n-2bn,1\le s,s'\le S}.
\end{equation}
Recall the covariance matrix $\Xi$ for the Gaussian vector $\Z$ in (\ref{eq_cov}). By (\ref{eq_thm2_x_def}) and (\ref{eq_thm2_GAform}), we similarly define
\begin{equation}
    \label{eq_cov_local_G}
    \Xi_{i,s,i',s'}^{\diamond}=(|\L_s||\L_{s'}|)^{-1/2}\One_{j\in\L_s\cap\L_{s'}}\Xi_{i,i'}.
\end{equation}
We aim to provide the GA theorem under the null with a Two-Way MOSUM applied (cf. Theorem \ref{thm2_constanttrend}). This result shall enable us to find the critical value of our proposed Two-Way MOSUM test statistic. We denote each element in $\Z^{\diamond}$ by $\Z_{\varphi}^{\diamond}$, where
\begin{equation}
    \label{eq_node_pair}
    \varphi=(i,s)\in\N, \quad \text{for } \N=\{bn+1,\ldots,n-bn\}\times \{1,\ldots,S\}.
\end{equation}
By the GA, the distribution of $\max_{\varphi\in\N}\Z_{\varphi}^{\diamond}$ shall approximate the one of our test statistic $\Q_n^{\diamond}$ under the null with large $p$, that is
\begin{equation}
    \label{eq_GA2}
    \PP(\Q_n^{\diamond}\le u)\approx \PP\big(\max_{\varphi\in\N}\Z_{\varphi}^{\diamond}\le u\big).
\end{equation}

\begin{remark}[Comparison of $\ell^2$ MOSUM and $\ell^2$-clustered Two-Way MOSUM]\label{22}
When the breaks only occur in a {portion} of time series, aggregation of all $p$ dimensions would cause power loss. This situation is frequently encountered when a spatial group structure exists and only a few groups have breaks. Our proposed Two-Way MOSUM {can} account for this situation by taking the $\ell^2$-norm within each spatial group. To explicitly show the differences, we present a simulated example with two different proportions of jumps and compare the testing powers in Figure \ref{fig_sim_compare_cluster}. We simulate $n=100$ observations of $p=10$, $20$, $30$, $50$, $70$ and $100$. The number of spatial groups is $S=5$ and each group size is $0.3p$. Specifically, we let $\L_s=2p(s-1)/10+\{1, 2, \ldots, 0.3p\}$, $1\le s\le 4$ and $\L_5=\{0.7p+1,0.7p+2,\ldots,p\}$. In Figure \ref{fig_sim_compare_cluster}(a), two groups $\L_2$ and $\L_5$ contain breaks at the same time $\tau=50$. In Figure \ref{fig_sim_compare_cluster}(b), breaks only exist in one group $\L_3$ at $\tau=50$. The errors in both two figures are generated from MA($\infty$) models defined in (\ref{eq_epsilon_linear}) with $\eta_t\sim t_9$ and jump sizes are $0.2$ for each dimension. We let the window size $bn=20$. All the reported powers in Figure \ref{fig_sim_compare_cluster} are averaged over $1000$ samples. We defer a more detailed power comparison to Remark \ref{rmk_power_l2cluster}.
\end{remark}

\begin{figure}[!htbp]
\centering
\includegraphics[width=1\linewidth]{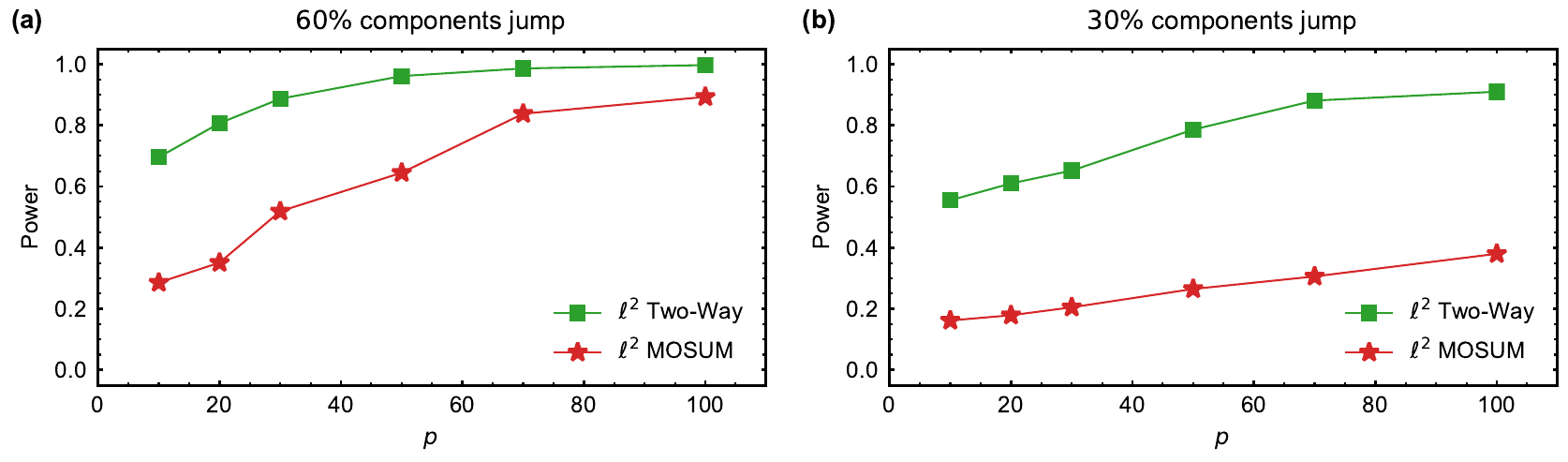}  
\caption{Power comparison of $\ell^2$ MOSUM and Two-Way MOSUM.}
\label{fig_sim_compare_cluster}
\end{figure}

\subsection{Gaussian Approximation for Two-Way MOSUM}
Given the updated statistics for spatially clustered signals, we further formalize our GA theory in this setting. Theoretically, it shall be noted that when signals are dense within the clusters and sparse among clusters, our Two-Way MOSUM improves the one adopted in Section \ref{sec_test}. 
\begin{theorem}[GA for Two-Way MOSUM]
    \label{thm2_constanttrend}
    Suppose that {Assumptions \ref{asm_finitemoment}-- \ref{asm_sec_indep}} and \ref{asm_nbd_size_ratio} are satisfied. Then, under the null hypothesis, for $\Delta_0^{\diamond}=(bn)^{-1/3}\log^{2/3}(nS)$,
    $$\Delta_1^{\diamond} = \Bigg(\frac{(nS)^{4/q}\log^7(pn)}{|\L_{\text{min}}|} \Bigg)^{1/6},  \quad \Delta_2^{\diamond} = \Bigg(\frac{(nS)^{4/q}p^{2/q}\log^3(pn)}{|\L_{\text{min}}|} \Bigg)^{1/3},$$
    we have
    \begin{equation}
        \sup_{u\in\RR}\Big|\PP(\Q_n^{\diamond}\le u)-\PP\big({\max_{\varphi\in\N}\Z_{\varphi}^{\diamond}}\le u\big)\Big|\lesssim \Delta_0^{\diamond}+\Delta_1^{\diamond}+\Delta_2^{\diamond},
    \end{equation}
    where $\N$ is defined in (\ref{eq_node_pair}), and the constant in $\lesssim$ is independent of $n,p,b$. If in addition, $\log(nS)=o\{(bn)^{1/2}\}$ and
    \begin{equation}
        \label{eq_thm21_o1}
        (nS)^4p^2|\L_{\text{min}}|^{-q}\log^{3q}(pn){\rightarrow0},
    \end{equation}
    then we have
     \begin{equation}
        \label{eq_thm21_result}
        \sup_{u\in\RR}\Big|\PP(\Q_n^{\diamond}\le u)-\PP\big({\max_{\varphi\in\N}\Z_{\varphi}^{\diamond}}\le u\big)\Big| \rightarrow 0.
    \end{equation}
\end{theorem}

% \textcolor{blue}{(We can absorb $S$ into $p$ in the logarithm since each spatial window spans adjacent components, which restricts $S$ to be no larger than $p^2$. If $S=o(p)$, it might be worth keeping $S$ in $\Delta_0$ for a slightly better rate?)} {\color{red} give an example on the relative rate of $|\L_{\text{min}}|$ and $q$ to show that $S$ can be of order $p^2$ (see Figure \ref{fig_ordering})}
One shall note that the GA results in Theorem \ref{thm1_constanttrend} is a special case of Theorem \ref{thm2_constanttrend}. Specifically, when $|\L_{\text{min}}|=p$ and $S=1$, which indicates that there does not exist any group structure and all $p$ time series belong to the same group, condition (\ref{eq_thm11_o1}) can be implied by (\ref{eq_thm21_o1}), and the same convergence rate of GA can be achieved. We extend the above theorem to nonlinear processes with general spatial temporal dependency in Theorem \ref{thm3_nonli}.

\begin{remark}[Allowed neighborhood number and size]
In Theorem \ref{thm2_constanttrend}, we can allow the minimum neighborhood size $|\L_{\text{min}}|$ to be of a polynomial order of the sample size $n$, and its order depends on the moment parameter $q$ defined in Assumption \ref{asm_finitemoment}. In particular, let $|\L_{\text{min}}|\asymp n^{\nu_2}$, for some $\nu_2>0$. 
Then, if $\nu_2>4/(q-2)$, expression (\ref{eq_thm21_result}) holds. The larger the moment parameter $q$ is, the larger minimum group size $|\L_{\text{min}}|$ we can allow. In addition, one shall note that the allowed number of neighborhoods $S$ can be as large as $O(p^2)$, while a bigger $S$ keeps more detailed local structure at the expense of time to inspect more windows. 
\end{remark}

Next, we consider the alternative hypothesis that there exists at least a break. Since the temporal-spatial moving windows can be overlapped, for the identification of breaks, we pose the following assumption on the separation of break locations.
\begin{assumption}[Temporal-spatial separation]
    \label{asm_spatial_sep}
    For any two breaks located at $(\tau_r,s_r)$ and $(\tau_{r'},s_{r'})$, $1\le r\neq r'\le R $, assume that there does NOT exist any moving window $\S_{\tau,s}$, $bn+1\le \tau\le n-bn$, $1\le s\le S$, such that $\S_{\tau,s}\cap\W_{\tau_r,s_r}\neq\emptyset$ and $\S_{\tau,s}\cap\W_{\tau_{r'},s_{r'}}\neq\emptyset$, where $\W(\tau_r,s_r)$ is defined in Definition \ref{def_set}.
\end{assumption}
Assumption \ref{asm_spatial_sep} can be viewed as an analogue of the condition $b\ll \kappa_n$ below Definition \ref{def_separation} when we apply a Two-Way MOSUM. To see this, consider the trend function in (\ref{eq_trend}). For any two breaks located at $\tau_k$ and $\tau_{k'}$, $b\ll \kappa_n$ guarantees that there is no moving window $\S_{\tau,:}$ intersects both $\W_{\tau_k,:}$ and $\W_{\tau_{k'},:}$, where $\S_{\tau,:}$ (resp. $\W_{\tau,:}$) is the sliding window (resp. influenced set) spanning all $p$ components. This adheres to the separation requirement in Assumption \ref{asm_spatial_sep}.

To achieve consistent {estimation} of the temporal and spatial break locations, for any given significance level $\alpha\in(0,1)$, we can choose the threshold value $ \omega^{\diamond}$ to be the quantile of the Gaussian limiting distribution indicated by Theorem \ref{thm2_constanttrend}, that is
\begin{equation}
    \label{eq_thm2_critical}
     \omega^{\diamond}=\inf_{r\ge0}\Big\{r:\PP\big(\max_{\varphi\in\N}\Z_{\varphi}^{\diamond}>r\big)\le \alpha\Big\}.
\end{equation}
Hence, we shall reject the null hypothesis if $\Q_n^{\diamond}> \omega^{\diamond}$. To evaluate the power of our localized test, consider the alternative hypothesis that there exists at least a break, i.e., $\underline{d}\neq 0$. We provide the power limit of our localized change-point detection in the following corollary.
\begin{corollary}[Power of a Two-Way MOSUM]
    \label{cor_local_power}
    Under Assumptions \ref{asm_finitemoment}--\ref{asm_sec_indep}, \ref{asm_nbd_size_ratio} and \ref{asm_spatial_sep}, if (\ref{eq_thm21_o1}) holds and
    $$\max_{1\le s\le S}\max_{1\le k\le K}n(u_{k+1}-u_k)|\Lambda^{-1}\gamma_k|_{2,s}^2\gg \sqrt{|\L_{\text{min}}|\log (nS)},$$
    where $|\cdot|_{2,s}$ is defined in Definition \ref{def_linear_nbdnorm}, then the testing power $\PP(\Q_n^{\diamond}> \omega^{\diamond})\rightarrow 1$.
\end{corollary}

\begin{remark}[Detailed power comparison of $\ell^2$ MOSUM and $\ell^2$-clustered Two-Way MOSUM]
    \label{rmk_power_l2cluster}
    This comment is complementary to Remark \ref{22}. Here we compare the testing power of the MOSUM aggregating all time series and the Two-Way MOSUM. Specifically, we consider a case where $p$ time series belong to $S$ groups and breaks only occur to one group. Suppose that all the series in this group jump with the same size $\vartheta''$, and we use the (Two-Way) MOSUM with bandwidth $b$ for detection. Then, to ensure the power tending to $1$, by Corollary \ref{cor_local_power}, $\Q_n^{\diamond}$ only requires $\vartheta''\gg \log^{1/4}(nS)|\L_{\text{min}}|^{-1/4}(bn)^{-1/2}$, while $\Q_n$ needs a stronger condition by Corollary \ref{cor_power} that $ \vartheta''\gg (p\log (n))^{1/4}|\L_{\text{min}}|^{-1/2}(bn)^{-1/2}$. 
\end{remark}

\subsection{Estimation of Change Points with Spatial Localization}\label{sec_estimation_nbd}

Providing the GA for the Two-Way MOSUM statistics, we gather the detailed steps of a change-point estimation procedure in Algorithm \ref{algo2}. Specifically, we extend Algorithm \ref{algo1} to the cases with cross-sectional localization via Two-Way MOSUM, and we shall expect to obtain spatial locations of change points besides the temporal ones. One follow-up theorem shows the consistency properties of some break statistics in this setup.

We denote the minimum break size over time and spatial neighborhoods by
\begin{equation}
    \label{eq_min_breaksize_nbd}
    \delta_p^{\diamond} =\min_{1\le r\le R }|\Lambda^{-1}\gamma_r|_2,
\end{equation}
and assume that this minimum break size is lower bounded as follows:
\begin{assumption}[Signal]
\label{asm_local_min_breaksize}
    $\min_{0\le r\le R}n(u_{r+1}-u_r)|\Lambda^{-1}\gamma_r|_2^2 \gg \sqrt{|\L_{\text{min}}|\log (nS)}$.
\end{assumption}
Let us consider the simple example that within any spatial neighborhood $\L_s$, $1\le s\le S$, the jump size of each time series is the same, denoted by $\vartheta\in \RR$. Then, Assumption \ref{asm_local_min_breaksize} means $n\kappa_n\vartheta^2\gg \sqrt{\log (nS)/|\L_{\text{min}}|}$, which is a {weaker} requirement of the signal strength for each series with breaks similar to Assumption \ref{asm_min_breaksize}.

To implement Algorithm \ref{algo2}, by the definition of $\bar{c}_s^{\diamond}$ in (\ref{eq_thm2_c_def}) and the similar arguments in Remark \ref{rmk_centering}, one {can} take $\bar{c}_s^{\diamond}=2|\L_s|/(bn)$, which still {ensures} the consistency. Also, similar to Algorithm \ref{algo1}, the selection of bandwidth parameter $b$ can follow the suggestions in Remark \ref{remark_b}, and the long-run variance can be estimated by a robust M-estimation method. The consistency results of the estimated number and temporal-spatial locations of breaks as well as the break sizes are all provided.

\begin{algorithm}[hbt!]
   \caption{$\ell^2$ Multiple Change-Point Detection via a Two-Way MOSUM}
   \label{algo2}
   \KwData{Observations $Y_1,Y_2,\ldots,Y_n$; spatial neighborhoods $\B_s$, $s=1,\ldots,S$; bandwidth parameter $b$; threshold value $ \omega^{\diamond}$}
   \KwResult{Estimated number of breaks $\hat R $; estimated break locations $(\hat\tau_r,\hat s_r)$, $r=1,\ldots,\hat R $; estimated jump vectors $\hat\gamma_r$; estimated minimum break size $\hat\delta_p^{\diamond}$}
   $\Q_n^{\diamond} \gets \max_{1\le s\le S}\max_{bn+1\le i\le n-bn}|\B_s|^{-1/2}(|V_i|_{2,s}^2-\bar{c}_s^{\diamond})$\;
   \eIf{$\Q_n^{\diamond}< \omega^{\diamond}$}{
    $\hat R =0$; STOP\;
   }{
   $r \gets 1$; $\A_1^{\diamond}\gets\{\V_{\tau,s},\,bn+1\le \tau\le n-bn,1\le s\le S: |\B_s|^{-1/2}(|V_{\tau}|_{2,s}^2-\bar{c}_s^{\diamond})> \omega^{\diamond}\}$\;
    \While{$\A_r^{\diamond}\neq\emptyset$}{
   $(\hat\tau_r,\hat s_r)\gets \arg\max_{\V_{\tau,s}\in\A_r^{\diamond}}|\B_s|^{-1/2}(|V_{\tau}|_{2,s}^2-\bar{c}_s^{\diamond})$; $\hat\gamma_r\gets \hat\mu_{\hat\tau_r-bn}^{(l)}- \hat\mu_{\hat\tau_r+bn-1}^{(r)}$\;

   $\A_{r+1}^{\diamond}\gets\A_r^{\diamond}\setminus \big\{\V_{\tau,s}:\,\text{there exists } bn+1\le i\le n-bn$, $1\le l\le S$, such that  $\S_{i,l}\cap\V_{\hat\tau_r,\hat s_r}\neq\emptyset \,\text{and } \S_{i,l}\cap\V_{\tau,s}\neq\emptyset\big\}$; $r \gets r+1 $\;
   }
   $\hat R \gets \max_{r\ge1}\{r:\A_r^{\diamond}\neq\emptyset\}$; $\hat\delta_p^{\diamond} \gets \min_{1\le r\le \hat R }\big||\Lambda^{-1}\hat\gamma_r|_2^2-\bar{c}\big|^{1/2}$\;
   }
\end{algorithm}

\begin{proposition}[Temporal-spatial consistency]
    \label{prop_consistency}
    Let $q\ge8$. Suppose that Assumptions \ref{asm_finitemoment}--\ref{asm_sec_indep}, \ref{asm_nbd_size_ratio}--\ref{asm_local_min_breaksize} and condition (\ref{eq_thm21_o1}) hold. If $|\L_{\min}|^{-1}\delta_p^{\diamond2}\ge3 \omega^{\diamond}$, $ \omega^{\diamond}\gg \log^{1/2}(n)(bn)^{-1}$ and $\max_{1\le r\le R}|\Lambda^{-1}\gamma_r|_q/|\Lambda^{-1}\gamma_r|_2=O(1/R^{1/q})$, 
    then we have the following results.
    \begin{itemize}
        \item[(i)] (Number of breaks). $\PP(\hat{R }=R )\rightarrow1$.
        \item[(ii)] (Time stamps of breaks).
        $\max_{1\le r\le R }|\hat \tau_r-\tau_{r^*}|\cdot|\Lambda^{-1}\gamma_r|_2^2/(1+\Phi_r) = O_{\PP}\big\{\log^2(nS)\big\}$, where $\Phi_r = |\L_{\text{min}}|/(bn|\Lambda^{-1}\gamma_r|_2^2)$ and $r^*=\arg\min_i\big|(\hat\tau_r,\hat s_r)-(\tau_i,s_i)\big|$.
        If in addition, $\delta_p^{\diamond2}/|\L_{\text{min}}|\gtrsim1/(bn)$, then
        $$\max_{1\le r\le R }|\hat \tau_r-\tau_{r^*}|\cdot|\Lambda^{-1}\gamma_r|_2^2/(bn)= O_{\PP}\big\{\log^2(nS)(bn)^{-1}\big\}.$$ 
        \item[(iii)] (Spatial locations of breaks).
        If there exists a constant $c_{\gamma}>0$ such that $|\gamma_{r,j}|/|\gamma_{r,j'}|\le c_{\gamma}$, for all $1\le r\le R$ and $j,j'\in \L_{s_r}$,
        then
        $$\max_{1\le r\le R }|\L_{\hat s_r}\ominus\L_{s_{r^*}}|\cdot|\Lambda^{-1}\gamma_r|_2^2/|\L_{\text{min}}|= O_{\PP}\big\{\log^2(nS)(bn)^{-1}\big\},$$
        where
        $\L_{\hat s_r}\ominus\L_{s_{r^*}}=(\L_{\hat s_r}\setminus\L_{s_{r^*}})\cup(\L_{s_{r^*}}\setminus\L_{\hat s_r})$ and $r^*=\arg\min_i\big|(\hat\tau_r,\hat s_r)-(\tau_i,s_i)\big|$. 
        \item[(iv)] (Break sizes). $\max_{1\le r\le R }\big||\Lambda^{-1}(\hat\gamma_r-\gamma_{r^*})|_2^2-\bar{c}\big| = O_{\PP}\big\{\big(p\log (nS)\big)^{1/2}(bn)^{-1}\big\}$. This also implies that $|\hat\delta_p^{\diamond}-\delta_p^{\diamond}| = O_{\PP}\big\{\big(p\log (nS)\big)^{1/4}(bn)^{-1/2}\big\}.$
    \end{itemize}
\end{proposition}
Proposition \ref{prop_consistency} (i) indicates the consistency of the estimator for the number of significant breaks; (ii) and (iii) show that we can consistently recover both the spatial break neighborhood $\L_{s_r}$ and the temporal break stamp $\tau_r$; (iv) suggests that the sizes of break vector $\gamma_r$ can also be estimated consistently.
Note that in Proposition \ref{prop_consistency} (ii), $|\hat \tau_r-\tau_{r^*}|\cdot|\Lambda^{-1}\gamma_r|_2^2/(bn)$ indicates the temporal precision, and the spatial precision is represented by $|\L_{\hat s_r}\ominus\L_{s_{r^*}}|\cdot|\Lambda^{-1}\gamma_r|_2^2/|\L_{\text{min}}|$ in (iii).
Both results are normalized by their window widths respectively and the two consistency rates are of the same order.
%Therefore, when we similarly consider the normalized temporal precision $|\hat\tau_r-\tau_{r^*}|/bn$ which is divided by the temporal window length $bn$, we achieve exactly the same consistency rate as the spatial one when $\Phi=o(1)$.

\begin{remark}[Comparison of consistency rates with Theorem \ref{thm_consistency}]
We see that the temporal consistency rate of $|\hat \tau_r-\tau_{r^*}|\cdot|\Lambda^{-1}\gamma_r|_2^2$ in Proposition \ref{prop_consistency} (ii) is similar to that in Theorem \ref{thm_consistency} except for an additional $S$ term in the log factor. For the break size, the convergence rate of $|\hat\delta_p^{\diamond}-\delta_p^{\diamond}|$ similarly admits an additional $S$ term in the log factor compared to $|\hat\delta_p-\delta_p|$ in Theorem \ref{thm_consistency}. Both two $S$ terms result from the maximization over all $S$ spatial neighborhoods in the estimators.
% stays the same, since the size is still defined by an $\ell^2$-norm and the noises from all $p$ series dominate the estimation error.
\end{remark}
% {\color{red}
% \subsection{Long run variance estimation???}}

\section{Nonlinear Time Series with Cross-Sectional Dependence}\label{sec_test_nonlinear}

In this section, we present three generalizations. Firstly, we expand the linear series given in (\ref{eq_epsilon_linear}) to accommodate a nonlinear scenario (see Eq. (\ref{eq_epsilon_nonlinear})) for a broader range of time series models. Secondly, we move beyond the linear ordering in coordinates by introducing a more comprehensive space in the spatial dimension, denoted as $\L_0\subset\ZZ^v$ (where $v\ge1$ is a fixed integer). Lastly, we generalize the GA from earlier sections by allowing weak cross-sectional dependence in the underlying {error process}. We will begin with the definition of the new spatial space and the nonlinear model, {proceed} with the conditions for both temporal and cross-sectional dependence structures, and ultimately, present our primary theoretical findings and the rationale behind the proof strategy.

\subsection{Multi-Dimensional Spatial Space}
To detect breaks in $\L_0$, we shall first provide a generalized notion of spatial window accordingly.
%We introduce a more general spatial space $\L_0=\L_{0,n}\subset\ZZ^v$ for the time series. 
In particular, denote $\B_s$, $1\le s\le S$, as spatial neighborhoods, which is a generalization of $\L_s$ in the previous section. Without loss of generality, we focus on hyper rectangles,
\begin{equation}
    \label{eq_nbd_Zv}
    \B_s=I_{s,1}\times I_{s,2}\times\ldots\times I_{s,v},
\end{equation}
where $I_{s,r}=I_{s,r,n} = [n_{s,r}^-,n_{s,r}^+]$ is some interval on $\ZZ$ whose end points $n_{s,r}^-$ and $n_{s,r}^+$ can depend on $n$. Different $\B_s$ are allowed to be overlapped and $S$ can go to infinity as $p\rightarrow\infty$. We define $I_r = I_{r,n} = [\min_s n_{s,r}^-,\max_s n_{s,r}^+]$ and let $\B_0 = I_1\times I_2\times\ldots\times I_v$, which implies
\begin{equation}
    \label{eq_nbd_all_Zv}
    \cup_{1\le s\le S}\B_s\subset\B_0.
\end{equation}
Suppose that the total number of locations in $\B_0\cap\L_0$ denoted as $|\B_0\cap\L_0|$ is $p=p_n$, which can go to infinity as $n$ increases. 
We consider the time series model
\begin{equation}
    \label{eq_nonli_model}
    Y_t(\bbell)=\mu_{\bbell}(t/n)+\epsilon_t(\bbell), \qquad t=1,...,n, \quad  \bbell\in\B_0\cap\L_0.
\end{equation}
% We denote  $\B_{s_k}$, $1\leq k\leq K $ as the $k$-th spatial neighborhood with breaks.
% Our main interest is to detect potential change points occurring on the trend functions at time $u_k$ and location $\bbell$ belonging to $\B_{s_k}$,\\
% {\color{red}
Our main objective is to identify possible change points in the trend functions 
\begin{equation}
    \label{eq_nonli_trend}
    \mu_{\bbell}(u)=\mu_{\bbell,0} + \sum_{k=1}^{K }\gamma_{k,\bbell}\One_{\{u\ge u_k, \bbell\in\B_{s_k}\cap\L_0\}},
\end{equation}
where $K$ and $u_1,\ldots,u_{K }$ are defined similarly to those in (\ref{eq_trend}); $\mu_{\bbell,0}\in\RR$ represents the {benchmark level} when no break occurs, and $\gamma_{k,\bbell}\in\RR$ denotes the jump at time point $u_k$ and location $\bbell$ in the neighborhood $\B_{s_k}$, the $k$-th spatial neighborhood containing breaks.

% {\color{red}
We then introduce definitions to characterize the mass and volume of the spatial neighborhood $\B_s$.
% }
% {\color{red}
By working with the spatial location $\bbell$, we can bypass the linear ordering presented in Section \ref{sec_test_spatial}. This notion of spatial location is similar to the general definition of spatial change-points in a $v$-dimensional spatial lattice used in studies such as \textcite{yu2022optimal} and \textcite{madrid2021lattice}.
% }
\begin{definition}[Spatial neighborhood] \label{def_sp_nbd}
    (i) (Mass).
    Define the mass of the spatial neighborhood $\B_s$ by the number of series in $\B_s$, i.e. $|\B_s\cap\L_0|$, where $|\cdot|$ is the number of elements in a Borel set. Denoted by $B_{\text{min}}$ and $B_{\text{max}}$ the sizes of the smallest and biggest spatial neighborhoods, respectively, i.e.,
    $$B_{\text{min}}=\min_{1\le s\le S}|\B_s\cap\L_0|,\quad B_{\text{max}}=\max_{1\le s\le S}|\B_s\cap\L_0|,$$
    which satisfy $B_{\text{max}}/B_{\text{min}} \le c$, for some constant $c\ge 1$.
    (ii) (Volume).
    Define the volume of the spatial neighborhood $\B_s$ as $\lambda(\B_s)$, where $\lambda(\cdot)$ is the Lebesgue measure of a Borel set.
\end{definition}
Following \textcite{matsuda_fourier_2009}, we introduce the following density assumption on the spatial space $\L_0$ that makes it possible to extend the asymptotic properties in regular space in $\ZZ^v$ to ones in irregular space. 
% In the following we relax Assumption \ref{asm_nbd_size_ratio}.
\begin{assumption}[Density of spatial space $\L_0$]
    \label{asm_density}
    Let $\bbell_j$, $j=1,\ldots,p$, be the spatial locations in $\L_0\subset\ZZ^v$ on which $Y_t(\bbell_j)$ is observed, $t=1,\ldots,n$. Assume that each $\bbell_j$ can be written as $\bbell_j=(A_1u_{j,1},\ldots,A_vu_{j,v})^{\top}.$ 
    % {\color{red} Explain $A_r$}
    Here, $\mathbf{u}_j=(u_{j,1},\ldots,u_{j,v})^{\top}$ is a sequence of i.i.d. random vectors with a density function $g(\mathbf{x})$ with a compact support in $[0,1]^v$. We assume that $A_r\rightarrow\infty$ as $p\to \infty$, for all $r=1,\ldots,v$. Also, for all $\mathbf{x}\in[0,1]^v$, 
    $c_1\le g(\mathbf{x}) \le c_2,$
    for some constants $c_1,c_2>0$.
\end{assumption}
Here we only require the density function $g(\mathbf{x})$ to be uniformly bounded from both sides, which is a weaker condition compared to Assumption 1 in \textcite{matsuda_fourier_2009}, where they aim to perform Fourier analysis for irregularly spaced data on $\RR^v$ and more restrict assumptions such as the existence of higher-order derivatives of $g(\mathbf{x})$ are therefore desired. Differently, our goal is to perform block approximation in the spatial direction to deal with the cross-sectional dependence (cf. Remark \ref{rmk_block}), which in fact only requires that, for any hyper-rectangle $\A\subset\ZZ^d$ satisfying $\lambda(\A)\rightarrow\infty$, there exist constants $c_1,c_2>0$ such that
$c_1\le |\A\cap\L_0|/\lambda(\A)\le c_2.$
One can view this condition as a special case of Assumption \ref{asm_density}. Also, when it breaks down to a simple space with linear ordering, the linear spatial neighborhood $\L_s$ defined in Definition \ref{def_linear_nbd} can be represented by $\L_s=\B_s\cap\L_0$ and we have $|\L_s|=|\B_s\cap\L_0|=\lambda(\B_s)$, $1\le s\le S$, that is $g(\mathbf{x})\equiv 1$ for all $\mathbf{x}\in[0,1]^v$ with $v=1$.
Concerning the shape of spatial neighborhoods, we pose the following assumption to eliminate the degenerate case which holds little relevance in the context of spatial statistics.
\begin{assumption}[Neighborhood shape]
    \label{asm_nbd_shape}
    There exists a constant {$c\geq 1$, such that for each neighborhood $\B_s$, $\max_{1\le r\le v}(n_{s,r}^+ - n_{s,r}^-)\le c \min_{1\le r\le v}(n_{s,r}^+ - n_{s,r}^-)$.}
\end{assumption}
It is worth noting that our approach is not limited to hyper-rectangles, provided that Assumptions \ref{asm_density} and \ref{asm_nbd_shape} are satisfied. In this context, we can have a general concept of aggregation for our statistics within the spatial space.
%{\color{red} comment on the shape }

% \begin{definition}[Neighborhood norm]
%     \label{def_nbd_norm}
%     For any vector $Z\in\RR^p$ with $Z=(Z(\bbell))_{\bbell\in\B_0\cap\L_0}^{\top}$ and neighborhood $\B_s$, $1\le s\le S$, we define the neighborhood-norm of $Z$ as
%     $$|Z|_{2,\B_s} = \Big(\sum_{\bbell\in\B_s\cap\L_0}Z^2(\bbell)\Big)^{1/2}.$$
% \end{definition}

\subsection{Generalized Two-Way MOSUM}
We now introduce a generalized Two-Way MOSUM, designed to accommodate the multi-dimensional spatial space constructed in the previous section. Let $\epsilon_t = (\epsilon_t(\bbell))_{\bbell\in\B_0\cap\L_0}^\top$, $t\in\ZZ$. We denote the long-run covariance matrix of $\{\underline\epsilon_t\}_{t\in\ZZ}$ and the corresponding diagonal matrix by
\begin{equation}
    \label{eq_nonli_longrun}
    \Sigma = \big(\sigma(\bbell_1,\bbell_2)\big)_{\bbell_1,\bbell_2\in\B_0\cap\L_0},
    % = \sum_{h=-\infty}^{\infty}\Gamma(h)
    \quad \text{and} \quad \Lambda = \text{diag}\big(\sigma(\bbell)\big)_{\bbell\in\B_0\cap\L_0},
\end{equation}
respectively, where $\sigma(\bbell)=\sigma(\bbell,\bbell)$ representing the long-run variance of the component $\epsilon_t(\bbell)$. To test for the existence of breaks, we denote $\hat\mu_i^{(l)}(\bbell) = \sum_{t=i-bn}^{i-1}Y_t(\bbell)/(bn), \,\, \hat\mu_i^{(r)}(\bbell) = \sum_{t=i}^{i+bn-1}Y_t(\bbell)/(bn)$, and evaluate a jump statistic defined by
\begin{equation}
    \label{eq_nonli_V}
    V_i(\bbell)= \sigma^{-1}(\bbell)\big(\hat\mu_i^{(l)}(\bbell) - \hat\mu_i^{(r)}(\bbell)\big).
\end{equation}
Note that $V_i(\bbell)$ comprises the signal part $\EE[V_i(\bbell)]$ and the noise part $V_i(\bbell)-\EE[V_i(\bbell)]$. 
Under the null hypothesis, where no break exists and $\EE[V_i(\bbell)]=0$, we define the centering term of the $\ell^2$-aggregation of $V_i(\bbell)$ within the neighborhood $\B_s$ as
\begin{equation}
    \label{eq_center_nbd}
    c_{\B_s}= \sum_{\bbell\in\B_s\cap\L_0}c(\bbell), \quad \text{where }c(\bbell) = \text{Var}[V_i(\bbell)].
\end{equation}
Subsequently, we propose the following test statistic
\begin{equation}
    \label{eq_teststats_nbd}
    \tilde\Q_n  = \max_{1\le s\le S} \max_{bn+1\le i\le n-bn}Q_{i,\B_s},\quad\textrm{where}
    \quad Q_{i,\B_s}=\frac{1}{\sqrt{|\B_s\cap \L_0|}}\Big(\sum_{\bbell\in\B_s\cap \L_0}V_i^2(\bbell)-c_{\B_s}\Big).
\end{equation}
Under the null hypothesis $\mathcal H_0$, since $\EE [V_i(\bbell)]=0$, we can rewrite $Q_{i,\B_s}$ into
\begin{equation}
    \label{eq_test_nonli}
    Q_{i,\B_s} = \frac{1}{\sqrt{|\B_s\cap \L_0|}}\sum_{\bbell\in\B_s\cap \L_0}x_i(\bbell), \quad \text{where }x_i(\bbell) = \big(V_i(\bbell)-\EE[V_i(\bbell)]\big)^2-c(\bbell).
\end{equation}
It shall be noted that when $v=1$, the test statistic $\tilde \Q_n$ reduces to $\Q_n^{\diamond}$ in Section \ref{sec_test_spatial}.

\subsection{Dependence Structure}
Although it is quite convenient to assume that the errors are {cross-sectionally} i.i.d., it is  unrealistic to ignore the spatial dependence.
The assumption on cross-sectional independence in previous sections can be relaxed accordingly to allow for a weak spatial dependence case.  In this section, we extend the GA in Section \ref{sec_test_spatial} to the cases where the underlying errors are allowed to be cross-sectionally weakly dependent. This will allow us to  evaluate the critical values of the test statistics $\tilde\Q_n$ accordingly. 

Suppose that the stationary noise process $\{\epsilon_t(\bbell)\}_{t\in\ZZ}$ in (\ref{eq_model}) is of the form:
\begin{equation}
    \label{eq_epsilon_nonlinear}
    \epsilon_t(\bbell) = f\big(\eta_{t-k,\bbell-\bbell'}; \, k\ge 0,\, \bbell'\in \ZZ^v\big).
\end{equation}
Here $\eta_{i,\mathbf{s}}$, $i\in\ZZ$, $\mathbf{s}\in\ZZ^v$, are i.i.d. random variables, and $f(\cdot)$ is an $\RR$-valued measurable function such that $\epsilon_t(\bbell)$ is well-defined. We assume throughout the paper that $\EE[\epsilon_t(\bbell)]=0$ and $\max_{\bbell\in\B_0\cap\L_0}\|\epsilon_t(\bbell)\|_q<\infty$, for some $q\ge 4$. 
Next, we introduce the functional dependence measures to characterize the temporal and spatial dependence structure of $\epsilon_t(\bbell)$. Let {$(\eta_{i,\mathbf{s}}')_{i\in\ZZ,\,\mathbf{s}\in\ZZ^v}$} be an i.i.d. copy of  $(\eta_{i,\mathbf{s}})_{i\in\ZZ,\,\mathbf{s}\in\ZZ^v}$. Specifically, we consider the \textit{temporal} and \textit{temporal-spatial} coupled versions of $\epsilon_t(\bbell)$ defined respectively by
\begin{equation*}
    \epsilon_t^*(\bbell) = f\big(\eta_{t-k,\bbell-\bbell'}^*; k\ge0,\,\bbell'\in\ZZ\big), \quad \text{and}\quad \epsilon_t^{**}(\bbell) = f\big(\eta_{t-k,\bbell-\bbell'}^{**}; k\ge0,\,\bbell'\in\ZZ\big),
\end{equation*}
where for any $i\ge0$ and $\mathbf{s}\in\ZZ^v$,
\begin{equation*}
    \eta_{i,\mathbf{s}}^* = \begin{cases}
\eta_{i,\mathbf{s}}, \quad \text{if } i\neq0, \\
\eta_{i,\mathbf{s}}', \quad \text{if } i=0.
\end{cases}, \quad \text{and} \quad \eta_{i,\mathbf{s}}^{**} = \begin{cases}
\eta_{i,\mathbf{s}}, \quad \text{if } i\neq0 \text{ and }\mathbf{s}\neq0, \\
\eta_{i,\mathbf{s}}', \quad \text{if } i=0 \text{ or }\mathbf{s}=0.
\end{cases}
\end{equation*}
Following \textcite{wu_nonlinear_2005}, we generalize the functional dependence measures as follows
\begin{align}
    \label{eq_func_dep}
    \theta_{t,\bbell,q} =  \big\|\epsilon_t(\bbell) - \epsilon_t^*(\bbell)\big\|_q, \quad \delta_{t,\bbell,q} = \big\|\epsilon_t(\bbell) - \epsilon_t^{**}(\bbell)\big\|_q.
\end{align}
Note that $\theta_{t,\bbell,q} $ represents the change measure of dependence by perturbing solely in the temporal direction, while $\delta_{t,\bbell,q} $ denotes the  counterpart which perturbs in both the temporal and spatial directions. 
% \textcolor{blue}{(relationship between $\theta$ and $\delta$...)}

To account for the temporal and cross-sectional dependence structure of $\{\epsilon_t(\bbell)\}_{t\in\ZZ}$, we shall impose the following assumptions on $\theta_{t,\bbell,q}$ and $\delta_{t,\bbell,q}$. The assumptions essentially require the algebraic decay of dependence both in the temporal and the spatial directions and are controlling the tail behavior of the noise terms.

\begin{assumption}[Finite moment]
    \label{asm_nonli_finitemoment}
    Assume $\max_{\bbell\in\B_0\cap\L_0}\|\epsilon_t(\bbell)\|_q<\infty$, for $q\ge 8$.
\end{assumption}
\begin{assumption}[Temporal dependence]
    \label{asm_nonli_temp_dep}
    There exist some constants $C>0$ and $\beta>0$, such that, for all $h\ge0$, $\max_{\bbell\in\B_0\cap\L_0}\sum_{k\ge h}\theta_{k,\bbell,q}/\sigma(\bbell)\le C(1\vee h)^{-\beta},$ for $q\ge 8$. 
\end{assumption}

\begin{assumption}[Weak cross-sectional dependence]
    \label{asm_nonli_sec_dep}
    Assume that there exist some constants $C'>0$ and $\xi>1$, such that, for all $m\ge0$, 
    \begin{equation}
        \label{eq_nonli_sec_dep}
        \sum_{k\ge 0}\Big(\sum_{\{\bbell\in\B_0\cap\L_0:\, |\bbell|_2\ge m\}}\delta_{k,\bbell,q}^2/\sigma^2(\bbell)\Big)^{1/2} \le C'(1\vee m)^{-\xi}.
    \end{equation}
\end{assumption}
It shall be noted that we can also switch the temporal and spatial aggregation in the above assumption. Specifically, (\ref{eq_nonli_sec_dep}) also implies 
\begin{equation}
    \label{eq_nonli_sec_dep_11}
    \Big(\sum_{\{\bbell\in\B_0\cap\L_0:\, |\bbell|_2\ge m\}}\Delta_{0,\bbell,q}^2/\sigma^2(\bbell)\Big)^{1/2}
        \le C'(1\vee m)^{-\xi}, \quad \text{where } \Delta_{0,\bbell,q} = \sum_{k\ge 0}\delta_{k,\bbell,q}.
\end{equation}
To see this, for any $n\in\NN$, we denote the partial sum $\Delta_{n,\bbell,q}=\sum_{k= 0}^n\delta_{k,\bbell,q}$. Let $\underline{\Delta}_{n,q}=(\Delta_{n,\bbell,q})^{\top}_{\{\bbell\in\ZZ^v:\, |\bbell|_2\ge m\}}$, and $\underline{\delta}_{k,q}=(\delta_{k,\bbell,q})^{\top}_{\{\bbell\in\ZZ^v:\, |\bbell|_2\ge m\}}$. By the triangle inequality, $|\underline{\Delta}_{n,q}|_2 = |\sum_{k=0}^n\underline{\delta}_{k,q}|_2\le\sum_{k=0}^n|\underline{\delta}_{k,q}|_2$. Since $\sum_{k\ge0}|\underline{\delta}_{k,q}|_2<\infty$, we let $n\rightarrow\infty$ and achieve the desired result. Furthermore, (\ref{eq_nonli_sec_dep_11}) also indicates the decay of cross-sectional long-run correlation as depicted in Lemma \ref{lemma_nonli_longrun_corr}.

\begin{lemma}[Decay of long-run correlation]
\label{lemma_nonli_longrun_corr}
Assume that condition (\ref{eq_nonli_sec_dep_11}) holds. Then, for any $\bbell_1,\bbell_2\in\B_0\cap\L_0$, the long-run correlation between $\epsilon_t(\bbell_1)$ and $\epsilon_t(\bbell_2)$, denoted by $\tilde\rho(\bbell_1,\bbell_2)$, decays at a polynomial rate as $|\bbell_1-\bbell_2|_2$ increases, that is
\begin{equation}
    \label{eq_nonli_sec_dep2}
    \tilde\rho(\bbell_1,\bbell_2)=\sigma(\bbell_1,\bbell_2)/\big(\sigma(\bbell_1)\sigma(\bbell_2)\big)= O\big\{|\bbell_1-\bbell_2|_2^{-2\xi}\big\}.
\end{equation}
\end{lemma}
It is worth noticing that the decay assumption regarding spatial correlation is widely prevalent in spatial statistics. See for example, \textcite{stein1999interpolation} and \textcite{rasmussen_gaussian_2006} follow a similar pattern that the covariance between variables decreases as their corresponding distance in the input space increases.

\subsection{Gaussian Approximation with Weak Cross-Sectional Dependence}

This subsection is devoted to the GA in the high-dimensional setting under the above-mentioned general framework. In the case when $p$ is fixed, we provide an invariance principle in Appendix \ref{subsec_fixed_p}. When $p$ grows to infinity as $n$ increases, to derive the limiting distribution of the test statistic $\tilde\Q_n$ under the null, we introduce the centered Gaussian random vector 
\begin{equation}
    \label{eq_Z_nonli}
    \tilde\Z=(\tilde\Z_{bn+1,1},\ldots,\tilde\Z_{n-bn,1},\ldots,\tilde\Z_{bn+1,S},\ldots,\tilde\Z_{n-bn,S})^{\top},
\end{equation}
with the covariance matrix
\begin{equation}
    \label{eq_cov_Z_nonli}
    \tilde\Xi=(\tilde\Xi_{i,s,i',s'})_{1\le i,i'\le n-2bn,1\le s,s'\le S}.
\end{equation}
Let $\pi_{s,s',\bbell_1,\bbell_2}=(|\B_s\cap\L_0||\B_{s'}\cap\L_0|)^{-1/2}\One_{\bbell_1,\bbell_2\in\B_s\cap\B_{s'}\cap\L_0}$, and $\tilde\Xi_{i,s,i+\zeta bn,s'}$ equals to
\begin{equation}
    \label{eq_nonli_cov}(bn)^{-2}\sum_{\bbell_1,\bbell_2\in\B_0\cap\L_0}\pi_{s,s',\bbell_1,\bbell_2}
    \begin{cases}
        (15\zeta^2-20\zeta+8)\tilde\rho^2(\bbell_1,\bbell_2) + 3\zeta^2-4\zeta, & 0<\zeta \le 1, \\
        (3\zeta^2-12\zeta+12)\tilde\rho^2(\bbell_1,\bbell_2) -\zeta^2+4\zeta -4, & 1< \zeta\le 2,\\
        0, & \zeta>2.
    \end{cases}
\end{equation}
We defer the detailed evaluation of (\ref{eq_nonli_cov}) to Lemma \ref{lemma_cov_thm3}. Note that, if for all $\bbell,\bbell_1,\bbell_2\in\B_0\cap\L_0$, $\tilde\rho(\bbell,\bbell)=1$ and $\tilde\rho(\bbell_1,\bbell_2)=0$,  $\bbell_1\neq\bbell_2$, which denotes the case with no spatial dependence, then (\ref{eq_nonli_cov}) is the same as (\ref{eq_cov_local_G}). Recall $\N$ defined in (\ref{eq_node_pair}). We denote each element in $\tilde\Z$ by $\tilde\Z_{\varphi}$, where $\varphi=(i,s)\in\N$. Similar to the cross-sectionally independent case, we can approximate the limiting distribution of $\tilde\Q_n$ under the null by the one of $\max_{\varphi}\tilde\Z_{\varphi}$. We refer to Remark \ref{rmk_block} in Appendix \ref{sec_proofs} for the proof strategies based on the block approximation. 

\begin{theorem}[GA with weak cross-sectional dependence]
    \label{thm3_nonli}
    Suppose that Assumptions \ref{asm_density}--\ref{asm_nonli_sec_dep} hold. Then, under the null hypothesis, for $\tilde\Delta_0=(bn)^{-1/3}\log^{2/3}(nS)$, $\tilde\Delta_1 = c_{p,n}^{-(q-4)/(3q)}$, $\tilde\Delta_2 = c_{p,n}^{-1/(8v)}\log(pn)$, where
    \begin{equation}
        \label{eq_cpn}
        c_{p,n}=p^{\frac{-2}{q-4}}B_{\text{min}}^{\frac{q}{q-4}}(nS)^{-(\frac{4}{q-4}+\frac{2v}{q\xi})}\big(\log(pn)\big)^{-(\frac{(2+q)v}{2q\xi}+\frac{3q}{q-4})},
    \end{equation}
    we have
    $$\sup_{u\in\RR}\Big|\PP(\tilde\Q_n\le u)-\PP\big(\max_{\varphi\in\N} \tilde\Z_{\varphi}\le u\big)\Big| \lesssim \tilde\Delta_0 + \tilde\Delta_1 + \tilde\Delta_2,$$
    where $\N$ is defined in (\ref{eq_node_pair}), and the constant in $\lesssim$ is independent of $n,p,b$ and $S$.
    If in addition, $\log(nS)=o\{(bn)^{1/2}\}$ and $\log^{8v}(pn)=o(c_{p,n})$, then $$\sup_{u\in\RR}\Big|\PP(\tilde\Q_n\le u)-\PP\big(\max_{\varphi\in\N} \tilde\Z_{\varphi}\le u\big)\Big|\rightarrow0.$$
\end{theorem}
\begin{remark}[Comparison with Theorem \ref{thm2_constanttrend}]
    Note that when $v=1$, we have $|\L_{\text{min}}|=B_{\text{min}}$, and if $\xi\rightarrow\infty$, it indicates the cross-sectional independence setting. Hence, Theorem \ref{thm2_constanttrend} can be viewed as a special case of Theorem \ref{thm3_nonli}. Specifically, $c_{p,n}\rightarrow\infty$ in (\ref{eq_cpn}) boils down to the condition (\ref{eq_thm21_o1}), which implies $c_{p,n}^{q-4}=p^{-2}B_{\text{min}}^q(nS)^{-4}\log^{-3q}(pn)\rightarrow\infty$, and we can achieve the same approximation rate up to a logarithm term.
\end{remark}

Moving on to the alternative hypothesis, we can set the detection threshold $\tilde\omega$ as the critical value of $\max_{\varphi\in\N} \tilde\Z_{\varphi}$ determined by the Gaussian limiting distribution presented in Theorem \ref{thm3_nonli}. Specifically, we set $\tilde\omega$ as $\inf_{r\ge0}\{r:\PP(\max_{\varphi\in\N} \tilde\Z_{\varphi}>r)\le \alpha\}$, for significant level $\alpha\in(0,1)$. We reject the null hypothesis if $\tilde\Q_n>\tilde\omega$. For any time point $i$ that satisfies $|i-\tau_k|\le bn$, we define the weighted break as $d_i(\bbell)=(1-|i-\tau_k|/(bn))\sigma^{-1}(\bbell)\gamma_{k,\bbell}$. Under the alternative, where there exists $i$ and $\bbell$ such that $d_i(\bbell)\neq0$, we refer to Corollary \ref{cor_nonli_power} in Appendix \ref{sec_power_append} for the power limit of our test. Further algorithm for detecting and identifying breaks can be developed accordingly.

\section{Application}\label{sec_data}
This section is devoted to the real-data analysis to illustrate our proposed method for multiple change-point detection. We apply Algorithm \ref{algo1} to a stock-return dataset and use Algorithms \ref{algo1} and \ref{algo2} to a COVID-19 dataset. Due to the space limit, we defer the results of the stock-return data to Appendix \ref{subsec_stock_append}.

Analyzing 812 days of daily COVID-19 case numbers in the US, we identified three significant breaks (Figure \ref{fig_covid_main} top): March 2020 (first outbreak), October 2020 (Delta variant), and December 2021 (Omicron variant). Further, we consider four geographic regions of the US as per the guidelines of the CDC: Northeast, Midwest, South, and West. A map of these four regions is available in Figure \ref{fig_region_us} in Appendix \ref{subsec_covid_append}. By our algorithm, each region exhibited different break time stamps, with the Northeast and West experiencing early outbreaks due to major international airports. The Midwest was the first to encounter the Delta variant, while the Northeast initially faced the Omicron variant. Our detection algorithm effectively captured these variations (Figure \ref{fig_covid_main} bottom), demonstrating the efficacy of our proposed testing procedures in identifying breaks over time and across diverse locations. For more detailed information, please refer to Appendix \ref{subsec_covid_append}.

\vspace{-0.3cm}
\begin{figure}[!htbp]
\centering
\includegraphics[width=.95\linewidth]{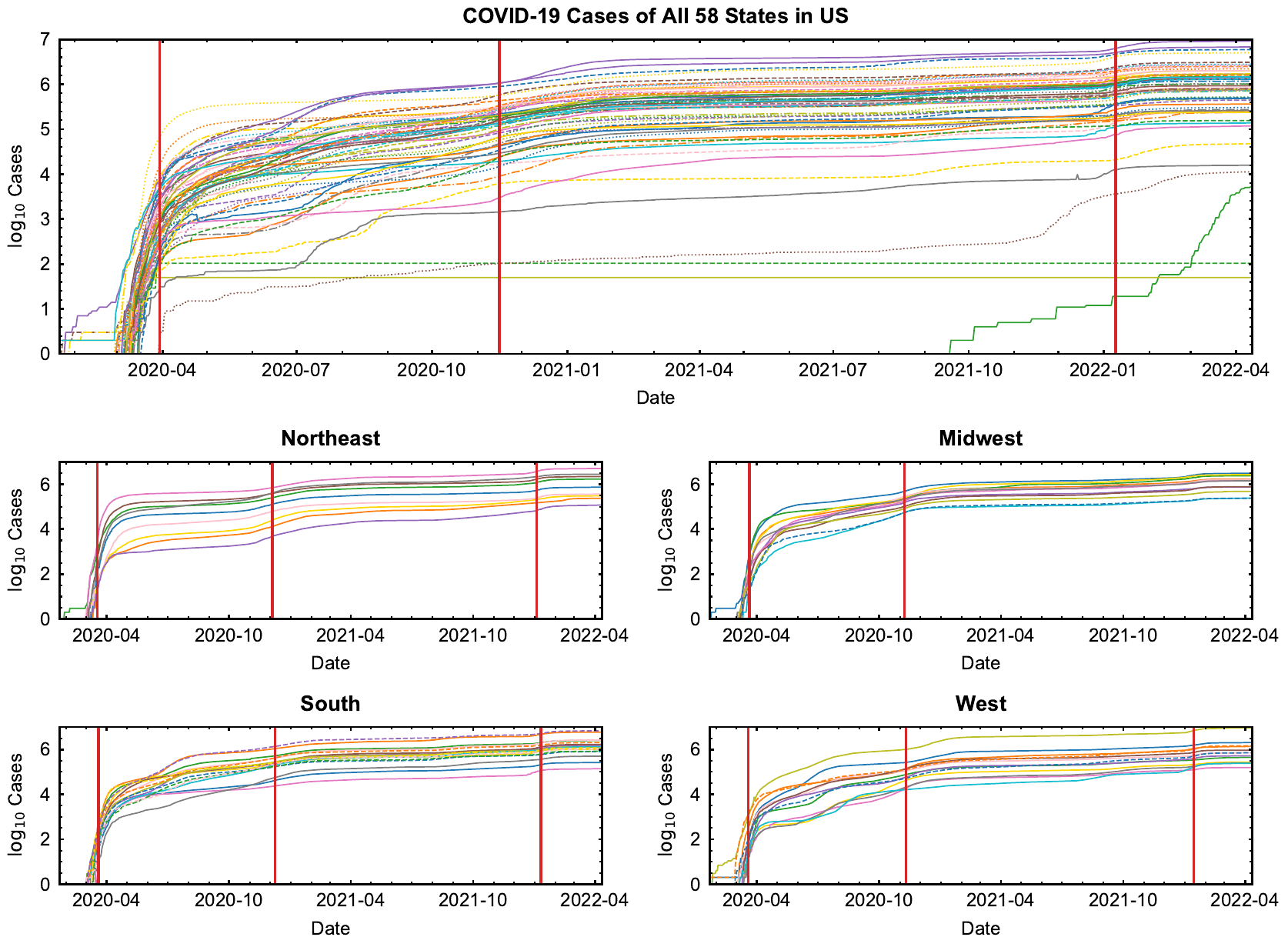}
\vspace{-0.3cm}
\caption{\textbf{Top:} Algorithm \ref{algo1} detected three change points which are 2020.03.30, 2020.11.16 and 2022.01.09. \textbf{Bottom:} Algorithm \ref{algo2} found different change points for four regions: Northeast (2020.03.18, 2020.12.05, 2022.01.04); Midwest (2020.03.21, 2020.11.08); South (2020.03.20, 2020.12.09, 2022.01.10); West (2020.3.19, 2020.11.10, 2022.01.14).}
\label{fig_covid_main}
\end{figure}

\clearpage

\printbibliography

\newpage

\begin{appendices}
\section{Simulation and Application}\label{sec_sim}

\subsection{Estimation of Long-Run Covariance {Matrices}}\label{subsec_longrun}

Recall that in previous sections, we assumed that the long-run covariance matrix $\Sigma$ is known, which, however, is hardly realistic in practice. Therefore, we utilize a robust M-estimation introduced by \textcite{catoni_challenging_2012} and extended by \textcite{chen_inference_2022} for the long-run covariance matrix to ensure our Gaussian approximation theory can still work even when the long-run covariance matrix is not given. {Note that the classical covariance matrix estimation procedure is not directly applicable in our setting due to possible breaks.}

First, we group $n$ observations into {$N_0=\lfloor n/m\rfloor-1$} blocks with each block {having same size} $m\in\NN$. We denote the index set of the observations in the $k$-th block by $\M_k=\{t\in\NN: km+1\le t\le (k+1)m\}$ and define the average of the observations in the block $\M_k$ as
$$\psi_k=\sum_{t\in\M_k}Y_t/m.$$
Under the classic setting without change points, we {can} estimate the long-run covariance matrix by
$$\sum_{k=1}^{N_0}(m/2)(\psi_k-\psi_{k-1})(\psi_k-\psi_{k-1})^{\top}/N_0,$$
where the difference between $\psi_k-\psi_{k-1}$ is aimed to cancel out the trends, since our trend function $\mu(\cdot)$ is piece-wise constant. But when there exist break points which cannot be canceled out by simply taking the difference, this estimator would fail as it is contaminated by the breaks. Hence, we consider a robust version of M-estimator (\cite{chen_inference_2022}) that fits well to our interest to estimate the long-run covariance matrix.

To this end, we {let} $\psi_k=(\psi_{k,1},\psi_{k,2},\ldots,\psi_{k,p})^{\top}$, and for $k=1,\ldots,N_0$, we define 
\begin{equation}
    \label{eq_longrun_element}
    \hat\sigma_{i,j,k}^2=(m/2)(\psi_{k,j}-\psi_{k-1,j})(\psi_{k,j}-\psi_{k-1,j}).
\end{equation}
Now, we introduce the M-estimation zero function of the long-run covariance matrix to be, for some $\alpha_{i,j}>0$,
\begin{equation}
    \label{eq_longrun_zero}
    h_{i,j}(u)=\sum_{k=1}^{N_0}\frac{1}{N_0}\phi_{\alpha_{i,j}}(\hat\sigma_{i,j,k}^2-u),
\end{equation}
where $\phi_{\alpha}(x)=\alpha^{-1}\phi(\alpha x)$ and 
\begin{equation}
    \label{eq_longrun_phi}
    \phi(x)= \begin{cases}
        \log(2), \qquad & x\ge 1,\\
        -\log(1-x+x^2/2), \qquad & 0\le x< 1,\\
    {\log(1+x+x^2/2)}, \qquad & -1\le x< 0,\\
        -\log(2), \qquad & x< -1.\\
    \end{cases}
\end{equation}
As indicated by \textcite{catoni_challenging_2012}, {the function $|\phi(x)|$} is bounded by $\log(2)$, and it is also Lipschitz continuous with the Lipschitz constant bounded by 1. Additionally, the non-decreasing influence function $\phi(x)$ has neat envelopes in the form that
\begin{equation}
    -\log(1-x+x^2/2) \le \phi(x) \le \log(1+x+x^2/2).
\end{equation}

To {obtain} a consistent estimator of the long-run covariance matrix $\Sigma$, we solve the equation $h_{i,j}(u) = 0$ for each $1\le i,j\le p$, and use the root to be the estimator of the element with indices $(i,j)$ in the long-run covariance matrix. We denote this estimator by $\hat\sigma_{i,j}^2$. Since the equation $h_{i,j}(u) = 0$ may have more than one solution, in which case any of them can be used to define $\hat\sigma_{i,j}^2$. Finally, we collect the estimates of all the elements and combine them into the long-run covariance matrix,
\begin{equation}
    \label{eq_longrun_est}
    \hat\Sigma=(\hat\sigma_{i,j}^2)_{1\le i,j\le p},\quad \text{and } \hat\Lambda=\text{diag}(\hat\sigma_{1,1},\ldots,\hat\sigma_{p,p}).
\end{equation}
In particular, we define $\bar\sigma_i^2=2\sum_{N_0/4\le k\le 3N_0/4}\hat\sigma_{i,i,k}/N_0$ and set $\alpha_{i,j}$ in expression (\ref{eq_longrun_zero}) to be $\bar\sigma_i\bar\sigma_j\sqrt{m/n}$.
Here we present the consistency result of the estimated long-run covariance matrix.

\begin{theorem}[Consistency of the estimated long-run covariance matrix]
    \label{thm_cov_precision}
    Suppose that conditions in Theorem \ref{thm1_constanttrend} hold and the number of breaks $K=o\{(n\log(np))^{1/4}\}$. We take $m=\sqrt{n/\log(np)}$. {Then}
    $$|\Lambda^{-1}(\hat\Sigma-\Sigma)\Lambda^{-1}|_{\max}=O_{\PP}\big\{(n^{-1}\log(np))^{\frac{\beta}{\beta+1}} + n^{-1/4}\log^{1/2}(np) + K(n\log(np))^{-1/4}\big\},$$
    where $|A|_{\max}=\max_{1\le i,j\le n}|a_{i,j}|$, for matrix $A=(a_{i,j})_{1\le i,j\le n}$.
\end{theorem}

% By the above theorem, we have $\max_{1 \le j \le p}|\hat{\sigma}_{j,j}-\sigma_{j,j}|/\sigma_{j,j}=o_{\PP}(1)$. 

% Then, by Slutsky's theorem, when we replace the long-run variance matrix $\Lambda$ in Theorems \ref{thm1_constanttrend}--\ref{thm3_nonli} by $\hat\Lambda$, the approximation results still hold.

\begin{proof}[Proof of Theorem \ref{thm_cov_precision}]
The proof strategies essentially follow Proposition 2.4 in \textcite{catoni_challenging_2012} and Theorem 5 in \textcite{chen_inference_2022}. The main difference is that we {relax} the assumption which restricts the number of breaks to be finite in \textcite{chen_inference_2022}. In other words, we allow {$K\rightarrow\infty$.} First, we consider the set
\begin{equation}
    \T = \{k \mid \M_k \text{ or }\M_{k-1} \text{ contains change breaks}\}.
\end{equation}
One can see that $|\T|\le 2K_0$. {Let} $N_0^* = N_0 - |\S|$. For the block estimators without change points, we define the corresponding M-estimation zero function as
\begin{equation}
    h_{i,j}^*(u) = \frac{1}{N_0^*}\sum_{k\notin\T}\phi_{\alpha_{i,j}}(\hat \sigma_{i,j,k}-u).
\end{equation}
For $\hat\sigma_{i,j,k}$, we define the block-mean and block-variance respectively by
\begin{equation}
    \tilde \sigma_{i,j} = \frac{1}{N_0^*}\sum_{k\notin\T}\EE\hat\sigma_{i,j,k}, \quad v_{i,j}^2 = \frac{1}{N_0^*}\sum_{k\notin\T}\EE\hat\sigma_{i,j,k}^2 - \tilde \sigma_{i,j}^2.
\end{equation}
For $1\le i, j\le p$, we define the envelope bound functions
\begin{align}
    B_{i,j}^+(u,x) & = \tilde \sigma_{i,j} - u +\alpha_{i,j}\big[(\tilde \sigma_{i,j}-u)^2 + v_{i,j}^2\big]/2 + x, \nonumber \\
    B_{i,j}^-(u,x) & = \tilde \sigma_{i,j} - u -\alpha_{i,j}\big[(\tilde \sigma_{i,j}-u)^2 + v_{i,j}^2\big]/2 - x.
\end{align}
By applying the result of Step 1 in Theorem 5 \textcite{chen_inference_2022}, we can bound the expected influence functions without change points in the following way:
\begin{equation}
    \label{eq_thm5_step1}
    B_{i,j}^-(u,0) \le \EE h_{i,j}^*(u) \le B_{i,j}^+(u,0), \quad \text{for all }1\le i, j\le p.
\end{equation}
Further, following Step 2 in Theorem 5 \textcite{chen_inference_2022}, we can show that the estimated influence function $h_{i,j}^*(u)$ is concentrated around its mean $\tilde \sigma_{i,j}$, that is, for some constant $c_0>0$ and $x\gtrsim \sqrt{N_0^*\log(N_0^*)}$, we have the tail probability
\begin{align}
    \label{eq_thm5_step2}
    & \quad \sum_{i,j=1}^p\PP\Big(\sup_{|u-\sigma_{i,j}|\le c_0}|h_{i,j}^*(u) - \EE h_{i,j}^*(u)| \ge x\sigma_{i,i}^{1/2}\sigma_{j,j}^{1/2}/N_0^*\Big) \nonumber \\
    & \lesssim p^2\big[N_0^*\log^{q/4}(n)x^{-q/2} + e^{-x^2/(cN_0^*)}\big],
\end{align}
where $c$ and the constant in $\lesssim$ are independent of $n$ and $p$. Next, we shall bound the difference between the estimator $\tilde\sigma_{i,j}$ and the true long-run covariance $\sigma_{i,j}$. Since $\tilde\sigma_{i,j}$ is defined on the blocks without change points, we can similarly follow Step 3 in Theorem 5 in \textcite{chen_inference_2022}. The only difference is that in our setting, we have assumed that the trend function is piece-wise constant while they assumed a smooth trend. Therefore, we can obtain that
\begin{equation}
    \label{eq_thm5_step3}
    \max_{1\le i,j\le p}|\tilde \sigma_{i,j} - \sigma_{i,j}| = O(m^{-\beta/(\beta+1)}\sigma_{i,i}^{1/2}\sigma_{j,j}^{1/2}), \quad \text{and } v_{i,j}^2=O(\sigma_{i,i}\sigma_{i,i}).
\end{equation}
Lastly, we consider the blocks with change points. Since $|\T|\le 2K$ and $|\phi(\cdot)|_{\infty}\le\log(2)$, for any $1\le i,j\le p$, it follows that
\begin{equation}
    |N_0h_{i,j}(u)/N_0^* - h_{i,j}^*(u)| \le 2\log(2)K/(\alpha_{i,j}N_0^*).
\end{equation}
This, together with expressions (\ref{eq_thm5_step1}) and (\ref{eq_thm5_step2}) with $x=\sqrt{N_0^*\log(np)}$, yields, with probability tending to 1, for all $1\le i,j\le p$, and $|u-\sigma_{i,j}|\le c_0$,
\begin{equation}
    \label{eq_thm5_B_bounds}
    B_{i,j}^-(u,\Delta) \le N_0N_0^*h_{i,j}(u) \le B_{i,j}^+(u,\Delta),
\end{equation}
where
\begin{equation}
    \Delta = (x/N_0^*)\sigma_{i,i}^{1/2}\sigma_{j,j}^{1/2} = \sigma_{i,i}^{1/2}\sigma_{j,j}^{1/2}\sqrt{\log(np)/N_0^*}.
\end{equation}
It shall be noted that to ensure the existence of real roots for $B_{i,j}^+(u,\Delta)$, we need to have
\begin{equation}
    \label{eq_thm5_delta}
    \alpha_{i,j}^2v_{i,j}^2 + 2\alpha_{i,j}\Delta \le 1.
\end{equation}
Suppose that expression (\ref{eq_thm5_delta}) holds, and we denote the smaller real root by $u^+$, then we have
\begin{equation}
    u^+ \le \tilde\sigma_{i,j} +\alpha_{i,j}v_{i,j}^2 + 2\Delta.
\end{equation}
We take $\alpha_{i,j}^*=\alpha_{i,j}\sigma_{i,i}^{1/2}\sigma_{j,j}^{1/2}$. Since $\sigma_{i,i}$ is lower bounded for any $i$, it follows from expression (\ref{eq_thm5_step3}) that 
\begin{equation}
    \sigma_{i,i}^{-1/2}\sigma_{j,j}^{-1/2}(u^+-\sigma_{i,j}) = O\Big\{m^{-\beta/(\beta+1)} + \alpha_{i,j}^* + \sqrt{\log(np)/N_0^*} + mK/(\alpha_{i,j}^*n)\Big\}.
\end{equation}
We can obtain a similar bound for the larger real root $u^-$. Then, when expression (\ref{eq_thm5_B_bounds}) is satisfied, we can bound the long-run estimator by $u^-\le \hat\sigma_{i,j}\le u^+$. We let $\alpha_{i,j}^* = (m/n)^{1/2}$. Then, with probability tending to 1, we have
\begin{equation}
    \sigma_{i,i}^{-1/2}\sigma_{j,j}^{-1/2}(\hat\sigma_{i,j}-\sigma_{i,j}) \lesssim m^{-\beta/(\beta+1)} + \alpha_{i,j}^* + \sqrt{\log(np)/N_0^*} + mK/(\alpha_{i,j}^*n),
\end{equation}
uniformly over $1\le i,j\le p$, where $N_0^*\ge \lfloor (n-m)/m\rfloor -2K$. Also note that there exist some constants $c_1,c_2>0$ such that $c_1\le \bar\sigma_{j,j}/\sigma_{j,j}\le c_2$ for any $i,j$ with probability tending to 1, which completes the proof.
\end{proof}

\subsection{Simulation Under the Null}\label{subsec_sim_null}

First, we show the Gaussian approximation results under the null. Three different models are considered to generate the errors $\epsilon_t\in\RR^p$, which are 
(i) i.i.d., (ii) AR(1), and (iii) MA($\infty$). Specifically, we consider the AR(1) model $\epsilon_{t,j}=\phi_j\epsilon_{t-1,j} + \eta_t$, where $\phi_1,\ldots,\phi_p$ uniformly range from 0.6 to 0.9, and $\eta_i$ follow two different types of distributions, which are $N(0,1)$ and $t_9$. {For each component series $\epsilon_{t,j}$,} we assume that it takes the form
$$\epsilon_{t,j}=\sum\limits_{k=0}^Ta_{k,j}\eta_{t-k,j}.$$
Throughout all the simulation study in this paper, we let $T=300$ which is larger than the number of observations $n=200$. Furthermore, we denote $a_{\cdot,j}=(a_{1,j},\ldots,a_{T,j})^{\top}$ and generate $a_{\cdot,j}$ by computing $a_{\cdot,j} = \psi_j(1,2,\ldots,T)^{-\beta}$, where $(\psi_1,\ldots,\psi_p)$ uniformly range from 0.5 to 0.9. We set $\beta$ to be 2 which determines the decay speed of the temporal dependence.

% Note that since in this section, we provide the Gaussian approximation under the null, the window size $bn$ will only affect the precision of estimated long-run covariance matrix.

First, we intuitively present the distributions of our test statistic $\Q_n$ and the corresponding Gaussian counterparts based on 1000 Monte-Carlo replicates, respectively. Note that in practice, when a dataset is given, one only needs to calculate the test statistics $\Q_n$ once, and calculate the threshold value by generating Gaussian counterparts via 1000 Monte-Carlo replicates and computing the critical value. For each new data scenario, we need to simulate a new critical value as the detection threshold. Here, we generate the distributions of $\Q_n$ (i.e. "Tn" in Figure \ref{fig_sim_null}) and Gaussian counterparts (i.e. "GS" in Figure \ref{fig_sim_null}). We let the sample size $n=200$, $p=50$ and window size $bn=30$. As shown in Figure \ref{fig_sim_null}, for all three different models and two different types of tails, the null distribution of the test statistic $\Q_n$ (\textit{purple}) coincides with the one of its Gaussian counterparts (\textit{orange}). The \textit{pink} area is the overlapped part of two distributions. {The histograms show} that the two distributions are in general similar to the large overlapped areas, which supports the the Gaussian approximation theorem for our test statistics. 

\begin{figure}[!htbp]
\centering
\begin{subfigure}{.45\textwidth}
  \centering
  \includegraphics[width=.8\linewidth]{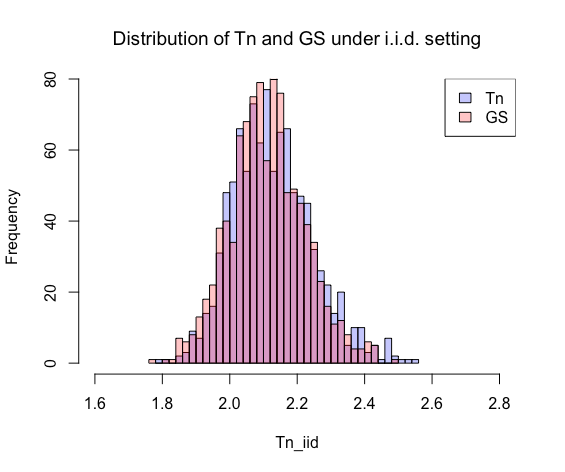}  
  \caption{i.i.d. errors $\sim N(0,1)$.}
  \label{fig_iid_normal}
\end{subfigure}
\begin{subfigure}{.45\textwidth}
  \centering
  \includegraphics[width=.8\linewidth]{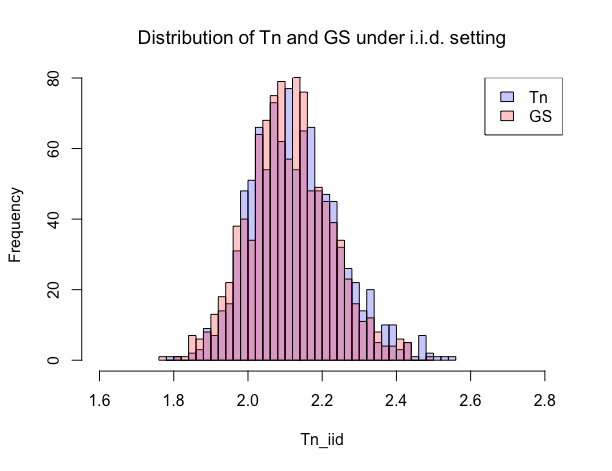}  
  \caption{i.i.d. errors $\sim t_9$.}
  \label{fig_iid_t5}
\end{subfigure}

\begin{subfigure}{.45\textwidth}
  \centering
  \includegraphics[width=.8\linewidth]{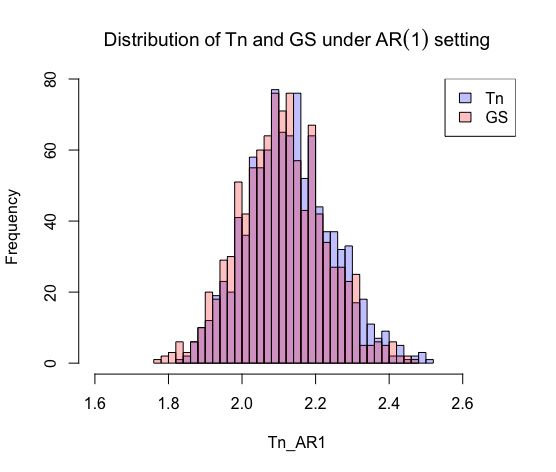}  
  \caption{AR(1) errors $\sim N(0,1)$.}
  \label{fig_AR1_normal}
\end{subfigure}
\begin{subfigure}{.45\textwidth}
  \centering
  \includegraphics[width=.8\linewidth]{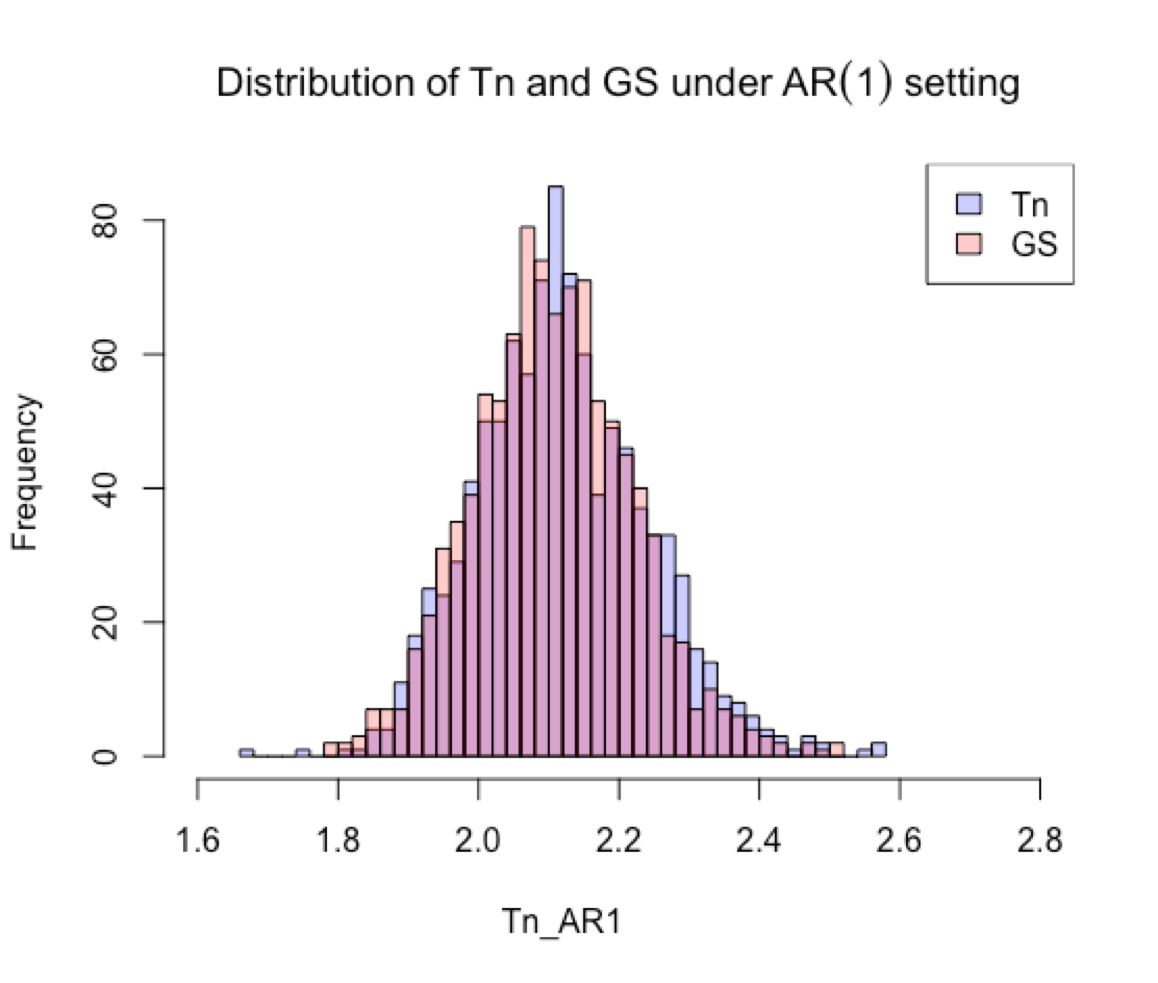}  
  \caption{AR(1) errors $\sim t_9$.}
  \label{fig_AR1_t5}
\end{subfigure}

\begin{subfigure}{.45\textwidth}
  \centering
  \includegraphics[width=.8\linewidth]{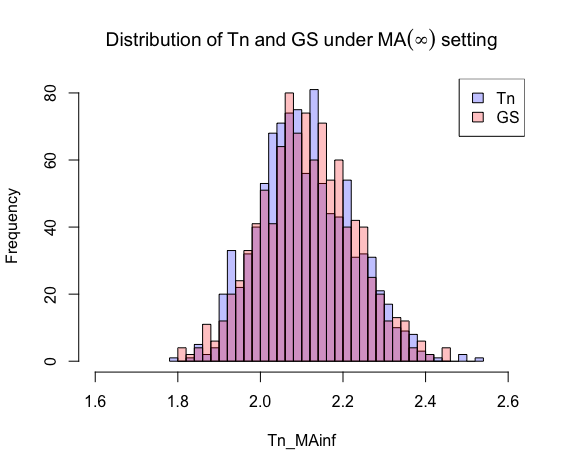}  
  \caption{MA($\infty$) errors $\sim N(0,1)$.}
  \label{fig_MA_normal}
\end{subfigure}
\begin{subfigure}{.45\textwidth}
  \centering
  \includegraphics[width=.8\linewidth]{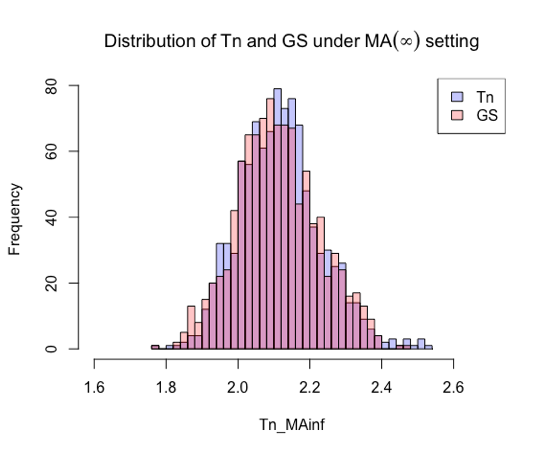}  
  \caption{MA($\infty$) errors $\sim t_9$.}
  \label{fig_MA_t5}
\end{subfigure}
\caption{Distributions of test statistics $\Q_n$ and Gaussian counterpart $G_n$ under the null in different settings ($n=200,p=50,bn=30$).}
\label{fig_sim_null}
\end{figure}

Further, we report the empirical sizes for the distributions of the test statistics $\Q_n$ based on $95\%$ quantiles of the Gaussian counterparts. We again consider the three different models and two different types of tails that we used for Figure \ref{fig_sim_null}. Let the sample size $n=200$ and the window size $bn=30$. We consider three different numbers of components, that is $p=50$, 200 and 400. As shown in Table \ref{table_sim_null_size}, the empirical sizes are close to 0.05 under different scenarios. These results, again, demonstrate the validity of our proposed Gaussian approximation theorem.

\begin{table}[!htbp]
\begin{center}
  \caption{Averaged empirical sizes under the null over 1000 Monte-Carlo replicates ($n=200,bn=30$).}
  \label{table_sim_null_size}
  \renewcommand{\arraystretch}{2}
  \begin{tabular}{c | c  c || c | c  c || c | c  c } 
  \hline\hline
   $p=50$ & $N(0,1)$ & $t_9$ & $p=200$ & $N(0,1)$ & $t_9$ & $p=400$ & $N(0,1)$ & $t_9$  \\ [1ex] 
  \hline
  i.i.d. & 0.0501 & 0.0503 & i.i.d. & 0.0498 & 0.0494 & i.i.d. & 0.0502 & 0.0507 \\ [1ex]
  \hline
  AR(1) & 0.0507 & 0.0495 & AR(1) & 0.0505 & 0.0506 & AR(1) & 0.0493 & 0.0510 \\ [1ex]
  \hline
  MA($\infty$) & 0.0489 & 0.0483 & MA($\infty$) & 0.0486 & 0.0477 & MA($\infty$) & 0.0481 & 0.0471 \\ [1ex]
  \hline\hline
  \end{tabular}
\end{center}
\end{table}

\subsection{Simulation with Change Points}\label{subsec_sim_alter}

Now we present the simulation study of our proposed change-point detection method. We shall start with the single change-point case to show the precision of break-time estimation. Figure \ref{fig_sim_thm1} illustrates the estimation of a single break located at the time point $\tau=50$ which is the center of the x-axis in each histogram. The number of observations is $n=200$. We consider three different numbers of components $p=50$, $200$ and $400$ (where $p=400$ is larger than the sample size $n=200$, and this case can be considered as a high-dimensional setting), two different window sizes $bn=20$ or $30$, three jump sizes 1, 0.6 and 0.3, and the decay rate of the dependence of the time series is $\beta=2$. All the errors are MA($\infty$) with $t_9$ tails. Throughout this section we consider $1000$ simulation samples.

\begin{figure}[!htbp]
\centering
\begin{subfigure}{1\textwidth}
  \centering
  \includegraphics[width=1\linewidth]{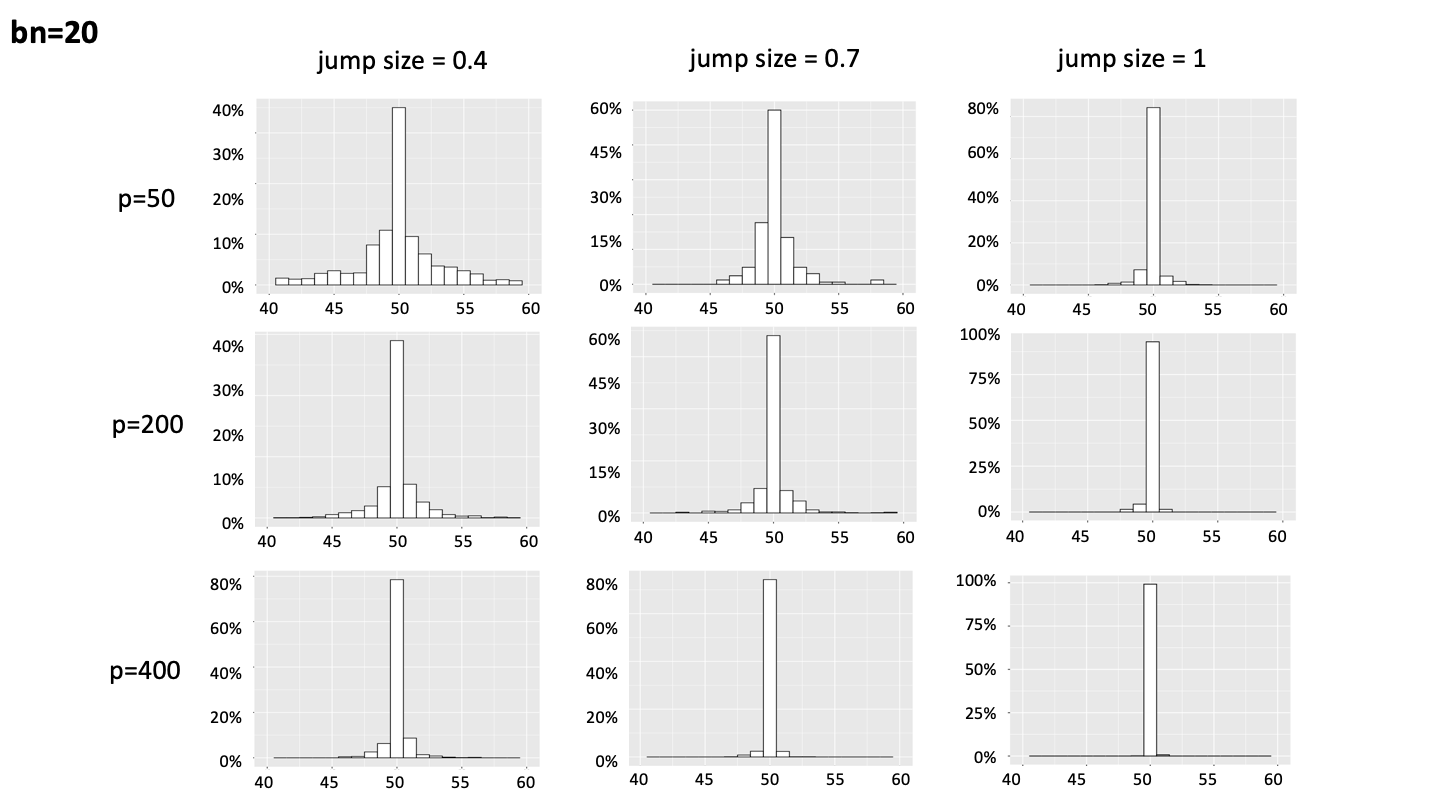}  
  \caption{$bn=20$.}
  \label{fig_thm1_bn20}
\end{subfigure}
\begin{subfigure}{1\textwidth}
  \centering
  \includegraphics[width=1\linewidth]{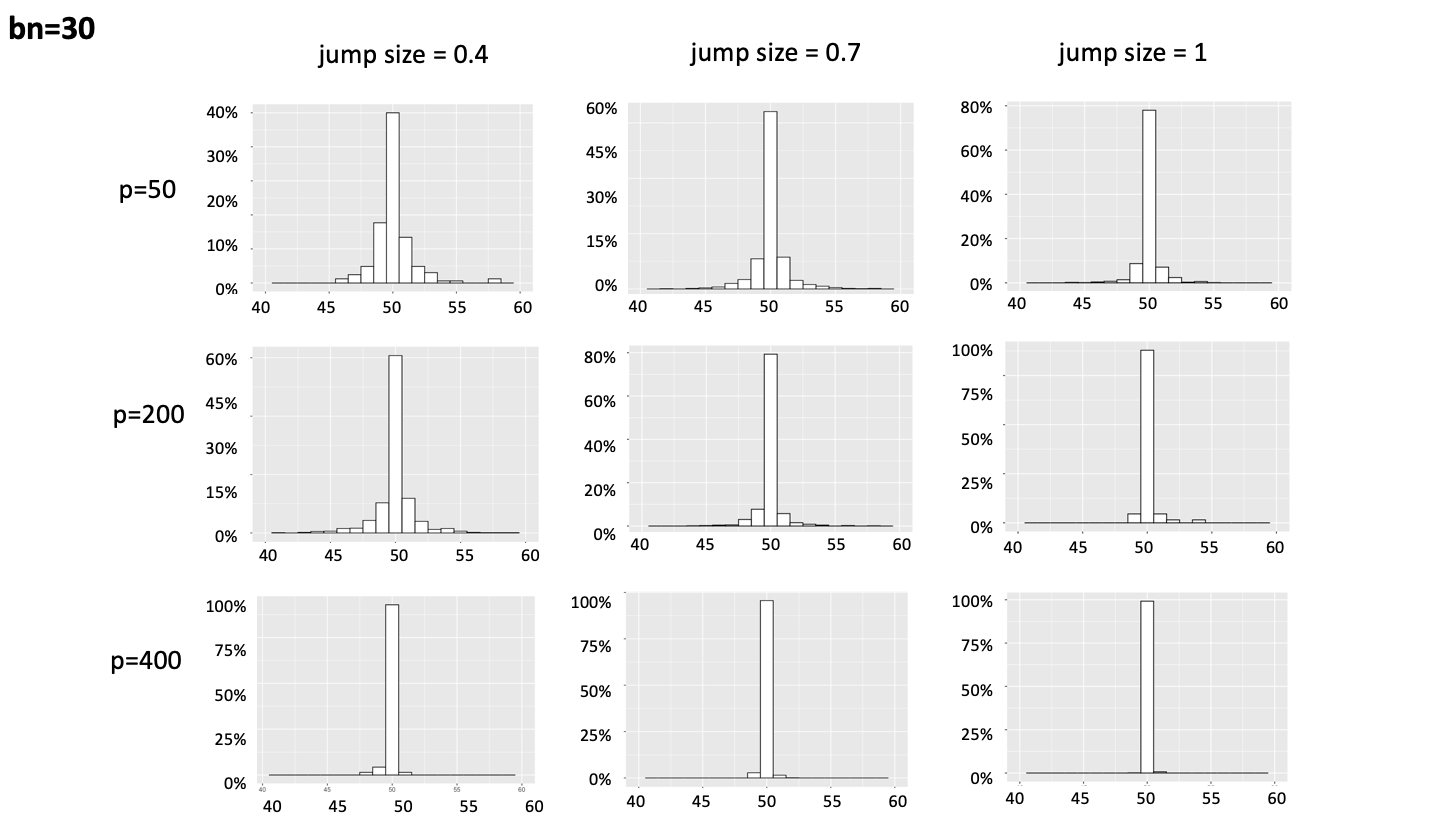}  
  \caption{$bn=30$.}
  \label{fig_thm1_bn30}
\end{subfigure}
\caption{Distributions of estimated break time stamps using Algorithm \ref{algo1}. MA($\infty$) errors $\sim t_9$, $\tau=50$ and all $p$ dimensions jump.}
\label{fig_sim_thm1}
\end{figure}

To demonstrate the robustness of Algorithm \ref{algo1}, we set the number of breaks $K =3$ and break-time stamps $\tau_1=40$, $\tau_2=100$, $\tau_3=160$ for all cases, Moreover, we consider two window sizes $bn=20$ and 30, and three two different jump sizes 1, 0.7 and 0.4. Note that for a simpler setting, in each sample, the jump sizes of all $p$ dimensions are set to be the same. Two different decay rates of the moving average model are shown, $\beta=3$ and 1.5.  We report the averaged difference between the estimated number of change points and the true number of breaks $|\hat K -K |$ (AN) in Table \ref{table_sim_multiple_AN}. The averaged distances between the estimated and true break temporal locations $\sum_{k=1}^{\hat K }|\hat\tau_k-\tau_k^*|$ (AT) under different scenarios are shown in Table \ref{table_sim_multiple_AT}.
The algorithm produces accurate break estimation in terms of all measures in different cases. For example, the estimation precision improves with increasing $p$ and jump sizes. 
It is worth noting that our proposed algorithm is computationally efficient as one only needs to estimate breaks once. Also, we do not need a second step aggregation after estimating the break statistics.

\begin{table}[!htbp]
\begin{center}
    \caption{Averaged Number of breaks (AN) over 1000 samples under different scenarios by Algorithm \ref{algo1}.}
    \label{table_sim_multiple_AN}
    \renewcommand{\arraystretch}{2}
    \begin{tabular}{ p{1.5cm}|p{1.5cm}p{1.5cm}p{1.5cm}|p{1.5cm}|p{1.5cm}p{1.5cm}p{1.5cm}  }
    \hline\hline
    \multicolumn{4}{c|}{jump size = 2, $n=200$} & \multicolumn{4}{c}{jump size = 1, $n=200$} \\
    \hline
     & $p=50$ & $p=200$ & $p=400$ &  & $p=50$ & $p=200$ & $p=400$ \\
    \hline
    $bn=20$   & 0.059  &  0.034 &  0.006 & $bn=20$  & 0.388  & 0.302 &  0.254 \\
    $bn=30$   & 0.021  & 0.013 &  0 & $bn=30$   & 0.259  &  0.187 &  0.099 \\
    \hline
    \multicolumn{4}{c|}{jump size = 0.7, $n=200$} & \multicolumn{4}{c}{jump size = 0.4, $n=200$} \\
    \hline
     & $p=50$ & $p=200$ & $p=400$ &  & $p=50$ & $p=200$ & $p=400$ \\
    \hline
    $bn=20$   & 0.844  &  0.759 &  0.530 & $bn=20$  & 0.998  & 0.806 &  0.629 \\
    $bn=30$   & 0.734  & 0.683 &  0.348 & $bn=30$   & 0.887  &  0.723 &  0.451 \\
    \hline\hline
    \end{tabular}
\end{center}
\end{table}

\begin{table}[!htbp]
\begin{center}
    \caption{Averaged Temporal error by $n$ (AT$/n$) over 1000 samples under different scenarios by Algorithm \ref{algo1}.}
    \label{table_sim_multiple_AT}
    \renewcommand{\arraystretch}{2}
    \begin{tabular}{ p{1.5cm}|p{1.5cm}p{1.5cm}p{1.5cm}|p{1.5cm}|p{1.5cm}p{1.5cm}p{1.5cm}  }
    \hline\hline
    \multicolumn{4}{c|}{jump size = 2, $n=200$} & \multicolumn{4}{c}{jump size = 1, $n=200$} \\
    \hline
     & $p=50$ & $p=200$ & $p=400$ &  & $p=50$ & $p=200$ & $p=400$ \\
    \hline
    $bn=20$   & 5.26e-03  & 4.49e-03 &  8.11e-04 & $bn=20$  & 5.27e-02  & 4.36e-02 &  7.97e-03 \\
    $bn=30$   & 4.19e-03  &  3.05e-03  &  0 & $bn=30$   & 3.42e-02  &  2.76e-02 &  6.28e-03 \\
    \hline
    \multicolumn{4}{c|}{jump size = 0.7, $n=200$} & \multicolumn{4}{c}{jump size = 0.4, $n=200$} \\
    \hline
     & $p=50$ & $p=200$ & $p=400$ &  & $p=50$ & $p=200$ & $p=400$ \\
    \hline
    $bn=20$   & 9.11e-02  & 7.45e-02 &  3.93e-02 & $bn=20$  & 6.23e-01  & 3.74e-01 &  9.35e-02 \\
    $bn=30$   & 8.37e-02  &  5.26e-02  &  1.44e-02 & $bn=30$   & 4.97e-01  &  1.73e-01 &  7.98e-02 \\
    \hline\hline
    \end{tabular}
\end{center}
\end{table}

Similarly, we present the simulation results by Algorithm \ref{algo2} where our proposed Two-Way MOSUM is used to account for cross-sectionally clustered jumps. We investigate two proportions of jumps, i.e. 60\% and 30\% and prefix two corresponding group structures with each spatial neighborhood size around $60\%p$ and $30\%p$, respectively. Specifically, we consider three different dimensions and the corresponding group structures with the number of spatial neighborhoods $S=4$ as follows: For the 60\% case, we choose 1) $p=30$ with the group structure $\{(1:18),(5:22),(9:26),(13:30)\}$, $B_{\text{min}}=18$, 2) $p=70$,  $\{(1:42),(10:51),(19:60),(29:70)\}$, $B_{\text{min}}=42$ and 3) $p=200$, $\{(1:120),(27:147),(54:175),(81:200)\}$, $B_{\text{min}}=120$. For the 30\% case, we consider 1) $p=30$, $\{(1:9),(8:16),(15:23),(22:30)\}$, $B_{\text{min}}=9$, 2) $p=70$,  $\{(1:21),(17:37),(33:53),(50:70)\}$, $B_{\text{min}}=21$ and 3) $p=200$, $\{(1:60),(41:100),(81:140),(121:200)\}$, $B_{\text{min}}=60$. For all the aforementioned cases, we let the number of observations $n=200$, total number of breaks $R =3$, temporal-spatial break locations $(\tau_1,s_1)=(40,2)$, $(\tau_2,s_2)=(100,4)$, $(\tau_3,s_3)=(160,2)$, and window size $bn=30$. We consider two different jump sizes, i.e. 2 and 1 for each dimension. Two different decay rates of the moving average model are checked, i.e. $\beta=3$ and $1.5$. 
Table \ref{table_sim_multiple_nbd_AN} shows the averaged difference between the estimated number of change points and the true number of breaks $|\hat R -R |$ (nbd-AN). Table \ref{table_sim_multiple_nbd_AT} provides the averaged distances between the estimated and true break temporal locations $\sum_{r=1}^{\hat R }|\hat\tau_r-\tau_{r^*}|$ (nbd-AT) under different scenarios. For the precision of the spatial break locations, we report the averaged spatial uncovered rate $\sum_{r=1}^{\hat R }\big|(\B_{\hat s_r}\setminus\B_{s_{r^*}})\cup(\B_{s_{r^*}}\setminus\B_{\hat s_r})\big|$ (nbd-AS) in Table \ref{table_sim_multiple_nbd_AS} for different cases. Overall our algorithm performs well across all the cases. We can see that the estimation accuracy improves over increasing $p$ and jump sizes. Moreover, the cross-sectionally clustered jump would also affect the estimation accuracy.

\begin{table}[!htbp]
\begin{center}
    \caption{nbd-AN over 1000 samples under different scenarios by Algorithm \ref{algo2}.}
    \label{table_sim_multiple_nbd_AN}
    \renewcommand{\arraystretch}{2}
    \begin{tabular}{ p{1.5cm}|p{1.5cm}p{1.5cm}p{1.5cm}|p{1.5cm}|p{1.5cm}p{1.5cm}p{1.5cm}  }
    \hline\hline
    \multicolumn{4}{c|}{jump size = 2, $n=200$} & \multicolumn{4}{c}{jump size = 1, $n=200$} \\
    \hline
    $bn=30$ & $p=50$ & $p=200$ & $p=400$ & $bn=30$ & $p=50$ & $p=200$ & $p=400$ \\
    \hline
    60\%   & 0.109  & 0.085 &  0 & 60\%   & 0.304  & 0.276 &  0.138 \\
    30\%   & 0.217  & 0.110 &  0.032 & 30\%   & 0.422  & 0.368 &  0.196 \\
    \hline
    $bn=20$ & $p=50$ & $p=200$ & $p=400$ & $bn=20$ & $p=50$ & $p=200$ & $p=400$ \\
    \hline
    60\%   & 0.142  & 0.105 &  0.029 & 60\%   & 0.313  & 0.288 &  0.143 \\
    30\%   & 0.236  & 0.128 &  0.051 & 30\%   & 0.456  & 0.397 &  0.216 \\
    \hline\hline
    \multicolumn{4}{c|}{jump size = 0.7, $n=200$} & \multicolumn{4}{c}{jump size = 0.4, $n=200$} \\
    \hline
    $bn=30$ & $p=50$ & $p=200$ & $p=400$ & $bn=30$ & $p=50$ & $p=200$ & $p=400$ \\
    \hline
    60\%   & 0.678  & 0.561 &  0.147 & 60\%   & 0.857  & 0.710 &  0.423 \\
    30\%   & 0.876  & 0.738 &  0.213 & 30\%   & 0.989  & 0.806 &  0.472 \\
    \hline
    $bn=20$ & $p=50$ & $p=200$ & $p=400$ & $bn=20$ & $p=50$ & $p=200$ & $p=400$ \\
    \hline
    60\%   & 0.725  & 0.642 &  0.195 & 60\%   & 0.857  & 0.710 &  0.464 \\
    30\%   & 0.944  & 0.864 &  0.301 & 30\%   & 1.171  & 0.862 &  0.519 \\
    \hline\hline
\end{tabular}
\end{center}
\end{table}

\begin{table}[!htbp]
\begin{center}
    \caption{nbd-AT$/n$ over 1000 samples under different scenarios by Algorithm \ref{algo2}.}
    \label{table_sim_multiple_nbd_AT}
    \renewcommand{\arraystretch}{2}
    \begin{tabular}{ p{1.5cm}|p{1.5cm}p{1.5cm}p{1.5cm}|p{1.5cm}|p{1.5cm}p{1.5cm}p{1.5cm}  }
    \hline\hline
    \multicolumn{4}{c|}{jump size = 2, $n=200$} & \multicolumn{4}{c}{jump size = 1, $n=200$} \\
    \hline
    $bn=30$ & $p=50$ & $p=200$ & $p=400$ & $bn=30$ & $p=50$ & $p=200$ & $p=400$ \\
    \hline
    60\%   & 9.47e-05  & 7.96e-05 &  0 & 60\%   & 6.55e-04  & 4.02e-04 &  8.74e-05 \\
    30\%   & 7.62e-04  & 4.51e-04 &  8.00e-05 & 30\%   & 8.37e-03  & 6.14e-03 &  1.01e-03 \\
    \hline
    $bn=20$ & $p=50$ & $p=200$ & $p=400$ & $bn=20$ & $p=50$ & $p=200$ & $p=400$ \\
    \hline
    60\%   & 1.10e-04  & 9.78e-05 &  0 & 60\%   & 7.63e-04  & 4.89e-04 &  9.16e-05 \\
    30\%   & 9.54e-04  & 2.49e-04 &  9.05e-05 & 30\%   & 9.21e-03  & 7.03e-03 &  1.43e-03 \\
    \hline\hline
    \multicolumn{4}{c|}{jump size = 0.7, $n=200$} & \multicolumn{4}{c}{jump size = 0.4, $n=200$} \\
    \hline
    $bn=30$ & $p=50$ & $p=200$ & $p=400$ & $bn=30$ & $p=50$ & $p=200$ & $p=400$ \\
    \hline
    60\%   & 5.78e-03  & 3.19e-03 &  7.50e-04 & 60\%   & 3.69e-02  & 2.87e-02 &  9.94e-03 \\
    30\%   & 1.36e-02  & 9.42e-03 &  3.52e-03 & 30\%   & 2.71e-01  & 9.98e-02 &  5.51e-02 \\
    \hline
    $bn=20$ & $p=50$ & $p=200$ & $p=400$ & $bn=20$ & $p=50$ & $p=200$ & $p=400$ \\
    \hline
    60\%   & 6.53e-03  & 4.72e-03 &  8.84e-04 & 60\%   & 4.72e-02  & 3.95e-02 &  1.01e-02 \\
    30\%   & 2.66e-02  & 1.21e-02 &  5.30e-03 & 30\%   & 3.28e-01  & 1.03e-01 &  6.96e-02 \\
    \hline\hline
\end{tabular}
\end{center}
\end{table}

\begin{table}[!htbp]
\begin{center}
    \caption{nbd-AS$/B_{\text{min}}$ over 1000 samples under different scenarios by Algorithm \ref{algo2}.}
    \label{table_sim_multiple_nbd_AS}
    \renewcommand{\arraystretch}{2}
    \begin{tabular}{ p{1.5cm}|p{1.5cm}p{1.5cm}p{1.5cm}|p{1.5cm}|p{1.5cm}p{1.5cm}p{1.5cm}  }
    \hline\hline
    \multicolumn{4}{c|}{jump size = 2, $n=200$} & \multicolumn{4}{c}{jump size = 1, $n=200$} \\
    \hline
    $bn=30$ & $p=50$ & $p=200$ & $p=400$ & $bn=30$ & $p=50$ & $p=200$ & $p=400$ \\
    \hline
    60\%   & 6.72e-05  & 5.03e-05 &  0 & 60\%   & 3.89e-04  & 2.90e-04 &  7.13e-05 \\
    30\%   & 5.18e-04  & 6.47e-04 &  8.92e-05 & 30\%   & 4.41e-03  & 1.02e-03 &  7.79e-04 \\
    \hline
    $bn=20$ & $p=50$ & $p=200$ & $p=400$ & $bn=20$ & $p=50$ & $p=200$ & $p=400$ \\
    \hline
    60\%   & 8.01e-05  & 7.31e-05 &  1.99e-05 & 60\%   & 5.43e-04  & 4.68e-04 &  8.79e-05 \\
    30\%   & 6.24e-04  & 7.92e-04 &  9.34e-05 & 30\%   & 5.18e-03  & 2.20e-03 &  8.66e-04 \\
    \hline\hline
    \multicolumn{4}{c|}{jump size = 0.7, $n=200$} & \multicolumn{4}{c}{jump size = 0.4, $n=200$} \\
    \hline
    $bn=30$ & $p=50$ & $p=200$ & $p=400$ & $bn=30$ & $p=50$ & $p=200$ & $p=400$ \\
    \hline
    60\%   & 1.76e-03  & 1.19e-03 &  5.04e-04 & 60\%   & 1.28e-02  & 9.32e-03 &  4.21e-03 \\
    30\%   & 1.13e-02  & 8.64e-03 &  3.22e-03 & 30\%   & 9.38e-02  & 2.59e-02 &  3.97e-02 \\
    \hline
    $bn=20$ & $p=50$ & $p=200$ & $p=400$ & $bn=20$ & $p=50$ & $p=200$ & $p=400$ \\
    \hline
    60\%   & 2.11e-03  & 1.34e-03 &  9.55e-04 & 60\%   & 1.96e-02  & 1.63e-02 &  9.28e-03 \\
    30\%   & 2.45e-02  & 9.34e-03 &  4.50e-03 & 30\%   & 1.03e-01  & 4.10e-02 &  5.05e-02 \\
    \hline\hline
\end{tabular}
\end{center}
\end{table}

\begin{remark}[Extension to a multi-scale MOSUM]\label{rmk_multiscale}
There is an existing literature on the use of the MOSUM methodology for adaptive bandwidth selection in change-point detection with respect to different break locations, such as the studies by \textcite{messer2014multiple} and \textcite{cho2022two}. \textcite{messer2014multiple} employs a bottom-up approach, which is similar to our approach to bandwidth selection. However, our method differs from both of these approaches in that we use a single-scale MOSUM with a fixed bandwidth $b$ per test. Nonetheless, our algorithm could potentially benefit from combining with a multi-scale MOSUM algorithm, such as the one proposed by \textcite{chen_inference_2022}, which adjusts $b$ based on the relative break locations. We leave the rigorous investigation of such an extension to future research.
\end{remark}

We have added a power comparison on test power with respect to different choices of window sizes as shown in Table \ref{tbl_power}. We use the model $MA(\infty)$ with tails following $t_9$ distribution, and we consider the jump size 0.4 to focus our analysis in the more challenging territory, since the larger the jump size, the more prominent the change points. We have observed in Table \ref{tbl_power} that the power of the tests varies significantly based on the window sizes used. In general, when the minimum gap condition is not violated, that is, $b\ll \kappa_n$ with $\kappa_n$ defined in Definition \ref{def_separation}, the larger the bandwidth $b$ used, the better the performance of our proposed algorithm.

\begin{table}[!htbp]
    \centering
    \caption{Power comparison with different $n,p$ and window sizes $bn$.}
    \label{tbl_power}
    \renewcommand{\arraystretch}{2}
    \begin{tabular}{p{1cm} p{1cm}|p{1.5cm}|p{1.5cm}|p{1.5cm}|p{1.5cm}}    
    \hline\hline
    $n$ & $p$  & $bn=25$  & $bn=30$ & $bn=35$ & $bn=40$   \\ 
    \hline
    200 &  70  &0.64&0.93& 0.99&1  \\ 
    \hline
    200 &  90 & 0.26&0.79&1&1  \\
    \hline
    200 &  110 & 0.17&0.47&0.95&0.99\\
    \hline\hline
    250 &  70  & 0.56&0.88&1&1\\
    \hline
    250 &  90  & 0.29&0.8&0.95&1\\
    \hline
    250 & 110&   0.11&0.55&0.92&1\\
    \hline\hline
    300 &  70 &  0.45&0.97&1&1\\
    \hline
    300 &  90  & 0.33&0.79&0.98&1\\
    \hline
    300 &  110  & 0.09&0.43&0.92&1\\
    \hline\hline
    \end{tabular}
\end{table}

\subsection{Application: Stock Return Data}\label{subsec_stock_append}

We obtain a panel of 300 monthly stock returns for $20$ years from January 2000 - December 2020. The data source is the CRSP US Stock Databases. We apply our proposed Algorithm \ref{algo1} to this dataset for change-point testing. Figure \ref{fig:app_stock} presents 10 stock returns from the original data where two change points were detected by our Algorithm \ref{algo1}. From Figure \ref{fig:app_stock}, one could hardly tell from the original time series whether any breaks exist as the signals could be very weak across cross-sectional dimensions. However, by our proposed $\ell^2$ type test statistics, we successfully detected two change points.
The first one was in Jun, 2009, and the second one was in Feb, 2020, which corresponds to the stock rebound after the subprime crisis in 2008 and the stock crash due to the COVID-19 crisis respectively. These results demonstrate that we can effectively identify the critical economic dates which cause turbulence in the financial market.
\begin{figure}[htbp!]
    \centering
    \includegraphics[width=1\textwidth]{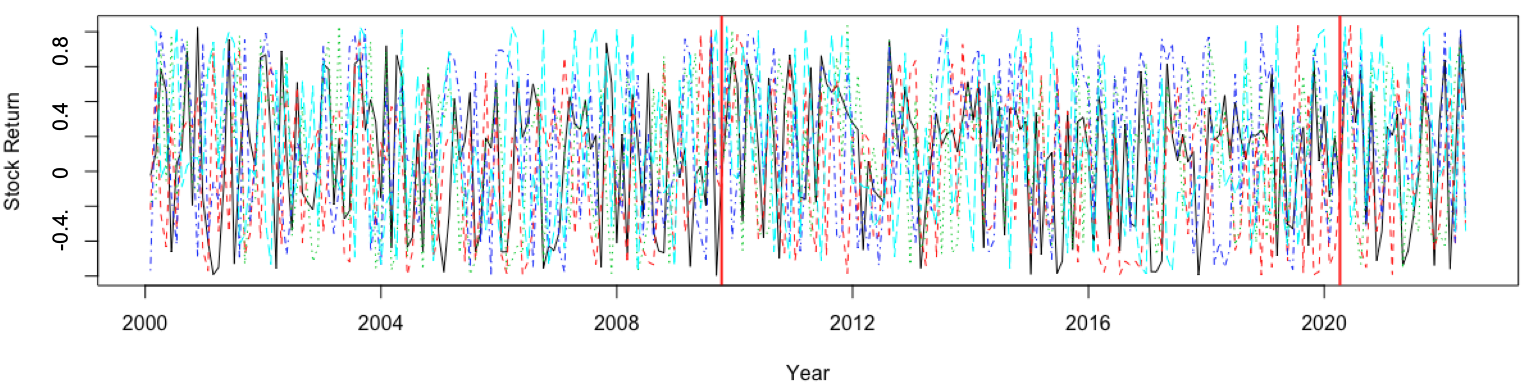}
    \caption{300 monthly stock returns from 2000.01 to 2020.12 with two change points detected by Algorithm \ref{algo1}, which are 2009.06 and 2020.01.}
    \label{fig:app_stock}
\end{figure}

\subsection{Application: COVID-19 Data}\label{subsec_covid_append}

In Figure \ref{fig_covid_stage}, we randomly picked four dates to show the COVID-19 maps reflecting the numbers of cases across the country, from which, one can see that break time stamps for each area could be different. Therefore, a change-point algorithm designed for spatially clustered signals is on demand. We apply the Two-Way MOSUM to identify and characterize breaks at different time stamps and spatial locations.

\begin{figure}[!htbp]
\centering
\begin{subfigure}{0.45\textwidth}
  \centering
  \includegraphics[width=1\linewidth]{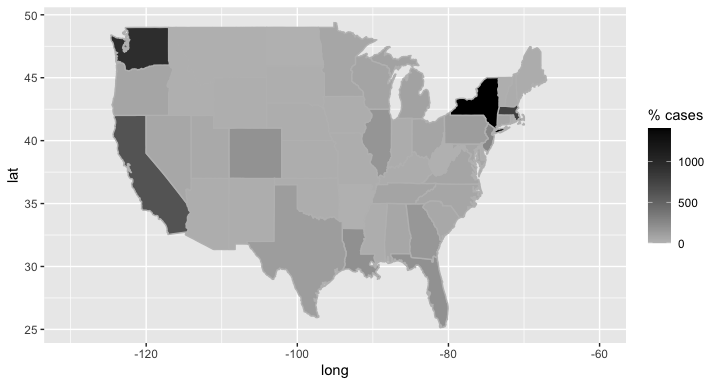}  
  \caption{2020.03.17}
  \label{fig_covid_stage1}
\end{subfigure}
\begin{subfigure}{0.45\textwidth}
  \centering
  \includegraphics[width=1\linewidth]{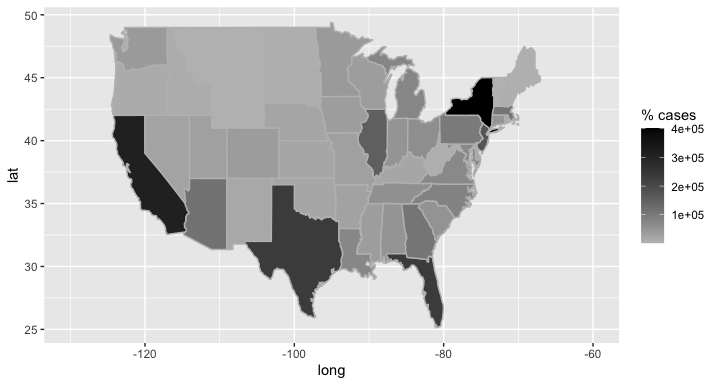}  
  \caption{2020.07.09}
  \label{fig_covid_stage2}
\end{subfigure}

\begin{subfigure}{0.45\textwidth}
  \centering
  \includegraphics[width=1\linewidth]{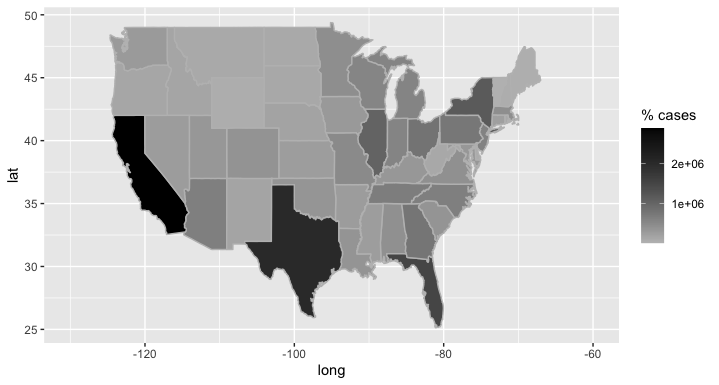}  
  \caption{2021.01.12}
  \label{fig_covid_stage3}
\end{subfigure}
\hskip 0.1cm
\begin{subfigure}{0.45\textwidth}
  \centering
  \includegraphics[width=1\linewidth]{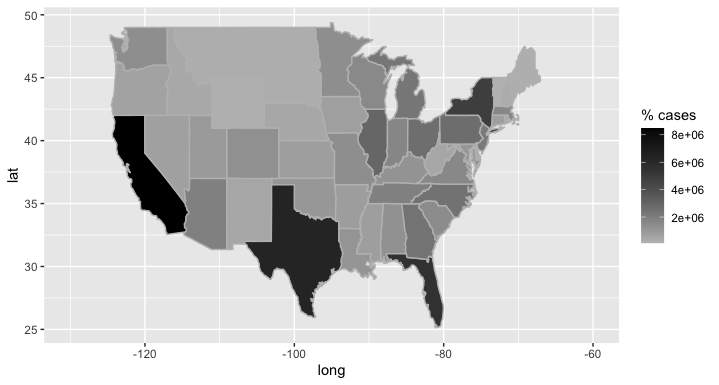}  
  \caption{2022.02.01.}
  \label{fig_covid_stage4}
\end{subfigure}

\caption{Maps of COVID-19 cases in the US on four different dates. 
}
\label{fig_covid_stage}
\end{figure}

In addition, we present the map of four geographic regions of the US by CDC (\url{https://www.cdc.gov/nchs/hus/sources-definitions/geographic-region.htm}) in Figure \ref{fig_region_us}, which were used in our application in Section \ref{sec_data}. We list the areas in each four regions.

\begin{itemize}
    \item[1] Region 1 (\textit{Northeast}): 
    Connecticut, Maine, Massachusetts, New Hampshire, Rhode Island, Vermont, New Jersey, New York, and Pennsylvania
    \item[2] Region 2 (\textit{Midwest}):
    Illinois, Indiana, Michigan, Ohio, Wisconsin, Iowa, Kansas, Minnesota, Missouri, Nebraska, North Dakota, and South Dakota
    \item[3] Region 3 (\textit{South}):
    Delaware, Florida, Georgia, Maryland, North Carolina, South Carolina, Virginia, District of Columbia, West Virginia, Alabama, Kentucky, Mississippi, Tennessee, Arkansas, Louisiana, Oklahoma, and Texas
    \item[4] Region 4 (\textit{West}):
    Arizona, Colorado, Idaho, Montana, Nevada, New Mexico, Utah, Wyoming, Alaska, California, Hawaii, Oregon, and Washington
\end{itemize}

\begin{figure}[!htbp]
  \centering
  \includegraphics[width=1\linewidth]{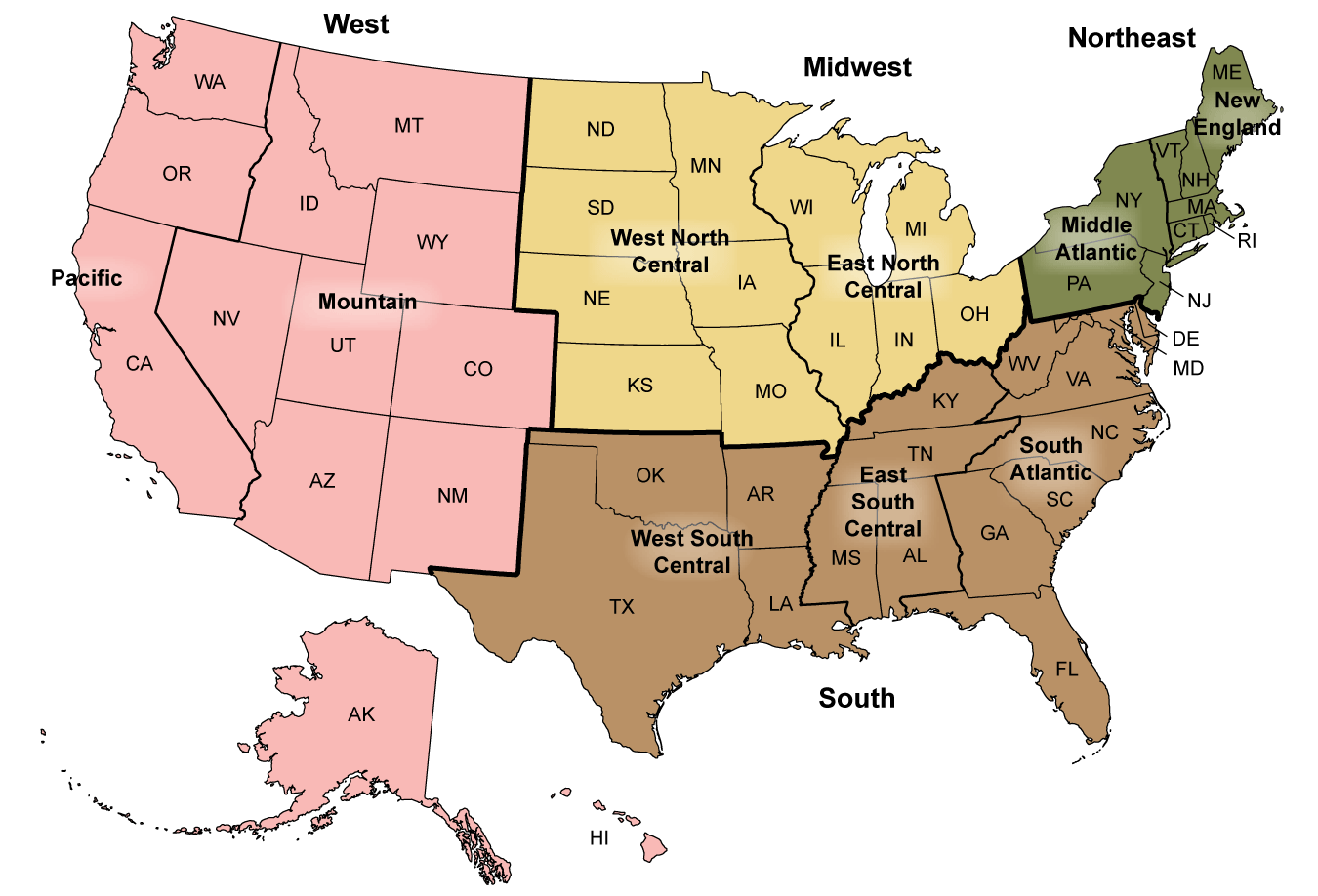}  
  \caption{Four geographic regions in the United States guided by CDC.}
  \label{fig_region_us}
\end{figure}

\clearpage

\section{Some Concrete Examples}\label{sec_example}

\subsection{Short-Range Dependent Linear Processes}\label{subsec_ex_temp_Dep}
Here we provide a specific example that fulfills Assumption \ref{asm_temp_dep}.

\begin{example}[Linear processes]
    \label{example_temp_dep}    
    Let $\eta_t\in\RR^p$ be i.i.d. random vectors with zero mean and identity covariance matrices. {Consider the linear process 
    \begin{equation}
        \label{eq_example_model}
        \mathcal{E}_t=\sum_{k\ge 0}\Lambda_k\eta_{t-k},
    \end{equation}
    where $\Lambda_k=\text{diag}(\lambda_{k,1},\ldots,\lambda_{k,p})$.} Set $\mathcal{E}_{t,j}$ to be the $j$-th component of $\mathcal{E}_t$, and $\mathcal{E}_{t,j}$ are independent for different $j$. Then, the long-run variance for $\mathcal{E}_{t,j}$ is $\sigma_{\mathcal{E},j}^2 = (\sum_{k\ge0}\lambda_{k,j})^2$. If for all $1\le j\le p$ and some constant $\alpha>1$,
    \begin{equation}
        \label{eq_example_condi}
    |\lambda_{k,j}|/\sigma_{\mathcal{E},j}\lesssim k^{-\alpha},
    \end{equation}
    where the constant in $\lesssim$ is independent of $k$ and $j$, then Assumption \ref{asm_temp_dep} holds for $\mathcal{E}_t$ with $\beta=\alpha-1$. To see this, one shall note that $|A_{k,j,\cdot}|_2=\lambda_{k,j}$ and $\sigma_j^2 = \sigma_{\mathcal{E},j}^2$, which indicates
    \begin{equation*}
        \sum_{k\ge i}|A_{k,j,\cdot}|_2 = \sum_{k\ge i}|\lambda_{k,j}| \lesssim i^{-\alpha+1} \sigma_{\mathcal{E},j} = i^{-\alpha+1}\sigma_j.
    \end{equation*}
\end{example}
To establish the Gaussian approximation for the $\ell^2$-based test statistic $\Q_n$, however, the cross-sectional independence is not necessary. Assumption \ref{asm_sec_indep} which requires the cross-sectional independence of the error process $\{\epsilon_t\}_{t\in\ZZ}$ has been relaxed to allow for weak cross-sectional dependence in Assumption \ref{asm_nonli_sec_dep}.

\subsection{Multiple Time Series}\label{subsec_fixed_p}
In Section \ref{sec_test_nonlinear}, we presented the Gaussian approximation results for the test statistic $\tilde Q_n$ defined in (\ref{eq_test_nonli}) in the high-dimensional setting. However, our analysis is not limited to large $p$. In this section, we provide an invariance principle for multiple time series, where the dimension $p$ remains fixed as $n$ diverges. 

Recall that the test statistic $\tilde \Q_n=\max_{1\le s\le S}\max_{bn+1\le i\le n-bn}Q_{i,\B_s}$, and under the null hypothesis, we can write $Q_{i,\B_s}$ into
\begin{align}
    Q_{i,\B_s} = \frac{1}{\sqrt{|\B_s\cap\L_0|}}\sum_{\bbell\in\B_0\cap\L_0}\Big[(bn\sigma(\bbell))^{-2}\Big(\sum_{t=i-bn}^{i-1}\epsilon_t(\bbell) - \sum_{t=i}^{i+bn-1}\epsilon_t(\bbell)\Big)^2 - c(\bbell)\Big]\One_{\bbell\in\B_s\cap\L_0},
\end{align}
where $\epsilon_t(\bbell)$ is defined in (\ref{eq_epsilon_nonlinear}) and $c(\bbell)$ is a centering term defined in (\ref{eq_center_nbd}). We aim to provide an invariance principle for the process $(Q_{i,\B_s})_{bn+1\le i\le n-bn, 1\le s\le S}$. Since $Q_{i,\B_s}$ is a quadratic form which is quite involved, we shall first establish a preliminary result for the $p$-dimensional partial sum process $(S_i)_{i\in\ZZ}$, where $S_i=\sum_{t=1}^i\epsilon_t$ with $\epsilon_t = (\epsilon_t(\bbell))_{\bbell\in\B_0\cap\L_0}^\top$. Recall the long-run covariance matrix of $\epsilon_t$ which is denoted by $\Sigma$ in (\ref{eq_nonli_longrun}). To facilitate identification for multiple time series, we first state an assumption on $\Sigma$ below.
\begin{assumption}[Non-degeneracy of long-run covariance matrix]
    \label{asm_multi_eigen}
    Denote the smallest eigenvalue of the long-run covariance matrix $\Sigma$ by {$\lambda_*$. Assume that $\lambda_*>0$.}
\end{assumption}
Next, we follow \textcite{karmakar_optimal_2020} to provide an invariance principle for $(S_i)_{i\ge1}$.
\begin{lemma}[Invariance principle for $(S_i)_{i\ge1}$]
    \label{lemma_finite_p}
    Suppose that Assumption \ref{asm_multi_eigen} is satisfied, Assumption \ref{asm_nonli_finitemoment} holds for some $q>2$, and Assumption \ref{asm_nonli_temp_dep} holds for
    \begin{equation}
        \label{eq_beta}
        \beta>\begin{cases}
            2, &\text{if }2<q<4,\\
            1+\frac{q^2-4+(q-2)\sqrt{q^2+20q+4}}{8q}, &\text{if }q\ge4.
        \end{cases}
    \end{equation}
    \noindent Then, there exists a probability space $(\Omega_c,\A_c,\PP_c)$ on which we can define a partial sum process $S_i^c=\sum_{t=1}^i\epsilon_t^c$, such that $(S_i^c)_{1\le i\le n}\overset{d}{=}(S_i)_{1\le i\le n}$, and for $q>2$,
    \begin{equation}
        \label{eq_prop_rate}
        \max_{i\le n}\big|S_i^c-\Sigma^{1/2}\BB(i)\big|_2 = o_{\PP}(n^{1/q}), \quad \text{in }(\Omega_c,\A_c,\PP_c),
    \end{equation}
    where $\BB(\cdot)$ is the $p$-dimensional standard Brownian motion.
\end{lemma}
Next, we shall proceed to provide an invariance principle for the random process $(Q_{i,\B_s})_{bn+1\le i\le n-bn,1\le s\le S}$. First, we introduce some necessary notation. Recall the probability space $(\Omega_c,\A_c,\PP_c)$ and the standard Gaussian process $\BB(i)$ in Lemma \ref{lemma_finite_p}. Then, on this probability space, we define
\begin{align}
    G_i^{c^-} & = (bn)^{-1}\Lambda^{-1}\Sigma^{1/2}(\BB(i-1)- \BB(i-bn-1)),\nonumber \\
    G_i^{c^+} & = (bn)^{-1}\Lambda^{-1}\Sigma^{1/2}(\BB(i+bn-1)- \BB(i-1)).
\end{align}
Further, recall the centering term $c_{\B_s}$ defined in (\ref{eq_center_nbd}), and we denote
\begin{equation}
    \G_{i,\B_s}^c = \frac{1}{\sqrt{|\B_s\cap\L_0|}}\Big(\big|G_i^{c^-} - G_i^{c^+}\big|_{2,\B_s}^2 - c_{\B_s}\Big),
\end{equation}
where $|\cdot|_{2,\B_s}$ is an nbd-norm similar to the one defined in Definition \ref{def_linear_nbdnorm}, that is, for a random vector $X=(X(\bbell))_{\bbell\in\B_0\cap\L_0}$,  $|X|_{2,\B_s}^2 = \sum_{\bbell\in\B_s\cap\L_0}X^2(\bbell)$. Then, we can approximate the joint distribution of the vector process $(Q_{i,\B_s})_{i,s}$ by a Gaussian process $(\G_{i,\B_s}^c)_{i,s}$ with an optimal approximation rate. We state this result in Proposition \ref{prop_finite_p} below.

\begin{proposition}[Gaussian approximation with fixed $p$]
    \label{prop_finite_p}
    Suppose that the conditions in Lemma \ref{lemma_finite_p} all hold. Then, there exists a probability space $(\Omega_c,\A_c,\PP_c)$ on which we can define a quadratic partial sum process $Q_{i,\B_s}^c$ and a Gaussian process $\G_{i,\B_s}^c$, such that $(Q_{i,\B_s}^c)_{i,s}\overset{\mathbb{D}}{=}(Q_{i,\B_s})_{i,s}$, and for $q>2$,
    \begin{equation}
        \label{eq_prop_rate2}
        \max_{1\le s\le S}\max_{bn+1\le i\le n-bn}bn\big|Q_{i,\B_s}^c-\G_{i,\B_s}^c\big| = o_{\PP}(n^{1/q}), \quad \text{in }(\Omega_c,\A_c,\PP_c).
    \end{equation}
\end{proposition}

\begin{proof}[Proof of Proposition \ref{prop_finite_p}]
Recall the partial sum process $S_i=\sum_{t=1}^i\epsilon_t$ and its counterpart $S_i^c$ on the probability space $(\Omega_c,\A_c,\PP_c)$. We define
\begin{equation}
    H_i^{c^-} = (bn)^{-1}\Lambda^{-1}(S_{i-1} - S_{i-bn-1}), \quad H_i^{c^+} = (bn)^{-1}\Lambda^{-1}(S_{i+bn-1} - S_{i-1}).
\end{equation}
Then, under the null hypothesis, we can write $Q_{i,\B_s}$ into
\begin{equation}
    Q_{i,\B_s} = \frac{1}{\sqrt{|\B_s\cap\L_0|}}\Big(\big|H_i^{c^-} -  H_i^{c^+}\big|_{2,\B_s}^2 - c_{\B_s}\Big).
\end{equation}
For the notation simplicity, we let
\begin{equation}
    \Delta H_i^c = H_i^{c^-} -  H_i^{c^+}, \quad \Delta G_i^c = G_i^{c^-} -  G_i^{c^+}.
\end{equation}
Since the dimension $p$ is fixed, it follows from Lemma \ref{lemma_finite_p} that
\begin{align}
    & \quad \max_{bn+1\le i\le n-bn}bn\big|\Delta H_i^{c\top}\Delta H_i^c - \Delta G_i^{c\top}\Delta G_i^c\big| \nonumber \\
    & \le \max_i bn\big|\Delta H_i^{c\top}(\Delta H_i^c - \Delta G_i^c)\big| + \max_i bn\big|(\Delta H_i^c - \Delta G_i^c)^{\top}\Delta G_i^{c}\big| \nonumber \\
    & = o_{\PP}(n^{1/q}).
\end{align}
{By applying similar arguments to the nbd-norm $|\cdot|_{2,\B_s}$,} one can achieve the desired result.
\end{proof}

The invariance principle in Proposition \ref{prop_finite_p} enables us to derive the threshold $\omega^*$ for change-point detection by implementing a multiplier bootstrap study on the Gaussian random process $(\G_{i,\B_s}^c)_{i,s}$. In particular, we can choose $\omega^*$ to be the critical value of $\max_{i,s}\G_{i,\B_s}^c$, that is, for some significant level $\alpha\in(0,1)$, we let
\begin{equation}
    \tilde\omega = \inf_{r\ge0}\big\{r: \, \PP\big(\max_{i,s}\G_{i,\B_s}^c > r\big)\le \alpha \big\}.
\end{equation}
We reject the null hypothesis if the test statistic $\tilde Q_n$ exceeds the threshold, i.e., $\tilde Q_n>\omega^*$.

\subsection{Two-Way MOSUM in One-Dimensional Spatial Space}\label{subsec_twoway_example}

To better illustrate the newly proposed Two-Way MOSUM, we focus on the simplest setting in this section, where the cross-sectional space is one-dimensional, and we provide two compelling examples.

First, we explain why prior knowledge of groups is not necessary for the Two-Way MOSUM. Consider $p$ component time series in a one-dimensional spatial space with a linear ordering, as depicted in Figure \ref{fig_ordering}. The larger the spatial distance between two component series, the further apart they are listed. In our grouping mechanism, only adjacent series can belong to the same spatial neighborhood. In other words, a series cannot skip its closest component series and group with another series further away in the same spatial neighborhood. Figure \ref{fig_ordering} illustrates that the grouping method in the top row is allowed, while the one in the bottom row is not. Due to this grouping mechanism naturally determined by the relative spatial distance between each pair of component series, the number of possible spatial neighborhoods is, at most, on the order of $p^2$, regardless of whether prior grouping knowledge is available. To see this, we can first consider the case when the group size is 1. Then, we can have at most $p$ groups. Similarly, for the group size $m$, $1\le m\le p$, the number of groups is at most $p-(m-1)$. Therefore, the total number of groups $S\le \sum_{m=1}^p[p-(m-1)]=p(p+1)/2 = O(p^2)$. 
% \textcolor{red}{(Do we need to remove this part? $p^2$ mentioned twice on Page 13 and 14.)}

When extending to a more general spatial space in $\ZZ^v$, for $v\ge1$, the largest group number $S$ in each dimension is $O(p^2)$, resulting in a total number of possible spatial groups $S= O(p^{2v})$. Note that the term $S$ appears in the logarithm in the Gaussian approximation rate in Theorem \ref{thm2_constanttrend} and consistency rate in Proposition \ref{prop_consistency}. This indicates that as long as $S$ is in the order of a polynomial of $p$, the asymptotic properties will not differ significantly in situations with or without prior grouping knowledge.

\begin{figure}[!htbp]
\centering
\includegraphics[width=0.8\linewidth]{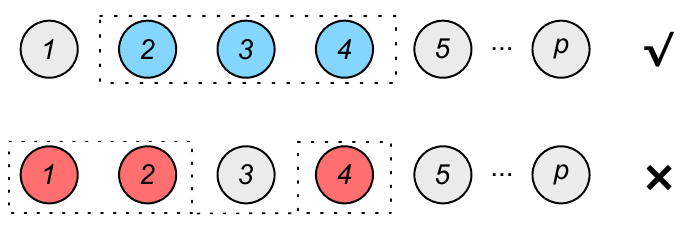}
\caption{Two-Way MOSUM: Only adjacent component series are allowed to be assigned to the same group. Blue circles represent an example group allowed in our setting, while red circles showcase an example group that fails our requirement. Gray circles are the component series to be assigned into other groups (not necessarily the same).}
\label{fig_ordering}
\end{figure}

Additionally, we present the temporal-spatial Two-Way moving window in Figure \ref{fig_twoway} for the simple case where the time series are in one-dimensional spatial space. Assume that the sizes of the spatial neighborhoods are $p_1, p_2, \ldots, p_S$ respectively, for the sizes defined as Definition \ref{def_linear_nbd}. As illustrated in Figure \ref{fig_twoway}, change points can be detected by moving temporal-spatial windows with a uniform temporal width of $2bn$ and varying spatial widths $p_1, \ldots, p_S$. Keep in mind that sliding windows in a one-way MOSUM move only in the temporal direction, favoring simultaneous jumps for all entities $1 \le j \le p$. In contrast, the more adaptable Two-Way MOSUM increases testing power when breaks occur in only a portion of the component series.
\begin{figure}[!htbp]
    \centering
    \includegraphics[width=0.7\linewidth]{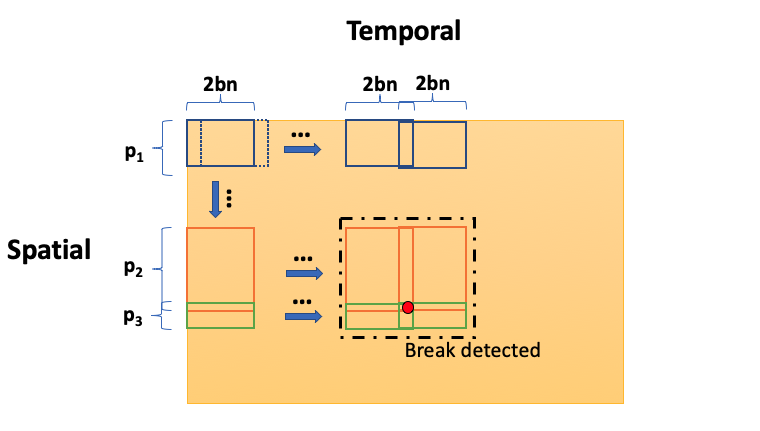}
    \caption{Two-Way MOSUM: windows move in both temporal and spatial directions.}
    \label{fig_twoway}
\end{figure}

\subsection{Linear Processes with Cross-Sectional Dependence}\label{subsec_linear_sec_dep}

As the cross-sectional dependence can be intricate, especially when using the newly proposed Two-Way MOSUM, we provide a concrete example in this section using a conventional $\ell^2$-based MOSUM for a linear process described in equation (\ref{eq_epsilon_linear}). Consequently, we only need to consider a one-dimensional spatial space. We present the asymptotic properties of the test statistic $\Q_n$, as defined in expression (\ref{eq_teststats}), allowing weak cross-sectional dependence. Throughout this section, we assume that $\{\epsilon_t\}_{t\in\ZZ}$ follows a linear process as defined in equation (\ref{eq_epsilon_linear}).

As a special case of Assumption \ref{asm_nonli_sec_dep}, we start with a condition on the strength of the cross-sectional dependence.
\begin{assumption}[Cross-sectional dependence]
    \label{asm_sec_dep}
    Assume that there exist constants $C'>0$ and $\xi>1$, such that, for all $r\ge0$,
    $$\max_{1\le j\le p}\sum_{k\ge0}\Big(\sum_{|s-j|\ge r}A_{k,j,s}^2/\sigma_j^2\Big)^{1/2}\le C'(r\vee1)^{-\xi}.$$
\end{assumption}
It is worth noticing that by a similar argument below Assumption \ref{asm_nonli_sec_dep}, for all $1\le j\le p$,
$$\big(\sum_{|s-j|\ge r}\tilde A_{0,j,s}^2/\sigma_j^2\big)^{1/2} \le \sum_{k\ge0}\big(\sum_{|s-j|\ge r}A_{k,j,s}^2/\sigma_j^2\big)^{1/2}.$$
Then, one can see that Assumption \ref{asm_sec_dep} also implies
\begin{equation}
    \label{eq_sec_dep2}
    \max_{1\le j\le p}\Big(\sum_{|s-j|\ge r}\tilde A_{0,j,s}^2/\sigma_j^2\Big)^{1/2}\le C'(r\vee1)^{-\xi}, \quad \text{where } \tilde A_{0,j,s}=\sum_{k\ge0}A_{k,j,s},
\end{equation}
which requires an algebraic decay rate on each row of the long-run covariance matrix $\tilde A_0$. This condition further indicates that, for any $1\le j,j'\le p$, the long-run correlation between the errors $\epsilon_{t,j}$ and $\epsilon_{t,j'}$ decays at a polynomial rate as $|j-j'|$ increases. Namely, for the long-run correlation denoted by $\rho_{j,j'}$, we have 
\begin{align}
\label{eq:rhojjpandbdd}
\rho_{j,j'}=\tilde A_{0,j,\cdot}^{\top}\tilde A_{0,j',\cdot}/(\sigma_j\sigma_{j'}) =O(|j-j'|^{-\xi}), 
\end{align}
where the constant in $O(\cdot)$ is independent of $j$ and $j'.$
To see this, we denote $$a_{1,j}^*=\sigma_j^{-1}\tilde A_{0,j, 1:(j+j')/2)}, \quad b_{1,j'}^*=\sigma_{j'}^{-1}\tilde A_{0,j',1:(j+j')/2},$$
and
$$a_{2,j}^*=\sigma_j^{-1}\tilde A_{0,j,(j+j')/2+1:p}, \quad b_{2,j'}^*=\sigma_{j'}^{-1}\tilde A_{0,j',(j+j')/2+1:p},$$ 
where $\tilde A_{0,j,1:k}$ denotes the 1 to $k$-th elements in the vector $\tilde A_{0,j,\cdot}$. For any $1\leq j<j'\leq p$, by Assumption \ref{asm_sec_dep},
$$|b_{1,j'}^*|_2=\Big(\sum_{s =1}^{(j+j')/2}\tilde A_{0,j,s}^2/\sigma_j^2\Big)^{1/2}
\leq \Big(\sum_{|s-j'|\leq |j'-j|/2}\tilde A_{0,j',s}^2/\sigma_j^2\Big)^{1/2}
=O(|j-j'|^{-\xi}),$$
and similarly, we have $|a_{2,j}^*|_2=O(|j-j'|^{-\xi})$. Then, it follows that
\begin{align}
    \label{eq_cov_dep_decay}
    & |\tilde A_{0,j,\cdot}^{\top}\tilde A_{0,j',\cdot}|/(\sigma_j\sigma_{j'})\lesssim |a_{1,j}^*|_2|b_{1,j'}^*|_2+ |a_{2,j}^*|_2|b_{2,j'}^*|_2 =O(|j-j'|^{-\xi}).
\end{align}
In fact, Assumption \ref{asm_sec_dep} is a very general condition on the spatial dependence. For example, it allows cross-sectionally $m$-dependent sequences. It is also worth noticing that too strong spatial dependence is not allowed in our condition, such as a factor structure. 

To achieve the limit distribution of the test statistic $\Q_n$ in expression (\ref{eq_teststats}) under the null hypothesis, 
we introduce a centered Gaussian random vector $\Z'=( \Z_{bn+1}',\ldots,\Z_{n-bn}')^{\top}\in\RR^{n-2bn}$ with covariance matrix $\Xi'=\EE(\Z'\Z'^{\top})\in\RR^{(n-2bn)\times (n-2bn)}.$ To be more specific, we define $\Xi'=(\Xi_{i,i'}')_{1\le i,i'\le n-2bn}\in\RR^{(n-2bn)\times (n-2bn)}$, where
\begin{equation}
    \label{eq_cov_dependent}
    \Xi_{i,i+\zeta bn}' = (bn)^{-2}\sum_{1\le j,j'\le p}
    \begin{cases}
        (15\zeta^2-20\zeta+8)\rho_{j,j'}^2 + 3\zeta^2-4\zeta, & 0<\zeta \le 1, \\
        (3\zeta^2-12\zeta+12)\rho_{j,j'}^2 -\zeta^2+4\zeta -4, & 1< \zeta\le 2,\\
        0, & \zeta>2,
    \end{cases}
\end{equation}
and $\rho_{j,j'}$ is defined in \eqref{eq:rhojjpandbdd}. The covariance matrix $\Xi'$ is asymptotically equal to the one of $(bn)^{-1}\sum_{j=1}^pX_j$. 
We defer the detailed evaluation of (\ref{eq_cov_dependent}) to Lemma \ref{lemma_cov_thm3}. One shall note that, if $\rho_{j,j}=1$ and $\rho_{j,j'}=0$, for all $1\le j\neq j'\le p$,  which denotes the case that no spatial dependency exists, then (\ref{eq_cov_dependent}) is same as (\ref{eq_cov}).

In the following theorem, we show that we can achieve a similar Gaussian approximation result for our test statistics $\Q_n$ as the one in the cross-sectionally independent case.
\begin{proposition}[Gaussian approximation with cross-sectional dependence]
    \label{thm3_constanttrend}
    Suppose that Assumptions \ref{asm_finitemoment}, \ref{asm_temp_dep} and \ref{asm_sec_dep} are satisfied.
    Let $c_{q,\xi}=4q\xi+q-2\xi-2$. 
    Then under the null hypothesis, for $\Delta_0'=(bn)^{-1/3}\log^{2/3}(n)$,
    $$\Delta_1' = \Big(\frac{n^{2\xi-4(\xi+1)/(3q)}}{p^{q(2\xi-1)/4}}\Big)^{1/c_{q,\xi}}\log^{7/6}(pn), \quad \Delta_2' = \Big(\frac{n^{4\xi}}{p^{(q-2)(2\xi-1)/2}}\Big)^{1/c_{q,\xi}}\log(pn),$$
    and
    $$\Delta_3'=\Big(\frac{n^{2\xi}}{p^{(q-2)(2\xi-1)/4}}\Big)^{1/c_{q,\xi}}\log^2(pn),$$
    we have
    \begin{equation}
        \sup_{u\in\RR}\Big|\PP(\Q_n\le u)-\PP\big(\max_{bn+1\le i\le n-bn}\Z'_i\le u\big)\Big|\lesssim \Delta_0'+\Delta_1'+\Delta_2'+\Delta_3',
    \end{equation}
    where the constant in $\lesssim$ is independent of $n,p,b$. 
    If in addition, $\log(n)=o\{(bn)^{1/2}\}$ and
    \begin{equation}
        \label{eq_thm31_o1}
        n^{8\xi}p^{(2-q)(2\xi-1)}\log^{2c_{q,\xi}}(pn)=o(1),
    \end{equation}
    then we have
    \begin{equation}
        \label{eq_thm31_result}
        \sup_{u\in\RR}\Big|\PP(\Q_n\le u)-\PP\big(\max_{bn+1\le i\le n-bn}\Z'_i\le u\big)\Big|\rightarrow 0.
    \end{equation}
\end{proposition}

\begin{remark}[Comparison with Theorem \ref{thm1_constanttrend}]
It is worth noticing that in Proposition \ref{thm3_constanttrend}, the rate of Gaussian approximation is affected by the spatial dependence rate $\xi$. In the extreme case, if $\xi\rightarrow\infty$ which indicates a drastic decay rate of the cross-sectional dependence, then according to (\ref{eq_thm31_o1}), the Gaussian approximation error tends to $0$ when $n^4p^{2-q}=o(1)$ (up to a logarithm factor). This condition adheres to the one shown in (\ref{eq_thm11_o1}) for Theorem \ref{thm1_constanttrend} where we have assumed the cross-sectional independence. {When $\xi$ is smaller,we can see that the condition has more demanding requirement on the rate of $p$ relative to $n$ to compensate for the dependency.}
\end{remark}

By the Gaussian approximation result in Proposition \ref{thm3_constanttrend}, we can reject the null hypothesis at the significant level $\alpha$, if the test statistic exceed the threshold $\omega$, i.e. $\Q_n>\omega$. Now we consider the alternative hypothesis that $\underline{d}\neq 0$ and give the following corollary to illustrate the power limit of our test.
\begin{corollary}[Power with cross-sectional dependence]
    \label{cor_sec_dep_power}
    Under Assumptions \ref{asm_finitemoment}, \ref{asm_temp_dep} and \ref{asm_sec_dep}, if (\ref{eq_thm31_o1}) holds, $b\ll \kappa_n$, and
    $$\max_{1\le k\le K}n(u_{k+1}-u_k)|\Lambda^{-1}\gamma_k|_2^2\gg \sqrt{p\log (n)},$$
    then the testing power $\PP(\Q_n>\omega)\rightarrow1$, as $n\wedge p\rightarrow\infty$.
\end{corollary}

\subsection{Power Limit in a Generalized Setting}\label{sec_power_append}
In Section \ref{sec_test_nonlinear}, we have extended the Two-Way MOSUM statistics to nonlinear time series and a general spatial space in $\ZZ^v$, for $v\ge1$. Weak cross-sectional dependence can also be allowed. Here, we provide the following corollary to demonstrate the power limit of our test proposed in Section \ref{sec_test_nonlinear} for such a generalized Two-Way MOSUM.
\begin{corollary}[Power of a generalized Two-Way MOSUM]
    \label{cor_nonli_power}
    Under Assumptions \ref{asm_density}--\ref{asm_nonli_sec_dep}, if $c_{p,n}\rightarrow0$, $b\ll \kappa_n$, and
    $$\max_{1\le s\le S}\max_{1\le k\le K}n(u_{k+1}-u_k)\sum_{\bbell\in\B_s\cap\L_0}\gamma_k^2(\bbell)/\sigma^2(\bbell) \gg \sqrt{B_{\text{min}}\log (nS)},$$
    then the testing power $\PP(\tilde\Q_n>\omega)\rightarrow1$, as $n\wedge p\rightarrow\infty$.
\end{corollary}
Corollary \ref{cor_nonli_power} ensures that the testing power tends to 1 as long as the minimum signal requirement is satisfied. One shall note that, similar to Corollaries \ref{cor_power} and \ref{cor_local_power}, this detection lower bound is sharp compared to the other methods listed in Table 1 in \textcite{cho2022two}. Accordingly, an algorithm similar to Algorithm \ref{algo2} for detecting and identifying breaks can be developed, and we shall omit this part here due to the similarity.

\clearpage

\section{Technical Proofs}\label{sec_proofs}

\subsection{Some Useful Lemmas}

\begin{lemma}[\cite{burkholder1988,rio_moment_2009}]
    \label{lemma_burkholder}
    Let $q>1,q'=\min\{q,2\}$, and $M_T=\sum_{t=1}^{T}\xi_t$, where $\xi_t\in\L^q$ are martingale differences. Then 
    $$\lVert M_T\rVert_q^{q'}\le K_{q}^{q'}\sum_{t=1}^{T}\lVert \xi_T\rVert_q^{q'},\quad\text{where } K_q=\max\{(q-1)^{-1},\sqrt{q-1}\}.$$
\end{lemma}

\begin{lemma}[\cite{chen_inference_2022}]
    \label{lemma_chen2019}
    Suppose that $X=(X_1,X_2,\ldots,X_v)^{\top}$ and $Y=(Y_1,Y_2,\ldots,Y_v)^{\top}$ are two centered Gaussian vectors in $\RR^n$. Let $d=(d_1,d_2,\ldots,d_v)^{\top}\in\RR^v$, and $\Delta=\max_{1\le i,j\le v}|\sigma_{i,j}^X-\sigma_{i,j}^Y|$, where $\sigma_{i,j}^X=\EE (X_iX_j)$ and $\sigma_{i,j}^Y=\EE (Y_iY_j)$. Assume that $Y_i$'s have the same variance $\sigma^2>0$. Then,
    $$\sup_{x\in\RR}\Big|\PP\big(|X+d|_{\infty}\le x) - \PP\big(|Y+d|_{\infty}\le x)\Big|\lesssim \Delta^{1/3}\log^{2/3}(v),$$
    where the constant in $\lesssim$ only depends on $\sigma$.
\end{lemma}

\begin{lemma}[\cite{Nazarov2003OnTM}]
    \label{lemma_nazarov}
    Let $X=(X_1,\ldots,X_s)^{\top}$ be a centered Gaussian vector in $\RR^s$. Assume $\EE(X_i^2)\ge b$, for some $b>0$, $1\le i\le s$. Then for any $e>0$ and $d\in\RR^s$, 
    \begin{equation}
        \label{eq_nazarov}
        \sup_{x\in \mathbb{R}}\PP\Big(\big||X+d|_{\infty}-x\big|\le e\Big) \le ce(\sqrt{2\log(s)}+2),
    \end{equation}
    where $c$ is a constant depending only on $b$.
\end{lemma}

\begin{lemma}[Bounds of $L^s$-norm of an infinite sum]
    \label{lemma_snorm}
    For matrices $A,B\in\RR^{\tilde p \times p}$ and two independent centered random vectors $x,y\in\RR^p$ whose elements are also independent. Let $A_j$ and $B_j$ be the $j$-th rows of $A$ and $B$, respectively. If $\mu_s'\le \lVert| xy^{\top}|_{\min}\rVert_s\le \lVert| xy^{\top}|_{\max}\rVert_s\le \mu_s$, then we have 
    
    \noindent(i) the upper bound,
    $$\Big\lVert\sum_{j=1}^{\tilde p}A_j^{\top}xy^{\top}B_j\Big\rVert_s^2 \lesssim \Big\lVert\sum_{j=1}^{\tilde p} B_jA_j^{\top}\Big\rVert_F^2 = \sum_{j_1,j_2=1}^{\tilde p}\Big(A_{j_2}^{\top}A_{j_1}B_{j_1}^{\top}B_{j_2}\Big),$$
    
    \noindent(ii) the lower bound,
    $$\Big\lVert\sum_{j=1}^{\tilde p}A_j^{\top}xy^{\top}B_j\Big\rVert_s^2 \gtrsim \Big\lVert\sum_{j=1}^{\tilde p} B_jA_j^{\top}\Big\rVert_F^2 = \sum_{j_1,j_2=1}^{\tilde p}\Big(A_{j_2}^{\top}A_{j_1}B_{j_1}^{\top}B_{j_2}\Big),$$
    where the constants in $\lesssim$ and $\gtrsim$ are independent of $p$ and $\tilde p$.
\end{lemma}

\begin{proof}
First, note that
\begin{equation}
    \Big\lVert\sum_{j=1}^{\tilde p}A_j^{\top}xy^{\top}B_j\Big\rVert_s^2 = \Big\lVert\text{tr}\Big(\big(\sum_{j=1}^{\tilde p}B_jA_j^{\top}\big)xy^{\top}\Big)\Big\rVert_s^2. \nonumber 
\end{equation}
Then, let $C=\sum_{j=1}^{\tilde p}B_jA_j^{\top}$, and $Z=xy^{\top}$, which along with Lemma \ref{lemma_burkholder} gives
\begin{align}
    \Big\lVert\sum_{j=1}^{\tilde p}A_j^{\top}xy^{\top}B_j\Big\rVert_s^2 & = \big\lVert \text{tr}(CZ)\big\rVert_s^2 = \Big\lVert 2\sum_{i}\sum_{j<i}C_{i,j}Z_{i,j} + \sum_{i}C_{i,i}Z_{i,i}\Big\rVert_s^2 
    \nonumber \\
    & \le 2\Big\lVert 2\sum_{i}\sum_{j<i}C_{i,j}Z_{i,j}\Big\rVert_s^2 + 2\Big\lVert\sum_{i}C_{i,i}Z_{i,i}\Big\rVert_s^2 \nonumber \\ 
    & \lesssim \sum_{i}\sum_{j<i}C_{i,j}^2\lVert Z_{i,j}\rVert_s^2 + \sum_{i}C_{i,i}^2\lVert Z_{i,i}\rVert_s^2 \lesssim \lVert C\rVert_F^2, \nonumber
\end{align}
where the constants in $\lesssim$ here and the rest of the proof are independent of $p$ and $\tilde p$. Note that
\begin{align}
    & \Big\lVert\sum_{j=1}^{\tilde p} B_jA_j^{\top}\Big\rVert_F^2 = \text{tr}\Big[\Big(\sum_{j=1}^{\tilde p}A_jB_j^{\top}\Big)\Big(\sum_{j=1}^{\tilde p}B_jA_j^{\top}\Big)\Big] =\sum_{j_1,j_2=1}^{\tilde p}\text{tr}\Big(A_{j_1}B_{j_1}^{\top}B_{j_2}A_{j_2}^{\top}\Big) \nonumber \\
    = & \sum_{j_1,j_2=1}^{\tilde p}\Big(A_{j_2}^{\top}A_{j_1}B_{j_1}^{\top}B_{j_2}\Big).
\end{align}
The desired upper bound is obtained. For the lower bound, we have
\begin{align}
    & \Big\lVert\sum_{j=1}^{\tilde p}A_j^{\top}xy^{\top}B_j\Big\rVert_s^2 
    \ge 2\Big\lVert 2\sum_{i}\sum_{j<i}C_{i,j}Z_{i,j}\Big\rVert_s^2
    \ge \sum_{i}\sum_{j<i}C_{i,j}^2\lVert Z_{i,j}\rVert_s^2 \gtrsim \lVert C\rVert_F^2, \nonumber
\end{align}
which completes the proof.
\end{proof}

\subsection{Proof of Theorem \ref{thm1_constanttrend}}\label{subsec_thm1proof}

We first provide an outline of the proof and then fulfill the details afterwards to enhance readability. To investigate the null limiting distribution of our test statistic $\Q_n$, we define $I_{\epsilon}$ as $\Q_n$ under the null, that is
\begin{equation}
    I_{\epsilon} := \max_{bn+1\le i \le n-bn}\Bigg(\Big|\sum_{t=i-bn}^{i-1}(bn)^{-1}\Lambda^{-1}\epsilon_t-\sum_{t=i}^{i+bn-1}(bn)^{-1}\Lambda^{-1}\epsilon_t\Big|_2^2-\bar{c}\Bigg).
\end{equation}
Therefore, it suffices to study $I_{\epsilon}$, which can be rewritten into the maximum of a sum over $j$, that is,
\begin{align}
    \label{eq_thm1_Iepsilon2}
    & I_{\epsilon} =\max_{bn+1\le i \le n-bn}\Bigg(\sum_{j=1}^{p}(bn)^{-2}\sigma_j^{-2}\Big(\sum_{t=i-bn}^{i-1}\epsilon_{t,j}-\sum_{t=i}^{i+bn-1}\epsilon_{t,j}\Big)^2-\bar{c}\Bigg),
\end{align}
where the centering term $\bar{c}$ of the $\ell^2$-norm in $I_{\epsilon}$ can be evaluated as follows:
\begin{align}
    \label{eq_thm1_center}
    & \bar{c}= \sum_{j=1}^p(bn)^{-2}\sigma_j^{-2}\EE\Big(\sum_{t=i-bn}^{i-1}\epsilon_{t,j}-\sum_{t=i}^{i+bn-1}\epsilon_{t,j}\Big)^2 =2p/(bn)+O\big\{p/(bn)^2\big\}.
\end{align}
Thus, after we scale $I_{\epsilon}$ by multiplying $bn/\sqrt{p}$, the rescaled $I_{\epsilon}$ can be written into the maximum of a sum of centered independent random variables $x_{i,j}$ defined in expression (\ref{eq_thm1_x_def}), that is,
\begin{equation}
    \label{eq_thm1_prime_Iepsilon2}
    \frac{bn}{\sqrt{p}} I_{\epsilon}=\max_{bn+1\le i \le n-bn}\frac{bn}{\sqrt{p}}\sum_{j=1}^{p}x_{i,j},
\end{equation}
where $x_{i,j}$ can also be written into
\begin{equation}
    \label{eq_thm1_x}
    x_{i,j}=\frac{1}{(bn\sigma_j)^2}\Bigg[\Big(\sum_{t=i-bn}^{i-1}\epsilon_{t,j}-\sum_{t=i}^{i+bn-1}\epsilon_{t,j}\Big)^2 - \EE\Big(\sum_{t=i-bn}^{i-1}\epsilon_{t,j}-\sum_{t=i}^{i+bn-1}\epsilon_{t,j}\Big)^2\Bigg].
\end{equation}
Motivated by the Gaussian approximation theorem in \textcite{chernozhukov_central_2017}, for each $x_{i,j}$, we consider its Gaussian counterpart $z_{i,j}$. Specifically, we define $Z_j=(z_{bn+1,j},\ldots,z_{n-bn,j})^{\top}\in\RR^{n-2bn}$, $1\le j\le p$, as independent centered Gaussian random vectors with covariance matrix $\EE(X_jX_j^{\top})\in\RR^{(n-2bn)\times(n-2bn)}$, where $X_j=(x_{bn+1,j},\ldots,x_{n-bn,j})^{\top}$. The analysis of the covariance matrix is deferred to Lemma \ref{lemma_cov_thm1}. We let
\begin{equation}
    \bar{z}_i=\sum_{j=1}^pz_{i,j}, \quad bn+1\le i\le n-bn.
\end{equation}
By the Gaussian approximation, we have the distribution of $I_{\epsilon}$ approached by the one of $\max_i\bar{z}_i$. Finally, we show that the maximum of a non-centered Gaussian distribution is continuous to complete our proof.

\begin{proof}[Proof of Theorem \ref{thm1_constanttrend}]
Now we provide the detailed proof of Theorem \ref{thm1_constanttrend}. Our main goal is to bound the left-hand side of expression (\ref{eq_thm11_result}). To this end, since $\Q_n=I_{\epsilon}$, it follows from expression (\ref{eq_thm1_prime_Iepsilon2}) and Lemma \ref{lemma_chen2019} that
\begin{align}
    & \quad \sup_{u\in\mathbb{R}}\big|\PP\big(\Q_n \le u\big)-\PP\big(\max_{bn+1\le i\le n-bn}\Z_i\le u\big)\big|\nonumber \\
    & \le \sup_{u\in \mathbb{R}}\big\vert\PP(I_{\epsilon}\le u)-\PP\big(\max_i\bar{z}_i\le u\big)\big\vert +\sup_{u\in \mathbb{R}}\big\vert\PP\big(\max_i\bar{z}_i\le u\big)-\PP\big(\max_i\Z_i\le u\big)\big\vert \nonumber \\
    & \lesssim \sup_{u\in \mathbb{R}}\Big|\PP\Big(\max_i\sum_{j=1}^{p}x_{i,j}\le u\Big)-\PP\Big(\max_i\sum_{j=1}^{p}z_{i,j}\le u\Big)\Big|+(bn)^{-1/3}\log^{2/3}(n), \nonumber 
\end{align}
where the last inequality follows from Lemma \ref{lemma_chen2019} and Lemma \ref{lemma_cov_thm3}. In view of expression (\ref{eq_epsilon_linear}) and the cross-sectional independence
of errors,  $\epsilon_{t,j}=\sum_{l\le  t}A_{t-l,j,j}\eta_{l,j}$. Then, we can rewrite $x_{i,j}$ into 
\begin{align}
    \label{eq_thm1_x_linear}
    & x_{i,j} = \frac{1}{(bn\sigma_j)^2}\Bigg[\Big(\sum_{l\le i+bn-1} a_{i,l,j}\eta_{l,j}\Big)^2 - \EE\Big(\sum_{l\le i+bn-1} a_{i,l,j}\eta_{l,j}\Big)^2\Bigg],
\end{align}
where
\begin{equation}
    \label{eq_thm11_defa}
    a_{i,l,j}=
    \begin{cases}
        \sum_{t=(i-bn)\vee l}^{i-1}A_{t-l,j,j}-\sum_{t=i\vee l}^{i+bn-1}A_{t-l,j,j} , & \quad \text{if } l\le i-1,\\
        \sum_{t=i\vee l}^{i+bn-1}A_{t-l,j,j}, & \quad \text{if } i\le l\le i+bn-1.
    \end{cases}
\end{equation}
Now we are ready to verify the conditions with regard to $bnx_{i,j}$ to apply the Gaussian approximation and it shall be noted that our sum is over $j$. First, for all $bn+1\le i\le n-bn$ and $s=6,8$ , we aim to bound $p^{-1}\sum_{j=1}^p\EE |bnx_{i,j}|^{s/2}$ from above. Let $\{D_{i,l,j}\}_l$ be a sequence of martingale differences with respect to $(\ldots,\eta_{l-2},\eta_{l-1},\eta_l)$, that is,
\begin{equation}
    \label{eq_thm11_md}
    D_{i,l,j}= 2a_{i,l,j}\eta_{l,j}\sum_{r<l} a_{i,r,j}\eta_{r,j} + a_{i,l,j}^2(\eta_{l,j}^2 - \EE\eta_{l,j}^2).
\end{equation}
This, along with the definition of $x_{i,j}$ in expression (\ref{eq_thm1_x_linear}) yields
\begin{align}
    \label{eq_thm11_m2step1}
    bnx_{i,j} & = \frac{1}{bn\sigma_j^2}\sum_{l\le i+bn-1} \Big(2a_{i,l,j}\eta_{l,j}\sum_{r<l} a_{i,r,j}\eta_{r,j} + a_{i,l,j}^2(\eta_{l,j}^2 - \EE\eta_{l,j}^2)\Big) \nonumber \\
    & = \frac{1}{bn\sigma_j^2}\sum_{l\le i+bn-1}D_{i,l,j}.
\end{align}
By applying Lemma \ref{lemma_burkholder}, we have, for $s=6,8$,
\begin{align}
    \label{eq_thm11_m2step2}
    \Big\lVert\sum_{l\le i+bn-1}D_{i,l,j} \Big\rVert_{s/2} & \lesssim \Big(\sum_{l\le i+bn-1}\lVert D_{i,l,j}\rVert_{s/2}^2 \Big)^{1/2} \nonumber \\
    & \le \Big[\sum_{l\le i+bn-1}\Big(2\lVert a_{i,l,j}\eta_{l,j}\rVert_{s/2}\Big\lVert\sum_{r<l}a_{i,r,j}\eta_{r,j}\Big\rVert_{s/2} \nonumber \\
    & \quad + a_{i,l,j}^2\big\lVert\eta_{l,j}^2 - \EE\eta_{l,j}^2\big\rVert_{s/2}\Big)^2\Big]^{1/2},
\end{align}
where the constants in $\lesssim$ here and in the rest of the proof are independent of $n,p$ and $b$, so are the ones in $\gtrsim$ and $O(\cdot)$. Then, given the i.i.d. sequence $\{\eta_{l,j}\}_l$, it follows from Lemma \ref{lemma_burkholder} and Assumption \ref{asm_finitemoment} that 
\begin{equation}
    \label{eq_thm11_m2step3part1}
    \Big\lVert\sum_{r<l}a_{i,r,j}\eta_{r,j}\Big\rVert_{s/2} \le \Big(\sum_{r<l}a_{i,r,j}^2\lVert\eta_{r,j}\rVert_{s/2}^2\Big)^{1/2} = \mu_{s/2}\Big(\sum_{r<l}a_{i,r,j}^2\Big)^{1/2}.
\end{equation}
The combination of expressions (\ref{eq_thm11_m2step1})--(\ref{eq_thm11_m2step3part1}) leads to, for the constant $s=6,8$,
\begin{equation}
    \label{eq_thm11_m2expression}
    \EE |bnx_{i,j}|^{s/2} \lesssim \frac{1}{(bn)^{s/2}\sigma_j^s}\Big(\sum_{l\le i+bn}a_{i,l,j}^2\sum_{r\le l}a_{i,r,j}^2 \Big)^s.
\end{equation}
Let $\tilde A_{0,j,j}= \sum_{t\ge 0}A_{t,j,j}$. By the definition of $a_{i,l,j}$ in expression (\ref{eq_thm11_defa}), when $i\le l\le i+bn$, we have
\begin{align}
    \label{eq_thm11_m2step4}
    \sum_{l\le i+bn}a_{i,l,j}^2 & = \sum_{l\le i+bn}\Big( \sum_{t=(i+1)\vee l}^{i+bn}A_{t-l,j,j}\Big)^2 \nonumber \\
    & \le \Big(\sum_{l\le i+bn}\Big|\sum_{t=(i+1)\vee l}^{i+bn}A_{t-l,j,j}\Big|\Big) \max_{l\le i+bn}\Big|\sum_{t=(i+1)\vee l}^{i+bn}A_{t-l,j,j}\Big| \le bn\tilde A_{0,j,j}^2,
\end{align}
where the last inequality is due to
$$\sum_{l\le i+bn}\Big|\sum_{t=(i+1)\vee l}^{i+bn}A_{t-l,j,j}\Big| = \sum_{t=i+1}^{i+bn}\Big|\sum_{l\le t}A_{t-l,j,j}\Big|=bn|\tilde A_{0,j,j}|.$$ Therefore, in view of expressions (\ref{eq_thm11_m2expression}) and (\ref{eq_thm11_m2step4}) as well as a similar argument for the case with $l\le i-1$, we have
\begin{align}
    \label{eq_thm11_m2result}
    p^{-1}\sum_{j=1}^p\EE |bnx_{i,j}|^{s/2} \lesssim p^{-1}\sum_{j=1}^p\frac{1}{(bn)^{s/2}\sigma_j^s} (bn\tilde A_{0,j,j}^2)^{s/2} = O(1).
\end{align}
Similarly, we can also give an upper bound for $\EE\big(\max_{bn+1\le i\le n-bn}|bnx_{i,j}|^{q/2}\big)$. Specifically, for all $1\le j\le p$, 
\begin{align}
    \label{eq_thm11_e2result}
    &\quad \EE\big(\max_{bn+1\le i\le n-bn}|bnx_{i,j}|^{q/2}\big) \le \sum_{i=bn+1}^{n-bn}\EE |bnx_{i,j}|^{q/2} \lesssim n.
\end{align}
Now we shall prove $p^{-1}\sum_{j=1}^p\EE (bnx_{i,j})^2$ bounded away from zero. By expression (\ref{eq_thm11_m2step1}), we have
\begin{align}
    \label{eq_thm11_m1step1}
    & \EE (bnx_{i,j})^2 = \frac{1}{(bn\sigma_j^2)^2}\sum_{l\le i+bn}\EE D_{i,l,j}^2,
\end{align}
which along with expression (\ref{eq_thm11_md}) yields
\begin{equation}
    \label{eq_thm11_m1part1and2}
    \EE (bnx_{i,j})^2 \ge \frac{4}{(bn\sigma_j^2)^2}\sum_{l\le i+bn}\sum_{r<l}a_{i,l,j}^2a_{i,r,j}^2\EE\eta_{l,j}^2\EE\eta_{r,j}^2 \gtrsim \frac{1}{(bn\sigma_j^2)^2}\sum_{l\le i+bn}\sum_{r<l}a_{i,l,j}^2a_{i,r,j}^2.
\end{equation}
It follows from the definition of $a_{i,l,j}$ in expression (\ref{eq_thm11_defa}) that
\begin{align}
    \label{eq_thm11_m1step2}
    & \sum_{l\le i+bn}\sum_{r<l}a_{i,l,j}^2a_{i,r,j}^2
    \ge \sum_{i-bn\le l\le i+bn}\sum_{r<l}a_{i,l,j}^2a_{i,r,j}^2=(bn\tilde A_{0,j,j}^2)^2 +o\big\{(bn)^2\big\}.
\end{align}
Thus, in view of expressions (\ref{eq_thm11_m1part1and2}) and (\ref{eq_thm11_m1step2}), $p^{-1}\sum_{j=1}^p\EE (bnx_{i,j})^2$ is lower bounded away from zero, i.e., for some constant $c>0$,
\begin{equation}
    \label{eq_thm11_m1result}
    p^{-1}\sum_{j=1}^p\EE (bnx_{i,j})^2 \ge cp^{-1}\sum_{j=1}^p\frac{1}{\big(bn\sigma_j^2\big)^2}(bn\tilde A_{0,j,j}^2)^2 =c.
\end{equation}

For any $r\le q/2$, we define
$$M_r:=\max_{bn+1\le i\le n-bn }\Big(p^{-1}\sum_{j=1}^p\lVert bnx_{i,j}\rVert_r^r\Big)^{1/r}\quad \textrm{and}\quad
\tilde M_r:=\max_{1\le j\le p}\Big\lVert\max_{bn+1\le i\le n-bn}bnx_{i,j}\Big\rVert_r.$$
The two definitions above along with expressions (\ref{eq_thm11_m2result}) and (\ref{eq_thm11_e2result}) yield
\begin{align}
    B_n & :=\max\Big\{M_3^3,M_4^2, \tilde M_{q/2} \Big\}\lesssim  n^{2/q}. \nonumber
\end{align}
Finally, we apply Proposition 2.1 in \textcite{chernozhukov_central_2017} to get the desired result.
\end{proof}

\subsection{Proof of Theorem \ref{thm_consistency}}\label{subsec_thm4proof}

\begin{proof}[Proof of Theorem \ref{thm_consistency} (i)]
Note that $1-\Phi(x)\le (2\pi)^{-1}x^{-1}e^{-x^2/2}$, for $x>0$, where $\Phi(\cdot)$ is the cumulative distribution function of a standard normal distribution. Recall that $\bar{z}_i=\sum_{j=1}^pz_{i,j}$ is a centered Gaussian variable with covariance matrix $\sum_{j=1}^p\EE(X_jX_j^{\top})$, which asymptotically takes the form of expression (\ref{eq_cov}). This gives
\begin{align}
    \label{eq_thm4_step1}
    &\PP\Big(\frac{bn}{\sqrt{p}}\max_{bn+1\le i\le n-bn}\bar{z}_i\ge u\Big) \le \sum_{i=bn+1}^{n-bn} \PP\Big(\frac{1}{\sqrt{p}}\sum_{j=1}^pbnz_{i,j} \ge u\Big) \nonumber \\
    \le & n(2\pi)^{-1/2}(\tilde\sigma/u)e^{-u^2/(2\tilde\sigma^2)},
\end{align}
where $\tilde\sigma^2=p^{-1}\sum_{j=1}^p\EE (bnz_{i,j})^2$ which converges to 8 by expression (\ref{eq_cov}). Therefore,
\begin{equation}
    \label{eq_thm4_step1result}
    \PP\Big(\max_{bn+1\le i\le n-bn}\bar{z}_i\ge  4(p\log{(n)})^{1/2}(bn)^{-1}\Big)\rightarrow0. 
\end{equation}
Define $\I := \{1\le i\le n: |i-\tau_k|>bn, \text{ for all }1\le k\le K \}$. We note that for any $i\in\I$, $\EE V_i=0$. It follows from Theorem \ref{thm1_constanttrend} that 
\begin{equation}
    \label{eq_thm4_outofset}
    \sup_{u\in\RR}\Big|\PP\big(\max_{i\in\I}\big(|V_i|_2^2-\bar{c}\big)\ge u\big) - \PP\big(\max_{i\in\I}\bar{z}_i\ge u\big)\Big|\rightarrow0,
\end{equation}
which along with expression (\ref{eq_thm4_step1result}) and the fact that $\max_{i\in\I}\bar{z}_i\le \max_{bn+1\le i\le n-bn}\bar{z}_i$ yields
$\PP\big(\max_{i\in\I}\big(|V_i|_2-\bar{c}\big)\ge \omega\big)\rightarrow0$, for all $\omega\ge 4(p\log{(n)})^{1/2}(bn)^{-1}$. Thus, we have
\begin{equation}
    \label{eq_thm4_set1}
    \PP\big(\forall t\in\A_1,\, \text{there exists } 1\le k\le K ,\, |t-\tau_k|\le bn\big)\rightarrow1,
\end{equation}
as $n\rightarrow\infty$. Recall that the weighted break size $d_{\tau_k}=\Lambda^{-1}\gamma_k$ under the alternative. We denote the left-side and right-side weighted sums of errors as following
\begin{equation}
\label{eq_U}
    U_{i,j}^{(l)}= \sum_{t=i-bn}^{i-1}\epsilon_{t,j}/(bn), \quad U_{i,j}^{(r)}=\sum_{t=i}^{i+bn-1}\epsilon_{t,j}/(bn).
\end{equation} 
Let $\bar x_i=\sum_{j=1}^px_{i,j}$ with $x_{i,j}$ in the form of expression (\ref{eq_thm1_x}). Then, we have
\begin{align}
    \label{eq_thm4_z_alter}
    & \min_{1\le k\le K }\big(|V_{\tau_k}|_2^2 -\bar{c}\big)= \min_{1\le k\le K }\Big(\sigma_j^{-2}\sum_{j=1}^p\big(U_{\tau_k,j}^{(l)}-U_{\tau_k,j}^{(r)}+\sigma_jd_{\tau_k,j}\big)^2-\bar{c}\Big) \nonumber \\
    \ge & \min_{1\le k\le K }|d_{\tau_k}|_2^2 - 2\max_{1\le k\le K }\Big|\sum_{j=1}^pd_{\tau_k,j}(U_{\tau_k,j}^{(l)}-U_{\tau_k,j}^{(r)})/\sigma_j\Big| \nonumber \\
    &\quad - \max_{1\le k\le K }\Big|\bar{c}-\sum_{j=1}^p\big((U_{\tau_k,j}^{(l)}-U_{i,j}^{(r)})/\sigma_j\big)^2\Big|\nonumber \\
    \ge & \delta_p^2 - 2\max_{1\le k\le K }\Big|\sum_{j=1}^pd_{\tau_k,j}(U_{\tau_k,j}^{(l)}-U_{\tau_k,j}^{(r)})/\sigma_j\Big| -\max_{1\le k\le K }|\bar x_{\tau_k}|.
\end{align}
Regarding the second term in expression (\ref{eq_thm4_z_alter}), as a direct consequence of Lemma \ref{lemma_tail_cross} (i), we can obtain the tail probability
\begin{align}
    \label{eq_thm4_GA}
    \PP\Big(\max_{1\le k\le K }\Big|\sum_{j=1}^pd_{\tau_k,j}(U_{\tau_k,j}^{(l)}-U_{\tau_k,j}^{(r)})/\sigma_j\Big|/|\Lambda^{-1}\gamma_k|_2 \ge e^q\log^{1/2}(n)(bn)^{-1/2}\Big)\rightarrow0.
\end{align}
For the last term in expression (\ref{eq_thm4_z_alter}), under the similar lines of the proof for Theorem \ref{thm1_constanttrend}, we can show that, under the null,
$$\sup_{u\in\RR}\Big|\PP\big(\max_{bn+1\le i\le n-bn}|\bar x_i|\le u\big)-\PP\big(\max_{bn+1\le i\le n-bn}|\bar{z}_i|\le u\big)\Big| \rightarrow0,$$
which together with expressions (\ref{eq_thm4_step1result}), (\ref{eq_thm4_z_alter}), (\ref{eq_thm4_GA}) and the conditions in Theorem \ref{thm_consistency} that $\delta_p^2\ge 3\omega$ and $\omega\gg  4(p\log{(n)})^{1/2}(bn)^{-1}$ gives
\begin{equation}
    \label{eq_thm4_exceed}
    \PP\Big(\min_{1\le k\le K }\big(|V_{\tau_k}|_2^2-\bar{c}\big)\ge \omega\Big)\rightarrow1.
\end{equation}
Namely, the test statistics shall exceed the threshold at the break points with probability approaching 1. Hence, we have
\begin{equation}
    \label{eq_thm4_inside}
    \PP(\tau_k\in\A_1,1\le k\le K )\rightarrow1.
\end{equation}
Let $\D(\tau,r)=\{t:\,|t-\tau|\le r\}$. It follows from expressions (\ref{eq_thm4_set1}) and (\ref{eq_thm4_inside}) that 
$$\PP\Big(\{\tau_1,\tau_2,\ldots,\tau_{K }\}\subset\A_1\subset\cup_{1\le k\le K }\D(\tau_k,bn)\Big)\rightarrow1,$$
as $n\rightarrow\infty$. Since for $k_1\neq k_2$, $|\tau_{k_1}-\tau_{k_2}|\gg bn$, it follows that for $\forall t\in\D(\tau_{k_1},bn)$ and $k_1\neq k_2$, for all large $n$, $\D(t,2bn)\cap\D(\tau_{k_2},2bn)=\emptyset$. We obtain the desired result.
\end{proof}

\begin{lemma}[Cross-terms of signals and noises]
    \label{lemma_tail_cross}
    Consider the break locations $\tau_k$ and the break sizes $\gamma_k$, $1\le k\le K$, in model (\ref{eq_model}), and the weighted noise sums $U_{i,j}^{(l)}$, $U_{i,j}^{(r)}$ in (\ref{eq_U}). Under Assumptions \ref{asm_finitemoment}--\ref{asm_sec_indep}, if $\max_{1\le k\le K}|\Lambda^{-1}\gamma_k|_q/|\Lambda^{-1}\gamma_k|_2=O(1/K^{1/q})$ for some $q\ge8$, we have the following results:\\
    (i) 
    $$\PP\Bigg(\max_{1\le k\le K }\frac{\big|\sum_{j=1}^p(\sigma_j^{-1}\gamma_{k,j})(U_{\tau_k,j}^{(l)}-U_{\tau_k,j}^{(r)})/\sigma_j\big|}{|\Lambda^{-1}\gamma_k|_2} \ge e^q\log^{1/2}(n)(bn)^{-1/2}\Bigg)\rightarrow0.$$
    (ii)  
    $$\PP\Bigg(\max_{1\le k\le K,1\le|i-\tau_k|\le bn}\frac{\big|\sum_{j=1}^p(\sigma_j^{-1}\gamma_{k,j})(\Delta U_{i,j}-\Delta U_{\tau_k,j})/\sigma_j\big|}{|i-\tau_k|^{1/2}|\Lambda^{-1}\gamma_k|_2} \ge e^q\log(n)(bn)^{-1}\Bigg)\rightarrow0,$$
    where $\Delta U_{i,j}=U_{i,j}^{(l)} - U_{i,j}^{(r)}$. 
\end{lemma}

\begin{proof}[Proof of Lemma \ref{lemma_tail_cross} (i)]
For notation simplicity, we denote the normalized break size by
\begin{equation}
    \label{eq_thm4_std_breaksize}
    \vartheta_{k,j} := (\sigma_j^{-1}\gamma_{k,j})/|\Lambda^{-1}\gamma_k|_2,
\end{equation}
for $1\le k\le K$ and $1\le j\le p$, which indicates that $\sum_{j=1}^p\vartheta_{k,j}^2=1$ for all $k$. In addition, we denote the difference of the noise sums by
\begin{equation}
    \label{eq_thm4_std_diff}
    \U_{k,j} := (U_{\tau_k,j}^{(l)}-U_{\tau_k,j}^{(r)})/\sigma_j = \Delta U_{\tau_k,j}/\sigma_j.
\end{equation}
Recall $a_{i,l,j}$ defined in expression (\ref{eq_thm11_defa}), and we rewrite $\U_{k,j}$ into a sum of martingale differences, that is
\begin{equation}
    \U_{k,j} = \sum_{l\le \tau_k+bn-1}a_{\tau_k,l,j}\eta_{l,j}/(bn\sigma_j).
\end{equation}
Thus, by Assumptions \ref{asm_finitemoment} and \ref{asm_temp_dep}, it follows from Lemma \ref{lemma_burkholder} and the similar arguments in expression (\ref{eq_thm11_m2step4}) that
\begin{align}
    \label{eq_thm4_qnorm}
    \|\U_{k,j}\|_q & \lesssim \Big(\sum_{l\le \tau_k+bn-1}\big\|a_{\tau_k,l,j}\eta_{l,j}/(bn\sigma_j)\big\|_q^2\Big)^{1/2} \nonumber \\
    & = \mu_q\Big(\sum_{l\le \tau_k+bn-1}a_{\tau_k,l,j}^2/(bn\sigma_j)^2\Big)^{1/2} \lesssim 1/\sqrt{bn}.
\end{align}
uniformly over $k$ and $j$, where the constant in $\lesssim$ only depends on $q$. Let $q\ge8$. Since the noises $\epsilon_{t,j}$ are independent over $j$ by Assumption \ref{asm_sec_indep}, it follows from Corollary 1.7 in \textcite{nagaev_large_1979} that, for all $x>y>0$, and $y_1,\ldots,y_p>0$ with $y=\max_jy_j$,
\begin{align}
    \label{eq_thm4_nagaev1}
    & \quad \PP\Big(\max_{1\le k\le K }\Big|\sum_{j=1}^p\vartheta_{k,j}\U_{k,j}\Big| \ge x\Big) \nonumber \\
    & \le K\max_{1\le k\le K}\Bigg[\sum_{j=1}^p\PP\big(|\vartheta_{k,j}\U_{k,j}| \ge y_j\big) + \exp\Bigg\{\frac{-(1-c_q)^2x^2}{2e^q\sum_{j=1}^p\vartheta_{k,j}^2\EE(\U_{k,j}^2\One_{\U_{k,j}\le y_j})}\Bigg\} \nonumber \\
    & \quad + \Bigg(\frac{\sum_{j=1}^p|\vartheta_{k,j}|^q\EE(\U_{k,j}^q\One_{0\le\U_{k,j}\le y_j})}{c_q xy^{q-1}}\Bigg)^{c_q x/y}\Bigg]=:\III_1 + \III_2 + \III_3,
\end{align}
where $c_q=q/(q+2)$. We let $x=e^q\log^{1/2}(n)(bn)^{-1/2}$, and $y_j=c_qx$ for each $j$. We shall bound the three terms $\III_1$--$\III_3$ separately. For the term $\III_2$, since $K<1/b$, by expression (\ref{eq_thm4_qnorm}), we have
\begin{align}
    \III_2 \lesssim K\max_{1\le k\le K} \exp\Bigg\{\frac{-(1-c_q)^2e^{2q}\log(n)/(bn)}{2e^q/(bn)}\Bigg\},
\end{align}
which converges to 0 as $n$ diverges. 

Next, for the term $\III_3$, we denote the ratio $\varpi_{k,q}=|\Lambda^{-1}\gamma_k|_q/|\Lambda^{-1}\gamma_k|_2$. It can be shown that for all $1\le k\le K$ and $q>2$, $\varpi_{k,q}\le 1$. Note that $\sum_{j=1}^p|\vartheta_{k,j}|^q = \varpi_{k,q}^q$, which along with expression (\ref{eq_thm4_qnorm})  yields
\begin{align}
    \III_3 \lesssim K\max_{1\le k\le K}\Bigg(\frac{\varpi_{k,q}^q/(bn)^{q/2}}{c_q^qe^{q^2}\log^{q/2}(n)/(bn)^{q/2}}\Bigg),
\end{align}
which also converges to 0 as $n$ diverges due to the condition that $\max_{1\le k\le K}\varpi_{k,q}^q=O(1/K)$.

Lastly, we bound the term $\III_1$. Since $\eta_{i,j}$ are i.i.d. random variables, we can apply the Nagaev-type inequality in \textcite{nagaev_large_1979} again. Specifically, by Assumptions \ref{asm_finitemoment}, \ref{asm_temp_dep}, and expression (\ref{eq_thm4_qnorm}), we have
\begin{align}
    \label{eq_thm4_nagaev2}
    \PP\big(|\vartheta_{k,j}\U_{k,j}| \ge y_j\big) 
    & \le \sum_{l\le \tau_k+bn-1}\PP\big(|\vartheta_{k,j}a_{\tau_k,l,j}\eta_{l,j}|/(bn\sigma_j) \ge z_{j,l}\big) \nonumber \\
    & \quad + \exp\Bigg\{\frac{-(1-c_q)^2y_j^2}{2e^q\vartheta_{k,j}^2\sum_{l\le\tau_k+bn-1}a_{\tau_k,l,j}^2/(bn\sigma_j)^2}\Bigg\} \nonumber \\
    & \quad + \Bigg(\frac{|\vartheta_{k,j}|^q\sum_{l\le\tau_k+bn-1}|a_{\tau_k,l,j}|^q/(bn\sigma_j)^q}{c_q y_jz_j^{q-1}}\Bigg)^{c_q y_j/z_j} \nonumber \\
    & =:\III_{k,j,1} + \III_{k,j,2}+\III_{k,j,3},
\end{align}
where $a_{i,l,j}$ is defined in expression (\ref{eq_thm11_defa}), $z_{j,l}>0$ and $z_j=\max_lz_{j,l}$. Recall that $y_j=c_qe^q\log^{1/2}(n)(bn)^{-1/2}$ for all $j$. We let $z_{j,l}=c_qy_j$ for all $l$. Note that, by Markov's inequality and Assumption \ref{asm_finitemoment}, we have
\begin{align}
    \PP\big(|\vartheta_{k,j}a_{\tau_k,l,j}\eta_{l,j}|/(bn\sigma_j) \ge z_{j,l}\big) & \le (bn)^{q/2}\log^{-q/2}(n)\EE\big(|\vartheta_{k,j}a_{\tau_k,l,j}\eta_{l,j}|/(bn\sigma_j)\big)^q \nonumber \\
    & \lesssim (bn)^{-q/2}\log^{-q/2}(n)|\vartheta_{k,j}|^q|a_{\tau_k,l,j}/\sigma_j|^q.
\end{align}
This, together with Assumption \ref{asm_temp_dep} and the similar arguments in expression (\ref{eq_thm11_m2step4}) gives
\begin{align}
    K\max_{1\le k\le K}\sum_{1\le j\le p}\III_{k,j,1} & \lesssim K\max_{1\le k\le K}\sum_{1\le j\le p}|\vartheta_{k,j}|^q\sum_{l\le \tau_k+bn-1}(bn)^{-q/2}\log^{-q/2}(n)|a_{\tau_k,l,j}/\sigma_j|^q \nonumber \\
    & \le K\max_{1\le k\le K}\varpi_{k,q}^q\max_{1\le j\le p}\sum_{l\le \tau_k+bn-1}(bn)^{-q/2}\log^{-q/2}(n)|a_{\tau_k,l,j}/\sigma_j|^q \nonumber \\
    & \lesssim K(bn)^{1-q/2}\log^{-q/2}(n),
\end{align}
which converges to 0 as $n$ diverges. For the term $\III_{k,j,2}$, we have
\begin{align}
    K\max_{1\le k\le K}\sum_{1\le j\le p}\III_{k,j,2} & \le K\max_{1\le k\le K}\sum_{1\le j\le p}\exp\Bigg\{\frac{-(1-c_q)^2c_q^2e^{2q}\log(n)/(bn)}{2e^q\vartheta_{k,j}^2/(bn)}\Bigg\} \nonumber \\
    & \le K\max_{1\le k\le K}\exp\Bigg\{\frac{-(1-c_q)^2c_q^2e^{2q}\log(n)}{2e^q}\Bigg\},
\end{align}
and it tends to 0 as $n\rightarrow\infty$. For the last term $\III_{k,j,3}$, it follows from Assumption \ref{asm_temp_dep} and the similar arguments in expression (\ref{eq_thm11_m2step4}) that
\begin{align}
    K\max_{1\le k\le K}\sum_{1\le j\le p}\III_{k,j,3} & \lesssim  K\max_{1\le k\le K}\sum_{1\le j\le p}\Bigg(\frac{|\vartheta_{k,j}|^q/(bn)^{q-1}}{c_q^{2q}e^{q^2} \log^{q/2}(n)/(bn)^{q/2}}\Bigg) \nonumber \\
    & \le K\max_{1\le k\le K}\Bigg(\frac{\varpi_{k,q}^q}{c_q^{2q}e^{q^2} \log^{q/2}(n)(bn)^{q/2-1}}\Bigg),
\end{align}
which converges to 0 as $n\rightarrow\infty$. The desired result is achieved.
\end{proof}

\begin{proof}[Proof of Lemma \ref{lemma_tail_cross} (ii)]
We first assume that $i< \tau_k$ and the other side can be similarly dealt with. Note that
\begin{align}
    (\Delta U_{i,j}-\Delta U_{\tau_k,j})/\sigma_j = (bn\sigma_j)^{-1}\Big(\sum_{t=i-bn}^{\tau_k-bn-1}\epsilon_{t,j} - 2\sum_{t=i}^{\tau_k-1}\epsilon_{t,j} + \sum_{t=i+bn}^{\tau_k+bn-1}\epsilon_{t,j}\Big).
\end{align}
For each $1\le j\le p$, $1\le k\le K$ and $1\le |i-\tau_k|\le bn$, we denote the partial sum
\begin{equation}
    S_{k,i,j} = \sum_{t=i-bn}^{\tau_k-bn-1}\epsilon_{t,j}/(bn\sigma_j).
\end{equation}
Let $d=\lceil \log(bn)/\log(2)\rceil$. Then, by Assumptions \ref{asm_finitemoment}, \ref{asm_temp_dep} and the similar argument in expression (\ref{eq_thm11_m2step4}), we have, for some $q\ge8$,
\begin{align}
    \label{eq_thm4_max_bd}
    \EE\Big(\max_{\{i:\,1\le|i-\tau_k|\le bn}\}\frac{|S_{k,i,j}|^q}{|i-\tau_k|^{q/2}}\Big)
    & \le \sum_{g=1}^d \EE\Big(\max_{\{i:\,2^{g-1}\le |i-\tau_k|\le 2^g-1}\}\frac{|S_{k,i,j}|^q}{|i-\tau_k|^{q/2}}\Big) \nonumber \\
    & \le \sum_{g=1}^d \frac{1}{(2^{g-1})^{q/2}}\EE\big(\max_{\{i:\,1\le|i-\tau_k|\le 2^g-1}\}|S_{k,i,j}|^q\big) \nonumber \\
    & \lesssim \sum_{g=1}^d \frac{1}{(2^{g-1})^{q/2}} (2^g)^{q/2}/(bn)^q \lesssim \log(bn)/(bn)^q.
\end{align}
We denote $\S_{k,j}=\max_{\{i:\,|i-\tau_k|\le bn\}}|S_{k,i,j}|/|i-\tau_k|^{1/2}$. Then, it follows from the Nagaev-type inequality in Corollary 1.7 in \textcite{nagaev_large_1979} that, for any $x>y>0$, and $y_1,\ldots,y_p>0$ with $y=\max_jy_j$,
\begin{align}
    & \quad \PP\Big(\max_{1\le k\le K }\Big|\sum_{j=1}^p\vartheta_{k,j}\S_{k,j}\Big| \ge x\Big) \nonumber \\
    & \le K\max_{1\le k\le K}\Bigg[\sum_{j=1}^p\PP\big(|\vartheta_{k,j}\S_{k,j}| \ge y_j\big) + \exp\Bigg\{\frac{-(1-c_q)^2x^2}{2e^q\sum_{j=1}^p\vartheta_{k,j}^2\EE(\S_{k,j}^2)}\Bigg\} \nonumber \\
    & \quad + \Bigg(\frac{\sum_{j=1}^p\vartheta_{k,j}^q\EE(\S_{k,j}^q)}{c_q xy^{q-1}}\Bigg)^{c_q x/y}\Bigg]=:\III_1' + \III_2' + \III_3',
\end{align}
where $c_q=q/(q+2)$. Let $x=e^q\log(n)(bn)^{-1}$ and $y_j=c_qx$ for all $j$. Recall the normalized break size $\vartheta_{k,j}$ defined in expression (\ref{eq_thm4_std_breaksize}) and the ratio $\varpi_{k,q}=|\Lambda^{-1}\gamma_k|_q/|\Lambda^{-1}\gamma_k|_2$. Note that by Markov's inequality and expression (\ref{eq_thm4_max_bd}), we have
\begin{align}
    \sum_{j=1}^p\PP\big(|\vartheta_{k,j}\S_{k,j}| \ge y_j\big) & \le \sum_{j=1}^py_j^{-q}|\vartheta_{k,j}|^q\EE(\S_{k,j}^q)\nonumber \\
    & \lesssim \varpi_{k,q}^q \log^{-q}(n)\log(bn).
\end{align}
As a direct consequence, we obtain
\begin{align}
    \III_1' \lesssim K\max_{1\le k\le K}\varpi_{k,q}^q \log^{-q}(n)\log(bn),
\end{align}
which converges to 0 as $n\rightarrow\infty$ since $K\max_{1\le k\le K}\varpi_{k,q}^q=O(1)$.

For the term $\III_2'$, since $\sum_{j=1}^p\vartheta_{k,j}^2=1$, it follows from expression (\ref{eq_thm4_max_bd}) that
\begin{align}
    \III_2' \lesssim K\exp\Bigg\{\frac{-(1-c_q)^2e^{2q}\log^2(n)/(bn)^2}{2e^q\log(bn)/(bn)^2}\Bigg\},
\end{align}
which converges to 0 as $n$ grows since $K<1/b$. Regarding to the term $\III_3'$, by implementing expression (\ref{eq_thm4_max_bd}), we obtain
\begin{align}
    \III_3' \lesssim K\max_{1\le k\le K}\Bigg(\frac{\varpi_{k,q}^q\log(bn)/(bn)^q}{c_q^qe^{q^2} \log^{q}(n)/(bn)^q}\Bigg).
\end{align}
This result also tends to 0 as $n$ diverges since $\max_{1\le k\le K}\varpi_{k,q}^q=O(1/K)$. The desired result is achieved.

\end{proof}

\begin{proof}[Proof of Theorem \ref{thm_consistency} (ii)]
Let $\mu_i^{(l)}$ (resp. $\mu_i^{(r)}$) be $\hat\mu_i^{(l)}$ (resp. $\hat\mu_i^{(r)}$) with $Y_i$ therein replaced by $\mu(i/n)$. Recall the definitions of $U_{i,j}^{(l)}$ and $U_{i,j}^{(r)}$ in expression (\ref{eq_U}) and let $U_i^{(l)}=(U_{i,j}^{(l)})_{1\le j\le p}$ and $U_i^{(r)}=(U_{i,j}^{(r)})_{1\le j\le p}$. Then, we define $\Delta\mu_i = \mu_i^{(l)}-\mu_i^{(r)}$ and $\Delta U_i = U_i^{(l)}-U_i^{(r)}$. For each $1\le k\le K $ and all $t$ such that $|t-\tau_k|\le bn$, we have $\Delta\mu_i=\Big(1-|i-\tau_{k}|/(bn)\Big)\gamma_{k}$. Therefore, we can decompose $V_{t,j}$ into the signal and error parts, that is
\begin{align}
    V_{t,j} & = \sigma_j^{-1}(\hat\mu_t^{(l)}-\hat\mu_t^{(r)})  = \sigma_j^{-1}\gamma_{k,j}\Big(1-\frac{1}{bn}|t-\tau_k|\Big) + \sigma_j^{-1}\Delta U_{t,j}.\nonumber 
\end{align}
% By the definition of $\hat\tau_k$, we shall expect that $ |V_{\tau_k}|_2<|V_{\hat\tau_k}|_2$, 
This leads to
\begin{align}
    & \sum_{j=1}^p (V_{\tau_k,j}^2 - V_{t,j}^2) = \sum_{j=1}^p\Bigg[\Big(\frac{2}{bn}|t-\tau_k| - \frac{1}{(bn)^2}|t-\tau_k|^2\Big)(\sigma_j^{-1}\gamma_{k,j})^2 \nonumber \\
    - & 2(\sigma_j^{-1}\gamma_{k,j})\sigma_j^{-1}(\Delta U_{t,j}-\Delta U_{\tau_k,j})  + 2(\sigma_j^{-1}\gamma_{k,j})|t-\tau_k|(bn)^{-1}\sigma_j^{-1}\Delta U_{t,j} \nonumber \\
    - & \sigma_j^{-2}(\Delta U_{t,j}^2-\Delta U_{\tau_k,j}^2)\Bigg] := m_{k,t,1}-m_{k,t,2}+m_{k,t,3}-m_{k,t,4}.
\end{align}
Now we investigate the four parts $m_{k,1}$ - $m_{k,4}$, respectively. Note that when $t=\tau_k$, all the four parts equal to zero. Therefore, we shall consider the case where $1\le |t-\tau_k|\le bn$. For the non-stochastic part $m_{k,1}$, we can directly achieve that
\begin{equation}
    m_{k,t,1} = \Big(\frac{2}{bn}|t-\tau_k| - \frac{1}{(bn)^2}|t-\tau_k|^2\Big)|\Lambda^{-1}\gamma_k|_2^2, \nonumber
\end{equation}
which gives
\begin{equation}
    \label{eq_thm4_m1}
    \min_{\substack{1\le k\le K \\ 1\le|t-\tau_k|\le bn}}m_{k,t,1}/\big(|t-\tau_k|\cdot|\Lambda^{-1}\gamma_k|_2^2\big) \gtrsim 1/(bn),
\end{equation}
uniformly over $t,k$. For $m_{k,t,2}$, by Lemma \ref{lemma_tail_cross} (ii), we have
\begin{align}
\label{eq_thm4_m2}
    \max_{\substack{1\le k\le K \\ 1\le|t-\tau_k|\le bn}}m_{k,t,2}/\big(|t-\tau_k|^{1/2}\cdot|\Lambda^{-1}\gamma_k|_2\big) = O_{\PP}\big\{\log(n)/(bn)\big\},
\end{align}
uniformly over $t,k$. 
By the similar arguments, for $m_{k,3}$, we have
\begin{align}
\label{eq_thm4_m3}
    \max_{\substack{1\le k\le K \\ 1\le|t-\tau_k|\le bn}}m_{k,t,3}/\big(|t-\tau_k|\cdot|\Lambda^{-1}\gamma_k|_2\big) = O_{\PP}\big\{\log(n)/(bn)^{3/2}\big\},
\end{align}
uniformly over $t,k$. Finally, we note that the part $m_{k,t,4}$ is centered since the expectations of two squared terms are cancelled out. Thus, its stochastic order is
\begin{align}
    & m_{k,t,4} = O_{\PP}\Big\{\sqrt{p}\big(|t-\tau_k|/(bn)^2+|t-\tau_k|^{1/2}/(bn)^{3/2}\big)\Big\}=O_{\PP}\Big\{\sqrt{p}|t-\tau_k|^{1/2}/(bn)^{3/2}\Big\}.\nonumber
\end{align}
Then it follows that
\begin{align}
    \label{eq_thm4_m4}
     \max_{\substack{1\le k\le K \\ 1\le|t-\tau_k|\le bn}}m_{k,t,4}/|t-\tau_k|^{1/2} =O_{\PP}\big\{ \sqrt{p}\log(n)/(bn)^{3/2}\big\}.
\end{align}
Let $\Gamma_k=p/(bn|\Lambda^{-1}\gamma_k|_2^2)$. By combining expressions (\ref{eq_thm4_m1})--(\ref{eq_thm4_m4}) and condition in Theorem \ref{thm_consistency}, we have, uniformly over $k$,
$$\max_{1\le k\le K }|\hat\tau_k-\tau_{k^*}|\cdot|\Lambda^{-1}\gamma_k|_2^2/(1+\Gamma_k)= O_{\PP}\big\{\log^2(n)\big\}.$$ 
\end{proof}

\begin{proof}[Proof of Theorem \ref{thm_consistency} (iii)]
Recall the definitions of $\mu_t^{(r)}$, $\mu_t^{(l)}$, $U_t^{(r)}$ and $U_t^{(l)}$ in the proof of (ii). Since $bn\gg (p\log(n))\delta_p^{-2}$,
$$\big|\mu_{\tau_k-bn}^{(l)}-\mu((\hat\tau_k-bn)/n)\big|_2=0,$$
and similarly, $\big|\mu_{\tau_k+bn-1}^{(r)}-\mu((\hat\tau_k+bn-1)/n)\big|_2=0$. Note that we also have
$$\big|\mu((\hat\tau_k-bn)/n)-\mu((\hat\tau_k+bn-1)/n)-\gamma_{k^*}\big|_2=0,$$
which yields
\begin{align}
    \label{eq_thm4_breaksize}
    \big||\Lambda^{-1}(\hat\gamma_k-\gamma_{k^*})|_2^2-\bar{c}\big| 
    & \lesssim \Big|\big|\Lambda^{-1}(\mu_{\hat\tau_k-bn}^{(l)}-\mu_{\hat\tau_k+bn-1}^{(r)}-\gamma_{k^*})\big|_2^2 \nonumber \\
    & \quad + \big|\Lambda^{-1}(U_{\hat\tau_k-bn}^{(l)}-U_{\hat\tau_k+bn-1}^{(r)})\big|_2^2 - \bar{c} \Big| \nonumber \\
    & = \Big|\big|\Lambda^{-1}(U_{\hat\tau_k-bn}^{(l)}-U_{\hat\tau_k+bn-1}^{(r)})\big|_2^2-\bar{c}\Big|.
\end{align}
It follows from the Gaussian approximation in Theorem \ref{thm1_constanttrend} and expression (\ref{eq_thm4_step1result}) that
$$\PP\Bigg(\Big|\big|\Lambda^{-1}(U_{\hat\tau_k-bn}^{(l)}-U_{\hat\tau_k+bn-1}^{(r)})\big|_2^2-\bar{c}\Big|\ge 4(p\log{(n)})^{1/2}(bn)^{-1}\Bigg)\rightarrow0,$$
as $n\rightarrow\infty$, which together with expression (\ref{eq_thm4_breaksize}) completes the proof.
\end{proof}

\subsection{Proof of Theorem \ref{thm2_constanttrend}}\label{subsec_thm2proof}

The main idea of this proof is similar to Theorem \ref{thm1_constanttrend}. To study the limiting distribution of our localized test statistic $\Q_n^{\diamond}$, we define $I_{\epsilon}^{\diamond}$ as $\Q_n^{\diamond}$ under the null, that is
\begin{equation}
    I_{\epsilon}^{\diamond} =: \max_{1\le s\le S}\max_{bn+1\le i\le n-bn}\frac{1}{\sqrt{|\L_s|}}\Bigg(\Big|\sum_{t=i-bn}^{i-1}(bn)^{-1}\Lambda^{-1}\epsilon_t-\sum_{t=i}^{i+bn-1}(bn)^{-1}\Lambda^{-1}\epsilon_t\Big|^2_{2,s}-\bar{c}_s^{\diamond}\Bigg).
\end{equation}
Then, we only need to investigate the limiting distribution of $I_{\epsilon}^{\diamond}$. Now, we scale $I_{\epsilon}^{\diamond}$ by multiplying $bn$ and write this rescaled $I_{\epsilon}^{\diamond}$ into the maximum of a sum of centered random variables $x^{\diamond}_{i,s,j}$ defined in (\ref{eq_thm2_x_def}), that is,
\begin{align}
    \label{eq_thm2_Iepsilon2}
    & bnI_{\epsilon}^{\diamond} =\max_{1\le s\le S}\max_{bn+1\le i\le n-bn}\frac{bn}{\sqrt{|\L_s|}}\sum_{j=1}^px^{\diamond}_{i,s,j},
\end{align}
where $x_{i,s,j}^{\diamond}$ is defined in expression (\ref{eq_thm2_x_def}) and can be written into
\begin{equation}
    \label{eq_thm2_xdiamond}
    x^{\diamond}_{i,s,j} = \frac{1}{(bn\sigma_j)^2}\Bigg[\Big(\sum_{t=i-bn}^{i-1}\epsilon_{t,j}-\sum_{t=i}^{i+bn-1}\epsilon_{t,j}\Big)^2 - \EE\Big(\sum_{t=i-bn}^{i-1}\epsilon_{t,j}-\sum_{t=i}^{i+bn-1}\epsilon_{t,j}\Big)^2\Bigg]\One_{j\in\L_s}.
\end{equation}
Similar to Theorem \ref{thm1_constanttrend}, given the form of the rescaled $I_{\epsilon}^{\diamond}$ above, we utilize the Gaussian approximation theorem in \textcite{chernozhukov_central_2017} to find its limiting distribution. Consider the Gaussian counterpart $z^{\diamond}_{i,s,j}$ for each $x^{\diamond}_{i,s,j}$. For $1\le j\le p$, we define $Z^{\diamond}_j=(z^{\diamond}_{bn+1,1,j},\ldots,z^{\diamond}_{n-bn,1,j},$ $\ldots,z^{\diamond}_{bn+1,S,j},\ldots,z^{\diamond}_{n-bn,S,j})^{\top}\in\RR^{S(n-2bn)}$, as independent centered Gaussian random vectors with covariance matrix $\EE(X^{\diamond}_jX^{\diamond\top}_j)\in\RR^{S(n-2bn)\times S(n-2bn)}$, where $X^{\diamond}_j=(x^{\diamond}_{bn+1,1,j},\ldots,$ $x^{\diamond}_{n-bn,1,j},\ldots,x^{\diamond}_{bn+1,S,j},\ldots,x^{\diamond}_{n-bn,S,j})^{\top}$. Note that although the length of the random vector $X^{\diamond}_j$ (resp. $Z^{\diamond}_j$) is $S(n-2bn)$, only the elements $x^{\diamond}_{i,s,j}$ (resp. $z^{\diamond}_{i,s,j}$) satisfying $j\in\L_s$ are nonzero. Moreover, the elements in the covariance matrix could be similarly developed by applying Lemma \ref{lemma_cov_thm1} as well. Let $\bar{z}_{i,s}^{\diamond}=|\L_s|^{-1/2}\sum_{j=1}^pz_{i,s,j}^{\diamond}$. Finally, by the standard argument of the Gaussian approximation, we can approach the distribution of $I_{\epsilon}^{\diamond}$ by the one of $\max_{i,s}\bar{z}_{i,s}^{\diamond}$, which completes the proof.

\begin{proof}[Proof of Theorem \ref{thm2_constanttrend}]
Now we give a rigorous proof of Theorem \ref{thm2_constanttrend}. We aim to bound the left-hand side of expression (\ref{eq_thm21_result}). Recall the Gaussian random variable $\Z_{\varphi}^{\diamond}$ defined in (\ref{eq_Z_local}), for $\varphi=(i,s)\in\N$. In this proof, we will use the notation $\Z_{\varphi}^{\diamond}=\Z_{i,s}^{\diamond}$ to explicitly indicate its dependence on $i$ and $s$. Since $\Q_n^{\diamond}=I_{\epsilon}^{\diamond}$, by Lemma \ref{lemma_chen2019}, we have 
\begin{align}
  & \sup_{u\in\mathbb{R}}\Big|\PP\big(\Q_n^{\diamond}\le u\big)-\PP\big(\max_{1\le s\le S}\max_{bn+1\le i\le n-bn}\Z_{i,s}^{\diamond}\le u\big)\Big| \nonumber\\
  \le & \sup_{u\in\mathbb{R}}\Big|\PP\big(\Q_n^{\diamond}\le u\big)-\PP\big(\max_{i,s}\bar{z}_{i,s}^{\diamond}\le u\big)\Big| + \sup_{u\in\mathbb{R}}\Big|\PP\big(\max_{i,s}\bar{z}_{i,s}^{\diamond}\le u\big) - \PP\big(\max_{i,s}\Z_{i,s}^{\diamond}\le u\big)\Big| \nonumber\\
  \lesssim & \sup_{u\in \mathbb{R}}\Big|\PP\Big(\max_{i,s}\sum_{j=1}^px^{\diamond}_{i,s,j}\le u\Big) -\PP\Big(\max_{i,s}\sum_{j=1}^pz^{\diamond}_{i,s,j}\le u\Big)\Big|+(bn)^{-1/3}\log^{2/3}(nS), \nonumber
\end{align}
where the last inequality holds by Lemma \ref{lemma_chen2019} and Lemma \ref{lemma_cov_thm1}. Recall the forms of $x_{i,j}$ and $x^{\diamond}_{i,s,j}$ under the null in expression (\ref{eq_thm1_x}) and (\ref{eq_thm2_xdiamond}) respectively. Then, it follows from expression (\ref{eq_thm11_m2result}) and Assumption \ref{asm_nbd_size_ratio} that, for some constants $k/2=3,4$
\begin{equation}
    \label{eq_thm21_m2result}
    \frac{1}{p}\sum_{j=1}^p\EE |bnx^{\diamond}_{i,s,j}|^{k/2} =  \frac{1}{p}\sum_{j\in\L_s}\EE|bnx_{i,j}|^{k/2}\lesssim |\L_s|/p.
\end{equation}
Similarly, by expressions (\ref{eq_thm11_e2result}) and (\ref{eq_thm21_m2result}), we can show that, for all $1\le j\le p$ and some constant $q>0$, 
\begin{align}
    \label{eq_thm21_e2result}
    \EE\Big(\max_{1\le s\le S}\max_{bn+1\le i\le n-bn}|bnx^{\diamond}_{i,s,j}|^{q/2}\Big) 
    \le \sum_{i=bn+1}^{n-bn}\sum_{1\le s\le S}\EE |bnx^{\diamond}_{i,s,j}|^{q/2} \lesssim nS.
\end{align}
Concerning the lower bound of $p^{-1}\sum_{j=1}^p\EE (bnx^{\diamond}_{i,s,j})^2$, expressions (\ref{eq_thm2_x_def}) and (\ref{eq_thm11_m1result}) entail that, for all $bn+1\le i\le n-bn$, $1\le s\le S$, and some constant $c>0$,
\begin{equation}
    \label{eq_thm21_m1result}
    \frac{p}{|\L_s|}\cdot\frac{1}{p}\sum_{j=1}^p\EE (bnx^{\diamond}_{i,s,j})^2 = 
    \frac{1}{|\L_s|}\sum_{j\in\L_s}\EE (bnx_{i,j})^2 \ge c.
\end{equation}

For any $r\le q/2$, we define
$$M_r^{\diamond}=\max_{1\le s\le S}\max_{bn+1\le i\le n-bn} \Big(p^{-1}\sum_{j=1}^p\big\lVert \sqrt{p/|\L_s|} bnx_{i,s,j}^{\diamond}\big\rVert_r^r\Big)^{1/r},\quad$$
and
$$\tilde M_r^{\diamond}=\max_{1\le j\le p}\Big\lVert\max_{1\le s\le S}\max_{bn+1\le i\le n-bn}\sqrt{p/|\L_s|}bnx_{i,s,j}^{\diamond}\Big\rVert_r.$$
In view of the two definitions above and the upper bounds derived in expressions (\ref{eq_thm21_m2result}) and (\ref{eq_thm21_e2result}), by Assumption \ref{asm_nbd_size_ratio}, we have
\begin{align}
    B_n & :=\max\Big\{M_3^{\diamond3},M_4^{\diamond2}, \tilde M_{q/2}^{\diamond} \Big\}\lesssim \sqrt{p/|\L_{\text{min}}}|(nS)^{2/q}. \nonumber
\end{align}
By Proposition 2.1 in \textcite{chernozhukov_central_2017}, we complete the proof.
\end{proof}

\subsection{Proof of Proposition \ref{prop_consistency}}\label{subsec_prop1proof}

\begin{proof}[Proof of Proposition \ref{prop_consistency} (i)]
Recall that $\bar{z}_{i,s}^{\diamond}=|\L_s|^{-1/2}\sum_{j=1}^pz_{i,s,j}^{\diamond}$ is a centered Gaussian variable with covariance matrix $|\L_s|^{-1}\sum_{j=1}^p\EE(X_j^{\diamond}X_j^{\diamond\top})$ which can be evaluated by Lemma \ref{lemma_cov_thm1} and expression (\ref{eq_cov_local_G}), where $z_{i,s,j}^{\diamond}=z_{i,j}\One_{j\in\L_s}$. This gives
\begin{align}
    \label{eq_prop1_step1}
    \PP\Big(\max_{1\le s\le S}\max_{bn+1\le i\le n-bn}bn\bar{z}_{i,s}^{\diamond}\ge u\Big) & \le \sum_{i=bn+1}^{n-bn} \sum_{s=1}^S \PP\Big(\frac{1}{\sqrt{|\L_s|}}\sum_{j=1}^pbnz_{i,s,j}^{\diamond} \ge u\Big) \nonumber \\
    & \le nS(2\pi)^{-1/2}(\tilde\sigma/u)e^{-u^2/(2\tilde\sigma^2)},
\end{align}
where $\tilde\sigma^2=\max_{1\le s\le S}|\L_s|^{-1}\sum_{j=1}^p\EE (bnz_{i,s,j}^{\diamond})^2$ which converges to $8$ by Lemma \ref{lemma_cov_thm1} and expression (\ref{eq_cov_local_G}). Therefore, by Assumption \ref{asm_nbd_size_ratio},
\begin{equation}
    \label{eq_prop1_step1result}
    \PP\Big(\max_{1\le s\le S}\max_{bn+1\le i\le n-bn}\bar{z}_{i,s}^{\diamond}\ge  4\log^{1/2}(n)(bn)^{-1}\Big)\rightarrow0. 
\end{equation}
Define $\I^{\diamond} := \{\V_{t,l}: 1\le t\le n,1\le l\le S, \,\forall 1\le r\le R , \S_{t,l}\cap\V_{\tau_r,s_r}=\emptyset \}$. We note that for any $\V_{i,s}\in\I^{\diamond}$, $(\EE V_i)_{j\in\L_s}=0$. It follows from Theorem \ref{thm2_constanttrend} that 
\begin{equation}
    \label{eq_prop1_outofset}
    \sup_{u\in\RR}\Big|\PP\big(\max_{\V_{i,s}\in\I^{\diamond}}|\L_s|^{-1/2}\big(|V_i|_{2,s}^2-\bar{c}_s^{\diamond}\big)\ge u\big) - \PP\big(\max_{\V_{i,s}\in\I^{\diamond}}\bar{z}_{i,s}^{\diamond}\ge u\big)\Big|\rightarrow0,
\end{equation}
where the operator $\max_{\V_{i,s}\in\I^{\diamond}}(\cdot)$ is defined as the maximum of $(\cdot)$ over all $(i,s)$ such that $\V_{i,s}\in\I^{\diamond}$.
This, along with expression (\ref{eq_prop1_step1result}) and the fact that $\max_{\V_{i,s}\in\I^{\diamond}}\bar{z}_{i,s}^{\diamond}\le \max_{bn+1\le i\le n-bn,1\le s\le S}\bar{z}_{i,s}^{\diamond}$ yields
$\PP\big(\max_{\V_{i,s}\in\I^{\diamond}}|\L_s|^{-1/2}\big(|V_i|_{2,s}^2-\bar{c}_s^{\diamond}\big)\ge \omega^{\diamond}\big)\rightarrow0$, for all $\omega^{\diamond}\ge 4\log^{1/2}(nS)(bn)^{-1}$. Thus, we have 
\begin{equation}
    \label{eq_prop1_mid}
    \PP(\A_1^{\diamond}\subset\I^{\diamond c})\rightarrow1,
\end{equation}
% \begin{equation}
%     \label{eq_prop1_set1}
%     \PP\big(\forall (\tau,l)\in\A_1^{\diamond},\exists 1\le r\le R , bn+1\le t\le n-bn,1\le s\le L, \S_{t,s}\cap\V_{\tau_r,l_r}\neq\emptyset,  \S_{t,s}\cap\V_{\tau,l}\neq\emptyset \big)\rightarrow1,
% \end{equation}
as $n\rightarrow\infty$, where $\A_1^{\diamond}$ is defined in Algorithm \ref{algo2}.
Recall $\W_{\tau,s}$ defined in Definition \ref{def_set}, which is the set of all the sliding windows $\S_{t,l}$ influenced by a true break located at $(\tau,s)$, which gives 
\begin{equation}
    \label{eq_I_c}
    \I^{\diamond c}=\cup_{1\le r\le R }\W_{\tau_r,s_r}.
\end{equation}
This, together with expression (\ref{eq_prop1_mid}) yields
\begin{equation}
    \label{eq_prop1_set1}
    \PP\big(\A_1^{\diamond}\subset\cup_{1\le r\le R }\W_{\tau_r,s_r}\big)\rightarrow1.
\end{equation}
Recall $U_{i,j}^{(l)}$, $U_{i,j}^{(r)}$ defined in expression (\ref{eq_U}) and the weighted break size $d_{\tau_k}=\Lambda^{-1}\gamma_k$ under the alternative. Let $\bar x_{i,s}^{\diamond}=|\L_s|^{-1/2}\sum_{j=1}^px_{i,s,j}^{\diamond}$ with $x_{i,s,j}^{\diamond}$ defined in expression (\ref{eq_thm2_xdiamond}). Then, we have
\begin{align}
    \label{eq_prop1_z_alter}
    & \min_{1\le r\le R }|\L_{s_r}|^{-1/2}\big(|V_{\tau_r}|_{2,s_r}^2-\bar{c}_{s_r}^{\diamond}\big) \nonumber \\
    = & \min_{1\le r\le R }|\L_{s_r}|^{-1/2}\Big(\sum_{j\in\L_{s_r}}\big(U_{\tau_r,j}^{(l)}-U_{\tau_r,j}^{(r)}+\sigma_jd_{\tau_r,j}\big)^2/\sigma_j^2-\bar{c}_{s_r}^{\diamond}\Big) \nonumber \\
    = & \min_{1\le r\le R }|\L_{s_r}|^{-1/2}\Big(|d_{\tau_r}|_{2,s_r}^2 - 2\sum_{j\in\L_{s_r}}d_{\tau_r,j}(U_{\tau_r,j}^{(l)}-U_{\tau_r,j}^{(r)})/\sigma_j \nonumber \\
    & + \sum_{j\in\L_{s_r}}(U_{\tau_r,j}^{(l)}-U_{\tau_r,j}^{(r)})^2/\sigma_j^2 -\bar{c}_{s_r}^{\diamond}\Big) \nonumber \\
    \gtrsim & \, |\L_{\text{min}}|^{-1/2}\delta_p^{\diamond2} - 2\max_{1\le r\le R }|\L_{s_r}|^{-1/2}\Big|\sum_{j\in\L_{s_r}}d_{\tau_r,j}(U_{\tau_r,j}^{(l)}-U_{\tau_r,j}^{(r)})/\sigma_j\Big| -\max_{1\le r\le R  }|\bar x_{\tau_r,s_r}^{\diamond}|.
\end{align}
Note that, by Lemma \ref{lemma_tail_cross} (i) and Assumption \ref{asm_nbd_size_ratio}, we obtain
\begin{align}
    \label{eq_prop_GA}
    & \quad \max_{1\le r\le R }|\L_{s_r}|^{-1/2}\Big|\sum_{j\in\L_{s_r}}d_{\tau_r,j}(U_{\tau_r,j}^{(l)}-U_{\tau_r,j}^{(r)})/\sigma_j\Big|/|\Lambda^{-1}\gamma_k|_2 \nonumber \\
    & = O_{\PP}\big(|\L_{\text{min}}|^{-1/2}\log^{1/2}(nS)/(bn)^{1/2}\big).
\end{align}
By the similar arguments in the proof for Theorem \ref{thm2_constanttrend}, we can achieve, under the null, 
$$\sup_{u\in\RR}\Big|\PP\big(\max_{1\le s\le S}\max_{bn+1\le i\le n-bn}|\bar x_{i,s}^{\diamond}|\le u\big)-\PP\big(\max_{1\le s\le S}\max_{bn+1\le i\le n-bn}|\bar{z}_{i,s}^{\diamond}|\le u\big)\Big| \rightarrow0,$$
which along with expressions (\ref{eq_prop1_step1result}), (\ref{eq_prop1_z_alter}), (\ref{eq_prop_GA}) and the conditions in Proposition \ref{prop_consistency} that $|\L_{\text{min}}|^{-1/2}\delta_p^{\diamond2}\ge3\omega^{\diamond}$ and $\omega^{\diamond}\gg 4\log^{1/2}(nS)(bn)^{-1}$ yields
\begin{equation}
    \label{eq_prop1_exceed}
    \PP\Big(\min_{1\le r\le R }|\L_{s_r}|^{-1/2}\big(|V_{\tau_r}|_{2,s_r}^2-\bar{c}_{s_r}^{\diamond}\big) \ge \omega^{\diamond} \Big)\rightarrow1.
\end{equation}
Namely, the test statistics shall exceed the threshold at the break points with probability approaching to 1, which gives 
\begin{equation}
    \label{eq_prop1_inside}
    \PP\big(\V_{\tau_r,s_r}\in\A_1^{\diamond},1\le r\le R \big)\rightarrow1.
\end{equation}
It follows from expressions (\ref{eq_prop1_set1}) and 
(\ref{eq_prop1_inside}) that 
\begin{equation}
    \label{eq_prop1_conclude}
    \PP\Big(\big\{\V_{\tau_1,s_1},\V_{\tau_2,s_2},\ldots,\V_{\tau_{R },s_{R }}\big\}\subset\A_1^{\diamond}\subset\cup_{1\le r\le R }\W_{\tau_r,s_r}\Big)\rightarrow1,
\end{equation}
as $n\rightarrow\infty$. 
% Recall Algorithm \ref{algo2}, for any  we define $\tilde\A_r^{\diamond}=\big\{(\tau,l):\,\text{there does NOT exist any } bn+1\le i\le n-bn$, $1\le s\le L$, such that  $\S_{i,s}\cap\V_{\hat\tau_r,\hat l_r}\neq\emptyset \,\text{and } \S_{i,s}\cap\V_{\tau,l}\neq\emptyset\big\}$ which is the set of 
Note that for any detected $\V_{\hat\tau,\hat s}$, there exists $1\le r\le R $ such that $\V_{\hat\tau,\hat s}\cap\S_{\tau_r,s_r}\neq\emptyset$. 

Finally, we show two claims for our removal step in Algorithm \ref{algo2}. First, in the $r$-th round, we remove all the vertical lines $\V_{\tau,s}$ affected by the true break located at $(\tau_r,s_r)$. Second, in each round, we remove only one true break. To see the first claim, we shall note that for any $\V_{\tau,s}\in\W_{\tau_r,s_r}$, we also have $\V_{\tau,s}\cap\S_{\tau_r,s_r}\neq\emptyset$. Therefore, by the definition of the removal step in Algorithm \ref{algo2}, any $\V_{\tau,s}\in\W_{\tau_r,s_r}$ shall be removed. 
To show the second claim, suppose that another true break $\V_{\tau_{r'},s_{r'}}$ is removed, $r\neq r'$. Then, there exists a window $\S_{\tau,s}$ such that $\V_{\hat\tau,\hat s}\cap\S_{\tau,s}\neq\emptyset$ and $\V_{\tau_{r'},s_{r'}}\cap\S_{\tau,s}\neq\emptyset$, which gives $\S_{\tau,s}\cap\W_{\tau_{r'},s_{r'}}\neq\emptyset$. Note that $\V_{\hat\tau,\hat s}\in\W_{\tau_r,s_r}$, and this together with $\V_{\hat\tau,\hat s}\cap\S_{\tau,s}\neq\emptyset$ yields $\S_{\tau,s}\cap\W_{\tau_r,s_r}\neq\emptyset$. Therefore, $\S_{\tau,s}$ intersects both $\W_{\tau_r,s_r}$ and $\W_{\tau_{r'},s_{r'}}$. This contradicts with Assumption \ref{asm_spatial_sep}, which completes the proof.
% For any $(t,s)\in \D^{\diamond}(\tau_{r},l_r)$, we have $\D^{\diamond}(\tau_{r},l_r)\subset \D^{\diamond}(t,s)$. Also, for $r_1\neq r_2$, $(\tau_{r_1},l_{r_1})\not\in\D^\diamond(\tau_{r_2},l_{r_2})$, which completes the proof.
\end{proof}

\begin{proof}[Proof of Proposition \ref{prop_consistency} (ii)]
Note that by Algorithm \ref{algo2} and Assumption \ref{asm_spatial_sep}, for all $1\le r\le R $, 
\begin{equation}
    \label{eq_prop_max}
    |\L_{\hat s_r}|^{-1/2}\big(|V_{\hat\tau_r}|_{2,\hat s_r}^2-\bar{c}_{\hat s_r}^{\diamond}\big)=\max_{\V_{\tau,s}\in\W_{\tau_r,s_r}}|\L_s|^{-1/2}\big(|V_{\tau}|_{2,s}^2-\bar{c}_s^{\diamond}\big).
\end{equation} 
Therefore, with probability tending to 1, we have
\begin{equation}
    \label{eq_prop_max1}
    \III_{r,1}:=|\L_{\hat s_r}|^{-1/2}\Big[ \big(|V_{\tau_r}|_{2,\hat s_r}^2-\bar{c}_{\hat s_r}^{\diamond}\big) - \big(|V_{\hat\tau_r}|_{2,\hat s_r}^2-\bar{c}_{\hat s_r}^{\diamond}\big)\Big] \le 0,
\end{equation}
and
\begin{equation}
    \label{eq_prop_max2}
    \III_{r,2}:= |\L_{s_r}|^{-1/2}\big(|V_{\hat\tau_r}|_{2,s_r}^2-\bar{c}_{s_r}^{\diamond}\big) -|\L_{\hat s_r}|^{-1/2}\big(|V_{\hat\tau_r}|_{2,\hat s_r}^2-\bar{c}_{\hat s_r}^{\diamond}\big) \le0.
\end{equation}
We shall use $\III_{r,1}\le 0$ to show the consistency of the temporal location estimator and apply $\III_{r,2}\le 0$ for the spatial one.

We first investigate $\III_{r,1}$. Recall that by Assumption \ref{asm_nbd_size_ratio}, all the spatial neighborhoods $\L_s$ have the same size $|\L_{\text{min}}|$ up to a constant factor. Thus, by applying similar arguments in Theorem \ref{thm_consistency}, for any $(t,s)$ such that $\V_{t,s}\in\W_{\tau_r,s_r}$, we have
\begin{align}
    \label{eq_prop_part1}
    & \quad \big(|V_{\tau_r}|_{2,s}^2-\bar{c}_{s}^{\diamond}\big) - \big(|V_{t}|_{2,s}^2-\bar{c}_{s_r}^{\diamond}\big) \nonumber \\
    & \ge  (bn)^{-1}|t-\tau_r||\Lambda^{-1}\gamma_r|_{2,s}^2  - O_{\PP}\Big\{(bn)^{-1}\log(nS)|t-\tau_r|^{1/2}|\Lambda^{-1}\gamma_r|_{2,s}\Big\} \nonumber \\
    & \quad - O_{\PP}\Big\{(bn)^{-3/2}\log(nS)\sqrt{|\L_{\text{min}}|}|t-\tau_r|^{1/2}\Big\}.
\end{align}
Denote $\Phi_r = |\L_{\text{min}}|/(bn|\Lambda^{-1}\gamma_r|_2^2)$. As a direct consequence of expression (\ref{eq_prop_part1}), we achieve the temporal consistency as follows:
\begin{equation}
    \label{eq_prop_temporal}
    \max_{1\le r\le R }|\hat \tau_r-\tau_{r^*}|\cdot|\Lambda^{-1}\gamma_r|_2^2/(1+\Phi_r) = O_{\PP}\big\{\log^2(nS)\big\},
\end{equation}
uniformly over $r$. 

Next, we study $\III_{r,2}$ to show the spatial consistency. Recall that $\Delta U_{t,j}=U_{t,j}^{(l)}-U_{t,j}^{(r)}$ and  $d_{t,j}=\sigma^{-1}\big(1-|t-\tau_r|/(bn)\big)\gamma_{r,j}$. In fact, by expression (\ref{eq_thm21_o1}) and Assumption \ref{asm_local_min_breaksize}, we have $(\log(nS))^2(bn)^{-1}\delta_p^{\diamond-2}\ll (\log(nS))^{3/2}|\L_{\text{min}}|^{-1/2}$, which implies that $|\hat\tau_r-\tau_r|/(bn)\rightarrow0$. Hence, $d_{\hat\tau_r,j}=\sigma_j^{-1}\gamma_{r,j}(1+o(1))$, and we shall proceed the proof by assuming that $d_{t,j}=\sigma^{-1}\gamma_{r,j}(1+o(1))$. Now, we decompose the break statistic into the signal part and the error part, that is
$$V_{t,j} =\Lambda^{-1}\Delta\mu_{t,j}= d_{t,j} +\sigma_j^{-1}\Delta U_{t,j}.$$
% We define the set of overlapped spatial coordinates as $\tilde\L_r=\L_{l_r}\cap\L_l$. 
Then, for all $(t,s)$ such that $\V_{t,s}\in\W_{\tau_r,s_r}$, it follows that
\begin{align}
    \label{eq_prop_part2_origin}
    & \quad  |\L_{s_r}|^{-1/2}\big(|V_t|_{2,s_r}^2-\bar{c}_{s_r}^{\diamond}\big) -|\L_s|^{-1/2}\big(|V_t|_{2,s}^2-\bar{c}_{s}^{\diamond}\big) \nonumber \\
    % = & |\L_{l_r}|^{-1/2}\sum_{j\in\L_{l_r}}\big[(d_{t,j}+\sigma_j^{-1}\Delta U_{t,j})^2 -\bar{c}_{l_r,j}^{\diamond}\big] -  |\L_l|^{-1/2}\sum_{j\in\L_l}\big[(d_{t,j}+\sigma_j^{-1}\Delta U_{t,j})^2-\bar{c}_{l,j}^{\diamond}\big] \nonumber \\
    & \ge \Big\{|\L_{s_r}|^{-1/2}\sum_{j\in\L_{s_r}}d_{t,j}^2-|\L_s|^{-1/2}\sum_{j\in\L_s}d_{t,j}^2\Big\} \nonumber \\
    & \quad + \Big\{ |\L_{s_r}|^{-1/2}\sum_{j\in\L_{s_r}}d_{t,j}\sigma_j^{-1}\Delta U_{t,j} - |\L_s|^{-1/2}\sum_{j\in\L_s}d_{t,j}\sigma_j^{-1}\Delta U_{t,j}\Big\} \nonumber \\
    & \quad + \Big\{ |\L_{s_r}|^{-1/2}\sum_{j\in\L_{s_r}}\big(\sigma_j^{-1}\Delta U_{t,j}^2-\bar{c}_{s_r,j}^{\diamond}\big) - |\L_s|^{-1/2}\sum_{j\in\L_s}\big(\sigma_j^{-1}\Delta U_{t,j}^2-\bar{c}_{s,j}^{\diamond}\big)\Big\} \nonumber \\
    & =: \III_{r,s,t,21} + \III_{r,s,t,22} + \III_{r,s,t,23}.\end{align}
We shall study the parts $\III_{r,s,t,21}$--$\III_{r,s,t,23}$ respectively. For simplicity, we denote $\tilde\A_{s,s_r}=\L_{s_r}\setminus \L_s$, $\tilde\C_{s,s_r}=\L_s\setminus \L_{s_r}$, and $\tilde\B_{s,s_r}=\L_{s_r}\cap \L_s$. Now we first investigate the signal part $\III_{r,s,t,21}$. Note that $d_{t,j}=0$ when $j\notin\L_{s_r}$ and by the condition in Proposition \ref{prop_consistency}, $d_{t,j}=d_0$ up to a constant factor. Then, we have, for any 
\begin{align}
    & \quad \III_{r,s,t,21} \nonumber \\
    & = \frac{1}{\sqrt{|\tilde\A_{s,s_r}| + |\tilde\B_{s,s_r}|}}\sum_{j\in\tilde\A_{s,s_r}\cup\tilde\B_{s,s_r}}d_{t,j}^2 - \frac{1}{\sqrt{|\tilde\B_{s,s_r}|+|\tilde\C_{s,s_r}|}}\sum_{j\in\tilde\B_{s,s_r}}d_{t,j}^2 \nonumber \\
    & = \frac{1}{\sqrt{|\tilde\A_{s,s_r}| + |\tilde\B_{s,s_r}|}}\sum_{j\in\tilde\A_{s,s_r}}d_{t,j}^2 + \Bigg[\frac{1}{\sqrt{|\tilde\A_{s,s_r}| + |\tilde\B_{s,s_r}|}} - \frac{1}{\sqrt{|\tilde\B_{s,s_r}|+|\tilde\C_{s,s_r}|}}\Bigg]\sum_{j\in\tilde\B_{s,s_r}}d_{t,j}^2 \nonumber \\
    &  \gtrsim \Bigg[\frac{|\tilde\A_{s,s_r}|}{\sqrt{|\tilde\A_{s,s_r}|+|\tilde\B_{s,s_r}|}} \nonumber \\
    & \quad + \frac{\big||\tilde\C_{s,s_r}|-|\tilde\A_{s,s_r}|\big|\cdot|\tilde\B_{s,s_r}|}{\sqrt{\big(|\tilde\B_{s,s_r}|+|\tilde\C_{s,s_r}|\big)\big(|\tilde\A_{s,s_r}|+|\tilde\B_{s,s_r}|\big)}\cdot\Big(\sqrt{|\tilde\B_{s,s_r}|+|\tilde\C_{s,s_r}|}+\sqrt{|\tilde\A_{s,s_r}|+|\tilde\B_{s,s_r}|}\Big)}\Bigg]d_0^2 \nonumber \\
    & \gtrsim \frac{|\tilde\A_{s,s_r}|+\big||\tilde\C_{s,s_r}|-|\tilde\A_{s,s_r}|\big|}{\sqrt{|\tilde\B_{s,s_r}|}}d_0^2,
\end{align}
where the last inequality holds if we have $|\tilde\A_{s,s_r}|=o(|\tilde\B_{s,s_r}|)$ and $|\tilde\C_{s,s_r}|= o(|\tilde\B_{s,s_r}|)$. Actually, by Assumption \ref{asm_local_min_breaksize} and expressions (\ref{eq_prop_max1}) and (\ref{eq_prop_max2}), it can be shown that these two conclusions hold when $(t,s)=(\hat\tau_r,\hat s_r)$. We shall continue with the proof by assuming that these two conditions.
Next, for the cross-term $\III_{r,s,t,22}$, we have
\begin{align}
    &\quad |\III_{r,s,t,22}| \nonumber \\
    & =  \Bigg|\frac{1}{\sqrt{|\tilde\B_{s,s_r}|+|\tilde\C_{s,s_r}|}}\sum_{j\in\tilde\B_{s,s_r}}d_{t,j}\sigma_j^{-1}\Delta U_{t,j}
    \nonumber \\
    & \quad -\frac{1}{\sqrt{|\tilde\A_{s,s_r}| + |\tilde\B_{s,s_r}|}}\sum_{j\in\tilde\A_{s,s_r}\cup\tilde\B_{s,s_r}}d_{t,j}\sigma_j^{-1}\Delta U_{t,j}\Bigg| \nonumber \\
    & \le \Bigg|\frac{1}{\sqrt{|\tilde\B_{s,s_r}|+|\tilde\C_{s,s_r}|}}-\frac{1}{\sqrt{|\tilde\A_{s,s_r}|+|\tilde\B_{s,s_r}|}}\Bigg|\cdot\Big|\sum_{j\in\tilde\B_{s,s_r}}d_{t,j}\sigma_j^{-1}\Delta U_{t,j}\Big|
    \nonumber \\
    & \quad + \frac{1}{\sqrt{|\tilde\A_{s,s_r}| + |\tilde\B_{s,s_r}|}}\Big|\sum_{j\in\tilde\A_{s,s_r}}d_{t,j}\sigma_j^{-1}\Delta U_{t,j}\Big|, \nonumber 
\end{align}
which together with the similar arguments in Lemma \ref{lemma_tail_cross} (ii) implies that
\begin{align}
    & |\III_{r,s,t,22}| \nonumber \\
    = & O_{\PP}\Bigg\{\frac{d_0}{\sqrt{bn}} \Bigg[\frac{\sqrt{|\tilde\A_{s,s_r}|}}{\sqrt{|\tilde\A_{s,s_r}|+|\tilde\B_{s,s_r}|}}\Bigg] \nonumber \\
    & +  \frac{\big||\tilde\A_{s,s_r}|-|\tilde\C_{s,s_r}|\big|\cdot\sqrt{|\tilde\B_{s,s_r}|}}{\sqrt{\big(|\tilde\B_{s,s_r}|+|\tilde\C_{s,s_r}|\big)\big(|\tilde\A_{s,s_r}|+|\tilde\B_{s,s_r}|\big)}\cdot\Big(\sqrt{|\tilde\B_{s,s_r}|+|\tilde\C_{s,s_r}|}+\sqrt{|\tilde\A_{s,s_r}|+|\tilde\B_{s,s_r}|}\Big)}\Bigg\} \nonumber \\
    = & O_{\PP}\Bigg\{\frac{d_0}{\sqrt{bn}}\Bigg(\frac{\big||\tilde\A_{s,s_r}| - |\tilde\C_{s,s_r}|\big|}{|\tilde\B_{s,s_r}|} + \sqrt{\frac{|\tilde\A_{s,s_r}|}{|\tilde\B_{s,s_r}|}}\Bigg)\Bigg\}.
\end{align}
For the error part $\III_{r,s,t,23}$, it follows that
\begin{align}
    |\III_{r,s,t,23}| & = \Bigg|\frac{1}{\sqrt{|\tilde\B_{s,s_r}|+|\tilde\C_{s,s_r}|}}\sum_{j\in\tilde\B_{s,s_r}\cup\tilde\C_{s,s_r}}\big(\sigma_j^{-1}\Delta U_{t,j}^2-\bar{c}_{l,j}^{\diamond}\big) \nonumber \\
    & \qquad - \frac{1}{\sqrt{|\tilde\A_{s,s_r}|+|\tilde\B_{s,s_r}|}}\sum_{j\in\tilde\A_{s,s_r}\cup\tilde\B_{s,s_r}}\big(\sigma_j^{-1}\Delta U_{t,j}^2-\bar{c}_{s_r,j}^{\diamond}\big)\Bigg| \nonumber \\
    & \le \Bigg| \frac{1}{\sqrt{|\tilde\B_{s,s_r}|+|\tilde\C_{s,s_r}|}} - \frac{1}{\sqrt{|\tilde\A_{s,s_r}| + |\tilde\B_{s,s_r}|}} \Bigg|\cdot\Big|\sum_{j\in\tilde\B_{s,s_r}}\big(\sigma_j^{-1}\Delta U_{t,j}^2-\bar{c}_{l,j}^{\diamond}\big)\Big| \nonumber \\
    & \qquad + \frac{1}{\sqrt{|\tilde\B_{s,s_r}|+|\tilde\C_{s,s_r}|}} \Big|\sum_{j\in\tilde\C_{s,s_r}}\big(\sigma_j^{-1}\Delta U_{t,j}^2-\bar{c}_{s,j}^{\diamond}\big)\Big| \nonumber \\
    & \qquad + \frac{1}{\sqrt{|\tilde\A_{s,s_r}|+|\tilde\B_{s,s_r}|}} \Big|\sum_{j\in\tilde\A_{s,s_r}}\big(\sigma_j^{-1}\Delta U_{t,j}^2-\bar{c}_{s_r,j}^{\diamond}\big)\Big|, \nonumber
\end{align}
which gives
\begin{align}
    |\III_{r,s,t,23}| =O_{\PP}\Bigg\{\frac{1}{bn}\Bigg(\frac{\big||\tilde\A_{s,s_r}| - |\tilde\C_{s,s_r}|\big|}{|\tilde\B_{s,s_r}|} + \frac{\sqrt{|\tilde\A_{s,s_r}|}+\sqrt{|\tilde\C_{s,s_r}|}}{\sqrt{|\tilde\B_{s,s_r}|}}\Bigg)\Bigg\}.
\end{align}
We shall note that under the alternative hypothesis, it can be shown that by Assumption \ref{asm_local_min_breaksize}, for all $(t,s)$ satisfying $\V_{t,s}\in\W_{\tau_r,s_r}$, $|V_t|_{2,s}^2-\bar{c}_s^{\diamond}\ge0$ with probability tending to 1 since the signal part dominates. Finally, by applying the results of $\III_{r,s,t,21}$--$\III_{r,s,t,23}$ to expression (\ref{eq_prop_part2_origin}), we have
\begin{align}
    \label{eq_prop_three_bounds}
    & \III_{r,s,t,21}+\III_{r,s,t,22}+\III_{r,s,t,23} \nonumber \\
    \gtrsim & \frac{|\tilde\A_{s,s_r}| + \big||\tilde\A_{s,s_r}|-|\tilde\C_{s,s_r}|\big|}{\sqrt{|\tilde\B_{s,s_r}|}} d_0^2 - O_{\PP}\Bigg\{\frac{1}{\sqrt{bn}}\Bigg(\frac{\big||\tilde\A_{s,s_r}| - |\tilde\C_{s,s_r}|\big|}{|\tilde\B_{s,s_r}|} + \sqrt{\frac{|\tilde\A_{s,s_r}|}{|\tilde\B_{s,s_r}|}}\Bigg)d_0\Bigg\} \nonumber \\
    & - O_{\PP}\Bigg\{\frac{1}{bn}\Bigg(\frac{\big||\tilde\A_{s,s_r}| - |\tilde\C_{s,s_r}|\big|}{|\tilde\B_{s,s_r}|} + \frac{\sqrt{|\tilde\A_{s,s_r}|}+\sqrt{|\tilde\C_{s,s_r}|}}{\sqrt{|\tilde\B_{s,s_r}|}}\Bigg)\Bigg\}.
\end{align}
By employing Nagaev's inequality presented in \textcite{nagaev_large_1979}, similar to our approach in Lemma \ref{lemma_tail_cross} (ii), we can derive the upper bound for $\max_{1\le r\le R}\max_{{(t,s):,\V_{t,s}\in\W_{\tau_r,s_r}}}|\III_{r,s,t,22}|$. The detailed calculations supporting this claim are omitted due to their relevance. Furthermore, we can combine the results obtained for the term $\III_{r,s,t,23}$ with the Gaussian approximation in \textcite{chernozhukov_central_2017} to establish a bound for $\max_r\max_{(t,s)}|\III_{r,s,t,23}|$. Finally, we obtain
\begin{equation}
    \label{eq_prop_spatial}
    \max_{1\le r\le R }\big(|\tilde\A_{\hat s_r,s_r}| + |\tilde\C_{\hat s_r,s_r}|\big)/|\L_{\text{min}}|=O_{\PP}\Bigg\{\frac{\log^2(nS)}{bn|\L_{\text{min}}|d_0^2}\Bigg\},
\end{equation}
uniformly over $r$. The desired result is achieved.

\end{proof}

\begin{proof}[Proof of Proposition \ref{prop_consistency} (iii)]
Recall the definitions of $\mu_t^{(r)}$, $\mu_t^{(l)}$, $U_t^{(r)}$ and $U_t^{(l)}$ in the proof of (ii) for Theorem \ref{thm_consistency}. Since $bn\gg(|\L_{\text{min}}|\log(nS))^{1/2}\delta_p^{\diamond2}$, 
$$\big|\mu_{\hat\tau_r-bn}^{(l)}-\mu((\hat\tau_r-bn)/n)\big|_2=0,$$
and similarly, $\big|\mu_{\tau_r+bn-1}^{(r)}-\mu((\hat\tau_r+bn-1)/n)\big|_2=0$. By Assumption \ref{asm_nbd_size_ratio} and the condition in Proposition \ref{prop_consistency}, it follows that
$$\big|\mu((\hat\tau_r-bn)/n)-\mu((\hat\tau_r+bn-1)/n)-\gamma_{r^*}\big|_2=o(1),$$
which together with Assumption \ref{asm_nbd_size_ratio} yields
\begin{align}
    \label{eq_prop_breaksize}
    & \big||\Lambda^{-1}(\hat\gamma_r-\gamma_{\tau_{r^*}})|_2^2-\bar{c}\big| 
    \lesssim \Big|\big|\Lambda^{-1}(\mu_{\hat\tau_r-bn}^{(l)}-\mu_{\hat\tau_r+bn-1}^{(r)}-\gamma_{r^*})\big|_2^2 \nonumber \\
    + & \big|\Lambda^{-1}(U_{\hat\tau_r-bn}^{(l)}-U_{\hat\tau_r+bn-1}^{(r)})\big|_2^2 -\bar{c} \Big| =\Big|\big|\Lambda^{-1}(U_{\hat\tau_r-bn}^{(l)}-U_{\hat\tau_r+bn-1}^{(r)})\big|_2^2-\bar{c}\Big|+o(1).
\end{align}
By the Gaussian approximation in Theorem \ref{thm2_constanttrend} and expression (\ref{eq_prop1_step1result}), we have
$$\PP\Big(\Big|\big|\Lambda^{-1}(U_{\hat\tau_r-bn}^{(l)}-U_{\hat\tau_r+bn-1}^{(r)})\big|_2^2-\bar{c}\Big|\ge 4(p\log{(nS)})^{1/2}(bn)^{-1}\Big)\rightarrow0,$$
as $n\rightarrow\infty$, which along with expression (\ref{eq_prop_breaksize}) completes the proof.
\end{proof}

\subsection{Proof of Lemma \ref{lemma_nonli_longrun_corr}}

To prepare the proofs for nonlinear time series, we first introduce the predictive dependence measures following \textcite{wu_nonlinear_2005}.

\noindent\textbf{Predictive dependence measure:}

Let $\nu:\ZZ\rightarrow\ZZ^v$ be a bijection and $l=\nu^{-1}(\bbell)$. We denote the \textit{temporal} projection operator by $\P_k\cdot$, $k\ge0$, with
\begin{align}
    \label{eq_temp_proj}
    \P_k\epsilon_t(\bbell) & = \EE[\epsilon_t(\bbell)\mid\F_{k,l}] - \EE[\epsilon_t(\bbell)\mid \F_{k-1,l}],
\end{align}
and denote the \textit{temporal-spatial} projection operator by $\P_{k,l'}\cdot$, $k\ge0$, $l'\in\ZZ$, with
\begin{align}
    \label{eq_tempspatial_proj}
    \P_{k,l'}\epsilon_t(\bbell) & = \EE[\epsilon_t(\bbell)\mid \F_{k,l'}] - \EE[\epsilon_t(\bbell)\mid \F_{k,l'-1}] \nonumber \\
    & \qquad - \big(\EE[\epsilon_t(\bbell)\mid \F_{k-1,l'}] - \EE[\epsilon_t(\bbell)\mid \F_{k-1,l'-1}]\big), 
\end{align}
where $\F_{k,l'} = (\eta_{i,\nu(j)};\, i\le k, \,j\le l')$. 
Define the predictive dependence measures
\begin{align}
    \label{eq_pred_dep}
    \phi_{t,l,q}  = \big\|\P_0\epsilon_t(\bbell)\big\|_q, \quad
    \psi_{t,l,q} = \big\|\P_{0,0}\epsilon_t(\bbell)\big\|_q.
\end{align}
Note that the projections $(\P_k)_{k\ge 0}$ and $(\P_{k,l'})_{k\ge0,l'\in\ZZ}$ induce martingale differences with respect to $(\F_{k,l})_k$ and $(\F_{k,l'})_{k,l'}$, respectively. In view of Jensen's inequality, the functional and predictive dependence measures satisfy $$\phi_{k,l,q}\le \theta_{k,\bbell,q} \quad \text{and} \quad \psi_{k,l,q} \le \delta_{k,\bbell,q}.$$

Now we provide the proof for Lemma \ref{lemma_nonli_longrun_corr}.

\begin{proof}[Proof of Lemma \ref{lemma_nonli_longrun_corr}]

Let $l_1=\nu^{-1}(\bbell_1)$ and $l_2=\nu^{-1}(\bbell_2)$. Define
\begin{equation}
    \label{eq_nonli_Psi0}
    \Psi_{0,l,q} = \sum_{k\ge0}\psi_{k,l,q}.
\end{equation}
We shall first note that by the definition of the long-run covariance in (\ref{eq_nonli_longrun}) and the orthogonality of the projections $\P_{k,l'}$, we have
\begin{align}
    \label{eq_nonli_lonrun_form}
    |\sigma(\bbell_1,\bbell_2)| &= \Big|\sum_{h=-\infty}^{\infty}\EE\big(\epsilon_t(\bbell_1)\epsilon_{t+h}(\bbell_2)\big)\Big| \nonumber \\
    & = \Big|\sum_{h=-\infty}^{\infty}\EE\Big(\sum_{k_1\ge0}\sum_{l_1'\in\ZZ}\P_{t-k_1,l_1'}\epsilon_t(\bbell_1)\sum_{k_2\ge0}\sum_{l_2'\in\ZZ}\P_{t+h-k_2,l_2'}\epsilon_{t+h}(\bbell_2)\Big)\Big| \nonumber \\
    & = \Big|\sum_{h=-\infty}^{\infty}\sum_{k\ge0}\sum_{l'\in\ZZ}\EE\big(\P_{t-k,l'}\epsilon_t(\bbell_1)\P_{t+h-k,l'}\epsilon_{t+h}(\bbell_2)\big)\Big| \nonumber \\
    & = \Big|\sum_{l'\in\ZZ}\sum_{k\ge0}\sum_{h=-k}^{\infty}\EE\big(\P_{t-k,l'}\epsilon_t(\bbell_1)\P_{t+h-k,l'}\epsilon_{t+h}(\bbell_2)\big)\Big| \nonumber \\
    & \le \sum_{l'\in\ZZ}\Psi_{0,l_1-l',2}\Psi_{0,l_2-l',2},
\end{align}
where the last inequality holds by Cauchy-Schwarz inequality. Then, we denote $|\bbell_1-\bbell_2|_2=m$ and by (\ref{eq_nonli_lonrun_form}), $\sigma(\bbell_1,\bbell_2)$ can be decomposed into two parts, that is
\begin{align}
    \label{eq_nonli_sec_dep_twoparts}
    \sum_{l'\in\ZZ}\Psi_{0,l_1-l',2}\Psi_{0,l_2-l',2} & = \sum_{\{l'\in\ZZ: \, |\bbell_1-\nu(l')|_2< m/2\}}\Psi_{0,l_1-l',2} \Psi_{0,l_2-l',2} \nonumber \\
    & \quad + \sum_{\{l'\in\ZZ: \, |\bbell_1-\nu(l')|_2\geq  m/2\}}\Psi_{0,l_1-l',2} \Psi_{0,l_2-l',2} \nonumber \\
    & =:\M_1+\M_2.
\end{align}
Since $|\bbell_1-\bbell_2|_2=m$,  
$|\bbell_1-\nu(l')|_2<m/2$ implies 
$|\bbell_2-\nu(l')|_2\geq m/2.$ Note that $\Psi_{0,l-l',q}\le \Delta_{0,\bbell,q}$.
By Cauchy-Schwarz inequality and Assumption \ref{asm_nonli_sec_dep},  
\begin{align}
    \label{eq_nonli_sec_dep_part1}
    \M_1
    \leq & \sum_{\{l'\in\ZZ: \, |\bbell_2-\nu(l')|_2\geq m/2\}}\Psi_{0,l_1-l',2}\Psi_{0,l_2-l',2}\nonumber \\
    \le &  \Big(\sum_{l'\in\ZZ}\Psi_{0,l_1-l',2}\Big)^{1/2}\Big(\sum_{\{l'\in\ZZ: \, |\bbell_2-\nu(l')|_2> m/2\}}\Psi_{0,l_2-l',2}\Big)^{1/2} \nonumber \\
    = & O\big\{(m/2)^{-2\xi}\sigma(\bbell_1)\sigma(\bbell_2)\big\}.    
\end{align}
Similar argument leads to the same bound for $\M_2$ in (\ref{eq_nonli_sec_dep_twoparts}) and we obtain (\ref{eq_nonli_sec_dep2}) by combining the two parts. 
\end{proof}

\subsection{Proof of Theorem \ref{thm3_nonli}}\label{subsec_thm3proof}
Since the proof of Theorem \ref{thm3_nonli} is quite involved, we shall first introduce the proof strategies and then follow up with the rigorous details. 

\begin{remark}[Bernstein's blocks]\label{rmk_block}
The proofs of Theorem \ref{thm3_nonli} and Proposition \ref{thm3_constanttrend} essentially utilize the classical idea involving Bernstein's blocks (see e.g. \textcite{de_jong_block_1997}).
This technique was originally introduced to expect that the central limit theorem holds if a sequence of random vectors is more likely to be independent when they are far apart. Inspired by this device, suppose that our $p$ time series are cross-sectionally $m$-dependent, and we split them into a sequence of big blocks and small blocks. The small blocks are positioned between the big blocks. Consequently, when we get rid of the small blocks, we create the asymptotic independence among the big blocks. We could choose proper sizes for big and small blocks to ensure that the sum of all small blocks is asymptotically negligible. The rigorous proof of applying this technique shall be illustrated in this section.
\end{remark}

Recall the definition of $\epsilon_t(\bbell)$ in expression (\ref{eq_epsilon_nonlinear}) and we denote the $m$-approximation of $\epsilon_t(\bbell)$ by $\epsilon_t^{(m)}(\bbell)$, that is
$$\epsilon_t^{(m)}(\bbell) = f\big(\eta_{t-k,\bbell'}, \, k\ge 0, \, |\bbell'-\bbell|_2\le m/2\big).$$
Define the weighted sums of errors on the left and right side of time point $i$ respectively by
\begin{equation}
    \label{eq_U_nonli}
    U_i^{(l)}(\bbell) = \sum_{t=i-bn}^{i-1}\epsilon_t(\bbell)/(bn) \quad U_i^{(r)}(\bbell) = \sum_{t=i}^{i+bn-1}\epsilon_t(\bbell)/(bn).
\end{equation}
Let $x_i^{(m)}(\bbell)$ (resp. $I_{\epsilon}^{(m)}$, $U_i^{-(m)}(\bbell)$, $U_i^{+(m)}(\bbell)$) be $x_i(\bbell)$ (resp. $I_{\epsilon}$, $U_i^-(\bbell)$, $U_i^+(\bbell)$) with $\epsilon_t(\bbell)$ therein replaced by $\epsilon_t^{(m)}(\bbell)$, and define
$$\I_{i,\B_s}^{(m)}=\frac{1}{\sqrt{|\B_s\cap\L_0|}}\Big|\sum_{\bbell\in\B_s\cap\L_0} x_i^{(m)}(\bbell)\Big|,\qquad \I_{\epsilon}^{(m)}=\max_{1\le s\le S}\max_{bn+1\le i\le n-bn}\I_{i,\B_s}^{(m)}.$$

\begin{remark}[$m$-dependence of $x_i^{(m)}(\bbell)$]
    \label{rmk:mdepepsilont}
    By construction, for any $|\bbell_1-\bbell_2|_2>m$, $\epsilon_t^{(m)}(\bbell_1)$ is independent of $\epsilon_t^{(m)}(\bbell_2)$. As a consequence, $x_i^{(m)}(\bbell_1)$ is independent of $x_i^{(m)}(\bbell_2)$ for any $|\bbell_1-\bbell_2|_2>m$ as well. To see this, note that $
    \epsilon_t^{(m)}(\bbell_1)$ only depends on $\eta_{k,\bbell_0}$ for $k\leq t$ and $|\bbell_1-\bbell_0|_2\leq m/2.$ When $|\bbell_1-\bbell_2|_2>m$, there does not exist any $\bbell_0\in\ZZ^v,$ such that both $|\bbell_1-\bbell_0|_2$ and $|\bbell_2-\bbell_0|_2$ are upper bounded by $m/2.$ Hence, $\epsilon_t^{(m)}(\bbell_1)$ is independent of $\epsilon_t^{(m)}(\bbell_2)$. 
\end{remark}

Now, we divide each dimension $r$ of $\ZZ^v$ into consecutive big blocks and small blocks, $1\leq r\leq v$. We let the length of each small block be $m$ and big block be $Lm$, where $L$ is some large positive integer. 
% The small blocks are placed between the big blocks to create a form of asymptotic independence between the big blocks and the sum of small blocks shall be asymptotically negligible. 
Specifically, we define the $k$-th big block for the $r$-th dimension as the interval
\begin{equation}
    \label{eq_bigblock}
    \Theta_{r,k}=\big((k-1)(L+1)m, (k-1)(L+1)m+Lm\big],
\end{equation}
and the small block as the interval
\begin{equation}
    \label{eq_smallblock}
    \theta_{r,k}=\big((k-1)(L+1)m+Lm, k(L+1)m\big],
\end{equation}
for $k\in\ZZ$. 
% We define $N=\prod_{1\le r\le v} K_r$, which is the total number
Let $\underline{k}=(k_1,k_2,\ldots,k_v)^{\top}$, and define the hyper-block indexed by $\underline{k}$ as 
\begin{equation}
    \label{eq_block}
    \Omega(\underline{k}) = \big(\Theta_{1,k_1}\cup\theta_{1,k_1}\big) \times \big(\Theta_{2,k_2}\cup\theta_{2,k_2}\big) \times \cdots\times \big(\Theta_{v,k_v}\cup\theta_{v,k_v}\big).
\end{equation}
We consider the union of all the hyper-blocks which contains points in spatial space $\L_0$ and denote it by $\B_0$, that is,
\begin{equation}
    \label{eq_B0}
    \B_0 = \cup_{\big\{\underline{k}\in\ZZ^v: \Omega(\underline{k})\cap\L_0 \neq \emptyset\big\}}\Omega(\underline{k}).
\end{equation}
Let $N$ be the total number of hyper-blocks in $\B_0$.
% W.l.o.g., we assume that $\B_0=[0,n_1]\times[0,n_2]\times\ldots\times[0,n_v]$, where $n_v\in\ZZ^+$ can go to infinity as $n$ grows. 
% $K_{r}=\lceil n_r/((L+1)m)\rceil$
Note that each hyper-block contains a big hyper-block and the remaining small blocks. We define the big hyper-block indexed by $\underline{k}$ as
\begin{equation}
    \label{eq_Omega}
    \Omega^{\dagger}(\underline{k}) = \Theta_{1,k_1} \times \Theta_{2,k_2} \times \cdots\times \Theta_{v,k_v}.
\end{equation}
Denote $G_i(\underline{k})$ as the summation of $x_i^{(m)}(\bbell)$ within the big hyper-block indexed by $\underline{k}$,
\begin{align}
\label{eq:def_Gibs}
G_i(\underline{k})=\sum_{\{\bbell\in\Omega^{\dagger}(\underline{k})\cap\L_0\}}x_i^{(m)}(\bbell).
\end{align}
Further, we combine all $x_i^{(m)}(\bbell)$ within the big hyper-blocks in $\B_s$, denote it by $\tilde \I_{i,\B_s}$ after rescaling, and define the maximum of $\tilde \I_{i,\B_s}$ over all the spatial neighborhoods and moving windows by $\tilde \I_{\epsilon}$, that is
\begin{equation}
    \label{eq_In_tilde}
    \tilde \I_{i,\B_s}= \frac{1}{\sqrt{|\B_s\cap\L_0|}} \sum_{\{\underline{k}: \,\Omega^{\dagger}(\underline{k})\subset\B_s\}} G_i(\underline{k}), \qquad \tilde \I_{\epsilon} = \max_{1\le s\le S} \max_{bn+1\le i\le n-bn}\tilde \I_{i,\B_s}.
\end{equation}

\begin{remark}[Independence of big hyper-blocks $G_i(\underline{k})$]
\label{rmk_ind_block}
Note that $G_i(\underline{k})$, the summations of $x_i^{(m)}(\bbell)$ within big hyper-blocks, are independent for different $\underline{k}$. To see this, note that
$x_i^{(m)}(\bbell)$ only depends on $\eta_{t,\bbell'}$ for $|\bbell-\bbell'|_2\leq m/2.$ Hence,
$G_i(\underline{k})$ only depends on $\eta_{t,\bbell'}$ with $|\bbell'-\bbell|_2\leq m/2$ and $\bbell\in \Omega^{\dagger}(\underline{k})\cap\L_0.$ 
For any $\tilde{\underline k} \neq \underline{k},$ the $G_i(\tilde{\underline k})$ only depends on $\eta_{t,\tilde \bbell'}$ with  
$|\tilde \bbell'-\tilde \bbell|_2\leq m/2$ and $\tilde\bbell\in \Omega^{\dagger}(\tilde{\underline k})\cap\L_0.$ 
Therefore, $$|\tilde \bbell'-\bbell'|_2
\geq |\tilde \bbell-\bbell|_2
-|\tilde \bbell'-\tilde\bbell|_2
-|\bbell'-\bbell|_2\geq (m+1)-m/2-m/2=1.$$ 
This indicates that the innovations $\eta_{t,\bbell'}$ which $G_i(\underline{k})$ depends on do not overlap for different $\underline{k}$. As a direct consequence, $G_i(\underline{k})$ are independent of each other.
\end{remark}

\begin{proof}[Proof of Theorem \ref{thm3_nonli}]
Now we provide the rigorous proof of Theorem \ref{thm3_nonli}. Recall the Gaussian random variable $\tilde\Z_{\varphi}$ defined in (\ref{eq_Z_nonli}), for $\varphi=(i,s)\in\N$. In this proof, we will use the notation $\tilde\Z_{\varphi}=\tilde\Z_{i,\B_s}$ to explicitly indicate its dependence on $i$ and $s$. Recall the definition of $x_i(\bbell)$ in expression (\ref{eq_test_nonli}). For brevity, we define
\begin{equation}
    \label{eq_x_nonli_nbd}
    x_{i,s}^{\star}(\bbell) = x_i(\bbell)\One_{\bbell\in\B_s\cap\L_0}.
\end{equation}
We denote the Gaussian counterpart for each $x_{i,s}^{\star}(\bbell)$ by $z_{i,s}^{\star}(\bbell)$. Then, we let $Z^{\star}(\bbell)=(z^{\star}_{bn+1,1}(\bbell),\ldots,$ $z^{\star}_{n-bn,1}(\bbell),\ldots,z^{\star}_{bn+1,S}(\bbell),\ldots,z^{\star}_{n-bn,S}(\bbell))^{\top}\in\RR^{S(n-2bn)}$, $\bbell\in\B_0\cap\L_0$, which are the independent centered Gaussian random vectors with covariance matrix $\EE(X^{\star}(\bbell)X^{\star\top}(\bbell))\in\RR^{S(n-2bn)\times S(n-2bn)}$, where $X^{\star}(\bbell)=(x^{\star}_{bn+1,1}(\bbell),\ldots,x^{\star}_{n-bn,1}(\bbell),\ldots,$ $x^{\star}_{bn+1,S}(\bbell),\ldots,x^{\star}_{n-bn,S}(\bbell))^{\top}$. Note that although the length of the random vector $X^{\star}(\bbell)$ (resp. $Z^{\star}(\bbell)$) is $S(n-2bn)$, only the elements $x^{\star}_{i,s}(\bbell)$ (resp. $z^{\star}_{i,s}(\bbell)$) satisfying $\bbell\in\B_s\cap\L_0$ are nonzero. Moreover, the elements in the covariance matrix could be similarly developed by applying Lemma \ref{lemma_cov_thm3} as well. Let 
\begin{equation}
    \label{eq_z_nonli_nbd_ave}
    \bar{z}_{i,s}^{\star}=|\B_s\cap\L_0|^{-1/2}\sum_{\bbell\in\B_s\cap\L_0}z_{i,s}^{\star}(\bbell).
\end{equation}

Recall the definition $\tilde \I_{\epsilon}$ in expression (\ref{eq_In_tilde}) and the fact that $\Q_n=\I_{\epsilon}$ under the null hypothesis. Then, for any $\alpha>0$, we have
\begin{align}
    \label{eq_thm3_nonli_goal}
    & \quad \sup_{u\in\RR}\big[\PP\big(\tilde\Q_n \ge u\big) - \PP\big(\max_{1\le s\le S}\max_{bn+1\le i\le n-bn}\tilde\Z_{i,\B_s} \ge u\big) \big] \nonumber \\
    & \le \PP\big(bn|\I_{\epsilon}-\tilde \I_{\epsilon}| \ge \alpha\big) + \sup_{u\in\RR}\big|\PP\big(\tilde \I_{\epsilon} \le u\big)-\PP\big(\max_{i,s}\bar{z}_{i,s}^{\star}\le u\big)\big| \nonumber \\
    & \qquad + \sup_{u\in\RR}\PP\big(\big|\max_{i,s}bn\tilde\Z_{i,\B_s}-u\big|\le \alpha\big) + \sup_{u\in\RR}\big|\PP\big(\max_{i,s}\bar{z}_{i,s}^{\star}\le u\big) - \PP\big(\max_{i,s}\tilde\Z_{i,\B_s} \le u\big)\big|  \nonumber \\
    & =: \tilde\III_1+\tilde\III_2+\tilde\III_3+\tilde\III_4.
\end{align}
It directly follows from Lemmas \ref{lemma_chen2019} and \ref{lemma_cov_thm3} that $\tilde \III_4\lesssim (bn)^{-1/3}\log^{2/3}(nS)$. Next, we shall investigate the parts $\tilde\III_1$--$\tilde\III_3$ separately. For the simplicity of notation, we denote the ratio of volumes of neighborhoods $\B_s$ and $\B_0$ by $\rho_s$, that is
\begin{equation}
    \label{eq_volume_ratio}
    \rho_s = \lambda(\B_s)/\lambda(\B_0), \quad \rho_{\text{min}} = \min_{1\le s\le S}\rho_s.
\end{equation}
\noindent Define
\begin{equation}
    \label{eq_Delta}
    \Delta_1^{\star} = \Bigg(\frac{(nS)^{4/q}\rho_{\text{min}}^{-1}\log^7(nSN)}{N}\Bigg)^{1/6}, \quad \Delta_2^{\star} = \Bigg(\frac{(nS)^{4/q}\rho_{\text{min}}^{-1}\log^3(nSN)}{N^{1-4/q}}\Bigg)^{1/3}.
\end{equation}
Let $\alpha = z+L^{-1/2}\log^{1/2}(nS)$. For the part $\tilde\III_1$, by Lemmas \ref{lemma_nonli_m_approx} and \ref{lemma_nonli_block}, we have
\begin{align}
    \label{eq_thm3_nonli_part1}
    \tilde\III_1 & \le  \PP\big(bn|\I_{\epsilon}-\I_{\epsilon}^{(m)}| \ge z\big) + \PP\big(bn|\I_{\epsilon}^{(m)}-\tilde\I_{\epsilon}| \ge L^{-1/2}\log^{1/2}(nS)\big) \nonumber \\
    & \lesssim  nSm^{-q\xi/2}z^{-q/2} + \Delta_1^{\star} + \Delta_2^{\star}. 
    % \lesssim & (nS)^{2/(2+q)}m^{-q\xi/(2+q)}
\end{align}
By the Gaussian approximation in Lemma \ref{lemma_nonli_GA}, we have
\begin{align}
    \label{eq_thm1_part21}
    \tilde\III_2 \lesssim \Delta_1^{\star} + \Delta_2^{\star}.
\end{align}
For the part $\tilde\III_3$, by applying Lemma \ref{lemma_nazarov}, we obtain
\begin{equation}
    \label{eq_thm1_part3}
    \tilde\III_3 \lesssim \alpha\sqrt{\log(nSN)}.
\end{equation}
Let $z=(nS)^{2/(2+q)}m^{-q\xi/(2+q)}$ and combine the results from parts $\tilde\III_1$--$\tilde\III_3$, we have
\begin{align*}
    \tilde\III_1+\tilde\III_2+\tilde\III_3 \lesssim (nS)^{2/(2+q)}m^{-q\xi/(2+q)}\log^{1/2}(nSN) + \Delta_1^{\star} + \Delta_2^{\star} + L^{-1/2}\log(nSN).
\end{align*}
Recall $c_{p,n}$ defined in (\ref{eq_cpn}). We choose $L = c_{p,n}^{1/(4v)}$, $m= c_{p,n}^{1/(4v)}(nS)^{2/(q\xi)}(\log(pnS))^{\frac{2+q}{2q\xi}}$, and $N=B_{\text{min}}(Lm)^{-v}$ to balance the three parts $\tilde\III_1$--$\tilde\III_3$.

Finally, we combine the results from the parts $\tilde \III_1$--$\tilde \III_4$, and apply similar arguments for the other side of the inequality in expression (\ref{eq_thm3_nonli_goal}), which completes the proof.

% 

% By Lemma \ref{lemma_chen2019} and similar arguments in Lemmas \ref{lemma_m_approx} and \ref{lemma_GA}, it follows that
% \begin{align}
%     \label{eq_thm1_part22}
%     \III_{22} \le & \sup_{u\in\RR}\big|\PP\big(\tilde \I_{\epsilon}'\le u\big) - \PP\big(|\tilde \Z|_{\infty}\le u\big)\big| + \sup_{u\in\RR}\big|\PP\big(|\tilde \Z|_{\infty}\le u\big) - \PP\big(|\Z|_{\infty}\le u\big)\big| \nonumber \\
%     \lesssim & 
% \end{align}
\end{proof}

\begin{lemma}[$m$-dependent approximation]
\label{lemma_nonli_m_approx}
Suppose that Assumptions \ref{asm_nonli_finitemoment}--\ref{asm_nonli_sec_dep} hold. For constant $\xi>1$, we have
$$\PP\Big(\max_{1\le s\le S}\max_{bn+1\le i\le n-bn}bn\big|\I_{i,\B_s}-\I_{i,\B_s}^{(m)}\big|>z\Big)\lesssim nS z^{-q/2}m^{-q\xi/2},$$
where the constant in $\lesssim$ is independent of $n,p,b$ and $S$.
\end{lemma}

\begin{proof}
First, we note that, for any $z\in\RR$,
\begin{align}
    & \PP\Big(\max_{1\le s\le S}\max_{bn+1\le i\le n-bn}bn\big|\I_{i,\B_s}-\I_{i,\B_s}^{(m)}\big|>z\Big) \nonumber \\
    \le & \sum_{1\le s\le S}\sum_{bn+1\le i\le n-bn}\PP\Big(bn\big|\I_{i,\B_s}-\I_{i,\B_s}^{(m)}\big|>z\Big).
\end{align}
We aim to derive the upper bound of $\|\I_{i,\B_s}-\I_{i,\B_s}^{(m)}\|_{q/2}$, for $q\ge4$. Recall the definitions of $U_i^-(\bbell)$, $U_i^+(\bbell)$ in expression (\ref{eq_U_nonli}) and $U_i^{-(m)}(\bbell)$, $U_i^{+(m)}(\bbell)$ defined below (\ref{eq_U_nonli}). Also, recall the operator $\E_0(\cdot)=\cdot -\EE(\cdot)$. Then, it follows that 
\begin{align}
    \label{eq_nonli_mdep}
    & bn\big|\I_{i,\B_s}-\I_{i,\B_s}^{(m)}\big| \nonumber \\
    = &  \frac{bn}{\sqrt{|\B_s\cap\L_0|}}\Big|\sum_{\bbell\in\B_s\cap\L_0}\big[x_i(\bbell) - x_i^{(m)}(\bbell)\big]\Big| \nonumber \\
    = & \frac{bn}{\sqrt{|\B_s\cap\L_0|}}\Big|\sum_{\bbell\in\B_s\cap\L_0}\E_0\big[\big(U_i^-(\bbell) - U_i^+(\bbell)\big)^2 - \big(U_i^{-(m)}(\bbell) - U_i^{+(m)}(\bbell)\big)^2\big]\Big| \nonumber \\
    \lesssim &  \frac{bn}{\sqrt{|\B_s\cap\L_0|}}\Big|\sum_{\bbell\in\B_s\cap\L_0}\E_0\big[\big(U_i^-(\bbell)\big)^2 - \big(U_i^{-(m)}(\bbell)\big)^2\big] \nonumber \\
    & \quad + \sum_{\bbell\in\B_s\cap\L_0}\E_0\big[\big(U_i^+(\bbell)\big)^2 - \big(U_i^{+(m)}(\bbell)\big)^2\big]\Big|.
\end{align}
Note that we can decompose each $\epsilon_t(\bbell)$ into a sequence of martingale differences $\P_{k,l'}\epsilon_t(\bbell)$, that is
\begin{equation}
    \label{eq_MD_decompose}
    \epsilon_t(\bbell) = \sum_{k\le t}\sum_{l'\in\ZZ}\P_{k,l'}\epsilon_t(\bbell).
\end{equation}
This, along with the definition of $U_i^-(\bbell)$ in expression (\ref{eq_U_nonli}) gives
\begin{align}
    \label{eq_MD_Ul}
    \E_0\big[\big(U_i^-(\bbell)\big)^2\big] & = \frac{1}{(bn\sigma(\bbell))^2}\sum_{k_1,k_2\le i-1}\sum_{l_1',l_2'\in\ZZ}\E_0\big[M_{i,k_1,l_1'}(\bbell)M_{i,k_2,l_2'}(\bbell)\big],
\end{align}
where $M_{i,k,l'}(\bbell)$ is a sequence of martingale differences with respect to $k$ and $l'$ defined as
\begin{equation}
    \label{eq_MD}
    M_{i,k,l'}(\bbell) = \sum_{t=(i-bn)\vee k}^{i-1}\P_{k,l'}\epsilon_t(\bbell).
\end{equation}
Let $M_{i,k,l'}^{(m)}(\bbell)$ be $M_{i,k,l'}(\bbell)$ with $\epsilon_t(\bbell)$ therein replaced by $\epsilon_t^{(m)}(\bbell)$. 
Then, for $q\ge 4$, by Burkholder's inequality, we have
\begin{align}
    \label{eq_nonli_mdep_goal}
    & \quad \Big\|\sum_{\bbell\in\B_s\cap\L_0}\E_0\big[\big(U_i^-(\bbell)\big)^2 - \big(U_i^{-(m)}(\bbell)\big)^2\big]\Big\|_{q/2} \nonumber \\
    & = \frac{1}{(bn)^2}\Big\|\sum_{\bbell\in\B_s\cap\L_0}\sum_{k_1,k_2\le i-1}\sum_{l_1',l_2'\in\ZZ}\E_0\big[M_{i,k_1,l_1'}(\bbell)M_{i,k_2,l_2'}(\bbell) \nonumber \\
    & \quad - M_{i,k_1,l_1'}^{(m)}(\bbell)M_{i,k_2,l_2'}^{(m)}(\bbell)\big]/\sigma^2(\bbell)\Big\|_{q/2} \nonumber \\
    & \lesssim \frac{1}{(bn)^2}\Big(\sum_{k_1,k_2\le i-1}\sum_{l_1',l_2'\in\ZZ}\Big\|\sum_{\bbell\in\B_s\cap\L_0}\E_0\big[M_{i,k_1,l_1'}(\bbell)M_{i,k_2,l_2'}(\bbell) \nonumber \\
    & \quad - M_{i,k_1,l_1'}^{(m)}(\bbell)M_{i,k_2,l_2'}^{(m)}(\bbell)\big]/\sigma^2(\bbell) \Big\|_{q/2}^2\Big)^{1/2}.
\end{align}
To simplify the notation, we let $A_{l',\bbell} = M_{i,k_1,l'}(\bbell)/\sigma(\bbell)$, $B_{l',\bbell} = M_{i,k_2,l'}(\bbell)/\sigma(\bbell)$, and set $A_{l',\bbell}^{(m)}$, $B_{l',\bbell}^{(m)}$ to be the corresponding $m$-dependent versions. Note that
\begin{align}
    \label{eq_nonli_mdep_goal_simple}
    & \quad \sum_{l_1',l_2'\in\ZZ}\Big\|\sum_{\bbell\in\B_s\cap\L_0} A_{l_1',\bbell}(B_{l_2',\bbell} - B_{l_2',\bbell}^{(m)})\Big\|_{q/2}^2 \nonumber \\
    & \le \sum_{l_1',l_2'\in\ZZ}\Big\|\Big(\sum_{\bbell\in\B_s\cap\L_0} A_{l_1',\bbell}(B_{l_2',\bbell} - B_{l_2',\bbell}^{(m)})\Big)^2\Big\|_{q/2} \nonumber \\
    & = \sum_{l_1',l_2'\in\ZZ}\Big\|\sum_{\bbell_1,\bbell_2\in\B_s\cap\L_0} \big[A_{l_1',\bbell_1}(B_{l_2',\bbell_1} - B_{l_2',\bbell_1}^{(m)})\big] \cdot \big[A_{l_1',\bbell_2}(B_{l_2',\bbell_2} - B_{l_2',\bbell_2}^{(m)})\big]\Big\|_{q/2} \nonumber \\
    & \le \sum_{\bbell_1,\bbell_2\in\B_s\cap\L_0} \sum_{l_1'\in\ZZ}\big\|A_{l_1',\bbell_1}A_{l_1',\bbell_2}\big\|_{q} \cdot \sum_{l_2'\in\ZZ}\big\|(B_{l_2',\bbell_1}-B_{l_1',\bbell_1}^{(m)})\cdot(B_{l_2',\bbell_2} - B_{l_2',\bbell_2}^{(m)})\big\|_{q},
\end{align}
where the first inequality is by Jensen's inequality and the second inequality holds due to H\"{o}lder's inequality. Recall the bijection $\nu: \ZZ\rightarrow\ZZ^v$. If $|\bbell_1-\nu(l_1')|_2 < |\bbell_1-\bbell_2|_2/2$, then $|\bbell_2-\nu(l_1')|_2 > |\bbell_1-\bbell_2|_2/2$. For $i-bn\le k_1 \le i-1$, it follows from the similar arguments in expressions (\ref{eq_nonli_sec_dep_twoparts}) and (\ref{eq_nonli_sec_dep_part1}) that
\begin{align}
    \label{eq_nonli_mdep_part1}
    \sum_{\{l_1'\in\ZZ:|\bbell_1-\nu(l_1')|_2<|\bbell_1-\bbell_2|_2/2\}}\big\|A_{l_1',\bbell_1}A_{l_1',\bbell_2}\big\|_{q} = O\big\{|\bbell_1-\bbell_2|_2^{-\xi}\big\}.
\end{align}
Similarly, for $i-bn\le k_2\le i-1$, we can obtain
\begin{align}
    \label{eq_nonli_mdep_part2}
    \sum_{l_2'\in\ZZ}\big\|(B_{l_2',\bbell_1}-B_{l_1',\bbell_1}^{(m)})\cdot(B_{l_2',\bbell_2} - B_{l_2',\bbell_2}^{(m)})\big\|_{q} = O\big\{(m/2)^{-2\xi}\big\}.
\end{align}
Implementing the results in expressions (\ref{eq_nonli_mdep_part1}) and (\ref{eq_nonli_mdep_part2}) into (\ref{eq_nonli_mdep_goal_simple}), we have
\begin{align}
    & \quad \sum_{k_1,k_2\le i-1}\sum_{l_1',l_2'\in\ZZ}\Big\|\sum_{\bbell\in\B_s\cap\L_0}\E_0\big[M_{i,k_1,l_1'}(\bbell)\big(M_{i,k_2,l_2'}(\bbell)- M_{i,k_2,l_2'}^{(m)}(\bbell)\big)\big]/\sigma^2(\bbell) \Big\|_{q/2}^2 \nonumber \\
    & \lesssim (bn)^2|\B_s\cap\L_0|m^{-2\xi}.
\end{align}
This, along with expression (\ref{eq_nonli_mdep_goal}) yields
\begin{align}
    \max_{1\le s\le S}\max_{bn+1\le i\le n-bn}\frac{bn}{\sqrt{|\B_s\cap\L_0|}}\Big\|\sum_{\bbell\in\B_s\cap\L_0}\E_0\big[\big(U_i^-(\bbell)\big)^2 - \big(U_i^{-(m)}(\bbell)\big)^2\big]\Big\|_{q/2} = O\{m^{-\xi}\}.
\end{align}
We apply the similar arguments to the part $\E_0\big[\big(U_i^+(\bbell)\big)^2 - \big(U_i^{+(m)}(\bbell)\big)^2\big]$ and insert the results back to expression (\ref{eq_nonli_mdep}), which gives the desired result.

\end{proof}

\begin{lemma}[Block approximation]
    \label{lemma_nonli_block}
    Under Assumptions \ref{asm_density}--\ref{asm_nonli_sec_dep}, 
    for $\Delta_1^{\star}$ and $\Delta_2^{\star}$ defined in expression (\ref{eq_Delta}),
    we have
    $$\PP\Big(\max_{1\le s\le S}\max_{bn+1\le i\le n-bn}bn\big|\I_{i,\B_s}^{(m)}-\tilde\I_{i,\B_s}\big|\ge  \sqrt{L^{-1}\log(nS)} \Big) \lesssim \Delta_1^{\star} + \Delta_2^{\star},$$
    where the constant in $\lesssim$ is independent of $n,p,b$ and $S$.
\end{lemma}

\begin{proof}
Recall the big hyper-block $\Omega^{\dagger}(\underline{k})$ defined in expression (\ref{eq_Omega}). We denote the remaining area after removing the big hyper-blocks from $\B_s$ as $\B_s^{\diamond}$, that is, for $1\le s\le S$,
\begin{equation}
    \label{eq_block_small}
    \B_s^{\diamond} = \B_s\setminus\cup_{\{\underline{k}:\,\Omega^{\dagger}(\underline{k})\subset\B_s\}}\Omega^{\dagger}(\underline{k}),
\end{equation}
Our goal is to derive an upper bound for $\max_{s,i} bn|\B_s\cap\L_0|^{-1/2}\big|\sum_{\bbell\in\B_s^{\diamond}\cap\L_0}x_i^{(m)}(\bbell)\big|$ by applying Gaussian approximation. For each big block $\Theta_{r,k}$ defined in expression (\ref{eq_bigblock}), we let $\Theta^{\diamond}_{r,k}$ be $\Theta_{r,k}$ with the last part of length $m$ removed, that is
\begin{equation}
    \label{eq_block_bigblock_short}
    \Theta^{\diamond}_{r,k} = \big((k-1)(L+1)m, (k-1)(L+1)m+(L-1)m\big].
\end{equation}
Define the small hyper-block indexed by $\underline{k}$ as
\begin{equation}
    \label{eq_block_first}
    \omega^{\diamond}_r(\underline{k}) = \Theta^{\diamond}_{1,k_1} \times \cdots \times \Theta^{\diamond}_{r-1,k_2}  \times \theta_{r,k_r} \times \Theta^{\diamond}_{r+1,k_{r+1}} \times \cdots \times \Theta^{\diamond}_{v,k_v},
\end{equation}
In fact, the union of $\omega_r^\diamond(\underline{k})$ is close to $\B_s^\diamond$. To see this,
we denote  $\B_{s,1}^{\diamond}=\cup_{\{r,\underline{k}:\,\omega^{\diamond}_r(\underline{k})\subset\B_s\}}\omega^{\diamond}_r(\underline{k})$ and let $\B_{s,2}^{\diamond} = \B_s^{\diamond}\setminus\B_{s,1}^{\diamond}$. Then
\begin{equation}
    \label{eq_block_dominate}
    \lambda(\B_s^{\diamond}) = \lambda(\B_{s,1}^{\diamond}) \big(1+O(1/L)\big).
\end{equation}
It follows that
\begin{align}
    \label{eq_block_lower_goal}
    \Big\| \sum_{\bbell\in\B_s^{\diamond}\cap\L_0}x_i^{(m)}(\bbell)\Big\|_2
    \ge & \Big\| \sum_{\bbell\in\B_{s,1}^{\diamond}\cap\L_0}x_{i}^{(m)}(\bbell) \Big\|_2 -\Big\| \sum_{\bbell\in\B_{s,2}^{\diamond}\cap\L_0}x_i^{(m)}(\bbell) \Big\|_2.
\end{align}
Recall the operator $\E_0(\cdot)=\cdot -\EE(\cdot)$. Since $x_i^{(m)}(\bbell)=\E_0\big[\big(U_i^{-(m)}(\bbell)-U_i^{+(m)}(\bbell)\big)^2\big]$, by the similar arguments in Remark \ref{rmk_ind_block}, $U_i^{-(m)}(\bbell)$ (or $U_i^{+(m)}(\bbell)$) within different small hyper-blocks $\omega_r^{\diamond}(\underline{k})$ are independent for different $r$ or $\underline{k}$. Therefore,
\begin{align}
\label{eq_block_ind}
    \Big\| \sum_{\bbell\in\B_{s,1}^{\diamond}\cap\L_0}\E_0\big[\big(U_i^{-(m)}(\bbell)\big)^2\big]\Big\|_2^2 
    = \sum_{\{r,\underline{k}:\,\omega_r^{\diamond}(\underline{k})\subset\B_{s,1}^{\diamond}\}} \Big\|\sum_{\bbell\in\omega_r^{\diamond}(\underline{k})\cap\L_0}\E_0\big[\big(U_i^{-(m)}(\bbell)\big)^2\big]\Big\|_2^2.
\end{align}
Recall that $M_{i,k,l'}^{(m)}(\bbell)$ is a sequence of martingale differences with respect to $k$ and $l'$, that is
$$
    M_{i,k,l'}^{(m)}(\bbell) = \sum_{t=(i-bn)\vee k}^{i-1}\P_{k,l'}\epsilon_t^{(m)}(\bbell).
$$
Then, it follows from Assumptions (dependence and moment) that
\begin{align}
    \label{eq_block_step1}
    & \quad  \Big\|\sum_{\bbell\in\omega_r^{\diamond}(\underline{k})\cap\L_0}\E_0\big[\big(U_i^{-(m)}(\bbell)\big)^2\big]\Big\|_2^2 \nonumber \\
    & =  \frac{1}{(bn)^4}\sum_{k_1,k_2\le i-1}\sum_{l_1',l_2'\in\ZZ} \Big\lVert \sum_{\bbell\in\omega_r^{\diamond}(\underline{k})\cap\L_0} \E_0\big[M_{i,k_1,l_1'}^{(m)}(\bbell)M_{i,k_2,l_2'}^{(m)}(\bbell)\big]/\sigma^2(\bbell)\Big\rVert_2^2 \nonumber \\
    & =  \frac{1}{(bn)^4}\sum_{\bbell_1,\bbell_2}  \Big\lVert \sum_{k_1,l_1'} \E_0\big[M_{i,k_1,l_1'}^{(m)}(\bbell_1)M_{i,k_2,l_1'}^{(m)}(\bbell_2)\big] \nonumber \\
    &\quad \cdot \sum_{k_2,l_2'} \E_0\big[M_{i,k_1,l_2'}^{(m)}(\bbell_1)M_{i,k_2,l_2'}^{(m)}(\bbell_2)\big]\Big\rVert_2/\big[\sigma(\bbell_1)\sigma(\bbell_2)\big]\nonumber \\
    & \asymp \frac{1}{(bn)^2}|\omega_r^{\diamond}(\underline{k})\cap\L_0|,
\end{align}
where the last $\asymp$ holds due to the decay of correlation indicated by Assumption \ref{asm_nonli_sec_dep}. This, together with similar arguments on $U_i^{+(m)}(\bbell)$ gives
\begin{equation}
    \label{eq_block_part1}
    \Big\| \sum_{\bbell\in\B_{s,1}^{\diamond}\cap\L_0}x_i^{(m)}(\bbell) \Big\|_2^2 \asymp (bn)^{-2}\sum_{\{r,\underline{k}:\,\omega_r^{\diamond}(\underline{k})\subset\B_{s,1}^{\diamond}\}} |\omega_r^{\diamond}(\underline{k})\cap\L_0|.
\end{equation}
This, along with similar arguments on the second part in expression (\ref{eq_block_lower_goal}) yields
\begin{align}
    \label{eq_block_lower_result}
    \Big\| \frac{bn}{\sqrt{|\B_s^{\diamond}\cap\L_0|}}\sum_{\bbell\in\B_s^{\diamond}\cap\L_0}x_i^{(m)}(\bbell)\Big\|_2^2 \gtrsim 1.
\end{align}
The upper bounds for the $q$-th moment of $bn|\B_s\cap\L_0|^{-1/2}\big|\sum_{\bbell\in\B_s^{\diamond}\cap\L_0}x_i^{(m)}(\bbell)\big| $ can be similarly derived. Hence, by Proposition 2.1 in \textcite{chernozhukov_central_2017} and the fact that $\max_{1\le s\le S}\lambda(\B_s^{\diamond})/\lambda(\B_s) \lesssim 1/L$, 
we have, for $\Delta_1$ and $\Delta_2$ defined in expression (\ref{eq_Delta}),
\begin{equation}
    \label{eq_block_result}
    \PP\Bigg(\max_{1\le s\le S}\max_{bn+1\le i\le n-bn} \frac{bn}{\sqrt{|\B_s^{\diamond}\cap\L_0|}}\Big|\sum_{\bbell\in\B_s^{\diamond}\cap\L_0}x_i^{(m)}(\bbell)\Big| \ge  \sqrt{\log(nS)} \Bigg) \lesssim \Delta_1+ \Delta_2,
\end{equation}
which further yields our desired result.

\end{proof}

\begin{lemma}[Gaussian approximation]
   \label{lemma_nonli_GA}
    Under Assumptions \ref{asm_density}--\ref{asm_nonli_sec_dep}, for $\Delta_1^{\star}$ and $\Delta_2^{\star}$ defined in expression (\ref{eq_Delta}), we have
    $$\sup_{u\in\RR}\Big|\PP\big(\max_{1\le s\le S} \max_{bn+1\le i\le n-bn}\I_{i,\B_s,b} \ge u\big) - \PP\big(\max_{1\le s\le S}\max_{bn+1\le i\le n-bn}|\bar{z}_{i,s}^{\star}| \ge u\big)\Big| \lesssim \Delta_1^{\star} + \Delta_2^{\star},$$
    where the constant in $\lesssim$ is independent of $n,p,b$ and $S$.
\end{lemma}

\begin{proof}
The goal of this lemma is to apply the Gaussian approximation theorem on the summation of $G_i(\underline{k})$, for all $\underline{k}$ such that $\Omega^{\dagger}(\underline{k})\subset \B_s$. We shall first derive the upper bound of $\|G_i(\underline{k})\|_{q/2}$. W.l.o.g., let us consider the first big hyper-block, i.e. $\underline {k}=\underline{1}.$
To this end, we shall divide this big hyper-block into small hyper-blocks by partitioning block $\Theta_{r,1},$ $1\leq r\leq v,$ into small ones $I_{r}(j_r)$, $1\le j_r\le L$, with expression
\begin{equation}
    \label{eq_GA_group_index}
    I_r(j_r) = \big(  (j_r-1)m,\, j_rm \big].
\end{equation}
Define $\omega(\underline{j})$ as the locations of series in the small group indexed by $\underline{j}$ in the first big block, that is
\begin{equation}
    \label{eq_omega}
    \omega(\underline{j}) = I_1(j_1)\times I_2(j_2) \times \cdots I_v(j_v).
\end{equation}
% Let $\bbell=(l_1,l_2,\ldots,l_v)^{\top}$. 
Further, we define the summation of $x_i^{(m)}(\bbell)$ within each small hyper-block as
\begin{equation}
    \label{eq_GA_group}
    g_i(\underline{j}) = \sum_{\bbell\in\omega(\underline{j})\cap\L_0} x_i^{(m)}(\bbell).
\end{equation}
Then $G_i(\underline{1})$ in expression (\ref{eq:def_Gibs}) can be written into
\begin{equation}
    \label{eq_GA_sum_group}
    G_i(\underline{1}) = \sum_{j_1,\ldots,j_v=1}^L g_i(\underline{j}).
\end{equation}
By the parity of index, we can divide all the $g_i(\underline{j})$ into $2^v$ parts, and for each part, each coordinate of index $\underline{j}$ are all even or odd. Note that
$g_i(\underline{j})$ within each part are independent by similar arguments in Remark \ref{rmk_ind_block}. W.l.o.g., we consider such a part where $j_1,\ldots,j_v$ are all odd, then by Burkholder's inequality,
\begin{align}
\label{eq_GA_qnorm_goal}
&\Big\|\sum_{j_1,\ldots,j_v \textrm{ are odd}} g_i(\underline{j})\Big\|_{q/2}^2
\lesssim
\sum_{j_1,\ldots,j_v \textrm{ are odd}}\big\| g_i(\underline{j})\big\|_{q/2}^2. \end{align}
Recall that $x_i^{(m)}(\bbell) = \E_0\big[(U_i^{-(m)}(\bbell)-U_i^{+(m)}(\bbell))^2\big],$ where $\E_0(\cdot)=\cdot -\EE(\cdot)$, and $M_{i,k,l'}(\bbell)$ is a sequence of martingale differences with respect to $k$ and $l'$, that is
$$
    M_{i,k,l'}(\bbell) = \sum_{t=(i-bn)\vee k}^{i-1}\P_{k,l'}\epsilon_t(\bbell).
$$ 
By Assumptions \ref{asm_nonli_finitemoment}--\ref{asm_nonli_sec_dep}, we have
\begin{align}
\label{eq_GA_qnorm_step1}
    & \quad\Big\lVert\sum_{\bbell\in\omega(\underline{j})\cap\L_0}\E_0\big[\big(U_i^{-(m)}(\bbell)\big)^2\big]\Big\rVert_{q/2}^2 \nonumber \\
    & \lesssim (bn)^{-4}\sum_{k_1,k_2\le i-1} \sum_{l_1',l_2'\in\ZZ}\Big\|\sum_{\bbell\in\Omega^{\dagger}(\underline{k})\cap\B_s\cap\L_0}\E_0\big[M_{i,k_1,l_1'}^{(m)}(\bbell)M_{i,k_2,l_2'}^{(m)}(\bbell)\big]/\sigma^2(\bbell)\Big\|_{q/2}^2 \nonumber\\
    & \lesssim (bn)^{-2}\big|\omega(\underline{j})\cap\L_0\big|,
\end{align}
% {\color{blue}
% The second inequality in expression (\ref{eq_GA_qnorm_step1}) follows from Assumptions (temporal dependence decay) and the third inequality follows from Assumption (spatial dependence decay). 
% }
The same upper bounds can be derived for the remaining $2^v-1$ parts by similar arguments. This, along with expressions (\ref{eq_GA_qnorm_goal}), (\ref{eq_GA_qnorm_step1}) yields
\begin{equation}
    \label{eq_GA_G_bd}
    \|G_i(\underline1)\|_{q/2}^2 \lesssim
    \sum_{j_1,\ldots,j_v=1}^L\|g_i(\underline{j})\|_{q/2}^2\leq (bn)^{-2}\sum_{j_1,\ldots,j_v=1}^L
    |\omega(\underline{j})\cap\L_0|
    =(bn)^{-2}|\Omega^{\dagger}(\underline 1)\cap\L_0|.
\end{equation}
The bound in expression (\ref{eq_GA_G_bd}) holds uniformly over $\underline{k}$. Therefore, we achieve
\begin{align}
    \label{eq_GA_qnorm_result}
    & \quad \frac{1}{N}\sum_{\{\underline{k}:\, \Omega^{\dagger}(\underline{k})\subset\B_s\}}\EE\Big|\frac{bn\sqrt{N}}{\sqrt{|\B_s\cap\L_0|}}G_i(\underline{k})\Big|^{q/2} \nonumber \\
    & \lesssim N^{q/4-1}\sum_{\{\underline{k}:\, \Omega^{\dagger}(\underline{k})\subset\B_s\}} \Big(\frac{|\Omega^{\dagger}(\underline{k})\cap\L_0|}{|\B_s\cap\L_0|}\Big)^{q/4}\nonumber\\
    & \lesssim N^{q/4-1}\sum_{\{\underline{k}:\, \Omega^{\dagger}(\underline{k})\subset\B_s\}} \Big(\frac{|\Omega^{\dagger}(\underline{k})\cap\L_0|}{|\B_s\cap\L_0|}\Big)
    \max_{\{\underline{k}:\, \Omega^{\dagger}(\underline{k})\subset\B_s\}} \Big(\frac{|\Omega^{\dagger}(\underline{k})\cap\L_0|}{|\B_s\cap\L_0|}\Big)^{q/4-1}.
\end{align}
Recall the ratio of volumes $\rho_s$, $1\le s\le S$, defined in expression (\ref{eq_volume_ratio}). By Assumption \ref{asm_density}, we have $\sum_{\{\underline{k}:\, \Omega^{\dagger}(\underline{k})\subset\B_s\}}|\Omega^{\dagger}(\underline{k})\cap\L_0|/|\B_s\cap\L_0|\lesssim 1$, and
\begin{align*}
\max_{\underline{k}} \frac{|\Omega^{\dagger}(\underline{k})\cap\L_0|}{|\B_s\cap\L_0|} 
\lesssim \max_{\underline{k}} \frac{\lambda(\Omega^{\dagger}(\underline{k}))}{\lambda(\B_s)}= \max_{\underline{k}}\frac{\lambda(\Omega^{\dagger}(\underline{k}))}{\lambda(\B_0)}\cdot \frac{\lambda(\B_0)}{\lambda(\B_s)} = N^{-1}\cdot\rho_s^{-1}.
\end{align*}
Implementing above to \eqref{eq_GA_qnorm_result}, we have
\begin{align}
    \label{eq_GA_qnorm_result2}
    \frac{1}{N}\sum_{\{\underline{k}:\, \Omega^{\dagger}(\underline{k})\subset\B_s\}}\EE\Big|\frac{bn\sqrt{N}}{\sqrt{|\B_s\cap\L_0|}}G_i(\underline{k})\Big|^{q/2} \lesssim \rho_s^{1-q/4}.
\end{align}
Further, following from the similar arguments to achieve expression (\ref{eq_GA_qnorm_result2}), we obtain
\begin{align}
    \label{eq_GA_maxnorm}
    & \quad \EE\Big(\max_{1\le s\le S} \max_{bn+1\le i\le n-bn} \Big|\frac{bn\sqrt{N}}{\sqrt{|\B_s\cap\L_0|}}G_i(\underline{k})\One_{\Omega^{\dagger}(\underline{k})\subset\B_s}\Big|^{q/2}\Big) \nonumber \\
    & \le  \sum_{1\le s\le S}\sum_{bn+1\le i\le n-bn} \EE\Big|\frac{bn\sqrt{N}}{\sqrt{|\B_s\cap\L_0|}}G_i(\underline{k})\One_{\Omega^{\dagger}(\underline{k})\subset\B_s\}}\Big|^{q/2}
    \lesssim nS\rho_{\text{min}}^{-q/4}.
\end{align}
Now we prove the lower bound of $N^{-1}\sum_{\{\underline{k}:\, \Omega^{\dagger}(\underline{k})\subset\B_s\}} \EE\big(bnN^{1/2}(|\B_s\cap\L_0|)^{-1/2}G_i(\underline{k})\big)^2$ is away from zero. By the orthogonality of the martingale differences $M_{i,k,l'}^{(m)}(\bbell)$ with respect to $k$ and $l'$, we have
\begin{align}
    \label{eq_GA_lower}
    & \quad \big\lVert G_i(\underline{k})\One_{\Omega^{\dagger}(\underline{k})\subset\B_s\}}\big\rVert_2^2 \nonumber \\
    & =  \frac{1}{(bn)^4}\sum_{k_1,k_2\le i-1}\sum_{l_1',l_2'\in\ZZ} \Big\lVert \sum_{\bbell\in\Omega^{\dagger}(\underline{k})\cap\B_s\cap\L_0} \E_0\big[M_{i,k_1,l_1'}^{(m)}(\bbell)M_{i,k_2,l_2'}^{(m)}(\bbell)\big]/\sigma^2(\bbell)\Big\rVert_2^2 \nonumber \\
    & = \frac{1}{(bn)^4}  \sum_{\bbell_1,\bbell_2} \EE \Big( \sum_{k_1,l_1'} \E_0\big[M_{i,k_1,l_1'}^{(m)}(\bbell_1)M_{i,k_1,l_1'}^{(m)}(\bbell_2)\big] \nonumber \\
    & \quad \cdot \sum_{k_2,l_2'} \E_0\big[M_{i,k_2,l_2'}^{(m)}(\bbell_1)M_{i,k_2,l_2'}^{(m)}(\bbell_2)\big]\Big)/\big[\sigma(\bbell_1)\sigma(\bbell_2)\big] \nonumber \\
    & = \frac{1}{(bn)^4}  \sum_{\bbell_1,\bbell_2} \EE \Big( \sum_{k,l'} \E_0\big[M_{i,k,l'}^{(m)}(\bbell_1)M_{i,k,l'}^{(m)}(\bbell_2)\big]\Big)^2/\big[\sigma(\bbell_1)\sigma(\bbell_2)\big] \nonumber \\
    & \asymp \frac{1}{(bn)^2}  \sum_{\bbell_1,\bbell_2} \sigma(\bbell_1,\bbell_2)/\big[\sigma(\bbell_1)\sigma(\bbell_2)\big] \nonumber \\
    & \asymp \frac{1}{(bn)^2}|\Omega^{\dagger}(\underline{k})\cap\L_0|.
\end{align}
where the first $\asymp$ follows from the temporal dependence decay in Assumption \ref{asm_nonli_temp_dep} and the $m$-dependent approximation in Lemma \ref{lemma_nonli_m_approx}, and the last $\asymp$ is due to the decay of long-run correlation in (\ref{eq_nonli_sec_dep2}).
Consequently, by expression (\ref{eq_GA_lower}) and Assumption (volume ratio), we obtain
\begin{align}
    \label{eq_GA_lower_result}
    \frac{1}{N}\sum_{\{\underline{k}:\Omega^{\dagger}(\underline{k})\subset\B_s\}}\EE\Big(\frac{bn\sqrt{N}}{\sqrt{|\B_s\cap\L_0|}}G_i(\underline{k})\Big)^2 \asymp \sum_{\{\underline{k}:\Omega^{\dagger}(\underline{k})\subset\B_s\}}\frac{|\Omega^{\dagger}(\underline{k})\cap\L_0|}{|\B_s\cap\L_0|} \gtrsim 1.
\end{align}
Finally, by expressions (\ref{eq_GA_qnorm_result}), (\ref{eq_GA_maxnorm}) and (\ref{eq_GA_lower_result}), we apply Proposition 2.1 in \textcite{chernozhukov_central_2017}, and the desired result is achieved.

% {\color{blue}
% By expression (\ref{eq_GA_maxnorm}), we have
% $$B_n\lesssim(nS)^{2/q}\rho_{\text{min}}^{-1/2},$$
% which is consistent with $B_n = \sqrt{p/B_{\text{min}}}(nS)^{2/q}$ in Theorem \ref{thm2_constanttrend}. Then, we have
% \begin{equation*}
%     \Delta_1 = \Bigg(\frac{(nS)^{4/q}\rho_{\text{min}}^{-1}\log^7(nSN)}{N}\Bigg)^{1/6}, \quad \Delta_2 = \Bigg(\frac{(nS)^{4/q}\rho_{\text{min}}^{-1}\log^3(nSN)}{N^{1-4/q}}\Bigg)^{1/3}
% \end{equation*}
% }

\end{proof}

\subsection{Proofs of Corollaries \ref{cor_power}--\ref{cor_sec_dep_power}}
\begin{proof}[Proof of Corollary \ref{cor_power}]
Under the conditions in Theorem \ref{thm1_constanttrend} and (\ref{eq_thm11_o1}), the testing power $\PP(\Q_n>\omega)$ satisfies
\begin{equation}
    \label{eq_power_expression}
    \PP(\Q_n>\omega)=\PP\big(\max_{bn+1\le i\le n-bn}\big(|V_i|_2^2-\bar{c}\big)> \omega\big)=\PP\big(\max_{bn+1\le i\le n-bn}\Z_i>\omega\big)+o(1).
\end{equation}
Recall the definitions of $U_{i,j}^{(l)}$, $U_{i,j}^{(r)}$ in expression (\ref{eq_U}), $U_i^{(l)}=(U_{i,j}^{(l)})_{1\le j\le p}$ and $U_i^{(r)}=(U_{i,j}^{(r)})_{1\le j\le p}$, and the weighted break size $d_i$ defined in (\ref{eq_EV_appr}). Then, we have
\begin{align}
    \label{eq_power_step1}
    \max_{bn+1\le i\le n-bn}\big(|V_i|_2^2-\bar{c}\big) & = \max_{bn+1\le i\le n-bn}\Big(\big|d_i + \Lambda^{-1}(U_i^{(l)}-U_i^{(r)})\big|_2^2 -\bar{c}\Big) \nonumber \\
    & \ge \max_{bn+1\le i\le n-bn}|d_i|_2^2 - \max_{bn+1\le i\le n-bn}\Big|2\sum_{j=1}^pd_{i,j}(U_{i,j}^{(l)}-U_{i,j}^{(r)})/\sigma_j\Big| \nonumber \\
    & \quad - \max_{bn+1\le i\le n-bn}\Big|\bar{c}-\sum_{j=1}^p(U_{i,j}^{(l)}-U_{i,j}^{(r)})^2/\sigma_j^2\Big|,
\end{align}
where by Lemma \ref{lemma_tail_cross} (i),
$$\max_{bn+1\le i\le n-bn}\Big|2\sum_{j=1}^pd_{i,j}(U_{i,j}^{(l)}-U_{i,j}^{(r)})/\sigma_j\Big|/|d_i|_2 = O_{\PP}\Big(\sqrt{\log(n)/(bn)}\Big),$$
and by the Gaussian approximation in Theorem \ref{thm1_constanttrend},
$$\max_{bn+1\le i\le n-bn}\Big|\bar{c}-\sum_{j=1}^p(U_{i,j}^{(l)}-U_{i,j}^{(r)})^2/\sigma_j^2\Big| = O_{\PP}\Big(\sqrt{p\log(n)/(bn)^2}\Big).$$
Therefore, we can see from expression (\ref{eq_power_step1}) that the power of our test depends on the vectors $d_i$'s, and $d_i$ is determined by the true jump sizes $\gamma_k$'s. Since $b\ll\min_{0\le k\le K}(u_{k+1}-u_k)$, when $\max_{1\le k\le K}n(u_{k+1}-u_k)|\Lambda^{-1}\gamma_k|_2^2\gg \sqrt{p\log(n)}$, 
% $\max_i|d_i|_2 \gg (p\log (n))^{1/4}(bn)^{-1/2}$, 
we have the testing power $\PP(\Q_n>\omega)\rightarrow1$, as $n\rightarrow\infty$.
\end{proof}

\begin{proof}[Proof of Corollary \ref{cor_local_power}]
Recall the neighborhood norm $|\cdot|_{2,s}$, $1\le s\le S$, defined in Definition \ref{def_linear_nbdnorm}. Suppose that the conditions in Theorem \ref{thm2_constanttrend} and (\ref{eq_thm21_o1}) hold. Then, the testing power $\PP(\Q_n^{\diamond}>\omega^{\diamond})$ satisfies
\begin{align}
    \label{eq_local_power_expression}
    \PP(\Q_n^{\diamond}>\omega^{\diamond}) & = \PP\Big(\max_{1\le s\le S}\max_{bn+1\le i\le n-bn}|\L_s|^{-1/2}\big(|V_i|_{2,s}^2-\bar{c}_s^{\diamond}\big)> \omega^{\diamond}\Big) \nonumber \\
    & = \PP\big(\max_{1\le s\le S}\max_{bn+1\le i\le n-bn}\Z_{i,s}^{\diamond}> \omega^{\diamond}\big)+o(1).
\end{align}
Note that
\begin{align}
    \label{eq_local_power_step1}
    & \quad \max_{1\le s\le S}\max_{bn+1\le i\le n-bn}|\L_s|^{-1/2}\big(|V_i|_{2,s}^2-\bar{c}_s^{\diamond}\big) \nonumber \\
     & \ge \max_{i,s}|\L_s|^{-1/2}|d_i|_{2,s}^2 - \max_{i,s}2|\L_s|^{-1/2}\Big|\sum_{j\in\L_s}d_{i,j}(U_{i,j}^{(l)}-U_{i,j}^{(r)})/\sigma_j\Big| \nonumber \\
    & \quad -  \max_{i,s}|\L_s|^{-1/2}\Big|\bar{c}_s^{\diamond}-\sum_{j\in\L_s}(U_{i,j}^{(l)}-U_{i,j}^{(r)})^2/\sigma_j^2\Big|,
\end{align}
where by Lemma \ref{lemma_tail_cross} (i),
$$\max_{i,s}|\L_s|^{-1/2}\Big|\sum_{j\in\L_s}d_{i,j}(U_{i,j}^{(l)}-U_{i,j}^{(r)})/\sigma_j\Big|/|d_i|_{2,s} = O_{\PP}\Big(\sqrt{\log(nS)/(bn|\L_{\text{min}}|)}\Big),$$
and by the Gaussian approximation in Theorem \ref{thm2_constanttrend},
$$\max_{i,s}|\L_s|^{-1/2}\Big|\bar{c}_s^{\diamond}-\sum_{j\in\L_s}(U_{i,j}^{(l)}-U_{i,j}^{(r)})^2/\sigma_j^2\Big| = O_{\PP}\Big(\sqrt{\log(nS)/(bn)^2}\Big).$$
As a result, Since $b\ll\min_{0\le k\le K}(u_{k+1}-u_k)$, when $\max_{1\le s\le S}n(u_{k+1}-u_k)|\Lambda^{-1}\gamma_k|_{2,s}^2\gg \sqrt{|\L_{\min}|\log(nS)}$,
% when $\max_{i,s}|d_i|_{2,s} \gg (|\L_{\text{min}}|\log (nS))^{1/4}(bn)^{-1/2}$, 
the testing power $\PP(\Q_n^{\diamond}>\omega^{\diamond})\rightarrow1$, as $n\rightarrow\infty$.
\end{proof}
The proofs for Corollaries \ref{cor_nonli_power} and \ref{cor_sec_dep_power} closely resemble the proofs for Corollaries \ref{cor_local_power} and \ref{cor_power}, respectively. However, they require a different approach to deal with the cross-sectional dependence when deriving the tail probability for the noise component using the Gaussian approximation. This approach is similar to the block approximation technique used in the proofs of Theorems \ref{thm3_nonli} and \ref{thm3_constanttrend}. Due to this similarity, we have omitted the proofs for Corollaries \ref{cor_nonli_power} and \ref{cor_sec_dep_power} here.

\subsection{Proof of Proposition \ref{thm3_constanttrend}}

The proof strategies for Proposition \ref{thm3_constanttrend} are similar to the ones used for Theorem \ref{thm3_nonli}. Specifically, the $v$-dimensional block approximation applied in the proof of Theorem \ref{thm3_nonli} can be reduced to a one-dimension version.
Recall expression (\ref{eq_thm1_prime_Iepsilon2}) that
$$\frac{bn}{\sqrt{p}}I_{\epsilon}=\max_{bn+1\le i\le n-bn}\Big|\frac{1}{\sqrt{p}}\sum_{j=1}^p x_{i,j}\Big|.$$
To extend the proof of Theorem \ref{thm1_constanttrend} with spatial dependence, we first construct an $l_p$-dependent approximation of $x_{i,j}$, denoted as $x_{i,j}^*.$ Divide the sequence $x_{i,j}^*$ into several consecutive big blocks. By dropping off some small block at the beginning of each big block, we obtain a sequence of independent blocks. Finally, we shall finish the proof by applying the Gaussian approximation theorem to this independent sequence.

\begin{proof}[Proof of Proposition \ref{thm3_constanttrend}]
First, we introduce some necessary definitions. Let $\epsilon_t^*=\sum_{k\ge0}A_k^*\eta_{t-k}$ to be the $l_p$-dependent approximation of $\epsilon_t$, where, for each $k\ge0$ and $1\le j\le p$, 
\begin{equation}
    \label{eq_thm31_step1_Ak}
    A_{k,j,s}^*= 
    \begin{cases}
        A_{k,j,s},  &|s-j|\le l_p/2, \\
        0,          & \text{otherwise}.
    \end{cases}
\end{equation}
The new coefficient matrices $A_k^*\in\RR^{p\times p}$ leads to the $l_p$-dependent sequence of $\epsilon_{t,j}^*$ over $j$. Let $x_{i,j}^*$ (resp. $I_{\epsilon}^*$) be $x_{i,j}$ (resp. $I_{\epsilon}$) with $\epsilon_{t,j}$ therein replaced by $\epsilon_{t,j}^*$, and define 
$$\frac{bn}{\sqrt{p}}I_{\epsilon}^*=\max_{bn+1\le i\le n-bn}\Big|\frac{1}{\sqrt{p}}\sum_{ j=1}^p x_{i,j}^*\Big|.$$
We then divide $\{x_{i,j}^*\}_{1\le j\le p}$, into a sequence of big blocks. For $1\le  k\le r_p$, we denote the index set of the $ k$-th big block by
$$\Theta_k:=\{( k-1)b_p+1, ( k-1)b_p+2, \cdots,  k b_p\}.$$ 
Each big block has $b_p$ elements and the number of big blocks is $r_p=\lfloor p/b_p\rfloor$. 
Moreover, define the index set of the $ k$-th small block by
$$\theta_k:=\{( k-1)b_p+1, ( k-1)b_p+2, \cdots, ( k-1)b_p+l_p\}.$$ 
Note that each small block is defined to be the first $l_p$ elements of the corresponding big block, and we have $\cup_{1\le k\le r_p} \Theta_k=\{1,\cdots,p\}$. We define the rescaled sum of $x_{i,j}^*$ within the $ k$-th chunked big block as
$$ y_{i, k} = \frac{1}{\sqrt{b_p}}\sum_{j\in\Theta_k\setminus\theta_k} x_{i,j}^*.$$
% as well as the shift term $\tilde d_i=\sum_{k=1}^{r_p} d_i d_{i,j}^*/\sqrt{b_p}$.  
Clearly, the sequence $\{ y_{i, k}\}$ generated from the chunked big blocks are independent over $ k$. We denote the maximum of the rescaled sum of $ y_{i, k}$ by $\tilde I_{\epsilon}$, that is,
$$\frac{bn}{\sqrt{p}}\tilde I_{\epsilon}=\max_{bn+1\le i\le n-bn}\Big|\frac{1}{\sqrt{r_p}}\sum_{ k=1}^{r_p}  y_{i, k}\Big|.$$
Also, set $\tilde I_g$ to be $\tilde I_{\epsilon}$ with $y_{i,k}$ therein replaced by $g_{i,k}$, where $g_{i,k}$ is centered Gaussian random variables and the covariance matrix of $g_k=(g_{bn+1,k},...,g_{n-bn,k})^{\top}$ is $\EE(y_ky_k^{\top})$, where $y_k=(y_{bn+1,k},...,y_{n-bn,k})^{\top}$. 

Now we are ready to proceed the proof. Let $r_p=\big(n^{4/q}p^{3(2\xi-1)/(2\xi+2)}\big)^{q(\xi+1)/c_{q,\xi}}$, and $l_p=(p^{3/2}/r_p)^{\frac{1}{1+\xi}}=\big(n^{-4}p^{3(q-1)}\big)^{1/c_{q,\xi}}$, where the constant $c_{q,\xi}=4q\xi+q-2\xi-2$. Recall that $\Q_n=I_{\epsilon}$ under the null. For any $\alpha>0$,
\begin{align}
    \label{eq_thm31_oneside}
    & \sup_{u\in\mathbb{R}}\Big[\PP\big(\Q_n\le u\big)-\PP\big(\max_{bn+1\le i\le n-bn}\Z'_i\le u\big)\Big] \nonumber \\
    \le & \PP\Big(bnp^{-1/2}|I_{\epsilon}-\tilde I_{\epsilon}|\ge \alpha \Big) + \sup_{u\in \mathbb{R}}\big\vert\PP(\tilde I_{\epsilon}\le u)-\PP\big(\max_i\Z'_i\le u\big)\big\vert \nonumber \\
    &\quad + \sup_{u\in\RR}\PP\big(\big|\max_i\Z'_i-u\big|\le \alpha\big)
    =: \III_1'+\III_2'+\III_3'.
\end{align}
We shall investigate the parts $\III_1'$--$\III_3'$ separately. We define $\alpha=c_0\big(n^{2\xi}p^{(2-q)(2\xi-1)/4}\big)^{1/c_{q,\xi}}$, where $c_0$ is some constant independent of $p$ and $n$. For the $\III_1'$ part, since $l_p=(p^{3/2}/r_p)^{\frac{1}{1+\xi}}$, we have $\sqrt{p}l_p^{-\xi}=r_pl_p/p$. Then, by Lemmas \ref{lemma_thm31_lp_dependence} and \ref{lemma_thm31_block}, we obtain
$$\III_1' \lesssim r_pl_pp^{-1}\log^2(pn)/\alpha =  \big(n^{2\xi}p^{(2-q)(2\xi-1)/4}\big)^{1/c_{q,\xi}}\log^2(pn).$$
For the $\III_2'$ part, note that
$$\III_2'\le \sup_{u\in \mathbb{R}}\big\vert\PP(\tilde I_{\epsilon}\le u)-\PP(\tilde I_g\le u)\big\vert + \sup_{u\in \mathbb{R}}\big\vert\PP(\tilde I_g\le u)-\PP\big(\max_{bn+1\le i\le n-bn}\bar{z}_i\le u\big)\big\vert =: \III_{21}'+\III_{22}'.$$
It follows from Lemma \ref{lemma_thm31_GS} that
\begin{align}
    \III_{21}' & \lesssim \Bigg(\frac{n^{4/q}\log^7(r_pn)}{r_p} \Bigg)^{1/6} + \Bigg(\frac{n^{4/q}\log^3(r_pn)}{r_p^{1-2/q}}\Bigg)^{1/3} \nonumber\\
    & \lesssim \Big(\frac{n^{2\xi-4(\xi+1)/(3q)}}{p^{q(2\xi-1)/4}}\Big)^{1/c_{q,\xi}}\log^{7/6}(pn) + \Big(\frac{n^{4\xi}}{p^{(2\xi-1)(q-2)/2}}\Big)^{1/c_{q,\xi}}\log(pn). \nonumber 
\end{align}
By Lemma \ref{lemma_chen2019} and the similar arguments in Lemmas \ref{lemma_thm31_lp_dependence} and \ref{lemma_thm31_block}, we have 
\begin{align*}
    \III_{22}' & \le \sup_{u\in \mathbb{R}}\big\vert\PP(\tilde I_g\le u)-\PP\big(\max_i\bar{z}_i\le u\big)\big\vert + \sup_{u\in \mathbb{R}}\big\vert\PP\big(\max_i\bar{z}_i\le u\big) - \PP\big(\max_i\Z'_i\le u\big)\big\vert  \\
    & \lesssim \big(n^{4\xi/3}p^{(2-q)(2\xi-1)/6}\big)^{1/c_{q,\xi}}\log^{2/3}(n) + (bn)^{-1/3}\log^{2/3}(n).
\end{align*}
Concerning the $\III_3'$ part, as a direct consequence of Lemma \ref{lemma_nazarov},
$$\III_3'\lesssim \alpha\sqrt{\log(pn)} = \big(n^{2\xi}p^{(2-q)(2\xi-1)/4}\big)^{1/c_{q,\xi}}\sqrt{\log(pn}).$$
By combining the parts $\III_1'$--$\III_3'$ and a similar argument for the other side of the inequality in expression (\ref{eq_thm31_oneside}), we achieve the desired result.
\end{proof}

\begin{lemma}[$l_p$-dependent approximation]
    \label{lemma_thm31_lp_dependence}
    Assume that conditions in Proposition \ref{thm3_constanttrend} hold. For some small block size $l_p=o(p)$ and constant $\xi>1$, we have
    \begin{equation}
    \label{eq_thm31_step1_goal}
    \frac{bn}{\sqrt{p}}\big|I_{\epsilon}^* -  I_{\epsilon}\big|=O_{\PP}\big(\sqrt{p}l_p^{-\xi}\log(np)\big).
\end{equation}
\end{lemma}

\begin{proof}
Recall that we have defined $U_{i,j}^{(l)}$ and $U_{i,j}^{(r)}$ in expression (\ref{eq_U}).
We further define $U_{i,j}^{(l)*}$ (resp. $U_{i,j}^{(r)*}$) as $U_{i,j}^{(l)}$ (resp. $U_{i,j}^{(r)}$) with $\epsilon_{t,j}$ therein replaced by $\epsilon_{t,j}^*$. The definitions of $x_{i,j}$ and $x_{i,j}^*$ lead to
\begin{align}
    \label{eq_thm31_step1_twoparts}
    & \max_{bn+1\le i\le n-bn}\frac{1}{\sqrt{p}}\Big|\sum_{j=1}^p x_{i,j} - \sum_{j=1}^p x_{i,j}^*\Big| \nonumber \\
    \le &  \frac{1}{\sqrt{p}}\sum_{j=1}^p\max_{bn+1\le i\le n-bn}| x_{i,j} - x_{i,j}^*| \nonumber \\
    \le & \frac{1}{\sqrt{p}}\sum_{j=1}^p\max_{bn+1\le i\le n-bn}\Big(bn\sigma_j^{-2}\big|(U_{i,j}^{(l)}-U_{i,j}^{(r)})^2-(U_{i,j}^{(l)*}-U_{i,j}^{(r)*})^2\big| \nonumber \\
    & + bn\sigma_j^{-2}\EE\big|(U_{i,j}^{(l)}-U_{i,j}^{(r)})^2-(U_{i,j}^{(l)*}-U_{i,j}^{(r)*})^2\big|\Big)
    =:\III_1+\III_2.
\end{align}
We shall study the two parts $\III_1$ and $\III_2$ respectively. First, we aim to apply the Gaussian approximation theorem in \textcite{chernozhukov_central_2017} to the part $\III_1$. We shall verify the conditions therein.
It follows from expression (\ref{eq_epsilon_linear}) and Lemma \ref{lemma_snorm} (ii) that
\begin{align}
    \label{eq_thm31_step1_part1}
     \Big\lVert (U_{i,j}^{(l)}-U_{i,j}^{(l)*})/\sigma_j\Big\rVert_2 
    = &(bn)^{-1}\Big\lVert\sum_{l\le i-1}\sum_{t=(i-bn)\vee l}^{i-1}(A_{t-l,j,\cdot}^{\top}-A_{t-l,j,\cdot}^{*\top}) \eta_l/\sigma_j\Big\rVert_2  \nonumber \\
    \gtrsim & (bn)^{-1}\Big(\sum_{l\le i-1}\Big|\sum_{t=(i-bn)\vee l}^{i-1}(A_{t-l,j,\cdot}-A_{t-l,j,\cdot}^*)/\sigma_j\Big|_2^2\Big)^{1/2},
\end{align}
where the constants in $\lesssim$ here and the rest of the proof are independent of $n,p$ and $b$, so are the ones in $\gtrsim$ and $O(\cdot)$. Recall that the errors $\{\epsilon_{t,j}\}_t$ and $\{\epsilon_{t,j}^*\}_t$ both have an algebraic decay rate of the temporal dependence, which yields
% \begin{align}
%     \label{eq_thm31_tail}
%     & \sum_{l\le i-1}\Big|\sum_{t=(i-bn)\vee l}^{i-1}(A_{t-l,j,\cdot}-A_{t-l,j,\cdot}^*)/\sigma_j\Big|_2^2 \le \Big(\sum_{l\le i-1}\Big|\sum_{t=(i-bn)\vee l}^{i-1}(A_{t-l,j,\cdot}-A_{t-l,j,\cdot}^*)/\sigma_j\Big|_2\Big)\nonumber \\
%     \cdot & \max_{l\le i-1}\Big|\sum_{t=(i-bn)\vee l}^{i-1}(A_{t-l,j,\cdot}-A_{t-l,j,\cdot}^*)/\sigma_j\Big|_2 \le bn\Big|(\tilde A_{0,j,\cdot} - \tilde A_{0,j,\cdot}^*)/\sigma_j\Big|_2^2\big(1+o(1)\big).
% \end{align}
% \textcolor{red}{New:}
\begin{align}
\label{eq_thm31_tail}
    & \sum_{l\le i-1}\Big|\sum_{t=(i-bn)\vee l}^{i-1}(A_{t-l,j,\cdot}-A_{t-l,j,\cdot}^*)/\sigma_j\Big|_2^2 = \sum_{l=i-bn}^{i-1}\Big|\sum_{t=l}^{i-1}(A_{t-l,j,\cdot}-A_{t-l,j,\cdot}^*)/\sigma_j\Big|_2^2 \nonumber \\
    & \quad + \sum_{l\le i-bn-1}\Big|\sum_{t=i-bn}^{i-1}(A_{t-l,j,\cdot}-A_{t-l,j,\cdot}^*)/\sigma_j\Big|_2^2 =:  \tilde \III_1+ \tilde \III_2.
\end{align}
For the part $\tilde \III_1$, by Assumption \ref{asm_temp_dep}, it follows that
\begin{align}
    \tilde \III_1 & = \sum_{l=i-bn}^{i-1}\Big|(\tilde A_{0,j,\cdot}-\tilde A_{0,j,\cdot}^*)/\sigma_j- \sum_{t=i}^{\infty}(A_{t-l,j,\cdot}-A_{t-l,j,\cdot}^*)/\sigma_j\Big|_2^2 \nonumber \\
    & = bn\big|(\tilde A_{0,j,\cdot}-\tilde A_{0,j,\cdot}^*)/\sigma_j\big|_2^2 +o(bn).
\end{align}
By the same arguments, we have $\tilde \III_2 = o(bn)$. By inserting the results of parts $\tilde\III_1$ and $\tilde \III_2$ into expression (\ref{eq_thm31_step1_part1}), we can obtain
\begin{align}
    \label{eq_thm31_step1_part1boundpre}
    &\big\lVert (U_{i,j}^{(l)}-U_{i,j}^{(l)*})/\sigma_j\big\rVert_2
    \gtrsim (bn)^{-1/2}\big|(\tilde A_{0,j,\cdot} - \tilde A_{0,j,\cdot}^*)/\sigma_j\big|_2,
\end{align}
which along with Assumption \ref{asm_sec_dep} leads to
\begin{align}
    \label{eq_thm31_step1_part1bound}
    &\big\lVert \sqrt{bn}(U_{i,j}^{(l)}-U_{i,j}^{(l)*})/\sigma_j\big\rVert_2
    \gtrsim l_p^{-\xi}.
\end{align}
Similarly, by Lemma \ref{lemma_snorm} (i), we can derive the upper bound for $|U_{i,j}^{(l)}-U_{i,j}^{(l)*}|/\sigma_j$, that is, for $s\ge4$,
\begin{align}
    \label{eq_thm31_lpupper}
    \big\lVert \sqrt{bn}(U_{i,j}^{(l)}-U_{i,j}^{(l)*})/\sigma_j\big\rVert_{s/2}\lesssim l_p^{-\xi}.
\end{align}
Therefore, we can apply the Gaussian approximation theorem in Proposition 2.1 by \textcite{chernozhukov_central_2017}, which enables us to approximate the limit distribution of $\max_{i,j}\sqrt{bn}|(U_{i,j}^{(l)}-U_{i,j}^{(l)*})/\sigma_j|$ by the maximum coordinate of a centered Gaussian vector in $\RR^{p(n-2bn)}$. Consequently, we obtain that, with probability tending to 1,
% By Gaussian approximation, we have
% \begin{align}
%     \label{eq_thm31_lpGA_left}
%     \frac{1}{\sqrt{p}}\sum_{i=bn+1}^{n-bn}\sum_{j=1}^p\PP\Big(\sqrt{bn}\sigma_j^{-1}\big|(U_{i,j}^{(l)}-U_{i,j}^{(l)*})\big|\ge (u/\sqrt{p})^{1/2}\Big) \lesssim n\sqrt{p}(2\pi)^{-1/2}\tilde\sigma(u/\sqrt{p})^{-1/2}e^{-u/(2\tilde\sigma^2\sqrt{p})},
% \end{align}
% where $\tilde\sigma \rightarrow l_p^{-\xi}$, which yields
% we could approximate the limit distribution of $\max_{i,j}|(U_{i,j}^{(l)}-U_{i,j}^{(l)*})/\sigma_j|$ by the one of a centered Gaussian vector $\mathcal{Z}^*\in\RR^{p(n-2bn)}$. Consequently,
\begin{align}
    \label{eq_thm31_lpGA_bound}
    \max_{\substack{bn+1\le i\le n-bn\\1\le j\le p}}\sqrt{bn}\big|(U_{i,j}^{(l)}-U_{i,j}^{(l)*})/\sigma_j\big| \lesssim l_p^{-\xi}\log(np).
\end{align}
Likewise, we can achieve, with probability tending to 1, $\max_{i,j}\sqrt{bn}\big|(U_{i,j}^{(l)}+U_{i,j}^{(l)*})/\sigma_j\big|\lesssim \log(np)$, which together with the similar arguments for the parts $|(U_{i,j}^{(r)*}-U_{i,j}^{(r)})/\sigma_j|$ and $|(U_{i,j}^{(r)}+U_{i,j}^{(r)*})/\sigma_j|$ yields
\begin{align}
    \label{eq_thm31_lpGAresult}
    &\max_{\substack{bn+1\le i\le n-bn\\1\le j\le p}}bn\sigma_j^{-2}\big|(U_{i,j}^{(l)}-U_{i,j}^{(r)})+(U_{i,j}^{(l)*}-U_{i,j}^{(r)*})\big|\nonumber \\
    & \qquad \qquad \cdot\big|(U_{i,j}^{(l)}-U_{i,j}^{(r)})-(U_{i,j}^{(l)*}-U_{i,j}^{(r)*})\big| \lesssim l_p^{-\xi}\log(np),
\end{align}
with probability tending to 1. For the part $\III_2$, it follows from Cauchy-Schwarz inequality and the similar argument in expression (\ref{eq_thm31_lpupper}) that
\begin{align}
    \label{eq_thm31_lpGAcenter}
    &\max_{\substack{bn+1\le i\le n-bn\\1\le j\le p}}bn\sigma_j^{-2}\EE\big|(U_{i,j}^{(l)}-U_{i,j}^{(r)})+(U_{i,j}^{(l)*}-U_{i,j}^{(r)*})\big|\nonumber \\
    & \qquad \qquad \cdot\big|(U_{i,j}^{(l)}-U_{i,j}^{(r)})-(U_{i,j}^{(l)*}-U_{i,j}^{(r)*})\big| \lesssim l_p^{-\xi}.
\end{align}
This, along with expressions (\ref{eq_thm31_step1_twoparts}) and (\ref{eq_thm31_lpGAresult}) implies that
\begin{align}
    \max_{bn+1\le i\le n-bn}\frac{1}{\sqrt{p}}\Big|\sum_{j=1}^p x_{i,j} - \sum_{j=1}^p x_{i,j}^*\Big|=O_{\PP}\big(\sqrt{p}l_p^{-\xi}\log(np)\big).
\end{align}
The desired result is achieved.
\end{proof}

\begin{lemma}[Block approximation]
    \label{lemma_thm31_block}
    Assume conditions in Proposition \ref{thm3_constanttrend}. For some big block size $b_p=o(p)$ and small block size $l_p=o(b_p)$, we have
    \begin{equation}
    \label{eq_thm31_step2_goal}
    \frac{bn}{\sqrt{p}}\big| \tilde I_{\epsilon} - I_{\epsilon}^*\big|=O_{\PP}\big(l_pb_p^{-1}\log(n)\big).
\end{equation}
\end{lemma}

\begin{proof}
% In this lemma, we aim to prove that $\tilde I_{\epsilon}$ is a good approximation to $ I_{\epsilon}^*$ by applying the technique Bernstein's blocks. To this end, we study the loss resulted from this approximation, which is the sum of small blocks with coefficient matrices $A_k^*$. 
Recall that the number of big blocks is denoted by $r_p=\lfloor p/b_p\rfloor$. First, we shall note that
\begin{align}
    \label{eq_thm31_blockGA}
    \max_{bn+1\le i\le n-bn}\Big|\frac{1}{\sqrt{r_p}}\sum_{ k=1}^{r_p}\frac{1}{\sqrt{b_p}}\sum_{j\in\theta_k}x_{i,j}^*\Big|\le \sum_{i=bn+1}^{n-bn}\Big|\frac{1}{\sqrt{r_p}}\sum_{ k=1}^{r_p}\frac{1}{\sqrt{b_p}}\sum_{j\in\theta_k}x_{i,j}^*\Big|,
\end{align}
and we aim to apply the Gaussian approximation theorem to $(r_pb_p)^{-1/2}\sum_{k=1}^{r_p}\sum_{j\in\theta_k}x_{i,j}^*$.
Since the chunked small blocks are independent, it follows that,
\begin{align}
    \label{eq_thm31_step2_bound}
    & \EE\Big(\frac{1}{\sqrt{r_p}}\sum_{ k=1}^{r_p}\frac{1}{\sqrt{b_p}}\sum_{j\in\theta_k}x_{i,j}^*\Big)^2 = \frac{1}{r_pb_p}\sum_{ k=1}^{r_p}\EE\Big(\sum_{j\in\theta_k}x_{i,j}^*\Big)^2.
\end{align}
We can rewrite $x_{i,j}^*$ into
\begin{align}
    \label{eq_thm31_cov}
    & x_{i,j}^* = \E_0\Big[\Big(\sum_{l\le i+bn}G_{i,l,j,\cdot}^{\top}\eta_l\Big)^2\Big],
\end{align}
where
\begin{equation}
    \label{eq_thm31_defG}
    G_{i,l,j,\cdot}=
    \begin{cases}
        (bn)^{-1/2}\sigma_j^{-1}\Big(\sum_{t=(i-bn)\vee l}^{i-1}A_{t-l,j,\cdot}^* - \sum_{t=i\vee l}^{i+bn-1}A_{t-l,j,\cdot}^*\Big), & \,\text{if } l\le i-1,\\
        (bn)^{-1/2}\sigma_j^{-1}\sum_{t=i\vee l}^{i+bn-1}A_{t-l,j,\cdot}^*, & \, \text{if } i\le l\le i+bn-1.
    \end{cases}
\end{equation}
Recall the operator $\E_0(\cdot)=\cdot -\EE(\cdot)$. Note that $\{\eta_l\}$ are i.i.d. by Assumption \ref{asm_finitemoment}. In view of expressions (\ref{eq_thm31_tail}) and (\ref{eq_thm31_cov}), it follows from Lemmas \ref{lemma_burkholder} and \ref{lemma_snorm} (ii) that
\begin{align}
    \label{eq_thm31_cov_part1}
    \EE\Big(\sum_{j\in\theta_k}x_{i,j}^*\Big)^2 & = \Big\lVert\sum_{j\in\theta_k}\sum_{l,r\le i+bn}G_{i,l,j,\cdot}^{\top}\E_0[\eta_l\eta_r^{\top}]G_{i,r,j,\cdot}\Big\rVert_2^2 \nonumber \\
    & = \sum_{l,r\le i+bn}\Big\lVert\sum_{j\in\theta_k}G_{i,l,j,\cdot}^{\top}\E_0[\eta_l\eta_r^{\top}]G_{i,r,j,\cdot}\Big\rVert_2^2 \nonumber \\
    & \gtrsim \sum_{j_1,j_2\in\theta_k}\big(\tilde A_{0,j_2,\cdot}^{*\top}\tilde A_{0,j_1,\cdot}^*\big)^2/(\sigma_{j_1}\sigma_{j_2})^2.
\end{align}
which along with expression (\ref{eq_thm31_step2_bound}) and Assumption \ref{asm_sec_dep} yields,
\begin{align}
    \label{eq_thm31_l2_bound}
    & \EE\Big(\frac{1}{\sqrt{r_p}}\sum_{ k=1}^{r_p}\frac{1}{\sqrt{b_p}}\sum_{j\in\theta_k}x_{i,j}^*\Big)^2 
    \gtrsim l_p/b_p.
\end{align}
Therefore, we can apply the Gaussian approximation in \textcite{chernozhukov_central_2017} to approximate the limiting distribution of $\max_i|(r_pb_p)^{-1/2}\sum_k\sum_jx_{i,j}^*|$ by the maximum of a centered Gaussian vector in $\RR^n$. Then, by expression (\ref{eq_thm31_blockGA}), we achieve that, with probability tending to 1,
\begin{align}
    \max_{bn+1\le i\le n-bn}\Big|\frac{1}{\sqrt{r_p}}\sum_{ k=1}^{r_p}\frac{1}{\sqrt{b_p}}\sum_{j\in\theta_k}x_{i,j}^*\Big| \lesssim  l_pb_p^{-1}\log(n),
\end{align}
% $$\max_{bn+1\le i\le n-bn}\Big|\frac{1}{\sqrt{r_p}}\sum_{ k=1}^{r_p}\frac{1}{\sqrt{b_p}}\sum_{j\in\theta_k}x_{i,j}^*\Big| =O_{\PP}\big(l_pb_p^{-1}n\big),$$
which completes the proof.
\end{proof}

\begin{lemma}[Gaussian approximation on independent blocks]
    \label{lemma_thm31_GS}
    Assume that conditions in Proposition \ref{thm3_constanttrend} hold. Recall the number of blocks $r_p=\lfloor p/b_p\rfloor$, where the big block size $b_p=o(p)$. Then, as $r_p\rightarrow\infty$, we have
    \begin{equation}
        \label{eq_thm31_step3_goal}
    \sup_{u\in \mathbb{R}}\big\vert\PP(\tilde I_{\epsilon}\le u)-\PP(\tilde I_g\le u)\big\vert\lesssim \Bigg(\frac{n^{4/q}\log^7(r_pn)}{r_p} \Bigg)^{1/6} + \Bigg(\frac{n^{4/q}\log^3(r_pn)}{r_p^{1-2/q}}\Bigg)^{1/3}.
    \end{equation}
\end{lemma}

\begin{proof}
We shall apply the Gaussian approximation theorem to $\tilde I_{\epsilon}$ that consists of $r_p$ independent big blocks $ y_{i, k}$. 
Note that in each big block $ y_{i, k}$ in $\tilde I_{\epsilon}$, there are $b_p-l_p$ many $x_{i,j}^*$'s. We further divide them into $M=\lceil (b_p-l_p)/l_p\rceil$ groups. For the $ k$-th big block, define the index set of the $m$-th group as
\begin{equation}
    \G_{m, k}:=( k-1)b_p+\Big\{ml_p+1,ml_p+2,...,(m+1)l_p\Big\}, 
\end{equation}
for all $m=1,...,M, \,  k=1,...,r_p$, and $\cup_{1\le m\le M}\G_{m,k}=\Theta_k\setminus\theta_k$. All the groups with odd $m$ are independent, so are the ones with even $m$. The grouping step illustrated above entails another view of $ y_{i, k}$ that
\begin{align}
    & y_{i, k} =  \frac{1}{\sqrt{b_p}}\sum_{m=1}^{M}\Big[\sum_{j\in\G_{m\text{(odd)}, k}} x_{i,j}^* +\sum_{j\in\G_{m\text{(even)}, k}}x_{i,j}^* \Big],
\end{align}
where $\G_{m\text{(odd)}, k}$ (resp. $\G_{m\text{(even)}, k}$) denotes $\G_{m, k}$ with an odd $m$ (resp. even $m$). We shall study the odd groups first, and then the even groups can be investigated along the same lines. For all $bn+1\le i\le n-bn$, $1\le  k\le r_p$ and some constant $s=6,8,$ we consider the upper bound of $ r_p^{-1}\sum_{ k=1}^{r_p}\EE|y_{i, k}|^{s/2}$. It follows from Lemma \ref{lemma_burkholder} that
\begin{align}
    \label{eq_thm31_m2_bound1}
    & \Big\lVert\sum_{m=1}^{M}\sum_{j\in\G_{m\text{(odd)}, k}} x_{i,j}^*\Big\rVert_{s/2} \lesssim \Big(\sum_{m=1}^{M}\Big\lVert\sum_{j\in\G_{m\text{(odd)}, k}}x_{i,j}^*\Big\rVert_{s/2}^2\Big)^{1/2}.
\end{align}
By applying Lemma \ref{lemma_burkholder} again and accounting for the fact that the sequence $\{\eta_l\}$ are i.i.d., one can show that
\begin{align}
    \label{eq_thm31_m2step1}
    & \Big\lVert\sum_{j\in\G_{m\text{(odd)}, k}} x_{i,j}^*\Big\rVert_{s/2}^2
    \lesssim \sum_{l,r\le i+bn-1}\Big\lVert\sum_{j\in\G_{m\text{(odd)}, k}}G_{i,l,j,\cdot}^{\top}\E_0[\eta_l\eta_r^{\top}]G_{i,r,j,\cdot}\Big\rVert_{s/2}^2. 
\end{align}
It follows by expression (\ref{eq_thm31_tail}) and Lemma \ref{lemma_snorm} (i) that
\begin{equation}
    \label{eq_thm31_snorm}
    \sum_{l,r\le i+bn-1}\Big\lVert\sum_{j\in\G_{m\text{(odd)}, k}}G_{i,l,j,\cdot}^{\top}\E_0[\eta_l\eta_r^{\top}]G_{i,r,j,\cdot}\Big\rVert_{s/2}^2 \lesssim \sum_{j_1,j_2\in\G_{m\text{(odd)},k}}\big(\tilde A_{0,j_2,\cdot}^{*\top}\tilde A_{0,j_1,\cdot}^*\big)^2/\big(\sigma_{j_1}\sigma_{j_2}\big)^2,
\end{equation}
which along with expressions (\ref{eq_thm31_m2_bound1}) and (\ref{eq_thm31_m2step1}) yields
\begin{align}
    \label{eq_thm31_m2step2}
    \Big\lVert\sum_{m=1}^{M}\sum_{j\in\G_{m\text{(odd)}, k}} x_{i,j}^*\Big\rVert_{s/2} \lesssim \Big(\sum_{m=1}^{M}\sum_{j_1,j_2\in\G_{m\text{(odd)},k}}\big(\tilde A_{0,j_2,\cdot}^{*\top}\tilde A_{0,j_1,\cdot}^*\big)^2/\big(\sigma_{j_1}^2\sigma_{j_2}^2\big) \Big)^{1/2}.
\end{align}
This together with a similar argument for the even groups and Assumption \ref{asm_sec_dep} lead to
\begin{align}
    \label{eq_thm31_m2result}
    r_p^{-1}\sum_{ k=1}^{r_p}\EE|y_{i, k}|^{s/2} \lesssim & \frac{1}{r_p}\sum_{ k=1}^{r_p}\frac{1}{b_p^{s/4}}\Big(\Big\lVert\sum_{m=1}^{M}\sum_{j\in\G_{m\text{(odd)}, k}} x_{i,j}^*\Big\rVert_{s/2}^{s/2} + \Big\lVert\sum_{m=1}^{M}\sum_{j\in\G_{m\text{(even)}, k}} x_{i,j}^*\Big\rVert_{s/2}^{s/2}\Big) \nonumber \\
    \lesssim & \frac{1}{r_p}\sum_{ k=1}^{r_p}\frac{2}{b_p^{s/4}}\Big(\sum_{m=1}^{M}\sum_{j_1,j_2\in\G_{m\text{(odd)},k}}\big(\tilde A_{0,j_2,\cdot}^{*\top}\tilde A_{0,j_1,\cdot}^*\big)^2/\big(\sigma_{j_1}^2\sigma_{j_2}^2\big) \Big)^{s/4} = O(1).
\end{align}
Similarly, we have 
\begin{equation}
    \label{eq_thm31_e2result}
    \EE\big(\max_{bn+1\le i\le n-bn}|y_{i, k}|^q\big) \le \sum_{i=bn+1}^{n-bn}\EE |y_{i, k}|^q \lesssim n.
\end{equation}
Now, we shall prove the lower bound of $r_p^{-1}\sum_{ k=1}^{r_p}\EE y_{i, k}^2$ away from zero. It follows by Lemma \ref{lemma_snorm} (ii) as well as the definition of $ y_{i, k}$ that, for each $bn+1\le i\le n-bn$ and $ k=1,...,r_p$,
\begin{align}
    \label{eq_thm31_m1expression}
    \EE y_{i, k}^2 \ge &  4\sum_{l\le i-1}\sum_{r<l}\EE\Big(\sum_{j\in\Theta_k\setminus\theta_k}G_{i,l,j,\cdot}^{\top}\E_0[\eta_l\eta_r^{\top}]G_{i,r,j,\cdot}\Big)^2 \nonumber \\
    \gtrsim &  \frac{1}{b_p}\sum_{l\le i-1}\sum_{r<l}\Big\lVert\sum_{j\in\Theta_k\setminus\theta_k}G_{i,l,j,\cdot}G_{i,r,j,\cdot}^{\top}\Big\rVert_F^2.
\end{align}
Since
\begin{equation}
    \sum_{l\le i-1}\sum_{r<l}\Big\lVert\sum_{j\in\Theta_k\setminus\theta_k}G_{i,l,j,\cdot}G_{i,r,j,\cdot}^{\top}\Big\rVert_F^2 = \Big\lVert\sum_{j\in\Theta_k\setminus\theta_k}\tilde A_{0,j,\cdot}^*\tilde A_{0,j,\cdot}^{*\top}/\sigma_j^2 \Big\rVert_F^2 -o(b_p),
\end{equation}
in view of expression (\ref{eq_thm31_m1expression}) and Lemma \ref{lemma_snorm} (ii), we get
\begin{align}
    \label{eq_thm31_m1result}
    & \frac{1}{r_p}\sum_{ k=1}^{r_p}\EE y_{i, k}^2 
    \gtrsim \frac{1}{r_pb_p}\sum_{ k=1}^{r_p}\sum_{j_1,j_2\in\Theta_k\setminus\theta_k}\big(\tilde A_{0,j_2,\cdot}^{*\top}\tilde A_{0,j_1,\cdot}^*\big)^2/(\sigma_{j_1}\sigma_{j_2})^2 \ge 1.
\end{align}

For any $r\le q$, in view of $M_r$ and $\tilde M_r$ defined in the proof of Theorem \ref{thm1_constanttrend}, expressions (\ref{eq_thm31_m2result}) and (\ref{eq_thm31_e2result}) imply that
$$B_n :=\max\Big\{M_3^3,M_4^2, \tilde{M}_{q/2} \Big\}\lesssim  n^{2/q}.$$
Then, Proposition 2.1 in \textcite{chernozhukov_central_2017} yields the desired result.
\end{proof}

\subsection{Covariance Matrices of Gaussian Processes}

Now we attach the detailed calculation of the covariance matrix for the Gaussian vectors $Z_j\in\RR^{n-2bn}$ in Theorems \ref{thm1_constanttrend} and \ref{thm3_constanttrend}, respectively. Note that to apply the Gaussian approximation, we only need to evaluate the covariance matrix of corresponding $X_j\in\RR^{n-2bn}$ derived from our test statistics $\Q_n$ and use this same covariance matrix to generate Gaussian random vector $Z_j$. For simplicity, we assume that the long-run covariance of the errors $(\epsilon_{t,j})_t$ is known. One can estimate this long-run covariance by applying the robust M-estimation method provided in Section \ref{subsec_longrun}.

\begin{lemma}[Covariance matrix of the independent summands in $\Q_n$]
    \label{lemma_cov_thm1}
    For $bn+1\le i\le n-bn$ and $1\le j\le p$, $x_{i,j}$ is defined in expression (\ref{eq_thm1_x}). If $x_{j_1}$ and $x_{j_2}$ are independent, $j_1\neq j_2$, then, for $c\ge 0$, we have 
    \begin{equation}
        (bn)^2\EE (x_{i,j}x_{i+cbn,j})= \begin{cases}
            18c^2-24c +8 +O(1/(bn)),\quad 0\le c <1, \\
             2c^2-8c +8+O(1/(bn)), \quad 1\le c <2, \\
            O(1/(bn)), \quad c \ge 2,
        \end{cases}
    \end{equation}
    where constant in $O(\cdot)$ is independent of $n,b,p,i,j.$
\end{lemma}

\begin{proof}
First, since the errors are cross-sectionally independent, it follows that the long-run covariance of $(\epsilon_{t,j})_t$ can be expressed as
\begin{align}
    \label{eq_cov1_longrun}
    \sigma_j^2=\sum_{h=-\infty}^{\infty}\EE(\epsilon_{t,j}\epsilon_{t+h,j}) = \sum_{h=-\infty}^{\infty}\sum_{k\ge0}A_{k,j,j}A_{k+h,j,j}=\Big(\sum_{k\ge0} A_{k,j,j}\Big)^2 =\tilde A_{0,j,j}^2,
\end{align}
where the last equality is a direct consequence of the definition of $\tilde A_{k,j,j}$. Recall that $x_{i,j}$'s can be rewritten in expression (\ref{eq_thm1_x_linear}) as the centered independent summands in $|V_i|_2^2$ under the null, which yields, for any $c\ge0$,
\begin{align}
    \label{eq_cov1_goal}
    (bn)^2\EE (x_{i,j}x_{i+cbn,j})
    & = (bn)^2\sigma_j^{-4}\EE\Bigg(\E_0\Big[\Big(\sum_{t=i-bn}^{i-1}\frac{\epsilon_{t,j}}{bn}-\sum_{t=i+1}^{i+bn}\frac{\epsilon_{t,j}}{bn}\Big)^2\Big] \nonumber \\
    & \quad \cdot \E_0\Big[\Big(\sum_{t=i-(1-c)bn}^{i-1+cbn}\frac{\epsilon_{t,j}}{bn}-\sum_{t=i+1+cbn}^{i+(1+c)bn}\frac{\epsilon_{t,j}}{bn}\Big)^2\Big]\Bigg).
\end{align}
Recall the definitions of $U_{i,j}^{(l)}$ and $U_{i,j}^{(r)}$ in expression (\ref{eq_U}). To evaluate expression (\ref{eq_cov1_goal}), we first study the expectation of $(U_{i,j}^{(l)}-U_{i,j}^{(r)})^2$. By Assumption \ref{asm_temp_dep} and $\EE\eta_{t,j}^2=1$, we have
\begin{equation}
    \label{eq_cov1_exp_part1}
    bn\EE (U_{i,j}^{(l)2})=(bn)^{-1}\sum_{k\le i-1}\Big(\sum_{t=(i-bn)\vee k}^{i-1}A_{t-k,j,j}\Big)^2 = \tilde A_{0,j,j}^2\big[1+O(1/(bn))\big],
\end{equation}
and $bn\EE (U_{i,j}^{(r)2})=\tilde A_{0,j,j}^2\big[1+O(1/(bn))\big]$. Similarly, we get
\begin{align}
    \label{eq_cov1_exp_part12}
    & bn\EE (U_{i,j}^{(l)}U_{i,j}^{(r)})= (bn)^{-1}\sum_{k\le i-1}\Big(\sum_{t=(i-bn)\vee k}^{i-1}A_{t-k,j,j}\Big)\Big(\sum_{t=i+1}^{i+bn}A_{t-k,j,j}\Big)=O(1/(bn)),
\end{align}
which together with expressions (\ref{eq_cov1_longrun}) and (\ref{eq_cov1_exp_part1}) gives
\begin{equation}
    \label{eq_cov1_exp}
    bn\EE [(U_{i,j}^{(l)}-U_{i,j}^{(r)})^2/\sigma_j^2] = 2+O(1/(bn)).
\end{equation}
Hence, when $c=0$, by expression (\ref{eq_cov1_exp}), we have the variance of $x_{i,j}$ to be
\begin{align*}
    (bn)^2\EE (x_{i,j}^2)
    &=\EE\Big\{\Big( bn\frac{(U_{i,j}^{(l)}-U_{i,j}^{(r)})^2-\EE(U_{i,j}^{(l)}-U_{i,j}^{(r)})^2 }{\sigma_j^2} \Big)^2\Big\}\nonumber\\
    &= \EE\Big\{\Big( bn\frac{(U_{i,j}^{(l)}-U_{i,j}^{(r)})^2}{\sigma_j^2}-2+O(1/(bn)) \Big)^2\Big\}\nonumber\\
    &= (bn)^2\EE [(U_{i,j}^{(l)}-U_{i,j}^{(r)})^4/\sigma_j^4] -4bn\EE[(U_{i,j}^{(l)}-U_{i,j}^{(r)})^2/\sigma_j^2]+4+O(1/(bn))\nonumber \\
    &= (bn)^2\EE [(U_{i,j}^{(l)}-U_{i,j}^{(r)})^4/\sigma_j^4] -4+O(1/(bn)),
\end{align*}
which along with the definition of $a_{i,l,j}$ in expression (\ref{eq_thm11_defa}) gives
\begin{align}
    \label{eq_cov1_var}
    (bn)^2\EE (x_{i,j}^2)
    &= \sigma_j^{-4}\EE\Big[\Big(\sum_{l\le i+bn} a_{i,l,j}\eta_{l,j}\Big)^4\Big] -4 +O(1/(bn))\nonumber\\ 
    &= \sigma_j^{-4}\sum_{l\le i+bn} a_{i,l,j}^4 + 6\sigma_j^{-4}\sum_{l_1\le i+bn}a_{i,l_1,j}^2\sum_{l_2<l_1} a_{i,l_2,j}^2 -4+O(1/(bn))\nonumber \\
    & =  6\sigma_j^{-4}\sum_{i-bn\le l_1\le i-1}a_{i,l_1,j}^2\sum_{i-bn\le l_2<l_1} a_{i,l_2,j}^2 \nonumber \\
    & \quad +6\sigma_j^{-4}\sum_{i+1\le l_1\le i+bn}a_{i,l_1,j}^2\sum_{i-bn\le l_2<l_1} a_{i,l_2,j}^2 -4 +O(1/(bn)) \nonumber \\
    & = 8+O(1/(bn)).
\end{align}
For $0<c\le1$, we have
\begin{align}
    \label{eq_cov1_0c1}
    & \quad (bn)^2\EE (x_{i,j}x_{i+cbn,j}) \nonumber \\
    & = \sigma_j^{-4}\EE\Big[\Big(\sum_{l\le i+bn} a_{i,l,j}\eta_{l,j}\Big)^2\Big(\sum_{l\le i+(1+c)bn} a_{i,l,j}\eta_{l,j}\Big)^2\Big] -4+O(1/(bn))\nonumber \\
    & = \sigma_j^{-4}\EE\Big[\Big(\sum_{l=i-bn}^{i-(1-c)bn} a_{i,l,j}\eta_{l,j} + \sum_{l=i-(1-c)bn+1}^{i} a_{i,l,j}\eta_{l,j} \nonumber \\
    & \quad + \sum_{l=i+1}^{i-1+cbn} a_{i,l,j}\eta_{l,j} + \sum_{l=i+1+cbn}^{i+bn} a_{i,l,j}\eta_{l,j} \Big)^2 \nonumber \\
    & \quad \cdot \Big(\sum_{l=i-(1-c)bn+1}^{i} a_{i,l,j}\eta_{l,j} + \sum_{l=i+1}^{i-1+cbn} a_{i,l,j}\eta_{l,j} \nonumber \\
    & \quad + \sum_{l=i+1+cbn}^{i+bn} a_{i,l,j}\eta_{l,j} + \sum_{l=i+bn+1}^{i+(1+c)bn} a_{i,l,j}\eta_{l,j} \Big)^2\Big] - 4+O(1/(bn)) \nonumber \\
    & =  18c^2-24c+8 + O(1/(bn)).
\end{align}
Similarly, for $1<c\le2$, it follows that
\begin{align}
    \label{eq_cov1_1c2}
    & (bn)^2\EE (x_{i,j}x_{i+cbn,j}) \nonumber \\
    = &  \sigma_j^{-4}\EE\Big[\Big(\sum_{l=i-bn}^{i-1} a_{i,l,j}\eta_{l,j} + \sum_{l=i}^{i+(c-1)bn-1} a_{i,l,j}\eta_{l,j} + \sum_{l=i+(c-1)bn}^{i+bn} a_{i,l,j}\eta_{l,j}  \Big)^2 \nonumber \\
    & \cdot \Big(\sum_{l=i+(c-1)bn}^{i+bn} a_{i,l,j}\eta_{l,j} + \sum_{l=i+bn+1}^{i-1+cbn} a_{i,l,j}\eta_{l,j} + \sum_{l=i+cbn}^{i+(c+1)bn} a_{i,l,j}\eta_{l,j} \Big)^2\Big] -4 +O(1/(bn)) \nonumber \\
    = & 2c^2-8c+8 + O(1/(bn)),
\end{align}
which along with a similar argument for $c>2$ and expression (\ref{eq_cov1_0c1}) completes the proof. 
\end{proof}

In cases with weak cross-sectional dependence, calculating the covariance matrix for $x_{i,j}$, as defined in (\ref{eq_thm1_x}), is considerably more complex. We present the calculation procedures only for the scenario where $\epsilon_t$ is linear, as defined in (\ref{eq_epsilon_linear}), in the lemma below. The nonlinear case can be addressed in a similar manner, and the same dominant terms as the linear case can be obtained. We choose not to include this part here for brevity.

\begin{lemma}[Covariance matrix of the dependent summands in $\Q_n$]
    \label{lemma_cov_thm3}
    For $bn+1\le i\le n-bn$ and $1\le j\le p$, $x_{i,j}$ is defined in expression (\ref{eq_thm1_x}). If the dependence between $x_{j}$ and $x_{j+h}$ satisfies Assumption \ref{asm_sec_dep}, then, for $c\ge 0$, we have 
    \begin{align}
        & (bn)^2\EE (x_{i,j}x_{i+cbn,j+h}) \nonumber \\
        &=\begin{cases}
            (15c^2-20c+8)(\tilde A_{0,j,\cdot}^{\top}\tilde A_{0,j+h,\cdot})^2/(\sigma_j^2\sigma_{j+h}^2) + 3c^2-4c +O(1/(bn)), & 0<c\le 1, \\
            (3c^2-12c+12)(\tilde A_{0,j,\cdot}^{\top}\tilde A_{0,j+h,\cdot})^2/(\sigma_j^2\sigma_{j+h}^2) -c^2+4c -4+O(1/(bn)), & 1< c\le 2,\\
            O(1/(bn)), & c>2. \nonumber 
        \end{cases}
    \end{align}
\end{lemma}

\begin{proof}
By the definition of $x_{i,j}$ in expression (\ref{eq_thm1_x}), we have 
\begin{align}
    \label{eq_cov3_goal}
    bnx_{i,j} = bn\sigma_j^{-2}\E_0\big[(U_{i,j}^{(l)} - U_{i,j}^{(r)})^2 \big]=\sigma_j^{-2}\Big(\sum_{l\le i+bn} B_{i,l,j,\cdot}^{\top}\eta_l\Big)^2 - \sigma_j^{-2}\EE\Big[\Big(\sum_{l\le i+bn} B_{i,l,j,\cdot}^{\top}\eta_l\Big)^2\Big],
\end{align}
where
\begin{equation}
    B_{i,l,j,\cdot}=
    \begin{cases}
        (bn)^{-1/2}\sum_{t=(i-bn)\vee l}^{i-1}A_{t-l,j,\cdot}-(bn)^{-1/2}\sum_{t=(i+1)\vee l}^{i+bn}A_{t-l,j,\cdot} , & \, \text{if } l\le i-1,\\
        (bn)^{-1/2}\sum_{t=(i+1)\vee l}^{i+bn}A_{t-l,j,\cdot}, & \, \text{if } i\le l\le i+bn.
    \end{cases}
\end{equation}
By applying the same argument as expression (\ref{eq_cov1_exp}) in Lemma \ref{lemma_cov_thm1}, we can get the second part in expression (\ref{eq_cov3_goal}), that is
\begin{equation}
    \label{eq_cov3_exp}
    \sigma_j^{-2}\EE\Big[\Big(\sum_{l\le i+bn} B_{i,l,j,\cdot}^{\top}\eta_l\Big)^2\Big] = 2+O(1/(bn)).
\end{equation}
For the covariance, the main idea of our proof is to extend the proof of the covariance for the temporal direction in Lemma \ref{lemma_cov_thm1} to the spatial direction. Thus, we first consider the case when $c=0$. Note that
\begin{align}
    \label{eq_cov3_var}
    & \quad (bn)^2\EE (x_{i,j}x_{i,j+h}) \nonumber \\
    & =  \sigma_j^{-2}\sigma_{j+h}^{-2}\EE \Big[\Big(\sum_{l\le i+bn} B_{i,l,j,\cdot}^{\top}\eta_l\Big)^2\Big(\sum_{r\le i+bn} B_{i,r,j+h,\cdot}^{\top}\eta_r\Big)^2\Big] -4+O(1/(bn)) \nonumber \\
    & =: \sigma_j^{-2}\sigma_{j+h}^{-2}\EE \Big[\Big(\sum_{l\le i+bn} P_l^{\top}\eta_l\Big)^2\Big(\sum_{r\le i+bn} Q_r^{\top}\eta_r\Big)^2\Big] -4+O(1/(bn)) \nonumber \\
    & = 4\sigma_j^{-2}\sigma_{j+h}^{-2}\Big(\sum_{l\le i+bn}\EE[(P_l^{\top}\eta_l)^2] \sum_{r<l}\EE[(Q_r^{\top}\eta_r)^2]  + 2\sum_{l\le i+bn}\sum_{r<l}\EE [P_l^{\top}\eta_l\eta_r^{\top}P_rQ_l^{\top}\eta_l\eta_r^{\top}Q_r]\Big) \nonumber \\
    & \quad -4+o(1) \nonumber \\
    & =: \sigma_j^{-2}\sigma_{j+h}^{-2}(\III_1^*+\III_2^*) -4+O(1/(bn)).
\end{align}
We shall investigate the two parts $\III_1^*$ and $\III_2^*$ separately. For the part $\III_1$, we have
\begin{align}
    \label{eq_cov3_var_part1}
    \III_1^* = &  4\sum_{l\le i+bn}\EE(P_l^{\top}\eta_l)^2 \sum_{r<l}\EE[(Q_r^{\top}\eta_r)^2] \nonumber \\
    = & 4\sum_{l\le i+bn} \text{tr}(P_lP_l^{\top}\EE[\eta_l\eta_l^{\top}]) \sum_{r<l}\text{tr}(Q_rQ_r^{\top}\EE[\eta_r\eta_r^{\top}]) \nonumber \\
    = & 4\sum_{l\le i+bn} P_l^{\top}P_l \sum_{r<l}Q_r^{\top}Q_r =4\sigma_j^2\sigma_{j+h}^2(1+O(1/(bn))).
\end{align}
For the part $\III_2$, it follows that
\begin{align}
    \label{eq_cov3_var_part2}
    \III_2^* & = 8\sum_{l\le i+bn}\sum_{r<l}\EE [P_l^{\top}\eta_l\eta_r^{\top}P_rQ_r^{\top}\eta_r\eta_l^{\top}Q_l]\nonumber \\
    & = 8\sum_{l\le i+bn} P_l^{\top}Q_l \sum_{r<l}P_r^{\top}Q_r = 8(\tilde A_{0,j,\cdot}^{\top}\tilde A_{0,j+h,\cdot})^2 +O(1/(bn)),
\end{align}
which together with expression (\ref{eq_cov3_var_part1}) yields
\begin{equation}
    \label{eq_cov3_var_result}
    (bn)^2\EE (x_{i,j}x_{i,j+h}) = 8(\tilde A_{0,j,\cdot}^{\top}\tilde A_{0,j+h,\cdot})^2/(\sigma_j^2\sigma_{j+h}^2) +O(1/(bn)).
\end{equation}
Now we are ready to combine the shifts in both temporal and spatial coordinates and evaluate the corresponding covariance. For $0<c\le 1$, we get
\begin{align}
    \label{eq_cov3_0c1}
    & \quad (bn)^2\EE (x_{i,j}x_{i+cbn,j+h}) \nonumber \\ 
    & =  \sigma_j^{-2}\sigma_{j+h}^{-2}\EE\Big[\Big(\sum_{l=i-bn}^{i-(1-c)bn} B_{i,l,j,\cdot}^{\top}\eta_l + \sum_{l=i-(1-c)bn+1}^{i} B_{i,l,j,\cdot}^{\top}\eta_l + \sum_{l=i+1}^{i-1+cbn} B_{i,l,j,\cdot}^{\top}\eta_l \nonumber \\
    & \quad + \sum_{l=i+1+cbn}^{i+bn} B_{i,l,j,\cdot}^{\top}\eta_l \Big)^2\Big(\sum_{r=i-(1-c)bn+1}^{i} B_{i,r,j+h,\cdot}^{\top}\eta_r + \sum_{r=i+1}^{i-1+cbn} B_{i,r,j+h,\cdot}^{\top}\eta_r \nonumber \\
    & \quad + \sum_{r=i+1+cbn}^{i+bn} B_{i,r,j+h,\cdot}^{\top}\eta_r 
    + \sum_{r=i+bn+1}^{i+(1+c)bn} B_{i,r,j+h,\cdot}^{\top}\eta_r \Big)^2\Big] -4 + O(1/(bn)) \nonumber \\
    & = (15c^2-20c+8)(\tilde A_{0,j,\cdot}^{\top}\tilde A_{0,j+h,\cdot})^2\sigma_j^{-2}\sigma_{j+h}^{-2} +3c^2-4c + O(1/(bn)).
\end{align}
Similarly, for $1< c\le 2$, we have
\begin{align}
    \label{eq_cov3_1c2}
    & \quad (bn)^2\EE (x_{i,j}x_{i+cbn,j+h}) \nonumber \\
    & = \sigma_j^{-2}\sigma_{j+h}^{-2}\EE\Big[\Big(\sum_{l=i-bn}^{i-1} B_{i,l,j,\cdot}^{\top}\eta_l + \sum_{l=i}^{i+(c-1)bn-1} B_{i,l,j,\cdot}^{\top}\eta_l + \sum_{l=i+(c-1)bn}^{i+bn} B_{i,l,j,\cdot}^{\top}\eta_l  \Big)^2 \nonumber \\
    & \quad \cdot \Big(\sum_{r=i+(c-1)bn}^{i+bn} B_{i,r,j+h,\cdot}^{\top}\eta_r + \sum_{r=i+bn+1}^{i-1+cbn} B_{i,r,j+h,\cdot}^{\top}\eta_r + \sum_{r=i+cbn}^{i+(c+1)bn} B_{i,r,j+h,\cdot}^{\top}\eta_r \Big)^2\Big] \nonumber \\
    & \quad -4+O(1/(bn)) \nonumber \\
    & = (3c^2-12c+12)(\tilde A_{0,j,\cdot}^{\top}\tilde A_{0,j+h,\cdot})^2\sigma_j^{-2}\sigma_{j+h}^{-2}-c^2+4c-4 +O(1/(bn)),
\end{align}
which together with expression (\ref{eq_cov3_0c1}) shows our desired result. When $h=0$, since $\sigma_j^2=\tilde A_{0,j,\cdot}^{\top}\tilde A_{0,j,\cdot}$, it follows that
$$(\tilde A_{0,j,\cdot}^{\top}\tilde A_{0,j+h,\cdot})^2 = \sigma_j^4,$$
which is a special case with cross-sectional independence, consistent with our result in Lemma \ref{lemma_cov_thm1}.
\end{proof}

\end{appendices}

\end{document}